\documentclass[leqno,a4paper,11pt]{article}

\usepackage{amsmath,amsthm,amssymb}
\usepackage[utf8]{inputenc}
\usepackage[T1]{fontenc}
\usepackage{comment}

\usepackage{amscd}

\usepackage{cancel}
\usepackage[normalem]{ulem}

\usepackage{enumerate}
\usepackage{bbm}
\usepackage{todonotes}
\usepackage{mathrsfs}
\usepackage{mathabx}
\usepackage{hyperref}
\usepackage{nicefrac}
\usepackage{tikz}
\usetikzlibrary{matrix}

\usepackage{url}
\makeatletter
\g@addto@macro{\UrlBreaks}{\UrlOrds}
\makeatother

\usepackage[sort,numbers]{natbib}

\providecommand{\noopsort}[1]{} 

\def\qed{\unskip\quad \hbox{\vrule\vbox
to 6pt {\hrule width 4pt\vfill\hrule}\vrule} }

\newcommand{\bez}{\nopagebreak\hspace*{\fill}
 \nolinebreak$\qed$\vspace{5mm}\par}

\newtheorem{Th}{Theorem}[section]
\newtheorem{Prop}[Th]{Proposition}

\newtheorem{Thx}{Theorem}

\newtheorem{Propx}[Thx]{Proposition}

\newtheorem{Remx}[Thx]{Remark}

\newtheorem{Lemma}[Th]{Lemma}
\newtheorem{Conj}{Conjecture}[section]
\newtheorem{Cor}[Th]{Corollary}
\newtheorem{Corx}[Thx]{Corollary}
\newtheorem{Fact}[Th]{Fact}

\theoremstyle{definition}
\newtheorem{Def}{Definition}[section]
\newtheorem{Remark}[Th]{Remark}

\newtheorem{Example}[Th]{Example}
\newtheorem{Question}[Conj]{Question}

\newcommand{\beq}{\begin{equation}}
\newcommand{\eeq}{\end{equation}}

\def\scalar(#1,#2){(#1\mid#2)}

\newcommand{\raz}{\mathbbm{1}}
\newcommand{\bax}{\overline{x}}
\newcommand{\baX}{\overline{X}}
\newcommand{\bamu}{\overline{\mu}}

\newcommand{\cs}{{\cal S}}
\newcommand{\ca}{{\cal A}}
\newcommand{\cb}{{\cal B}}
\newcommand{\cc}{{\cal C}}
\newcommand{\cd}{{\cal D}}
\newcommand{\ce}{{\cal E}}
\newcommand{\cf}{{\cal F}}

\newcommand{\cm}{{\cal M}}

\newcommand{\xbm}{(X,{\cal B},\mu)}

\newcommand{\ycn}{(Y,{\cal C},\nu)}
\newcommand{\zdk}{(Z,{\cal D},\kappa)}
\newcommand{\ot}{\otimes}
\newcommand{\ov}{\overline}
\newcommand{\la}{\lambda}

\newcommand{\bs}{\mathbb{S}}

\newcommand{\Q}{\mathbb{Q}}
\newcommand{\R}{{\mathbb{R}}}
\newcommand{\T}{{\mathbb{T}}}
\newcommand{\C}{{\mathbb{C}}}
\newcommand{\Z}{{\mathbb{Z}}}
\newcommand{\N}{{\mathbb{N}}}
\newcommand{\E}{{\mathbb{E}}}
\newcommand{\EE}{{\mathbb{E}}}
\newcommand{\PP}{{\mathbb{P}}}
\newcommand{\D}{{\mathbb{D}}}
\newcommand{\vep}{\varepsilon}
\newcommand{\va}{\varphi}

\newcommand{\mob}{\boldsymbol{\mu}}

\newcommand{\bfu}{\boldsymbol{u}}

\newcommand{\ck}{\mathcal{K}}

\theoremstyle{remark}

\newcommand{\bbN}{\mathbb{N}}

\newcommand{\bbZ}{\mathbb{Z}}

\renewcommand{\d}{\mathrm{d}}

\renewcommand{\a}{\alpha}
\renewcommand{\b}{\beta}

\newcommand{\ol}[1]{\overline{#1}}

\renewcommand{\t}[1]{\widetilde{#1}}

\begin{document}
\title{
Unveiling universality, encloseness, and orthogonality in dynamics}
\author{J.\ Aaronson, A.\,I.\ Danilenko\footnote{A.I.D. was  supported in part by the
``Long-term program of
support of the Ukrainian research teams at the
Polish Academy of Sciences carried out in
collaboration with the U.S. National Academy of
Sciences with the financial support of external
partners''.
}, J.\ Ku\l aga-Przymus,  M.\ Lema\'{n}czyk \vspace{10pt}\\  with an appendix ``Separable subsets for d-bar metric'' \\by T. Austin}

\maketitle

\begin{abstract}
Motivated by Sarnak's conjecture on M\"obius orthogonality, we investigate the general problem of orthogonality for a bounded sequence to topological models of characteristic classes of measure-preserving automorphisms. Our main observation is that whenever a strong form of such orthogonality holds in a system $(X,T)$ then the orthogonality holds for all topological systems in which each ergodic measure yields an automorphism  that is measure-theoretically isomorphic to one arising from an ergodic measure in $(X,T)$.
This leads us to study two purely dynamical problems. The first concerns  the  existence of universal topological models for characteristic classes of measure-preserving automorphisms. In particular, we show that the class of automorphisms with relative discrete spectrum over the identity factor--as well as several related classes including the weakly mixing case--admit such models. The second problem  we address is the existence of a common ergodic extension for a measurable family of ergodic automorphisms, with particular emphasis on the role of ergodic decomposition. We solve this problem by analysing separable subsets in $d$-bar distance. We also highlight potential applications to the orthogonality phenomena, especially those connected to Sarnak’s conjecture and its logarithmic variant.
In particular, we provide a purely ergodic proof of the equivalence—first established by T. Tao—between the local 1-Fourier uniformity of a sequence
$\bfu\colon\N\to\mathbb{D}$ and the vanishing of its second Gowers–Host–Kra seminorm, namely $\|\bfu\|_{u^2}=0$. Moreover, we show that if the set of all measure-theoretic eigenvalues of a zero entropy system $(X,T)$ is countable, then $(X,T)$ satisfies Sarnak's conjecture along a subsequence of full logarithmic density.

\end{abstract}

\tableofcontents


\section{Introduction}
The problems studied in this paper are motivated by the general form of Sarnak's conjecture on M\"obius orthogonality \cite{Sa} proposed in \cite{Ka-Ku-Le-Ru}: given a (bounded) sequence $\bfu \colon \N\to\D=\{z\in\C\colon |z|\leq1\}$ (whose mean $M(\bfu):=\lim_{N\to\infty}\frac1N\sum_{n<N}\bfu(n)$ exists and equals zero) and a topological system $(X,T)$,\footnote{That is, $T$ is a homeomorphism of a compact metric space $X$.} one says that $(X,T)$ {\em is $\bfu$-orthogonal} if
\begin{equation}\label{jsjm1}
\lim_{N\to\infty}\frac{1}{N}\sum_{n<N} f(T^n x)\bfu(n)=0
\end{equation}
for each $f\in C(X)$ and $x\in X$ (in symbols, $(X,T)\perp \bfu$).\footnote{Following Tao \cite{Ta00}, we also consider logarithmic orthogonality, in which Ces\`aro averages are replaced by logarithmic means, e.g.\ $(X,T)\perp_{\rm log}\bfu$ if
$\lim_{N\to\infty}\frac1{\log N}\sum_{n<N}\frac1nf(T^nx)\bfu(n)=0$
for all $f\in C(X)$ and $x\in X$ (and tacitly assuming that the logarithmic mean of $\bfu$ is zero). Other notions discussed later—such as the strong $\bfu$-MOMO property or Furstenberg systems (see Section~\ref{s:Sarnakc})—admit analogous logarithmic versions. As our results are based on the Lifting Lemma from \cite{Ka-Ku-Le-Ru}, which is also true in the logarithmic framework, all our results have logarithmic counterparts.} We can study this problem (for a fixed $\bfu$) with a single $(X,T)$ but usually we are interested
in a class $\mathcal{C}$ of topological systems and require $(X,T)\perp\bfu$ for all $(X,T)\in \mathcal{C}$ (in symbols, $\mathcal{C}\perp \bfu$). For example, the original Sarnak's conjecture deals with the problem $\mathcal{C}_{\rm {\bf{ZE}}}\perp\mob$, where $\mob \colon \N\to\{-1,0,1\}$ is the M\"obius function and $\mathcal{C}_{\rm {\bf{ZE}}}$ stands for the class
of topological systems with zero entropy. Realising that, by the variational principle, this class consists precisely of topological systems for which each invariant measure  yields a zero entropy measure-theoretic automorphism and that the class ${\rm {\bf ZE}}$ of zero entropy automorphisms is closed under taking joinings and factors, one can study a more general problem
\beq\label{generalS}
\mathcal{C}_{\cf}\perp\bfu\eeq
where $\cf$ is a {\em characteristic class} of automorphisms, i.e.\ a class closed under taking joinings and factors,\footnote{One of the reasons that we somehow ``forget'' about the multiplicative property of $\mob$ is that, since ${\rm {\bf{ZE}}}$ is characteristic, to show Sarnak's conjecture is equivalent to $\mathcal{C}_{\rm{\bf{ZE}}}\perp(\mob(n)\cdot v(n))_{n\geq1}$ for each deterministic sequence $v \colon \N\to\bs^1$ and the sequence $\mob\cdot v$ loses its multiplicative character, see \cite{Ka-Ku-Le-Ru}.} and $\mathcal{C}_{\cf}$ stands for the class of topological systems $(X,T)$ such that for each $T$-invariant measure $\nu\in M(X,T)$, the measure-theoretic automorphism $(X,\nu,T)$ belongs to $\cf$. To each characteristic class $\cf$ we can associate the class $\cf_{\rm ec}$ of automorphisms whose a.a.\ ergodic components belong to $\cf$. Then $\cf_{\rm ec}$ is also a characteristic class~\cite{Ka-Ku-Le-Ru} and  $(\cf_{\rm ec})_{\rm ec}=\cf_{\rm ec}$. Clearly, the larger (yet still proper) the characteristic class $\mathcal{F}$ is, the more difficult the orthogonality problem~\eqref{generalS} becomes. As proved in \cite{Ka-Ku-Le-Ru}, the class ${\rm {\bf{ZE}}}$  is the largest proper characteristic class, so the orthogonality problem culminates when we consider $\cf={\rm {\bf ZE}}$. For the classical multiplicative functions like M\"obius $\bfu=\mob$ and $\cf={\rm {\bf{ZE}}}$ the orthogonality $\cc_{\rm {\bf{ZE}}}\perp\mob$ is expected to hold (which is precisely Sarnak's conjecture \cite{Sa}), however, the orthogonality problem with $\bfu=\mob$ remains open even for much smaller  characteristic classes. In fact, while for the class ${\rm {\bf{DISP}}}$ of all automorphisms with discrete spectrum
\beq\label{disp17}\cc_{\rm {\bf DISP}}\perp \mob\eeq indeed holds~\cite{Fe-Ku-Le,Hu-Wa-Ye},\footnote{\label{f:disp} A spectral disjointness (see Section~\ref{f:spectral}) argument applies: from the averaged Chowla conjecture proved in \cite{Ma-Ra-Ta} we obtain the continuity of the spectral measure of function $\pi_0$ (see Section~\ref{s:Sarnakc}) for all Furstenberg systems of $\mob$ while on the side of $(X,T)$ all invariant measures ``produce'' discrete spectral measures for all $L^2$ functions.} it is still unknown (including the logarithmic case) whether $\cc_{{\rm {\bf DISP}}_{\rm ec}}\perp \mob$ \cite{Ka-Ku-Le-Ru} and, in the present paper, we try to throw some more light on the latter problem.

Let us return to the discussion for a general $\bfu$. Closely related to the orthogonality problems~\eqref{jsjm1} and~\eqref{generalS}, is the following notion: a topological system $(X,T)$ satisfies the {\em strong $\bfu$-MOMO property}\footnote{MOMO stands for ``mean orthogonality of moving orbits''.} \cite{Ab-Ku-Le-Ru} if for all increasing sequences $(b_k)$ with
\beq\label{zalbk}b_{k+1}-b_k\to\infty,\eeq for all $f\in C(X)$ and for each choice of $(x_k)\subset X$, we have
$$
\lim_{K\to\infty}\frac1{b_K}\sum_{k<K}\Big|\sum_{b_k\leq n<b_{k+1}}f(T^nx_k)\bfu(n)\Big|=0.
$$
Equivalently~\cite{Ka-Ku-Le-Ru},
\beq\label{jsjm2}
\lim_{K\to\infty}\frac1{b_K}\sum_{k<K}\Big\|\sum_{b_k\leq n<b_{k+1}}\bfu(n)\cdot f\circ T^n\Big\|_{C(X)}=0~\footnote{In the logarithmic case \eqref{jsjm2} is replaced by $\lim_{K\to\infty}\frac1{\log b_K}\sum_{k<K}\Big\|\sum_{b_k\leq n<b_{k+1}}\frac{\bfu(n)}n\cdot f\circ T^n\Big\|_{C(X)}=0$ for all $(b_k)$ satisfying~\eqref{zalbk} and $f\in C(X)$. Moreover, recall that in the crucial for us Lifting Lemma \cite{Ka-Ku-Le-Ru} the switch orbit set $\{b_k\colon k\geq1\}$ also satisfies \eqref{zalbk}.

It is not hard to see that if the family of sequences satisfying \eqref{zalbk} is extended to those for which the set $\{b_k\colon k\geq1\}$ has zero (natural) density or, for the logarithmic case, it has zero logarithmic density, we obtain exactly the same notions. Cf.\ also~\cite[Lemma 1.8]{Ka-Le-Ri-Te}.}
\eeq
for all $(b_k)$ and $f$ as above. Note that this implies convergence in \eqref{jsjm1} is uniform in $x\in X$.  The idea of strong $\bfu$-MOMO is to obtain $\bfu$-orthogonality not only for $(X,T)$ but also for many other topological systems which are ``close'' to $(X,T)$. By this we mean that, when compared to $(X,T)$, these systems may have some more orbits, however we control (using $(X,T)$) measurable dynamics coming from invariant measures which have quasi-generic points (see \cite{Ab-Le-Ru,Ab-Ku-Le-Ru,Ka-Ku-Le-Ru}). A particularly convenient instance of this, established in~\cite{Ka-Ku-Le-Ru}, is that, for each characteristic class $\cf$,
\beq\label{equivMOMO}
\mathcal{C}_{\cf_{\rm ec}}\perp\bfu \text{ iff each system in }\mathcal{C}_{\cf_{\rm ec}} \text{ enjoys the strong $\bfu$-MOMO property.}
\eeq
In particular, Sarnak's conjecture is equivalent to the strong $\mob$-MOMO property of all systems with zero entropy.\footnote{Nevertheless, even knowing that a zero entropy topological system is $\mob$-orthogonal to prove that it enjoys the strong $\mob$-MOMO property can be a considerable task. For example, the M\"obius orthogonality of horocycle flows has been known since \cite{Bo-Sa-Zi}, however, their strong  $\mob$-MOMO property remains an open problem. Moreover, the reader is warned that it is not true that whenever $(X,T)\perp \bfu$ then the system satisfies the strong $\bfu$-MOMO property. Indeed, in \cite{Do-Se}, it is shown that there are positive entropy systems orthogonal to $\mob$, but on the other hand, no positive entropy system can enjoy the strong $\bfu$-MOMO property (assuming the Chowla conjecture) \cite{Ka-Ku-Le-Ru}.}


The main motivation for the problems studied in the present paper is the following extension of earlier results on the stability of M\"obius orthogonality for uniquely ergodic models \cite{Ab-Ku-Le-Ru} of an ergodic automorphism and for systems with a countable set of ergodic measures \cite{Ka-Le-Ra}.

\begin{Thx}\label{theorem:a}
Assume that $(X,T)$ satisfies the strong $\bfu$-MOMO property. Let $(Z,R)$ be an arbitrary topological system such that for each $\nu\in M^e(Z,R)$ there is an ergodic measure $\nu'\in M^e(X,T)$ for which the measure-theoretic automorphisms $(Z,\nu,R)$ and $(X,\nu',T)$ are isomorphic. Then $(Z,R)\perp \bfu$.
The same result holds in the logarithmic case.
\end{Thx}

Theorem~\ref{theorem:a} suggests to study the problem of existence of universal models for a characteristic class (which is of independent interest in dynamics).
Given a class $\mathcal{F}$ of measure-preserving automorphisms\footnote{We consider automorphisms up to isomorphism.}, a {\em universal model for} $\mathcal{F}$ is a topological system $(X,T)$  such that:

\begin{enumerate}
\item[---] for every  $\nu\in M(X,T)$, the measure-theoretic system $(X,\nu,T)$ belongs to $\mathcal{F}$; and
\item[---]
 for every automorphism $R\in\mathcal{F}$ there exists a measure $\nu\in M(X,T)$ such that $(X,\nu,T)$ is isomorphic to $R$.
 \end{enumerate}

We recall that classically, the class ${\rm{\bf{ALL}}}$ of all automorphisms has a universal model (for example, the full shift with alphabet $[0,1]$),\footnote{Also the smallest non-trivial characteristic class ${\rm{\bf{ID}}}$ of all identities has a universal model, for example, the identity on $[0,1]$ is such.} while  Serafin's theorem \cite{Se} tells us that the class {\bf ZE} has no such a model. It is not hard to observe (see Proposition~\ref{p:noum})  that if a characteristic class $\cf$ has a universal model then it has to be an ec-class: $\cf=\cf_{\rm ec}$. Due to Theorem~\ref{theorem:a}, proving $\mathcal C_{\mathcal F}\perp \bfu$
is equivalent to showing the strong $\bfu$-MOMO property for one of the universal models of $\mathcal{F}$ (assuming there is a such). As already mentioned, the present challenge (on the way to prove Sarnak's conjecture in full generality) is to establish  $\cc_{{\rm{\bf{DISP}}}_{\rm ec}}\perp\mob$. A difficulty stems from the fact that for systems in $\cc_{{\rm{\bf{DISP}}}_{\rm ec}}$, non-ergodic measures may yield measure-theoretic automorphisms with continuous (possibly Lebesgue) spectrum in the orthocomplement of the $L^2$-space associated with the sigma-algebra of invariant sets, so the aforementioned spectral disjointness argument (see \eqref{disp17} and Footnote~\ref{f:disp}) does not work. In the light of Serafin's result, the following theorem looks rather surprising ($\T$ stands for the circle, we use both additive and multiplicative notations; which one is used should be clear from the context).

\begin{Thx}[see Theorem \ref{discr}, Remark \ref{Respect}, Lemma~\ref{p:momor} and Corollary~\ref{c:strongmomo}]\label{theorem:b}
Let $A \colon \T^{\N}\times\T^{\N}\times\T^{\N}\to \T^{\N}\times\T^{\N}\times\T^{\N}$ be the continuous group automorphism given by
$$
A((x_n),(y_n),(z_n)):=((x_n),(y_n),(x_n+z_n)).$$
Then:
\begin{enumerate}
\item[{\rm (i)}] $(\T^{\N}\times\T^{\N}\times\T^{\N}, A)$ is a universal model for ${\rm{\bf{DISP}}}_{\rm ec}$.
\item[{\rm (ii)}] For each function $\bfu \colon \N\to\mathbb{D}$, the topological system
$(\T^{\N}\times\T^{\N}\times\T^{\N},A)$ has the strong $\bfu$-MOMO property if and only if $A_2 \colon (x,y)\mapsto (x,x+y)$ on $\T^2$ has the strong $\bfu$-MOMO property. In fact, $\mathcal{C}_{{\rm {\bf{DISP}}}_{\rm ec}}\perp \bfu$ if and only if $A_2$ enjoys the strong $\bfu$-MOMO property.\footnote{An analogous result holds in the logarithmic case. In both cases, the strong $\mob$-MOMO property for $A_2$ is an open problem.}
\end{enumerate}
\end{Thx}

Recall that the main result of \cite{Ka-Ku-Le-Ru} asserts that $\mathcal{C}_{\cf_{\rm ec}}\perp\bfu$ is equivalent to the so called Veech condition for all Furstenberg systems of $\bfu$ (see Section~\ref{s:Sarnakc} for the notion of Furstenberg system of a bounded sequence). For the class $\mathcal{C}_{{\rm  {\bf{DISP}}}_{\rm ec}}$ the Veech condition is the vanishing of the second Gowers-Host-Kra seminorm: $\|\bfu\|_{u^2}=0$, see \cite{Ka-Ku-Le-Ru}.
There are some natural subclasses
$$\mathcal{C}_{{\rm  {\bf{DISP}}}(\Lambda)_{\rm ec}}\subset \mathcal{C}_{{\rm  {\bf DISP}}_{\rm ec}},$$ where $\Lambda\subset\T$ is a Borel subgroup ($\mathcal{C}_{{\rm  {\bf{DISP}}}(\Lambda)_{\rm ec}}$ consists of the topological systems whose all ergodic  measures yield discrete spectrum automorphisms whose point spectrum is contained in $\Lambda$). Note  however that once $\Lambda$ is proper, its Lebesgue measure is zero and this implies  that the maximal spectral type of the automorphisms associated to {\bf all} invariant measures have singular spectra.\footnote{Because ergodic components have the maximal spectral types supported by $\Lambda$, see Lemma~\ref{l:mutualsing}.} For the important case $\bfu=\mob$, we now deduce (like in \cite{Fe-Ku-Le}) that the {\bf logarithmic} Sarnak's conjecture holds for our systems:
\beq\label{lambdaec}
\mathcal{C}_{{\rm  {\bf{DISP}}}(\Lambda)_{\rm ec}}\perp_{\rm log}\mob.
\eeq
(See also Corollary~\ref{c:whenproper} to see what can be obtained for the original Sarnak's conjecture when $\Lambda$ is proper.)
 One of our tasks in the present paper is to determine the Veech condition for this smaller class for an arbitrary sequence $\bfu \colon \N\to\mathbb{D}$.

\begin{Propx}\label{theorem:c}
Assume that $\|\bfu\|_{u^1}=0$~\footnote{See Fact~\ref{p:ID}.} and let $\Lambda\subset\T$ be a Borel subgroup. Then, $\mathcal{C}_{{\rm  {\bf{DISP}}}(\Lambda)_{\rm ec}}\perp\bfu$  if and only if  for each Furstenberg system $(\D^{\Z},\kappa,S)$ of $\bfu$ and for a.e.\ of its ergodic component $\gamma$ the spectral measure $\sigma_{\pi_0,\gamma}$ of the function $\pi_0 \colon \D^{\Z}\to\D$, $\pi_0(y)=y_0$ has no atoms in $\Lambda$.\footnote{Note that when $\Lambda=\T$ then we  require that this spectral measure has no atoms at all, so the spectral measure $\sigma_{\pi_0,\gamma}$ is continuous for a.a.\ ergodic components which is equivalent to $\|\bfu\|_{u^2}=0$, see Section~\ref{s:ghk1} in Complements.} The same result holds in the logarithmic case.
\end{Propx}

(For the proof, see Proposition~\ref{prop:lambda-orth} and the fact that the Veech condition is the orthogonality of $\pi_0$ to $L^2$ of the relative $\Lambda$-Kronecker factor for each Furstenberg system of $\bfu$.)
Since for $\Lambda\neq\T$, we have $\mathcal{C}_{{\rm  {\bf{DISP}}}(\Lambda)_{\rm ec}}\perp_{\rm log}\mob$, Proposition~\ref{theorem:c} gives us a certain new information on the ergodic components of logarithmic Furstenberg systems of $\mob$.

\begin{Remx}\label{r:LambdaS}
{\em The same strategy can be applied in the Ces\`aro case when for $\Lambda$ additionally there exists a sequence $(q_n)$ with bounded prime volume, i.e.\ $\sup_n\sum_{p\in\PP, p|q_n}\frac1p<+\infty$, such that $\Lambda\subset\{x\in \T\colon \|q_nx\|\to 0\}$,\footnote{For example, take any bounded prime volume sequence $(q_n)$ satisfying $\sum_{n\geq1}(q_n/q_{n+1})^2<+\infty$ and consider $\Lambda:=\{x\in\T\colon \sum_{n\geq1}\|q_nx\|^2<+\infty\}$. This is an example of an uncountable H$_2$ group \cite{Ho-Me-Pa}.} then by \cite{Ka-Le-Ra}, Sarnak's conjecture holds for the class  ${\rm {\bf{DISP}}}(\Lambda)_{\rm ec}$: $\mathcal{C}_{{\rm {\bf{DISP}}}(\Lambda)_{\rm ec}}\perp\mob$, and Proposition~\ref{theorem:c} gives us an information on ergodic components of any Furstenberg system of $\mob$.}\end{Remx}

Returning  to the problem of existence of universal models in the context of ${\rm {\bf{DISP}}}(\Lambda)_{\rm ec}$, we will show the following result  (note that by taking $\Lambda=\T$, the first assertion of the result below shows that in Theorem~\ref{theorem:b}, we can obtain even more natural universal models for ${\rm{\bf{DISP}}}_{\rm ec}$).

\begin{Thx}[see Theorem~\ref{displambda}, Corollary~\ref{c:Fsigma}, Lemma~\ref{l:momor1}]\label{theorem:d}
For each  $F_\sigma$ subgroup $\Lambda\subset \T$ and each choice of closed, increasing $\Lambda_i\subset \Lambda$ satisfying $\Lambda=\bigcup_{i\geq1}\Lambda_i$, the topological system $\big((\prod_{i\geq1}\Lambda_i)\times\T^\infty\times\T^\infty, A|_{(\prod_{i\geq1}\Lambda_i)\times\T^\infty\times\T^\infty}\big)$ is a universal model for the class ${\rm {\bf{DISP}}}(\Lambda)_{\rm ec}$. Furthermore, if $\Lambda$ is proper then  $\big((\prod_{i\geq1}\Lambda_i)\times\T^\infty\times\T^\infty, A|_{(\prod_{i\geq1}\Lambda_i)\times\T^\infty\times\T^\infty}\big)$ satisfies the logarithmic strong $\mob$-MOMO property.\footnote{The logarithmic strong $\mob$-MOMO property also follows from \cite{Ka-Le-Ri-Te}.}
\end{Thx}

We also construct universal models for characteristic classes generated by certain weakly mixing automorphisms, see Theorem~\ref{t:misejo}.

Aiming at showing that a concrete dynamical system $(Z,R)$ is $\bfu$-ortho\-go\-nal, we could consider a ``local'' version of Theorem~\ref{theorem:a} in which for each Furstenberg system $\nu\in V(z_0)$ (i.e.\ $\nu$ is a measure for which a fixed $z_0\in Z$ is quasi-generic, see Section~\ref{s:Sarnakc}) we look at its ergodic decomposition $\nu=\int_{\Gamma}\nu_{\gamma}\,d\PP(\gamma)$ and then for the set $\{\nu_\gamma\colon \gamma\in\Gamma\}$ we look for a  strong $\bfu$-MOMO model which has ,,copies'' of $\nu_\gamma$ as ergodic measures. Instead, we will formulate a simplified version of it  as follows.

\begin{Thx}\label{theorem:e}
Let $(Z,R)$ be a topological system and $z_0\in Z$.  Assume that
for each $\nu\in V(z_0)$ with $\nu=\int_{\Gamma}\nu_\gamma\,d\PP(\gamma)$ denoting the ergodic decomposition, there exists a strong $\bfu$-MOMO system
$(X_\nu,T_\nu)$ and its invariant measure $\kappa_\nu$ such that for a.e.\ $\gamma\in\Gamma$,  $(R,\nu_\gamma)$ is a factor of $(T_\nu,\kappa_\nu)$. Then
$(Z,R)\perp \bfu$ at $z_0$, meaning \eqref{jsjm1} holds for all $f\in C(Z)$ at $z_0$.
\end{Thx}


The critical property here is to have a well-defined ``joining'' of  possibly {\bf uncountable} set of ergodic components $(R,\nu_\gamma)$. It raises one more natural problem which is of independent interest in dynamics. Namely,  we want to study non-ergodic systems which are close to ergodic ones in the sense that they are factors of
the systems of the form Id$\times Ergodic$. As we will see in Appendix, rather  surprisingly, the problem relates the existence of a common extension for ergodic components with the separability in $d$-bar distance.
\footnote{See also \cite{Ba-Ca-Kw-Op} and \cite{Ba-La} for some recent results on $d$-bar metric in dynamical context.} The following is a sample of consequences of results proved in Appendix.

\begin{Thx}[see the proof in Section~\ref{subsect}, cf.\ Proposition~\ref{p:inac}(B)]\label{theorem:f}
Let $(X,\nu,T)$ be an automorphism and let $\nu=\int_{\Gamma}\nu_\gamma\,d\PP(\gamma)$ denote its ergodic decomposition.
The following are equivalent:
\begin{enumerate}
\item[{\rm (i)}]
 $T$ is a factor of an automorphism of the form $\text{{\rm Id}}\times Ergodic$.
\item[{\rm (ii)}]
  There is an ergodic automorphism $(Y,\kappa,S)$ such that for $\PP$-a.e.\ $\gamma\in\Gamma$, the automorphism $(T,\nu_\gamma)$ is a factor of $(Y,\kappa,S)$.
\item[{\rm (iii)}]
 There is an  automorphism $(Y,\kappa,S)$ such that for $\PP$-a.e.\ $\gamma\in\Gamma$, the automorphism $(T,\nu_\gamma)$ is a factor of $(Y,\kappa,S)$.
  \end{enumerate}
\end{Thx}

We call systems satisfying (i) (or (ii) or (iii)) {\em uniformly
relatively ergodic} (URE) or {\em enclosed} \footnote{For the class ${\rm{\bf{DISP}}}$ this concept was implicitly considered in \cite{Hu-Ta-Xu}. Indeed, in the definition  therein, it is required that one can choose a subset of full measure of ergodic components so that the union of the set of eigenvalues is countable. In this case the group $\Lambda$ generated by this set is still countable. If $R$ denotes an ergodic automorphism with discrete spectrum with the group of eigenvalues $\Lambda$ then in fact $T$ is a factor of ${\rm Id}_{[0,1]}\times R$ (all ergodic components are factors of $R$), and actually $T$ must have discrete spectrum. See Corollary~\ref{p:random-factor?}, Theorem~\ref{pds} and Lemma~\ref{l:hutaxu1}.}.
A topological dynamical system is {\em topologically URE}\, if  each invariant measure yields a URE system.
Some 
examples and other basic properties of enclosed automorphisms
and topological URE systems are presented in Section~\ref{s:AC}.
 An extension of the notion of URE to Markov URE and piecewise Markov URE is also discussed in
 Section~\ref{s:AC}.



Returning to possible contributions toward a better understanding of the orthogonality problem, we obtain the following (${\rm UE}$ below stands for  the class of uniquely ergodic systems).

\begin{Corx}\label{corollary:g}
Let $\cf$ be a characteristic class. Assume that $({\rm UE}\cap \mathcal{C}_{\cf})\perp\bfu$.\footnote{\label{f:domomo} This condition guaranties that all uniquely ergodic systems in $\cf$ enjoy the strong-$\bfu$-MOMO property, see \cite{Ab-Le-Ru}, \cite{Ka-Le-Ri-Te}.} Then, for every topological system $(X,T)$ in $\mathcal{C}_{\cf_{\rm ec}}$
for which all visible measures\footnote{A measure $\nu\in M(X,T)$ is {\em visible} if it has a quasi-generic point $x$, i.e.\ along a subsequence $(N_k)$, we have $\frac1{N_k}\sum_{n<N_k}f(T^nx)\to\int f\,d\nu$ for each $f\in C(X)$. We write $\nu\in V(x)$.}  are enclosed, we have $(X,T)\perp \bfu$.
 If, in addition, for each $\kappa\in V(\bfu)$, $(X_{\bfu},\kappa,S)\in \mathcal{G}$, where $\mathcal{G}$ is another characteristic class containing in $\mathcal{F}$, then  $(X,T)\perp \bfu$ for all systems $(X,T)\in\mathcal{C}_{\cf_{\rm ec}}$ whose $\ca_{\mathcal{G}}$-factors\footnote{See Section~\ref{s:podr1} for the definition of $\ca_{\mathcal{G}}$.} for all visible measures are enclosed.
\end{Corx}

Denote by CT  the class of topological dynamical systems with countably many ergodic measures
and denote by TURE the class of topologically URE systems.
Of course, $\text{UE}\subset\text{CT}\subset\text{TURE}$.
Corollary~\ref{corollary:g} yields the following:
$$
({\rm UE}\cap\mathcal{C}_{\cf_{\rm ec}})\perp\bfu\;\Rightarrow ({\rm CT}\cap \mathcal{C}_{\cf_{\rm ec}})\perp\bfu\;\Rightarrow ({\rm TURE}\cap \mathcal{C}_{\cf_{\rm ec}})\perp\bfu.
$$
In particular, we generalize
\cite[Proposition 6.11]{Go-Le-Ru}, where only the first implication has been proved.

We recall that the assumption of Corollary~\ref{corollary:g} (for $\cf$={\bf ZE}) is satisfied for the logarithmic orthogonality and $\bfu=\mob$ by the break-through result of Frantzikinakis and Host \cite{Fr-Ho} and we obtain that the zero entropy
topologically URE systems
satisfy logarithmic Sarnak's conjecture.  In fact, as proved in \cite{Fr-Ho}, the Pinsker factors of logarithmic Furstenberg systems of $\mob$ are distal, so we can also apply the second part of the corollary for the class $\mathcal{G}={\bf DIST}$ of distal automorphisms to obtain that the logarithmic Sarnak's conjecture holds for all topological systems whose distal factors of automorphisms given by visible measures are enclosed.
 However, this particular result has already a recent strengthening by Huang, Tan and Xu \cite{Hu-Ta-Xu} to topological systems in which for each invariant measure, the Kronecker factors  of a.a.\ of its ergodic components  have a common extension (they call this property {\em almost countable spectrum}). In Section 8.1, we prove that, in fact, the property of almost countability does not depend on the choice of negligible subsets for all invariant measures: systems with almost countable spectrum are precisely those that have countable measure-theoretic spectrum (i.e.\ the union $\Lambda(X,T)$ of the sets of eigenvalues of $(X,\nu,T)$, $\nu\in M^e(X,T)$, is countable). For completeness, we state a refined version of the main result from \cite{Hu-Ta-Xu} and provide a shorter proof of it in Section~\ref{s:hutaxu}.

So far, thinking about applications, we mainly concentrated on finding properties of topological systems that are sufficient for the orthogonality with $\bfu$. However, another option is to focus on ergodic properties of Furstenberg systems of $\bfu$ themselves. This will ease (via some disjointness type argument) either proving orthogonality with all systems from $\mathcal{C}_{\cf}$ or with a certain subclass. Corollary~\ref{corollary:h} below provides a sample result that can be deduced that way.

In the context of the logarithmic Sarnak's conjecture, we recall a theorem of Frantzikinakis \cite{Fr}: {\em If all logarithmic Furstenberg systems of the M\"obius function $\mob$ are ergodic, then the logarithmic Sarnak's
conjecture holds.} Our goal is to extend this result by relaxing the ergodicity assumption, replacing it with suitable conditions on the (measurable) families of ergodic components of Furstenberg systems associated with $\mob$ (see Section~\ref{s:Markovure} for the notion of piecewise Markov URE).

\begin{Corx}\label{corollary:h}
Logarithmic Sarnak's conjecture holds if and only if for each logarithmic Furstenberg system $(X_{\mob},\kappa,S)$, the corresponding Pinsker factor is piecewise Markov URE. In particular, the logarithmic Sarnak's conjecture holds provided that, for the Pinsker factor of each Furstenberg system of $\mob$,  its a.a.\ ergodic components have a common extension.
\end{Corx}

\subsection{Universal models for characteristic classes}\label{s:podr1} Recall that a class $\mathcal{F}$ of measure-preserving automorphisms is a {\em characteristic class} if it is closed under taking (countable) joinings\footnote{Recall that if $R_i$ is an automorphism of $(Z_i,\cd_i,\kappa_i)$, $i=1,2$, then each $R_1\times R_2$-invariant measure $\rho$ with the marginals $\kappa_1$, $\kappa_2$, respectively,  is called a {\em joining} of $R_1$ and $R_2$: $\rho\in J(R_1,R_2)$. It yields a Markov operator $\Phi_\rho \colon L^2(Z_1,\kappa_1)\to L^2(Z_2,\kappa_2)$ determined by
$$
\int_{Z_1\times Z_2} f_1\ot f_2\,d\rho=\int_{Z_2}\Phi_\rho(f_1)f_2\,d\kappa_2\text{ for }f_i\in L^2(Z_i,\kappa_i),i=1,2,$$
and satisfying $U_{R_2}\circ\Phi_\rho=\Phi_\rho\circ U_{R_1}$,
\beq\label{jo2}\Phi_\rho \raz=\Phi_\rho^\ast\raz=\raz\text{ and }[f\geq0\Rightarrow\Phi_\rho(f)\geq0].\eeq
And vice versa: if $\Phi \colon L^2(Z_1,\kappa_1)\to L^2(Z_2,\kappa_2)$ is equivariant with $R_1$ and $R_2$ and satisfies the Markov properties \eqref{jo2}, then $\Phi=\Phi_\rho$ for a unique $\rho\in J(R_1,R_2)$. Recall that two automorphisms $R_1,R_2$ are (Furstenberg) {\em disjoint} if their only joining is the product measure; we then write $R_1\perp R_2$. Given a countable family $R_i$ of automorphisms of $(Z_i,\cd_i,\kappa_i)$, $i\geq1$, we define the set $J(R_1,R_2,\ldots)$ of joinings as the set of all $R_1\times R_2\times\cdots$-invariant measures having the projections $\kappa_i$ for each $i\geq 1$.} and factors.\footnote{Given an automorphism $R$ acting on $\zdk$, we identify factors with invariant sub-sigma-algebras of $\mathcal{D}$.} From now on, $\cf$ will always stand for a characteristic class. Every automorphism $R$ acting on a standard probability space $(Z,\mathcal{D},\kappa)$ possesses a largest (and hence unique) factor lying in $\mathcal{F}$ \cite{Le-Pa-Th,Ru}; we denote this factor by $\mathcal{A}_{\mathcal{F}}(R)$.\footnote{\label{f:reljoi} A crucial property of $\ca_{\cf}(R)$ is that whenever $R'$ (acting on $(Z',\cd',\kappa')$) belongs to $\cf$, we have $\int f\ot f'\,d\rho=\int \EE(f|\ca_{\cf}(R))\ot f'\,d\rho$ for each $\rho\in J(R,R')$, $f\in L^2(Z,\kappa)$ and $f'\in L^2(Z',\kappa')$, e.g.\ \cite{Ka-Ku-Le-Ru}. The role of such factors in the orthogonality problem~\eqref{jsjm1} will soon become clear, e.g.\ in Theorem~\ref{t:ajmt}.

Note moreover that if $\rho\in J(R_1,R_2)$ and $f_1\in L^2(\ca_{\cf}(R_1))$ then $\Phi_\rho(f_1)\in L^2(\ca_{\cf}(R_2))$. Indeed, this follows from the fundamental theorem on non-disjointness \cite{Le-Pa-Th} that $\Phi_\rho(f_1)$ belongs to the $L^2$-space of a common factor of $R_2$ and a certain infinite self-joining of $R_1|_{\ca_{\cf}(R_1)}$ (the latter is an element of $\cf$ and therefore the common factor is in $\cf$). It follows that each characteristic class is closed under taking Markov quasi-images: if $R_1\in \cf$, $\rho\in J(R_1,R_2)$, with $\Phi_\rho$ having dense image, then $R_2\in\cf$.  Note also that each characteristic class is closed under taking arbitrary joinings, meaning that whenever $R$ acting on $\zdk$ contains a (possibly uncountable) family $\cd_i$, $i\in I$, of factors which are in $\cf$, and $\cd=\bigvee_{i\in I}\cd_i$, then $R\in \cf$. Indeed, consider $\ca_{\cf}(R)$ which is an element in $\cf$.
}
The present paper takes first steps toward understanding which characteristic classes admit universal models, and, in particular, we show that non-trivial such classes do exist (e.g.\ Theorems~\ref{theorem:b} and~\ref{theorem:d}). It turns out (see Proposition~\ref{p:noum} below) that if $\cf$ has a universal model then $\cf$ must be an ec-class, i.e. \beq\label{restric1}\cf=\cf_{\rm ec},\eeq
where $\cf_{\rm ec}$ is the (characteristic) class of all automorphisms whose a.a.\ ergodic components lie in $\cf$. Since
\beq\label{zawD}
\mbox{
${\rm{\bf{DISP}}}\subsetneq  {\rm{\bf{DISP}}}_{\rm ec}$,}\eeq
this explains why ${\rm{\bf{DISP}}}$ has no universal model and it raises the natural question of whether ${\rm{\bf{DISP}}}_{\rm ec}$
admits one.
While indeed, we prove that ${\rm{\bf{DISP}}}_{\rm ec}$ admits a universal model (see Theorem~\ref{theorem:b}), one can study the universal model problem for related classes ${\rm{\bf{DISP}}}(\Lambda)$ of automorphisms with discrete spectrum whose set of eigenvalues is contained in a fixed Borel subgroup $\Lambda\subset\T$ (\eqref{zawD} holds for such classes with the equality whenever $\Lambda$ is countable). While a general case when $\Lambda$ is Borel remains unsolved, we solve positively the problem in case of $\sigma$-compact $\Lambda$ (see Theorem~\ref{theorem:d}) which shows a relationship of our investigations even with the non-singular ergodic theory as the groups of $L^\infty$-eigenvalues of ergodic non-singular automorphisms  are $\sigma$-compact \cite{Aa-Na}, \cite{Ho-Me-Pa} and  \cite{Na} Chapters 11 and 14.

Noticing that the intersection of characteristic classes is a characteristic class and following \cite{Be-Go-Ru} and \cite{Go-Le-Ru}, we also study the case of the smallest characteristic class  $\cf(T)$ containing a fixed automorphism $T$ (this class consists of factors of all countable self-joinings of $T$). For example, in the weak mixing case
in Theorem~\ref{t:misejo} we prove that the characteristic class generated by an MSJ automorphism\footnote{An ergodic automorphism $T$ has the {\em minimal self-joining} property (MSJ) \cite{Ju-Ru,Ru} if all its ergodic (infinite) self-joinings are Cartesian products of off-diagonal joinings, the latter being obtained as the images of $\mu$ under the maps
$$X\ni x\mapsto (x,T^{i_1}x,T^{i_2}x,\ldots)\in X^\infty$$
with $i_1,i_2,\ldots$ arbitrary integers.
For weakly mixing MSJ $T$, we have $T^m\perp T^n$ whenever $m\neq n$ \cite{Ju-Ru}. Moreover, such automorphisms enjoy the PID property, that is, for any other automorphisms $R,S$ and any $\rho$ a joining of $T$, $R$ and $S$ which is pairwise independent, $\rho$ is the product measure.
} admits a universal model, see Section~\ref{s:smomo} for precise results.

\subsection{(Markov) Uniform Relative Ergodicity} \label{s:Markovure}
An important  characterization of the URE property in terms of Ornstein's d-bar distance is done in Appendix, Corollary~\ref{cor:Tim}:

\begin{Th}
$T$ is URE (enclosed) if and only if   after removing a subset of measure zero of ergodic components, the remaining ergodic components yield a separable subset in the $d$-bar metric.
\end{Th}

From this theorem we deduce a characterization of the topological URE.
The following statements are equivalent (see Corollary~\ref{co:d-bar}):
\begin{enumerate}
\item[---]  a topological system $(X,T)$ is topologically URE,
\item[---]
the set  $M^e(X,T)$ of ergodic $T$-invariant measures is $d$-bar separable,
\item[---] all measures
from $M^e(X,T)$ yield automorphisms that have a common extension.
\end{enumerate}

Many consequences of the URE are provided in Section~\ref{s:AC}.
We mention here only one of them (Proposition~\ref{p:inac}(C)):
$$
\mbox{If $T$ is URE and $T\in\mathcal F_{\text{ec}}$ for a characteristic class $\mathcal F$    then $T\in\mathcal F$.}
$$

The URE property is also related to the recent paper \cite{Go-Le-Ru} in which the orthogonality of $\bfu$ to all uniquely ergodic systems: $\bfu\perp\,$UE, is studied (this problem is attributed to M. Boshernitzan in \cite{Co-Do-Se}). If $\bfu$ is such, it is not hard to see that there cannot be $\kappa\in V(\bfu)$  which yields a non-trivial ergodic Furstenberg system (see Section~\ref{s:Sarnakc} for notation $V(\bfu)$).
In fact, in \cite{Go-Le-Ru}, it is proved that $\bfu\perp\,$UE if and only if for each $\kappa\in V(\bfu)$ the function $\pi_0$\footnote{Recall that $\pi_0 \colon X_{\bfu}\to \mathbb{D}$, $\pi_0((y_n))=y_0$.} is orthogonal to all ergodic Markov quasi-images of $L^2(X,\mu)$:
\beq\label{orttoue}
\pi_0\perp \Phi_\rho(L^2(X,\mu))\eeq for all ergodic $(X,\mu, T)$ and all $\rho\in J(T,(S,\kappa))$.  This of course takes place if $\kappa$ yields a system disjoint from the class Erg of all ergodic automorphisms, but the real problem is to understand to which extent disjointness is needed, or in other words,  given a system non-disjoint with Erg, we follow a Vershik's idea \cite{Ver} of how much  the $L^2$-space of the system can be ``covered'' by ergodic Markov quasi-images,\footnote{In \cite{Go-Le-Ru}, the elements $f\in{\rm Im}(\Phi_\rho)$ ($R$ is ergodic, and $\rho\in J(R,T)$) are called {\em weakly ergodic}. The problem consists in understanding the orthocomplement of the space generated by the weakly ergodic elements. The reader can check that invariant functions are orthogonal to all weakly ergodic elements.}  see \cite{Fr-Le}.  For example,  we have the following.

\begin{Prop}\label{p:jjm}
Assume that $T\in{\rm Aut}\xbm$, and that a.a.\ ergodic components of $T$ are isomorphic. Then
\beq\label{Mproperty}
\ov{\rm span}\Bigg(\bigcup_{\rho\in J(R,T), R\in {\rm Erg}}{\rm Im}(\Phi_\rho)\Bigg)=L^2({\rm Inv}_\mu)^\perp.\eeq
All URE automorphisms satisfy the \eqref{Mproperty} property.
In fact, all piecewise Markov URE automorphisms (defined below) do.\footnote{
We would like to warn the reader that the URE property is not closed under taking (compact) group extensions (consider the Cartesian square of $(x,y)\mapsto (x+\alpha,x+y)$ on $\T^2$  with $\alpha$ irrational),  neither  direct products.}
\end{Prop}
That is, for enclosed automorphisms, only ``obvious'' (namely those which are $T$-invariant) functions from $L^2$ are not covered by the ergodic Markov quasi-images.

Instead of considering genuine factors in the definition of URE automorphism, another option is to use Markov quasi-images of automorphisms of the form Id$\times Ergodic$. We say that $T$ (acting on $\xbm$) is a {\em Markov URE} if there exist  an ergodic automorphism $S$  and a joining $\rho\in J(S\times \text{Id}_{[0,1]},T)$ such that
$$
L^2(X,\mu)\ominus L^2({\rm Inv}_\mu)\subset\ov{{\rm Im}(\Phi_\rho)}.
$$
Finally, we say that $T$ is {\em piecewise Markov URE} if there exist a family of ergodic automorphisms $\{S_i\colon i\in I\}$ and joinings $\rho_i\in J(S_i\times \text{Id}_{[0,1]},T)$, $i\in I$, such that
$$
L^2(X,\mu)\ominus L^2({\rm Inv}_\mu)\subset\ov{\rm span}\bigg(\bigcup_{i\in I}{{\rm Im}(\Phi_{\rho_i})}\bigg),$$
see also Footnote~\ref{f:jedenMarkov} for the RHS subspace above.
Using the same argument as before, it is easy to see that if $T$ is Markov URE then we can additionally claim that $S\in\cf$ (or all $S_i\in\cf$  if  $T$ is piecewise Markov URE). Clearly
$$
\mbox{URE $\Rightarrow$ Markov URE $\Rightarrow$ piecewise Markov URE.}
$$
More discussion on these notions is carried out in Section~\ref{s:AC} (see also Appendix and Section~\ref{s:open}).

\subsection{Sarnak's conjecture and  the general orthogonality problem}\label{s:Sarnakc}
Before returning  to the main results and formulating the other results of the paper more precisely and putting them in  context, let us recall some basics concerning the orthogonality problem.

Assume, we are given a sequence $\bfu \colon \mathbb{N}\to\mathbb{D}$, where $\mathbb{D}:=\{z\in\mathbb{C}:|z|\le 1\}$, and a characteristic class $\mathcal{F}$ of measure-preserving automorphisms.

Denote by $X_{\bfu}=\ov{\{S^n\bfu\colon n\in\Z\}}\subset \mathbb{D}^\Z$ the subshift generated by $\bfu$; here $S \colon \mathbb{D}^\Z\to \mathbb{D}^\Z$ stands for the left shift (and we treat $\bfu$ as a two sided sequence: $\bfu(-m)=\bfu(m)$ with (e.g.) $\bfu(0)=1$). Each weak$^\ast$-limit of empiric measures $\kappa:=
\lim_{k\to\infty}\frac1{N_k}
\sum_{n<N_k}\delta_{S^n\bfu}$ (the set of all such measures is denoted by $V(\bfu)$) gives an $S$-invariant measure on $X_{\bfu}$ which in turn yields a measure-theoretic automorphism $(X_{\bfu},\kappa,S)$, called a {\em Furstenberg system} of $\bfu$. By some abuse of vocabulary we may call $\kappa\in V(\bfu)$ a Furstenberg system of $\bfu$. An essential piece of information about statistical properties of $\bfu$ can be obtained from joinings
of Furstenberg systems of $\bfu$ with other automorphisms.

While  the orthogonality problem is difficult for the class {\bf ZE}, it admits a complete solution for the smallest non-trivial characteristic class, namely {\bf ID} \cite{Ka-Ku-Le-Ru}.
 To formulate the relevant result,  we recall that
given $\bfu$ and $(N_k)_{k=1}^\infty$   defining a Furstenberg system of $\bfu$, the corresponding Gowers-Host-Kra (GHK) seminorms are defined as follows (see, e.g.\ \cite{Ho-Kr}):
$$
\|\bfu\|^2_{u^1((N_k))}=\lim_{H\to\infty}\frac1H\sum_{h\leq H}\Big(\lim_{k\to\infty}\frac1{N_k}\sum_{n\leq N_k}\bfu(n)\ov{\bfu(n+h)}\Big),$$
and, for $s>1$,
$$
\|\bfu\|^{2^{s+1}}_{u^{s+1}((N_k))}=
\lim_{H\to\infty}\frac1H\sum_{h\leq H}\|
\bfu(\cdot+h)\cdot\ov{\bfu(\cdot)}\|^{2^s}_{u^s((N_k))}.$$
We say that $\|\bfu\|_{u^s}=0$ if $\|\bfu\|_{u^s((N_k))}=0$ is  zero for all relevant $(N_k)$.
The following is well known.
\begin{Fact}[see e.g.\ Corollary~5.3 and Lemma~5.5 in~\cite{Ka-Ku-Le-Ru}]\label{p:ID}
We have $\bfu\perp \mathcal{C}_{\rm {\bf ID}}$ if and only if $\|\bfu\|_{u^1}=0$.\end{Fact}
This is equivalent to the fact that $\pi_0\perp L^2({\rm Inv}_\kappa)$ for each $\kappa\in V(\bfu)$, where $\pi_0$ stands for the projection on the zero coordinate in $X_{\bfu}$ and ${\rm Inv}_\kappa$ denotes the sigma-algebra of invariant (mod~$\kappa$) subsets \cite{Ka-Ku-Le-Ru}. In what follows, when studying the orthogonality problem, we tacitly assume that $\|\bfu\|_{u^1}=0$. For the classical functions like M\"obius or Liouville, their first GHK norm is indeed equal to zero due to \cite{Ma-Ra}. Note also that the strong $\bfu$-MOMO property makes sense only when $\|\bfu\|_{u^1}=0$ (as \eqref{jsjm2} with $f=1$ is equivalent to $\|\bfu\|_{u^1}=0$, see Section~5.2 in \cite{Ka-Ku-Le-Ru}).

\subsection{Strong $\bfu$-MOMO property, further results and applications}\label{s:smomo}
While, by its very definition, the strong $\bfu$-MOMO property is combinatorial, actually, it is a purely dynamical concept. In fact, a joining flavour of the strong $\bfu$-MOMO property can be seen in the  following (by $\mathbb{A}$ we denote $\{e^{2\pi ij/3}\colon j=0,1,2\}$):
\begin{Prop}\label{p:rolemomo}
Let $(X,T)$ be a topological system.
Then
 it satisfies the strong $\bfu$-MOMO property if and only if for all $\nu\in M(X\times\mathbb{A},T\times {\rm Id}_{\mathbb{A}})$, $\kappa\in V(\bfu)$ and all joinings $\rho\in J((T\times {\rm Id}_{\mathbb{A}},\nu),(S,\kappa))$, we have
$$\int_{X\times\mathbb{A}\times X_{\bfu}}f\ot id_{\mathbb{A}}\ot\pi_0\,d\rho=0$$
for all $f\in C(X)$. The same property holds in the logarithmic case.\end{Prop}

Although this proposition is a rather direct consequence of the Lifting Lemma from \cite{Ka-Ku-Le-Ru}, we provide a proof for completeness.

This joining approach is often used to establish new results on the stability of $\bfu$-orthogonality in different topological models of a given measure-preserving automorphism as  seen in the proofs of most of our main results.

We have already mentioned that recently Huang, Tan, and Xu
\cite{Hu-Ta-Xu} proved (cf.\ Lemma~\ref{l:hutaxu1}) that each zero entropy topological system for which all invariant measures yield systems for which the relative (with respect to the sigma-algebra of invariant sets) Kronecker factor equals the Kronecker factor are logarithmically M\"obius orthogonal. We slightly strengthen this result by establishing the following short interval extension of the main result in~\cite{Hu-Ta-Xu}.
\begin{Th}\label{t:hutaxu1} Let $(Y,S)$ be any topological system of zero entropy.
\begin{enumerate}
\item[(a)] If for each $\nu\in M(Y,S)$, we have $\ck_{\rm rel}(S,\nu)=\ck(S,\nu)$ (i.e.,\ the relative (with respect to the sigma-algebra of invariant sets) Kronecker factor then $\mob\perp_{\rm log} (Y,S)$ (this was proved in~\cite{Hu-Ta-Xu}).
\item[(b)] If for each $\nu\in M(Y,S)$, we have $\ck_{\rm rel}(S,\nu)=\ck(S,\nu)$, then
$$
\lim_{H\to\infty}\limsup_{M\to\infty}\frac1{\log M}\sum_{m\leq M}\frac1m\Big\|\sum_{h\leq H}\mob(m+h)\cdot g\circ S^{m+h}\Big\|_{C(Y)}=$$$$
\lim_{H\to\infty}\limsup_{M\to\infty}\frac1{\log M}\sum_{m\leq M}\frac1m\sup_{y\in Y}\Big|\sum_{h\leq H}g(S^{m+h}y)\mob(m+h)\Big|=0$$
for all $g\in C(Y)$.
\end{enumerate}
\end{Th}
We note that the assertion in (b) is equivalent to the logarithmic strong $\mob$-MOMO of $(Y,S)$, see~\cite{Ka-Le-Ri-Te}.

\begin{Remark}
In fact, the proof of (a) in~\cite{Hu-Ta-Xu} shows that one can replace ``for each $\nu\in M(Y,S)$'' with ``for each {\bf logarithmically visible measure} $\nu\in M(Y,S)$'' (that is, $\nu\in \bigcup_{y\in Y}V^{\rm log}(y)$).
\end{Remark}

In~\cite[Theorem 1.1]{Go-Le-Ru1}, it is shown that logarithmic strong \textcolor{blue}{$\mob$}-MOMO property implies (uniform) M\"obius orthogonality along a subsequence of full logarithmic density. Therefore, it follows immediately from Theorem~\ref{t:hutaxu1} (see also Remark~\ref{r:ACequiv}) that we have the following.
\begin{Cor} \label{c:Sdensity}
Let $(X,T)$ be a topological system of zero entropy such that $\Lambda(X,T)$ is countable. Then there exists $A=A(X,T)\subset \mathbb{N}$ of full logarithmic density such that
\[
\lim_{A \ni N\to\infty}\left\|\frac{1}{N}\sum_{n\leq N}\mu(n)\cdot f\circ T^n \right\|_{C(X)}=0.
\]
In particular, M\"obius orthogonality holds along a subsequence of $N$ of full logarithmic density.
\end{Cor}

The assertion of Corollary~\ref{c:Sdensity} holds also in the case of singular spectrum. More precisely, the following holds.
\begin{Cor}\label{c:Sdens1}
Assume that $(X,T)$ is a topological system for which all Koopman operators associated to $(X,\nu,T)$, $\nu\in M(X,T)$, have singular maximal spectral types. Then Sarnak's conjecture holds in $(X,T)$ for a subsequence of logarithmic density~1.
\end{Cor}
Indeed, it was proved in \cite{Ka-Le-Ri-Te} that for the systems satisfying the assumptions of Corollary~\ref{c:Sdens1}, the logarithmic strong $\mob$-MOMO property holds. In view of \eqref{lambdaec} and Corollary~\ref{c:Sdens1}, we hence obtain the following.
\begin{Cor}\label{c:whenproper}
Let $\Lambda\subset\T$ be a proper Borel subgroup of the circle. Then
each $(X,T)\in\cc_{{\rm {\bf{DISP}}}(\Lambda)_{\rm ec}}$ satisfies Sarnak's conjecture along a subsequence of logarithmic density~1.\footnote{Notice that under the assumption that $\Lambda$ is $F_\sigma$, the assertion of Corollary~\ref{c:Sdens1} follows directly from Theorem~\ref{theorem:d}, \eqref{equivMOMO} and the logarithmic version of Theorem~\ref{theorem:a}.}
\end{Cor}

The main result in \cite{Hu-Ta-Xu} is applied to special topologically distal systems (in particular, to the class of classical continuous Anzai skew products, see below). We will show that Theorem~\ref{t:hutaxu1} applies to some other (continuous) extensions of rotations, the so called Rokhlin extensions \cite{Au-Le,Da-Le,Le-Pa,Le-Pa1}.

\begin{Th}\label{t:Rokh} Assume that $(X,T)$ is a uniquely ergodic, and let $G$ be an l.c.s.c.\ group. Let $\cs=(S_g)_{g\in G}$ be a continuous uniquely ergodic action on a compact metric space $Y$. Then, for each continuous cocycle $\va \colon X\to G$ (measurably) cohomologous to an ergodic cocycle\footnote{In other words, we assume that  $\va \colon X\to G$ is regular \cite{Sch}, that is, there are a cocycle $\psi \colon X\to H$, and a measurable $\xi \colon X\to G$ such that $\psi=\xi^{-1}\cdot\va\cdot \xi\circ T$ and the skew product $T_\psi$ acting on $X\times H$ by the formula $T_{\va}(x,h)=(Tx,\psi(x)h)$ is ergodic for the (in general, infinite) measure $\mu\ot m_H$, where $\mu$ is $T$-invariant and $m_H$ stands for a Haar measure on $H$. For $G$ non-compact not all cocycles are regular.} taking values in a (closed) cocompact  subgroup contained in the center,
the homeomorphism $T_{\va,\cs}(x,y)=(Tx,S_{\va(x)}(y))$ of the space $X\times Y$ is URE, in fact, all its ergodic components are isomorphic. In particular, $T_{\va,\cs}$ satisfies the logarithmic Sarnak's conjecture.\end{Th}
\noindent(For the proof, see Theorem~\ref{t:sm}.) In particular, the result applies for $G$ compact, as in this case all cocycles are regular \cite{Sch}. A classical playground to which we can apply the above is the world of Anzai skew products acting on $\T^2$:
$$
 T_{\va}\colon(x,y)\mapsto (x+\alpha,x+\va(x))
 $$
  with a cocycle $\va \colon \T\to\T$ continuous and $\alpha$ irrational, whose M\"obius orthogonality was first studied by Liu and Sarnak \cite{Li-Sa}  (see also \cite{Ku-Le,Wa}).  Note that for the original Sarnak's conjecture we can apply Theorem~\ref{theorem:a} for $C^{2+\epsilon}$ cocycles to obtain that:

   {\em For each continuous $\va$ which is (measure-theoretically) cohomologous to an ergodic $C^{2+\vep}$ cocycle, $T_{\va}$ is M\"obius orthogonal.}\footnote{Indeed, it is proved in \cite{Ka-Le-Ra} that ergodic $C^{2+\epsilon}$ Anzai skew products satisfy the strong $\mob$-MOMO property, so the original homeomorphism $T_\varphi$  has all its ergodic invariant measures giving rise to isomorphic automorphisms and having a strong $\mob$-MOMO model.}

Let us return to the problem of orthogonality involving a characteristic class concentrating on properties of Furstenberg systems of $\bfu$.
Let
$$
\mathcal{L}=\mathcal{L}_{\text{\rm{$\bfu$-MOMO}}}:=\{(X,T) \colon (X,T)\text{ satisfies \eqref{jsjm2} for all relevant }(b_k)\}\footnote{Note that $\mathcal{L}$ does not depend on $\cf$.}
$$
and let $\widehat{\mathcal{L}}$ stand for the class of automorphisms $R$ having a topological model belonging to
$\mathcal{L}$.\footnote{That is, there exist $(X,T)\in\mathcal{L}$ and $\mu\in M(X,T)$ such that $T$ acting on $(X,\mu)$ is measure-theoretically isomorphic to $R$.}
Building on the ideas developed in \cite{Ka-Ku-Le-Ru}, we establish the following.

\begin{Th}\label{t:ajmt} Let $\bfu \colon \N\to\mathbb{D}$, $\|\bfu\|_{u^1}=0$, and consider an ec-characteristic class $\cf_{\rm ec}$. Assume that for each Furstenberg system $\kappa\in V(\bfu)$, we have
\begin{multline}\label{nopo}
L^2(\ca(\cf_{\rm ec}),\kappa|_{\ca(\cf_{\rm ec})})\ominus L^2({\rm Inv}_\kappa)\subset\\
\ov{\rm span}\{{\rm Im}(\Phi)\colon \Phi\in J(R,S|_{\ca(\cf_{\rm ec})}), R\in\widehat{\mathcal{L}}\}.
\end{multline}
Then $\cc_{\cf_{\rm ec}}\perp\bfu$.\footnote{\label{f:jedenMarkov} Despite its ``largeness'', the space on the RHS in~\eqref{nopo} is generated by one Markov operator. Indeed, by choosing a countable dense subset in $\bigcup_{\Phi\in J(R,S|_{\ca(\cf_{\rm ec})}), R\in\widehat{\mathcal{L}}}{\rm Im}(\Phi)$, we can reduce the family of ``good'' $R$ to a countable family $\{R_i\colon i\geq1\}$ (acting on $(Z_i,\cd_i,\kappa_i)$) together with joinings $\rho_i\in J(R_i,S|_{\ca(\cf_{\rm ec})})$, $i\geq1$. Then, we consider $Z_0:=\bigsqcup_{i\geq 1}Z_i$ with a measure $\kappa_0=\sum_{i\geq1}a_i\kappa_i$ (with all $a_i>0$, $\sum_{i\geq1}a_i=1$) and the  corresponding automorphism $R_0$ acting on each $Z_i$ as $R_i$. It is not hard to see that $\rho:=\sum_{i\geq1}a_i\rho_i$ is a joining of $R_0$ with $S|_{\ca(\cf_{\rm ec})}$ and the closure of the image ${\rm Im}(R_0)$ is the RHS
in~\eqref{nopo}. It remains to show that $R_0\in\widehat{\mathcal{L}}$. But this follows from the fact that the classical shrinking wedge construction of strong $\bfu$-MOMO systems still enjoys the same property, see Proposition~\ref{thm:wedge-strongMOMO}.
Note also that if all $R_i$ belong to a fixed characteristic ec-class then also $R_0$ belongs to this class.

By assuming that $\kappa\in V^{\rm log}(\bfu)$ and considering $\mathcal{L}^{\rm log}$, we obtain a criterion for $\bfu\perp_{\rm log}\mathcal{C}_{\cf_{\rm ec}}$.}
\end{Th}
This result is a motivation to study the problem of ``covering'' of the $L^2$-space of an automorphism $T$ by the images of Markov operators arising from the joinings of $T$ with automorphisms belonging to a  certain (characteristic) class.
Note that in the special case when we have just one element $(R,Z)\in\mathcal{L}$ (depending on $\bfu$) which is a universal topological model for the class $\cf_{\rm ec}$ the assumption~\eqref{nopo} is satisfied, so the assertion of Theorem~\ref{t:ajmt} follows. However, in this case, the assertion also follows from Theorem~\ref{theorem:a}.





Recall that given an automorphism $T$, $\cf(T)$ denotes the smallest characteristic class containing $T$. This class consists of factors of all self-joinings of $T$.
 It is not hard to see that for a subgroup $\Lambda\subset\T$ countable,  ${\rm{\bf{DISP}}}(\Lambda)=\cf(T)$ for an ergodic automorphism $T$ with discrete spectrum equal to $\Lambda$.
However,  deciding whether, given an arbitrary ergodic
 ergodic $T$, the corresponding class $\cf(T)$ has a universal model remains a challenge.
 Consider another classical case of affine maps on $\T^2$: $T=A_{2,\alpha}$  with $\alpha$ irrational:
$$
A_{2,\alpha}(x,y)=(x+\alpha,x+y).$$
This topological system is uniquely ergodic, and we will show that the following holds:
\begin{Prop}\label{p:a2alfa}
\begin{enumerate}
\item[{\rm (i)}]
For each $\alpha$ irrational, ${\rm {\bf{DISP}}}_{\rm ec}\subset \cf(A_{2,\alpha})$.
\item[{\rm (ii)}]
Whenever $1,\alpha,\beta$ are rationally independent,
$$\big(\cf(A_{2,\alpha})\cap \cf(A_{2,\beta})\big)\cap{\rm Erg}= {\rm {\bf{DISP}}}_{\rm ec}\cap {\rm Erg}={\rm {\bf{DISP}}}\cap{\rm Erg}.$$
In particular,
$A_{2,\beta}\notin \cf(A_{2,\alpha})$.
\end{enumerate}
\end{Prop}
\noindent (For the proof, see Proposition~\ref{p:gaj1}, Lemma~\ref{l:gaj2} and Proposition~\ref{p:gaj20}.)

On the other hand, even in the weakly mixing case, we can find $T$ for which $\cf(T)$ seems to be fully understood.

\begin{Th}\label{t:misejo} Assume that $T$ acting on $\xbm$ has the minimal self-joining property. Then:
\begin{enumerate}
\item[{\rm (i)}]
A universal model for $\cf(T)$ is constructed explicitly.
 \item[{\rm (ii)}]
 $\cf(T)=\cf(T)_{\rm ec}$.\footnote{This follows from (i), cf.~\eqref{restric1}.}
\item[{\rm (iii)}]
 For each $\bfu \colon \N\to\mathbb{D}$
the above universal model enjoys the strong $\bfu$-MOMO if and only if
the system $([0,1]\times X_0^\infty,{\rm Id}_{[0,1]}\times T_0^{\times\infty})$ does ($(X_0,T_0)$ is any uniquely ergodic model of $(X,\mu,T)$). In fact, for each uniquely ergodic model $(X_0,T_0)$ of $T$,
$\bfu\perp \mathcal{C}_{\cf(T)}$ if and only if $ (X_0^\infty,T_0^{\times\infty})$ enjoys the strong $\bfu$-MOMO property.
\end{enumerate}
\end{Th}

\begin{Cor}\label{c:ortogtomsj}
For each {\bf multiplicative} $\bfu \colon \N\to\mathbb{D}$, we have
$$
\bfu\perp \mathcal{C}_{\cf(T)}$$
for each (weakly mixing) automorphism $T\in{\rm MSJ}$.
\end{Cor}

Let us now formulate one of our applications, namely,
a pure ergodic theory proof of the equivalence (below) first proved by Tao  in \cite[Section 4]{Ta1}\footnote{We thank N. Frantzikinakis for this information.} (together with a use of \cite{Gr-Ta-Zi}), see also Frantzikinakis \cite{Fr}.

\begin{Th}\label{t:m1}Assume that $\|\bfu\|_{u^1}=0$. Then $\|\bfu\|_{u^2}=0$ if and only if $\bfu$ satisfies the local 1-Fourier uniformity, i.e.
\beq\label{l1fu1}
\lim_{H\to\infty}\limsup_{M\to\infty}
\frac1M\sum_{m<M}\sup_{\alpha\in\R}
\Big|\frac1H\sum_{h<H}\bfu(m+h)e^{2\pi ih\alpha}\Big|=0.\footnote{An analogous result holds in the logarithmic case.}\eeq
\end{Th}

An interesting case arises when the assumption in the spirit of Proposition~\ref{p:jjm} applies to Furstenberg systems of $\bfu$.\footnote{\label{f:Mproperty} For example, it follows that if $\bfu\perp\,$UE and each of its Furstenberg systems $\kappa$ has a.a.\ ergodic components isomorphic then $\bfu$ has to be mean slowly varying:  $\lim_{N\to\infty}\frac1N\sum_{n<N}|\bfu(n+1)-\bfu(n)|=0$. Indeed, suppose that $\kappa=\lim_{k\to\infty}\frac1{N_k}\sum_{n<N_k}\delta_{S^n\bfu}$. We have
$$
\lim_{k\to\infty}\frac1{N_k}\sum_{n<N_k}|\bfu(n)-\bfu(n+1)|=\int|\pi_0-\pi_0\circ S|\,d\kappa.$$
We apply now Proposition~\ref{p:jjm}  and~\eqref{orttoue} to obtain that $\pi_0\in L^2({\rm Inv}_\kappa)$ and to conclude (note that this implies that all the $\kappa$s yield identity systems, cf.\ \cite{Fr-Le-Ru}, \cite{Go-Le-Ru}).}
As a consequence of Theorem~\ref{t:ajmt}, we obtain the following.

\begin{Cor}\label{p:m17}
Let $\cf$ be a characteristic class. Assume that $\|\bfu\|_{u^1}=0$ and moreover
\beq\label{almost}
(\forall\kappa\in V(\bfu))\;\;\;\;(S|_{\ca(\cf_{\rm ec})},\kappa|_{\ca(\cf_{\rm ec})})\text { is }\cf\text{-piecewise Markov URE.}
\eeq
Then, we have
$$
\bfu\perp {\rm UE}\cap\cf_{\rm ec}~\footnote{This condition  can  be expressed  in terms of self-joinings of Furstenberg systems of $\bfu$, see \cite{Go-Le-Ru}.}\;\;\Rightarrow\;\; \bfu\perp \cc_{\cf_{\rm ec}}.$$
In particular, the result holds if for each Furstenberg system $\kappa\in V(\bfu)$ the ergodic components have a common ergodic extension.\\
The result also holds in the logarithmic case.
\end{Cor}

Note that this result tells us that, assuming that \eqref{almost} holds, if we want to show the orthogonality of $\bfu$ to all topological systems whose all ergodic measures yield systems in $\cf$, we only need to check orthogonality to all uniquely ergodic systems which are in $\cf$. If we consider $\mob$ and the class ${\rm{\bf{DISP}}}$, then $\mob\perp_{\rm log} {\rm UE}\cap {\rm {\bf{DISP}}}_{\rm ec}$ indeed holds \cite{Fe-Ku-Le,Hu-Wa-Ye,Hu-Wa-Zh}, and (remembering  that Sarnak's conjecture for ${\rm{\bf{DISP}}}_{\rm ec}$ is equivalent to $\|\mob\|_{u^2}=0$ \cite{Ka-Ku-Le-Ru}),  we obtain Corollary~\ref{corollary:h}. (Note that for the necessity in Corollary~\ref{corollary:h}, we need to use Tao's theorem \cite{Ta1} about the equivalence of logarithmic Sarnak's conjecture
and the logarithmic Chowla conjecture--the latter gives in particular that $\mob$ has one logarithmic Furstenberg system, and this system is ergodic.)
As a matter of fact, one can obtain an almost direct proof of Corollary~\ref{corollary:h} from \cite{Fr-Ho},
\cite{Hu-Ta-Xu}, see Lemma~\ref{l:hutaxu6} below.




\begin{Remark}\label{r:logSchar} Yet, in Corollary~\ref{corollary:h}, we can replace the Pinsker factor by the relative (with respect to the sigma-algebra of invariant sets) Kronecker factor to obtain that: {\em Logarithmic Sarnak's conjecture holds if and only if for each $\kappa\in V^{\rm log}(\mob)$  the relative Kronecker factor $\ck_{\rm rel}(\kappa)$ is  ${\rm{\bf{DISP}}}_{\rm ec}$-piecewise Markov URE}.\end{Remark}



\section{Preliminaries}
\subsection{Disintegration and sub-sigma-algebras} Let $(X,\mathcal B,\mu)$ be a standard probability space.
Denote by Aut$(X,\mathcal B, \mu)$ (or ${\rm Aut}(X,\mu)$) the group of automorphisms (i.e.\ of $\mu$-preserving invertible transformations) of $X$.
We endow it with the weak topology. Let
$\mathcal E(X,\mu)$ denote the subset of ergodic automorphisms of $(X,\mu)$. Then $\mathcal E(X,\mu)$ is a dense $G_\delta$ subset of Aut$(X,\mu)$. 
We will not distinguish between objects (subsets, functions, transformations, sub-sigma-algebras, etc.) that agree almost everywhere.


We remind that given an arbitrary sub-sigma-algebra $\mathcal A\subset \mathcal B$, there is a standard probability space $(Y,\mathcal D,\nu)$ and
a measure preserving mapping $\tau\colon(X,\mathcal B,\mu)\to(Y,\mathcal D,\nu)$
such that $\mathcal A=\{\tau^{-1}(B)\colon B\in\mathcal D\}$ mod 0.
We will use sometimes the notation $(X,\mathcal B,\mu)/\mathcal A$
for $(Y,\mathcal D,\nu)$ or simply
$X/\mathcal A$ for $Y$.
Let $Y\ni y\mapsto\mu_y$ be the corresponding system of conditional measures on $X$.
This means that $\mu=\int_Y\mu_y\,d\nu(y)$ and $\mu_y(\tau^{-1}(y))=1$ at $\nu$-a.e.\ $y\in Y$.
We also note that $\int f\, d\mu_{\tau(x)}=\EE(f\mid \mathcal A)(x)$ at a.e.\ $x\in X$ for each bounded Borel function $f\colon X\to\Bbb R$.

There is a one-to-one correspondence between the sub-sigma-algebras in $\mathcal B$ and the closed $*$-subalgebras of the von Neumann algebra $L^\infty(X,\mu)$ acting on $L^2(X,\mu)$ via multiplication.
We note that $L^2(\mathcal A,\mu)$ is a closed $L^\infty(\mathcal A)$-invariant subspace of $L^2(X,\mu)$.
The disintegration of $\mu$ with respect to $\nu$  yields a decomposition of $L^2(X,\mu)$ into direct integral of Hilbert spaces:
\beq\label{integral}
L^2(X,\mu)=\int_{Y}^\oplus L^2(X,\mu_y)\,d\nu(y).
\eeq
As $L^2(Y,\nu)$ is a subspace of $L^2(X,\mu)$, the orthocomplement
$L^2(Y,\nu)^\perp$ of $L^2(Y,\nu)$ in $L^2(X,\mu)$ equals
$$
L^2(Y,\nu)^\perp=\int_{Y}^\oplus L^2_0(X,\mu_y)\,d\nu(y),
$$
where $L^2_0(X,\mu_y)$ is the subspace of  functions from $L^2(X,\mu_y)$
with 0 integral.
Given a sub-sigma-algebra $\mathcal G$ of $\mathcal B$,
we let $\mathcal G_y:=\{G\cap \tau^{-1}(y)\colon G\in\mathcal G\}$ for each $y\in Y$.
 Then $\mathcal G_y$ is a  sub-sigma-algebra of the fiber space $( \tau^{-1}(y),\mathcal B_y, \mu_y)$, where $\mathcal B_y=\mathcal B\cap \tau^{-1}(y)$.
 Of course, $L^2(\mathcal G_y,\mu_y)$ is a closed subspace
 of $L^2(X,\mu_y)$.
 If $\mathcal G\supset \mathcal A$ then (\ref{integral}) induces the following decomposition of $ L^2(\mathcal G,\mu)$:
 $$
 L^2(\mathcal G,\mu)=\int_Y^\oplus L^2(\mathcal G_y,\mu_y)\,d\nu(y).
 $$
This implies the following well-known statement.
\begin{Prop}\label{2factors}
 Let $\mathcal G$, $\mathcal G'$  be two sub-sigma-algebras
 of $\mathcal B$.
Assume  that $\mathcal G\cap\mathcal G'\supset \mathcal A$. Then  $\mathcal G=\mathcal G'$ if and only if
$\mathcal G_y=\mathcal G_y'$ for $\nu$-a.a. $y\in X/\mathcal A$.
\end{Prop}


\subsection{Isomorphisms and factors of non-ergodic automorphisms}\label{s:ergcom}

We first recall the following classical result (see e.g.\ \cite[Theorem~18.1]{Ke}) that will be utilized repeatedly in this paper.
\begin{Th} (Jankov-von Neumann Theorem) \label{th:JvN}
Let $X,Y$ be standard Borel spaces and $f \colon X\to Y$ a Borel map. Then, $f(X)$ is analytic and there exists a measurable selector $s \colon (f(X),\ca(f(X))\to (X,\cb(X))$, $f\circ s=\text{{\rm Id}}_{f(X)}$, where $\ca(f(X))$ denotes the sigma-algebra generated by analytic sets.
\end{Th}

Assume now  that $T\in{\rm Aut}\xbm$,  and let
$\mathcal I(T)={\rm Inv}_\mu$ denote
the sigma-algebra of all measurable subsets fixed by $T$.
Denote also
$$
(\baX,\ov{\cb},\bamu):= (X,\mathcal B,\mu)/\mathcal I(T),
$$
whence $T|_{{\rm Inv}_\mu}=\text{Id}_{\baX}$.
Let
\beq\label{edecomp}
\mu=\int_{\baX}\mu_{\bax}\,d\bamu(\bax)
\eeq
stand for the corresponding disintegration.
It is called the {\em ergodic decomposition} of $\mu$.
It follows from the uniqueness of disintegration that
 the conditional measures $\mu_{\bax}$  are $T$-invariant.
Besides, they are ergodic.
A suggestive picture for the ergodic decomposition appears when we assume additionally that $T$
 is aperiodic (i.e.\ the measure of periodic points equals zero). Then there is a standard nonatomic probability space $\ycn$
and a measurable mapping $\baX\ni \bax\mapsto T_{\bax}\in {\rm Aut}\ycn$
such that:
\begin{enumerate}[(a)]
\item
$\xbm=(\baX,\ov{\cb},\bamu)\ot\ycn$,
\item
$T_{\bax}$ is ergodic for $\bamu$-a.e.\ $\bax\in \baX$, i.e.\
$\baX\ni \bax\mapsto T_{\bax}\in \mathcal{E}\ycn,$
\item
the action $T$  is isomorphic to $T(\bax,y)=(\bax,T_{\bax}\,y)$ on the product space $(\baX\times Y,\bamu\otimes\nu)$.
\end{enumerate}
In other words, $T$ is relatively ergodic over the first coordinate sigma-algebra (i.e.\ the sigma-algebra of invariant sets).
In the general case, the space $\baX$ of ergodic components is partitioned into $\baX=\baX_1\sqcup\baX_2\sqcup\ldots \sqcup \baX_\infty$, with $\baX_k$ being the set of points with (the smallest) period $k$, where the action of $T$ (represented by $T_{\bax}$) on the ergodic components for $\bax\in\baX_k$ are rotations by~1 (mod k) on the finite set $\{0,1,\ldots,k-1\}$ equipped with the equidistribution, and the aperiodic part is represented by $\baX_\infty\times Y$ as above.

Assume that $T'\in {\rm Aut}(X',\mathcal B',\mu')$ is another automorphism, and that
$$
\baX'\ni \bax'\mapsto T'_{\bax'}\in \mathcal{E}\ycn
$$
corresponds to the ergodic decomposition $\mu'=\int_{\baX'}\mu'_{\bax'}\,d\bamu'(\bax')$ of $T'$.

 \begin{Prop}\label{fields} Let $T$ and $T'$ be two automorphisms of  $(X,\mathcal B,\mu)$ and
 $(X',\mathcal B',\mu')$, respectively.
 Then $T$ and $T'$ are isomorphic if and only if there is a measure space isomorphism
 $\phi\colon(\baX,\bamu)\to (\baX',\bamu')$ 
 such that $T_{\bax}$ is isomorphic to $T_{\phi(\bax)}'$ for $\bamu$-a.a. $\bax\in \baX$.
 \end{Prop}
 \begin{proof} If $T$ and $T'$ are isomorphic then there is a measure space isomorphism $R$ of $(X,\mu)\to (X',\mu')$ that intertwines  $T$ with $T'$.
 Of course, $R$ maps  $\mathcal I(T)$ onto  $\mathcal I(T')$.
 To simplify notation,\footnote{Note that two automorphisms are isomorphic iff the aperiodic parts are isomorphic and their $k$-periodic parts so are for all $k\geq1$.}  we assume that $T$ (and hence $T'$) are aperiodic.
 By a standard ergodic theory argument (e.g.\ \cite{EiWa}), it follows that there is an isomorphism  $\phi\colon\baX\ni \bax\mapsto \phi(\bax)\in \baX'$, $\phi_\ast\bamu=\bamu'$, and a measurable mapping
 $\baX\ni \bax\mapsto R_{\bax}\in\text{Aut}(Y,\nu)$ such that
 $$
 R(\bax,y)=(\phi(\bax),R_{\bax}y).
 $$
 Note that $R^{-1}(\bax',y)=(\phi^{-1}\bax',R^{-1}_{\phi^{-1}(\bax')}y)$. Since $R^{-1}T'R=T$, it follows that $R_{\bax}^{-1}T'_{\phi(\bax)}R_{\bax}=T_{\bax}$ at $\bamu$-a.e.\ $\bax$, as desired.

 Conversely, suppose that $T'_{\phi(\bax)}$ and $T_{\bax}$ are isomorphic for $\bamu$-a.e.\ $\bax\in \baX$.
 Consider the subset
 $$
 M:=\{(\bax,Q)\in \baX\times \text{Aut}(Y,\nu) : T'_{\phi(\bax)}Q=QT_{\bax}\}.
 $$
 Passing to a subset of full measure in $\baX$, we can assume without loss of generality that $\baX$ is standard Borel and  $\phi$ is Borel.
Then $M$ is a Borel subset of $\baX\times \text{Aut}(Y,\nu)$.
 Passing to a further subset of full measure if necessarily,
we can deduce from the assumption that the projection of $M$ to the first coordinate is onto.
 Hence, by the Jankov-von Neumann selection theorem (see Theorem~\ref{th:JvN}),
 there is a universally measurable mapping
 $$
 \baX\ni \bax\mapsto D_{\bax}\in \text{Aut}(Y,\nu)
 $$
  such that $(\bax, D_{\bax})\in M$ at all $\bax$.
  It follows that $T'_{\phi(\bax)}D_{\bax}=D_{\bax}T_{\bax}$ at $\bamu$-a.e. $\bax$, so the map $(\bax,y)\mapsto (\phi(\bax),D_{\bax}y)$ yields a desired isomorphism.
 \end{proof}

\begin{Cor}\label{c:jon17}
Assume that $T\in{\rm Aut}(X,\mu)$. Assume moreover that $\bamu$-a.a.\ ergodic components are isomorphic to a certain ergodic automorphism $R$.
Then $T$ is isomorphic to $\text{{\rm Id}}_{\baX}\times R$.
\end{Cor}
 \begin{proof} Assume that $\bamu$-a.a.\ ergodic components $T_{\bax}$ of $T$ are isomorphic to $R$. Then, apply Proposition~\ref{fields} to $T$ and $T':=\text{Id}_{\baX}\times R$.
 \end{proof}

 In the following proposition we describe factors of nonergodic dynamical systems.
 It follows straightforwardly from Proposition~\ref{2factors}.

 \begin{Prop}\label{p:faktory} The only factors of $T$ containing all invariant sets (of $T$) are given by a measurable choice of factors on ergodic components, i.e.\ if $\ce\subset \ov{\mathcal{B}}\ot\mathcal{C}$ is a $T$-invariant sigma-algebra containing $\ov{\mathcal{B}}\otimes \{\emptyset,Y\}$, then there exists a measurable choice\footnote{This means that the mapping $\widebar x\mapsto\Bbb E(\cdot|\mathcal E_{\widebar x})$ is measurable in the weak operator topology.}
\beq\label{can3}\bax\mapsto \mathcal{E}_{\bax}\subset \mathcal{C}\eeq of $T_{\bax}$-invariant sub-sigma-algebras, so that for $\widetilde{A}\in\ov{\mathcal{B}}\ot\mathcal{C}$,
$$\widetilde{A}\in \mathcal{E}\text{ iff } \widetilde{A}_{\bax}\in\mathcal{E}_{\bax}\text{ a.e.}$$
\end{Prop}

We will also need a non-invertible version of the ``only if'' part of Proposition~\ref{fields}.

  \begin{Prop}\label{p:faktory?}
  Let $T$ and $T'$ be two  transformations of  $(X,\mathcal B,\mu)$ and
 $(X',\mathcal B',\mu')$, respectively.
If there exists  a measure preserving mapping
 $\phi\colon(\baX,\bamu)\to (\baX',\bamu')$ 
 such that $T_{\phi(\bax)}'$ is isomorphic to a factor of $T_{\bax}$
 for $\bamu$-a.a. $\bax\in \baX$ then $T'$ is isomorphic to a factor of $T$.
  \end{Prop}

  \begin{proof}
    Suppose first that $T$ and $T'$ are aperiodic.
  The proof is only a slight modification of the proof of
  Proposition~\ref{fields}.
  Therefore, we will use the notation
  $$
  \baX\ni\bax\mapsto  T_{\bax}\in \text{Aut}(Y,\nu)\quad\text{
  and}\quad
  \baX'\ni\bax'\mapsto T'_{\bax'}\in\text{Aut}(Y,\nu)
  $$
   introduced there.
  Let End$(Y,\nu)$ denote the semigroup of
  $\nu$-preserving transformations of $(Y,\mathcal C,\nu)$.
  Endow it with the weak operator topology.
  Then End$(Y,\nu)$ is a semitopological semigroup.
  Replace  $M$ in the proof Proposition~\ref{fields} with the following subset
 $$
\widetilde M:= \{(\bax,Q)\in \baX\times \text{End}(Y,\nu) : T'_{\phi(\bax)}Q=QT_{\bax}\}
 $$
 and repeat  that proof almost literally to construct the factor mapping of $T$ onto $T'$.

 If $T$ and $T'$ have  periodic parts, there are partitions
 \begin{align*}
 \widebar X&=\left(\bigsqcup_{n=1}^\infty\bigsqcup_{m|n}\widebar X_{n,m}\right)\sqcup\left(\bigsqcup_{n=1}^\infty \widebar X_{\infty,n}\right)\sqcup \widebar X_{\infty,\infty}\quad\text{and}\\
  \widebar X'&=\left(\bigsqcup_{n=1}^\infty \widebar X'_n\right)\sqcup\widebar X'_\infty
 \end{align*}
 of $ \widebar X$ and $ \widebar X'$ respectively such that
 \begin{enumerate}
 \item[---]
 the $T$-ergodic component over every $\widebar x\in \widebar X_{n,m}$ consists of $n$ points,
  \item[---]
   the
 $T$-ergodic component over every $\widebar x\in \widebar X_{\infty,n}\sqcup\widebar X_{\infty,\infty}$ is nonatomic,
   \item[---]
 the $T'$-ergodic component over every $\widebar x'\in \widebar X_{n}'$ consists of $n$ points,
    \item[---]
    the
 $T'$-ergodic component over every $\widebar x\in \widebar  X'
 _\infty$ is nonatomic,
     \item[---] $\phi(\widebar X_{n,m})\subset \widebar X'_m$ for each $m|n$,
        \item[---]
        $\phi(\widebar X_{\infty,n})\subset \widebar X_n'$ for each $n$,
          \item[---]
          $\phi(\widebar X_{\infty,\infty})= \widebar X'_\infty$ and
           \item[---]
           $T_{\phi(\widebar x)}$ is isomorphic
           to a factor of
           $T_{\widebar x}$ for each $\widebar x\in\widebar X$.
 \end{enumerate}
  It follows from the above reasoning that the restriction  of $T'$
 to the union of components from $\widebar X'_\infty$
 is a factor of the restriction of $T$
  to the union of components from $\widebar X_\infty$ (because the two restrictions are aperiodic).
Constructing the factor mapping from the remaining part of $X$
  to the remaining part of $X'$ is easier because we should make a selection from a  fiber which is now finite and piecewise constant (on countably many pieces).
 For instance, on the ``piece'' corresponding to $\widebar X_{\infty,n}$, we should replace $\widetilde M$ with the set
  $$
  \widetilde M_n:=\{(\widebar x,f)\in \widebar X_{\infty,n}\times F(Y,\Bbb Z/n\Bbb Z) \,:\,  F\circ T_{\widebar x}=F+1 \bmod n\},
  $$
  where $F(Y,\Bbb Z/n\Bbb Z)$ is the group of measurable maps
  from $Y$ to $\Bbb Z/n\Bbb Z$.
  Then the fiber of $\widetilde M_n$ over $\widebar x$ is identified with the group of  the eigenfunctions
  of $T_{\widebar x}$  with eigenvalue $e^{2\pi i/n}$ that take values in the group generated by this eigenvalue.
  \end{proof}

From Proposition~\ref{p:faktory?} we deduce the following.

\begin{Cor}\label{p:random-factor?}
Let $(X,\mu,T)$ be a probability preserving automorphism.
Suppose that $(Y,\nu,R)$ is an ergodic automorphism so that almost every ergodic component  of
$(X,\mu,T)$ (as in \eqref{edecomp}) is a factor of $(Y,\nu,R)$, then  $(X,\mu,T)$ is a factor of $(\baX\times Y,\widebar\mu\otimes\nu,\text{{\rm Id}}\times R)$.
\end{Cor}

\subsection{Factor maps between automorphisms having isomorphic ergodic components}

The following lemma is well known. For the completeness of our argument we provide a short proof of it.
\begin{Lemma}\label{Krieger-non-ergodic}
Let $(X,\mu, T)$ be a probability preserving dynamical system.
Then there is a topological system $(Y,S)$ and a probability measure $\nu$ on $Y$ such that $(X,\mu, T)$ is isomorphic to $(Y,S,\nu)$.
\end{Lemma}
\begin{proof}
Take a Borel one-to-one function $f \colon X\to\Bbb T$.
Define a mapping $\phi \colon X\to\Bbb T^\Bbb Z$ by setting
$$
\phi(x):=(f(T^nx))_{n\in\Bbb Z}.
$$
Let $\nu:=\mu\circ\phi^{-1}$ and let $S$ be the two-sided shift on $\Bbb T^\Bbb Z$.
Then $\phi$ is an isomorphism of $(X,\mu, T)$ onto $(\Bbb T^\Bbb Z,\nu, S)$.
\end{proof}
Let $(X,\mu,T)\in \mathcal F_{\text{\rm ec}}$.
Applying Lemma~\ref{Krieger-non-ergodic}, we may assume without loss of generality that $X$ is a compact Polish space and $T$ is a homeomorphism of $X$.
Let $\mu=\int_{\widebar X}{\mu_{\widebar x}}\,d\widebar\mu(\widebar x)$
be the ergodic decomposition of $\mu$.
We will consider $\widebar X$ as a standard Borel space.

\begin{Prop}\label{p:claimR} Let $(X,\mu)$ and $(Y,\nu)$ be standard probability spaces and let $\pi\colon Y\to X$ be a measurable mapping such that $\mu=\nu\circ\pi^{-1}$.
Then there is a measurable mapping $\phi\colon X\times[0,1]\to Y$
such that $(\mu\otimes\text{{\rm Leb}})\circ\phi^{-1}=\nu$
and $\pi\circ\phi=\tau$, where $\tau\colon X\times[0,1]\to X$ is the first margin mapping.\end{Prop}
\begin{proof}
Let $\nu=\int_{X}\nu_xd\mu(x)$ be the disintegration of $\nu$ with respect to $\mu$ along $\pi$.
Then, because we deal with standard spaces, there exist
a sequence of subsets $X_1\supseteq X_2\supseteq\cdots$ and measurable mappings $y_n \colon X_n\to Y$ such that $\{y_n(x)\colon x\in X_n\}$ is the set of atoms
of $\nu_x$  at a.e.\ $x\in X$.
We now set  $m_n(x):=\nu_x(y_n(x))$ for each $n$ such that $x\in X_n$.
Without loss of generality we may assume that
$m_1(x)\ge m_2(x)\ge\cdots$.
Of course,  $\sum_{n}m_n(x)=:d(x)\le 1$.
Then
there exist
a probability measure $\widetilde\mu$ on $X\times[0,1]$ and  an isomorphism $\theta$ of  $(Y,\nu)$ onto $(X\times[0,1], \widetilde\mu)$ 
such that 
\begin{itemize}
\item
$\widetilde\mu=\int_X\delta_x\otimes\mu_xd\mu(x)$,
\item
where $\mu_x:=\sum_{n=1}^\infty m_n(x)\delta_{m_1(x)+\cdots+m_n(x)}+\text{Leb}_{[d(x),1]}$, and
\item
$\tau\circ\theta=\pi$.
\end{itemize}
Given $x\in X$, we now define a mapping $\varphi_x \colon [0,1]\to[0,1]$ by setting
$$
\varphi_x(t)=
\begin{cases}
m_1(x)
&\text{if } 0\le t< m_1(x), \\
m_k(x)
&\text{if } \sum_{n< k} m_n(x) <t \le \sum_{n\le k} m_n(x), \\
t&\text{if } t\ge d(x).
\end{cases}
$$
A straightforward verification shows that $\mu_x=\text{Leb}_{[0,1]}\circ \varphi_x^{-1}$ for a.e. $x\in X$.
We now define  a mapping
$\varphi \colon X\times[0,1]\to X\times[0,1]$  by setting
$$
\varphi(x,t):=(x,\varphi_x(t)).
$$
Then $\varphi$ is a measurable isomorphism of the space $(X\times[0,1], \mu\otimes\text{Leb})$ onto
$(X\times[0,1],\widetilde\mu)$.
It remains to put
$\phi:=\theta^{-1}\varphi$.
\end{proof}

Let $(X,\mu, T)$ and $(Y,\nu, S)$ be two dynamical systems.
Denote by $(\widebar X,\widebar\mu)$  and $(\widebar Y,\widebar \nu)$ the spaces of ergodic components of
 $(X,\mu, T)$ and $(Y,\nu, S)$ respectively.
 Let $\mu=\int_{\widebar X}\mu_{\widebar x}\,d\widebar\mu(\widebar x)$
 and
 $\nu=\int_{\widebar Y}\nu_{\widebar y}\,d\widebar\nu(\widebar x)$
 stand for the corresponding disintegrations.
 If $\psi \colon Y\to X$ is a factor mapping
 then $\psi$ induces a measurable map $\widebar\psi \colon (\widebar Y,\widebar \nu)\to(\widebar X,\widebar\mu)$ and
 a measurable field $\widebar Y\ni
 \widebar y\mapsto\psi_{\widebar y}$ of factor mapping $\psi_{\widebar y} \colon
 (Y,\nu_{\widebar y}, S)\to (X,\mu_{\widebar\psi(\widebar y)},T)$.


\begin{Prop}\label{p:Claim}
Let $(X,\mu, T)$ and $(Y,\nu, S)$
be two dynamical systems and let
$(\widebar X,\widebar\mu)$  and $(\widebar Y,\widebar \nu)$
stand for the spaces of their ergodic components.
Let $\widebar\psi \colon (\widebar Y,\widebar \nu)\to(\widebar X,\widebar\mu)$
be a measurable  onto mapping
and
let
$$
\psi_{\widebar y} \colon
 (Y,\nu_{\widebar y}, S)\to (X,\mu_{\widebar\psi(\widebar y)},T)
 $$
 be a measurable field of isomorphisms.
 Then, there exists a measurable mapping
 $\theta\colon(X\times[0,1],\mu\otimes\text{Leb})\to(Y,\nu)$
 such that $\theta\circ (T\times I)=S\circ\theta$.
\end{Prop}
\begin{proof}
By Proposition~\ref{p:claimR}, there is a measurable mapping
 $$
 \phi\colon (\widebar X\times[0,1], \widebar\mu\otimes\text{Leb})\to(\widebar Y,\widebar \nu)
 $$
  such that
 $\widebar\psi\circ\phi=\tau$, where $\tau\colon\widebar X\times[0,1]\to \widebar X$ is the first margin projection.
 We now set
 $$
 \theta(x,t):={\psi_{ \phi(\widebar x,t)}}^{-1}(x).
 $$
 for all $(x,t)\in X\times[0,1]$.
 Then
 $\theta(Tx,t)={\psi_{\phi(\widebar x,t)}}^{-1}(Tx)=S{\psi_{\phi(\widebar x,t)}}^{-1}(x)=
 S\theta(x,t)$.
 Let $\widebar y=\phi(\widebar x,t)$.
 Then $\widebar\psi\circ\phi(\widebar x,t)=\psi(\widebar y)$.
 Hence, $\widebar x=\widebar\psi(\widebar y)$.
 Therefore, the mapping $\psi_{\widebar y}^{-1}$ is an isomorphism of
$(X,\mu_{\widebar\psi(\widebar y)})$ onto  $(Y,\nu_{\widebar y})$.
The two spaces are fibers in $X$ and $Y$ over
$\widebar\psi(\widebar y)$ and $\widebar y$ respectively.
On the other hand,
$\widebar\mu=\widebar\nu\circ\widebar\psi^{-1}$
and  $\theta$ is constant on the fibers $\tau^{-1}(\{\widebar x\})$, $\widebar x\in \widebar X$.
We obtain  that $(\mu\otimes\text{Leb})\circ \theta^{-1}=\nu$.
\end{proof}

We would like to draw attention of the reader to  the difference between Proposition~\ref{p:Claim} and Proposition~\ref{p:faktory?}: the  factor mappings defined there ``act in the opposite directions''.

\subsection{Joinings of non-ergodic automorphisms}

Given two dynamical systems $(X,\mathcal B,\mu, T)$
and $(Y,\mathcal C,\nu, S)$, a {\it joining} of them is a $(T\times S)$-invariant measure on the product space $X\times Y$ such that the marginal projections of $\rho$ to $X$ to $Y$ are $\mu$ and $\nu$ respectively.
The dynamical system $(X\times Y,\mathcal B\otimes\mathcal C,\rho,T\times S)$ is also called a  joining of  $(X,\mathcal B,\mu, T)$
and $(Y,\mathcal C,\nu, S)$.
The set of all joinings of $T$ and $S$ is denoted by $J(T,S)$.
Given $\rho\in J(T,S)$, a linear operator $\Phi_\rho \colon  L^2(X,\mu)\to L^2(Y,\nu)$
is well defined by
\begin{equation}\label{eq:joi}
\langle \Phi_\rho f,g\rangle=\int_{X\times Y}f\otimes g\,d\rho.
\end{equation}
For all $f\in L^2(X,\mu)$ and $g\in L^2(Y,\nu).
$
Then
\begin{enumerate}[(i)]
\item
$\Phi_\rho (L^2(X,\mu)_+)\subset L^2_+(Y,\nu)$,
\item
$\Phi_\rho (1)=1=\Phi_\rho^* (1)$ and
\item
$\Phi_\rho (f\circ T)=(\Phi_\rho f)\circ S$ for each $f\in L^2(Y,\nu)$.
\end{enumerate}
Conversely, each bounded linear operator satisfying (i)--(iii) determines a joining $\rho\in J(T,S)$ such that (\ref{eq:joi}) holds.
We will write $J_2(T)$ for $J(T,T)$.

 Assume that $T\in {\rm Aut}\xbm$.
 Consider the ergodic decomposition of $T$ and
 let (a)-(c)  from \S\ref{s:ergcom} hold.
 Denote by $C_2(\nu)$ the set of 2-couplings on $Y$, i.e. the set of probability measures on $Y\times Y$ whose marginals are both equal to $\nu$ (i.e. the self-joinings of the identity on $Y$).
 Endow  $C_2(\nu)$ with the weak topology.
 Fix a coupling $\lambda\in C_2(\bamu)$  and a measurable  mapping
$$
\baX\times\baX\ni (\bax,\bax')\mapsto \rho_{\bax,\bax'}\in J(T_{\bax},T_{\bax'})\subset C_2(\nu)
$$
be measurable.
  Then, we define a measure
$$
\rho=\int_{\baX\times \baX}\delta_{\bax,\bax'}\otimes\rho_{\bax,\bax'}\,d\lambda(\bax,\bax').
$$
on the space $\baX\times\baX\times Y\times Y$ which is identified with
$X\times X$ via (a)
(see \cite{Go-Le-Ru} for the case $\lambda=\bamu\ot\bamu$).
Of course, $\rho\in J_2(T)$.
Self-joinings of that kind are called {\em canonical}.

\begin{Lemma}\label{l:canonical}   Every joining $\rho\in J_2(T)$
 is canonical (over $\lambda:=\rho|_{\baX\times \baX}$).\end{Lemma}
\begin{proof}Consider the disintegration of $\rho$ over
the factor $\baX\times \baX$:
there exists  a measurable mapping $\baX\times \baX\ni (\bax,\bax')\mapsto\rho_{\bax,\bax'}$ to the space of probability measures on $Y\times Y$ such that
\beq\label{rro}
\rho=\int_{\baX\times \baX}\delta_{\bax,\bax'}\otimes\rho_{\bax,\bax'}\,d\lambda(\bax,\bax').
\eeq
Let $\rho_{\bax,\bax'}^{(i)}$ denote the marginal of the fiber measure $\rho_{\bax,\bax'}$ onto the $i$-th coordinate, $i=1,2$.
Since $T\times T$ acts as identity on $\baX\times \baX$ (although this need not be the whole sigma-algebra of invariant sets for $\rho$), it follows that  $\rho_{\bax,\bax'}$ are $(T_{\bax}\times T_{\bax'})$-invariant.
Therefore, we only need to prove that
 $\rho_{\bax,\bax'}^{(i)}=\nu$ for each $i=1,2$ at $\lambda$-a.e. $(\bax,\bax')$.
Disintegrate $\lambda $ with respect to the first coordinate:
\beq\label{lla}
\lambda=\int_{\baX}\delta_{\bax}\otimes\lambda_{\bax}\,d\bamu(\bax).
\eeq
Substituting \eqref{lla} into \eqref{rro}, we obtain that
\beq\label{subst}
\rho=\int_{\baX}\delta_{\bax}\otimes\int_{\baX}(\delta_{\bax'} \otimes\rho_{\bax,\bax'})\,d\lambda_{\bax}(\bax')\,d\bamu(\bax).
\eeq
Applying the projection
$$
\baX\times\baX\times Y\times Y\ni(\bax,\bax',y,y')\mapsto (\bax,y)\in \baX\times Y
$$
 to the both sides of \eqref{subst} we obtain that
$$
\bamu\otimes\nu=
\int_{\baX}\delta_{\bax}\otimes\int_{\baX}\rho_{\bax,\bax'}^{(1)}\,d\lambda_{\bax}(\bax')\,d\bamu(\bax).
$$
Hence, by the uniqueness of disintegration, we obtain that
$$
\nu=\int_{\baX}\rho_{\bax,\bax'}^{(1)}\,d\lambda_{\bax}(\bax')
$$
at $\bamu$-a.e. $\bax$.
Since $\rho_{\bax,\bax'}^{(1)}\circ T_{\bax}=\rho_{\bax,\bax'}^{(1)}$ for
all $\bax'$ and $\nu$ is invariant and ergodic under
$T_{\bax}$, it follows that $\rho_{\bax,\bax'}^{(1)}=\nu$ for $\lambda_{\bax}$-a.e. $\bax'$.
Hence, by the Fubini theorem, $\rho_{\bax,\bax'}^{(1)}=\nu$ at $\lambda$-a.e.
$(\bax,\bax')$.
In a similar way one can prove that
 $\rho_{\bax,\bax'}^{(2)}=\nu$ at $\lambda$-a.e.
$(\bax,\bax')$.
 \end{proof}

\subsection{Mutual spectral singularity}
\label{f:spectral}
Given an automorphism $T$ acting on $\xbm$ and $f\in L^2\xbm$, the measure $\sigma_f$ on $\Bbb T$ whose Fourier transform is $\widehat{\sigma}_f(n)=\int_Xf\circ T^n\cdot\ov{f}\,d\mu$ for all $n\in\Z$, is called the {\em spectral measure} of $f$. Among spectral measures there are maximal ones (with respect to the absolute continuity relation). All maximal spectral measures are equivalent and their type is called the {\em maximal spectral type} of $T$. A classical fact in spectral theory is that for each $f,g\in L^2\xbm$
\beq\label{mutualsing}
\mbox{if $\sigma_f\perp \sigma_g$, then $f\perp g$.}\eeq
We will also need the following well-known fact.

\begin{Lemma}\label{l:mutualsing}
Assume that $T$ is an automorphism of $\xbm$ and let $\mu=\int_{\widebar X}\mu_{\widebar x}\,d\widebar \mu(\widebar x)$ be the ergodic decomposition of $\mu$.
Then for each $g\in L^2\xbm$,
\begin{equation}\label{rozk_spektr}
\sigma_g=\int \sigma_{g,\widebar x}\, d\widebar \mu(\widebar x),
\end{equation}
where $\sigma_{g,\widebar x}$ stands for the spectral measure of $g$ for the automorphism $(T,\mu_{\widebar x})$.\end{Lemma}
\begin{proof}
Note that the map $\widebar x\mapsto \sigma_{g,\widebar x}$ is measurable (cf.\ \cite{Be-Go-Ru}). Since
\begin{multline*}
\int z^n\,d\sigma_g=\int g\circ T^n\cdot\ov{g}\,d\mu\\
=\int\Big(\int g\circ T^n\cdot\ov{g}\,d\mu_{\widebar x}\Big)\,d\widebar\mu(\widebar x)=\int\Big(\int z^n\,d\sigma_{g,\widebar x}\Big)\,d\widebar\mu(\widebar x),
\end{multline*}
for each $j\in L^2(\sigma_g)$, we have $\int j(z)\,d\sigma_g(z)=\int\Big(\int j(z)\,d\sigma_{g,\widebar x}\Big)\,d\widebar\mu(\widebar x)$.
\end{proof}

\section{On the URE property}\label{s:AC}

\subsection{Basic properties and examples of URE automorphisms}\label{ss:Tim's}
Prior to an overview of examples and counterexamples let us give some general properties of the class of uniformly relatively ergodic automorphisms.
First of all, notice that an automorphism is URE if and only if  its aperiodic part is URE, so in what follows we tacitly assume that we consider only the aperiodic case.

\begin{Prop}\label{p:inac}
\begin{enumerate}
\item[{\rm (A)}] The class of URE automorphisms is closed under taking factors.
\item[{\rm (B)}] Let $T$ be an automorphism.
The following are equivalent.
\begin{enumerate}
\item[{\rm (i)}] $T$ is URE.
\item[{\rm (ii)}]
There is a set of full measure of ergodic components of $T$ having  a common ergodic extension.
\item[{\rm (iii)}]
 There is a countable set of ergodic components $(T,\mu_{\ov{x}_i})$, $i\geq1$ and $\rho\in J^e((T,\mu_{\ov{x}_1}),(T,\mu_{\ov{x}_2}),\ldots)$ such that $T$ is a factor of the product $\text{{\rm Id}}\times (\prod_{i\geq1}(T,\mu_{\ov{x}_i}),\rho)$.
\end{enumerate}
\item[{\rm (C)}] If $T\in \cf_{\rm ec}$ and  $T$ is URE then $T\in\cf$.

\end{enumerate}
\end{Prop}
\begin{proof}(A): This is obvious.

(B): (ii) $\Rightarrow$ (i) This implication follows from Corollary~\ref{p:random-factor?}.

(i) $\Rightarrow$ (ii) This follows from the fact that ergodic components of a factor are isomorphic to factors of ergodic components.

(i) $\Rightarrow$ (iii) (This is standard and follows from the separability of the space $(\cb,\mu)$, but we provide a proof for completeness.) We need to prove that given an automorphism $R$ acting on $\zdk$ and a family of factors $\ce_i\subset\cd$, $i\in I$, generating $\cd$, we can find a countable $I_0\subset I$ such that $\bigvee_{i\in I_0} \ce_i=\cd$. For this aim, note that $\cs:=\bigcup_{i\in I}\ce_i$ is closed under taking complement, so the semi-ring $\widetilde\cs$ generated by $\cs$ consists of finite intersections of elements from $\cs$. Then the corresponding algebra $\ca(\widetilde{\cs})$  consists of finite disjoint unions of sets from  $\widetilde\cs$. Since the space $(\cd,\kappa)$ (sets are considered modulo measure $\kappa$) is separable, we can choose a countable subset $\{F_j\}_{j\geq1}$ dense in $\ca(\widetilde\cs)$. By the form of the sets $F_j$, it follows that there exists a countable set $I_0\subset I$ such that $F_j\in  \bigvee_{i\in I_0}\ce_i$ for all $j\geq1$. Since $\cd$ is the closure of $\ca(\widetilde{\cs})$, the set $\{F_j\colon j\geq1\}$ is also dense in $(\cd,\kappa)$ and the claim follows.

(iii) $\Rightarrow$ (i) This follows straightforwardly from the definition of URE.

(C): This follows from~(B)(iii) since $T\in \cf_{\rm ec}$ is a factor of  $Id\times R$, where $R=(\prod_{i\geq1}(T,\mu_{\ov{x}_i}),\rho)$ for an ergodic joining $\rho$  of a family of $T$-ergodic components
$(T,\mu_{\ov{x}_i})\in \cf$, $i=1,2,\dots$.
Hence, $R\in \cf$.
This implies, in turn, that $T\in \cf$.
\end{proof}

\begin{Example}\label{ex:triv} Denote by $\mathcal{T}=\{T_i\}_{i\in I}$ the family of ergodic components of an automorphism $T$ of a standard probability space.

(i) If $|\mathcal{T}|\leq\aleph_0$ then $T$ is URE (consider any ergodic joining of $T_i$s).

(ii)  If $\mathcal{T}$ is uncountable and pairwise disjoint then $T$ cannot be URE.\footnote{Otherwise, in the $L^2$-space of the common extension $S$ (see Proposition~\ref{p:inac}), the automorphisms $T_i$ are represented by pairwise orthogonal $L^2$-subspaces which contradicts separability. In particular, by \cite{Go-Le-Ru}, the only URE automorphism disjoint with ERG is an identity.}
\end{Example}

\begin{Example} Let each ergodic component of an automorphism $(X,\mu,T)$ is Bernoulli.
Then $T$ is URE.
Indeed, each ergodic component is a factor of the Bernoulli shift of infinite entropy.
\end{Example}

\begin{Example}\label{e:e2} Assume that $\mu=\int_{\ov{X}}\mu_{\ov{x}}\,d\ov{\mu}(\ov{x})$ is the ergodic decomposition of an automorphism $T\in\cf_{\rm ec}$ and let $\ov{X}=\bigsqcup_{j\geq1}\ov{X}_j$ be a (measurable) countable partition of $\ov{X}$ such that, for each $j\geq1$ all ergodic components $(T,\mu_{\ov{x}})$, $\ov{x}\in \ov{X}_j$, are isomorphic to an (ergodic) automorphism $S_j\in\cf$. Take any ergodic joining $(S_1\times S_2\times\ldots,\rho)$ which is still an element of $\cf$. Then a.a.\ $(T,\mu_{\ov{x}})$ are factors of this joining and therefore $T$ is  $\cf$-URE. (Using Corollary~\ref{p:random-factor?}, see also Corollary~\ref{c:jon17}, it is not hard to see that $T$ will be a factor of $Id\times R$ with $R\in\cf$, so $T$ also belongs to $\cf$.)\end{Example}
This example might suggest that URE is the countability of ergodic components up to isomorphism, however this is not the case as the following two examples show.

\begin{Example}\label{e:e21}
Assume that we have a family $\{R_n\}$ of weakly mixing automorphisms defined on $\zdk$ which are pairwise disjoint and all of them enjoy the PID property \cite{Ju-Ru} (for example, $R_n=R^n$, $n\geq1$, where $R$ has the MSJ property).
Given a set $A\subset\N$, let $R_A$ denote for the Cartesian product
$$
R_A:=\prod_{n\in A}R_n\text{ considered with product measure.}$$
We obtain an uncountable family of automorphisms which have a common ergodic  extension, namely $\prod_{n\geq1}R_n$.
It follows from our assumptions that if $A\ne B$ then  $R_A$ and $R_B$ are not isomorphic.
Let $\Omega=\{0,1\}^{\N}$ considered with the Bernoulli  measure $\la:=\{0.5,0.5\}^{\otimes\Bbb N}$.
Then for $\la$-a.a. $\omega=(\omega_n)\in \Omega$, the support $A_\omega:=\{n\in\Bbb N\colon \omega_n\ne 0\}$ of $\omega$ is infinite.
Define now
an automorphism $T$ of the product space $(\Omega,\la)\ot(Z^{\N},\kappa^{\ot\N})$
by setting
$$
T(\omega,(z_n)_{n=1}^\infty):= (\omega,(R_{k_n}z_n)_{n=1}^\infty),
$$
where $\{k_1,k_2,\ldots\}=A_{\omega}$ and $k_1<k_2<\cdots$.
It is not hard to see that $(\Omega,\lambda)$ is identified naturally with the space of ergodic components of $T$ and for a.e. $\omega\in \Omega$, the corresponding (to $\omega$) ergodic component is isomorphic to $R_{A_\omega}$.
Thus, the family of ergodic components of $T$ admits (a.e.) a common extension.
Hence, $T$ is URE.
On the other hand, two ergodic components corresponding to $\omega$ and $\omega'$ respectively are isomorphic if and only if
 $A_\omega=A_{\omega'}$, i.e. $\omega=\omega'$.
 Hence, uncountably many ergodic components of $T$ are pairwise non-isomorphic.
\end{Example}

\begin{Example}\label{e:e21'} In the previous example, all ergodic components of $T$ are weakly mixing.
It is possible to construct a similar example  whose ergodic components have pure point spectrum.
For instance, fix a countable family $\{\lambda_n\colon n\in\Bbb N\}$ of  rationally independent irrationals.
Let $R_n$ stand for the $\lambda_n$-rotation on the circle.
Utilizing this new family $\{R_n\}_{n\in\Bbb N}$, we
repeat the construction from Example~\ref{e:e21} almost literally.
Then the corresponding  (to the new family) transformation $T$ is URE, it has uncountably many ergodic components, each of which has pure point spectrum.
Hence $T\in {\rm{\bf{DISP}}}_{\rm ec}$.
By Proposition~\ref{p:inac}(C), $T\in {\rm{\bf{DISP}}}$.
\end{Example}

\begin{Remark} We note that though the automorphism $T$ in Example~\ref{e:e21} and Example~\ref{e:e21'} has uncountably many pairwise non-isomorphic ergodic components, $T$
is isomorphic to the inverse limit of a sequence of automorphisms with finitely many ergodic components.
Indeed, for $N\in\Bbb N$, let $\Omega_N:=\{0,1\}^N$,
$\lambda_N:=\{0.5,0.5\}^{\otimes N}$.
Given $\omega\in\Omega_N$, let
$$
A_\omega:=\{n\in\{1,\dots,N\}\colon \omega_n\ne 0\}.
$$
We enumerate the elements of $A_\omega$ as follows:
$k_1(\omega)<k_1(\omega)<\cdots<k_{r(\omega)}(\omega)$, where  ${r(\omega)}$ is the cardinality of
$A_\omega$.
We now define a probability space $(X_N,\mu_N)$ in the following way:
$$
X_N:=\{(\omega,y)\colon \text{for all }\omega\in\Omega_N, y\in Z^{r(\omega)} \},\quad\mu_N:=\int_{\Omega_N} \delta_\omega\otimes \kappa^{\otimes r(\omega)}\,d\lambda_N(\omega).
$$
Next, we define an automorphism $T_N$ of $(X_N,\mu_N)$ by setting
$$
T_N(\omega,(z_n)_{n=1}^{r(\omega)}):= (\omega,(R_{k_n(\omega)}z_n)_{n=1}^{r(\omega)}).
$$
Of course, $T_N$ has finitely many  ergodic components.
The space of  these components is identified with  $\Omega_N$.
Consider the natural projections
$$
\pi_N \colon X_{N+1}\ni \big(\omega,(z_n)_{n=1}^{r(\omega)}\big)\mapsto \big(\omega|_{\{1,\dots,N\}},
(z_n)_{n=1}^{
r(\omega|_{\{1,\dots,N\}})}
\big)\in X_N,
$$
$N\in\Bbb N$.
Then $\mu_{N+1}\circ \pi_N^{-1}=\mu_N$  and $\pi_N T_{N+1}=T_N\pi_N$ for each $N$.
Thus, an inverse sequence of dynamical systems
$$
\begin{CD}
(X_1,\mu_1,T_1) @<\pi_1<< (X_2,\mu_2,T_2) @<\pi_2<< \cdots
\end{CD}
$$
is well defined.
It is straightforward to verify that the inverse limit of this system is $(X,\mu, T)$ from Example~\ref{e:e21}
or Example~\ref{e:e21'}.
\end{Remark}



\subsection{Ergodic decomposition of Rokhlin extensions and URE}

In this subsection we provide a class of Rokhlin extensions which are URE.
Let $G$ be  a locally compact second countable group and let $H$ be a co-compact
subgroup in the center of $G$.
It follows that $G$ is unimodular.
By $\lambda_G$ we denote a Haar measure on $G$.
Let $\cs=(S_g)_{g\in G}$ be a continuous action of  $G$ on a compact metric space $Y$.
Denote by $M(Y,\cs)$ and $M(Y, \cs|H)$ the simplices of $G$-invariant
 and $H$-invariant probability measures  on $Y$, respectively.
 Of course, the quotient group $G/H$ acts naturally on $M(Y, \cs|H)$:
 $$
 \nu\circ (gH):=\nu\circ S_g.
 $$
The following claim is well known.

\begin{Lemma}\label{l.an}
The mapping
$$
M(Y,\cs|H)\ni \nu\mapsto \int_{G/H}\nu\circ (gH)\,d\lambda_{G/H}(gH)\in M(Y, \cs)
$$
is an affine isomorphism.
\end{Lemma}

Moreover, given $\nu\in M^e(Y, \cs|H)$, let $G_\nu:=\{g\in G\colon \nu\circ g=\nu\}$.
Then $G_\nu$ is a co-compact subgroup of $G$ because
this subgroup contains $H$.
The composition $\nu\circ (G_\nu g)$ is well defined for each $g\in G$.
Of course,
$$
\nu\circ (G_\nu g)=\nu\circ gH
$$
 and $\nu\circ (G_\nu g)\perp \nu\circ (G_\nu g_1)$
whenever $gg_1^{-1}\not\in G_\nu$.
It follows that if
\begin{equation}\label{999}
\rho:=\int_{G_\nu\backslash G}\nu\circ z\,d\lambda_{G_\nu\backslash G}(z)
\end{equation}
then $\rho\in  M(Y, \cs|H)$ with~(\ref{999}) being its ergodic decomposition, $\nu\circ z\in M^e(Y,\cs|H)$ for each $z\in G_\nu\backslash G$.
Moreover, $\rho\in M^e(Y,\cs)$.

Assume that $T$ is a uniquely ergodic homeomorphism of a compact metric space $X$.
Denote by $\mu$ the unique ergodic $T$-invariant Borel probability measure on $X$.
Suppose that the above action $\cs$ is additionally uniquely ergodic.
Denote by $\nu$ the corresponding unique ergodic $\cs$-invariant probability measure on $Y$.
Fix a continuous cocycle $\phi \colon X\to G$.
Suppose that $\phi$ is $H$-regular, i.e.\ there is a Borel mapping $a \colon X\to G$
such that the cocycle $\psi:= a\circ T \cdot \phi  \cdot  a^{-1}$ of $T$ takes its values in $H$
and is ergodic (as an $H$-valued cocycle of $T$).
We define a homeomorphism  $T_{\phi, \cs} \colon X\times Y\to X\times Y$ by setting
$$
T_{\phi, \cs}(x,y):=(Tx, S_{\phi(x)}y).
$$
Our purpose is to prove the following theorem.

\begin{Th}\label{t:sm}
Under the aforementioned assumptions,
\begin{equation}\label{eq:cat}
M^e(X\times Y,T_{\psi,\cs})=\{\mu\otimes \kappa\colon\kappa\in  M^e(Y,\cs|H)\}
\end{equation}
and
the dynamical systems $(X\times Y, \alpha, T_{\phi, \cs})$,  $\alpha\in M^e(X\times Y, T_{\phi, \cs})$, are pairwise isomorphic.  In particular, $T_{\phi,\cs}$ is URE.
\end{Th}

\begin{proof}
First of all, consider a Borel mapping
$$
Q \colon X\times Y\ni (x,y)\mapsto(x,S_{a(x)}y)\in X\times Y.
$$
A straightforward verification yields that $QT_{\phi, \cs}Q^{-1}=T_{\psi, \cs}$.
Therefore, it suffices to describe the extremal points of the simplex $M(X\times Y, T_{\psi, \cs})$.
Since $T$ is uniquely ergodic, the first coordinate projection of every measure $\rho\in M(X\times Y, T_{\psi, S})$ is $\mu$.
On the other hand, the Mackey action  associated with $(X, T,\psi)$ is
the $G$-action by rotations on the quotient group $(G/H, \lambda_{G/H})$, where $\lambda_{G/H}$ is the Haar measure on $G/H$.
Denote this action by $\mathcal{W}=(W_g)_{g\in G}$.
It follows from Proposition~6~in \cite{Da-Le} that there is an affine isomorphism
$\Lambda$ of the simplex $M(G/H\times Y, \mathcal{W}\times \cs)$ onto the simplex
$M(X\times Y, T_{\psi, \cs})$\footnote{We use here the fact that $\mathcal{W}$  is uniquely ergodic on $G/H$.}.
Take $\rho\in M(G/H\times Y, \mathcal{W}\times \cs)$.
Since $\mathcal{W}$ is a factor of the dynamical system $(G/H\times Y, \mathcal{W}\times \cs,\rho)$,
it follows that  there exists a unique measure  $\kappa\in M(Y,\cs|H)$ such that $(G/H\times Y, \mathcal{W}\times \cs,\rho)$ is induced  (in G.~Mackey sense) from a dynamical system $(Y, \kappa, \cs|H)$, i.e.
$$
\rho=\int_{G/H}\delta_{gH}\otimes(\kappa\circ(gH))\,d\lambda_{G/H}(gH).
$$
It follows from the proof of  Proposition~6 in \cite{Da-Le} that
$\Lambda(\rho)=\mu\otimes\kappa$.
On the other hand, the mapping
$$
\Omega \colon M(G/H\times Y, \mathcal{W}\times \cs)\ni\rho\mapsto\kappa\in M(Y,\cs|H)
$$
is an affine isomorphism.
Hence, $\Lambda\circ \Omega^{-1}$ is an affine isomorphism of $M(Y,\cs|H)$
onto $M(X\times Y,T_{\psi, \cs})$ and for each $\kappa\in M(Y,\cs|H)$, we have that $\Lambda\circ \Omega^{-1}(\kappa)
=\mu\otimes\kappa$.
This yields (\ref{eq:cat}).

Take $\kappa\in M^e(Y,\cs|H)$.
Since $M(Y,\cs)=\{\nu\}$, it follows from Lem\-ma~\ref{l.an} that
$\int_{G/H}\kappa\circ (gH)\,d \lambda_{G/H}(gH)=\nu$.
Hence,
$\int_{G_\kappa\backslash G}\kappa\circ z\,d \lambda_{G_\kappa\backslash G}(z)=\nu$ is the ergodic decomposition of $(Y,\nu, (S_h)_{h\in H})$.
Thus, $\kappa$ is an ergodic component in the $\cs(H)$-ergodic decomposition of $\nu$.
Hence,
$$
M^e(Y, \cs\mid H)=\{\kappa\circ z\colon z\in G_\kappa\backslash G\}.
$$
We note that  the dynamical systems $(Y, \kappa\circ z, \cs|H)$, $z\in G_\kappa\backslash G$ are pairwise isomorphic.
Indeed, for each $g\in G$, the transformation $S_g$ conjugates $(Y, \kappa, \cs|H)$ with $(Y, \kappa\circ (G_\kappa g), \cs|H)$.
This implies, in turn, that the systems
$(X\times Y, T_{\psi, \cs}, \mu\otimes \kappa)$ and
$(X\times Y, T_{\psi, \cs}, \mu\otimes \kappa\circ z)$ are isomorphic for each
$z\in G_\kappa\backslash G$.
\end{proof}

In particular, if $Y=G$,  $G$ is compact Abelian, $S$ is the $G$-action on itself by rotations and $\phi$ is arbitrary then $T_{\phi,S}$
is an  Anzai skew product.

\subsection{Non-URE extensions and Cartesian products of URE-automor\-phisms}

     The URE property is a generalization of the {\it countablity of the set of ergodic components}.
However, despite the fact that  a finite extension of an automorphism with countably many ergodic components still have countably many ergodic components and, hence, is URE,
a finite extension of an {\it arbitrary} URE automorphism need not be URE as the following example shows.

\begin{Example}\label{e:22}
Let $T$ be an ergodic discrete spectrum automorphism on $(X,\mu)$.
Assume that for each $\omega\in\Omega:=\{0,1\}^\Bbb N$, there is a cocycle
$\phi_\omega$ for $T$ with values in $\Bbb Z/2\Bbb Z$ such that:

\begin{enumerate}
\item[---]
the mapping $\Omega\ni\omega\mapsto\phi_\omega$ is measurable,
\item[---]
the cocycle $\phi_\omega$ is ergodic for each $\omega$,
\item[---]
the skew product extensions $T_{\phi_\omega}$
and  $T_{\phi_{\omega'}}$ are not isomorphic if the sequences $\omega$ and $\omega'$
are not tail equivalent.\footnote{For a concrete example of such a situation, see e.g.\ \cite[\S4]{Kw},   (Kakutani's examples).}
\end{enumerate}

Let $\lambda$ denote the Bernoulli measure $\{0.5,0.5\}^{\otimes\Bbb N}$ on $\Omega$.
Consider the dynamical system  $(\Omega\times X,\lambda\otimes\mu, \text{Id}\times T)$.
By definition, it is URE.
On the other hand, consider
the following 2-point extension $R$ of $\text{Id}\times T$ acting on $(\Omega\times X\times\Z/2\Z,\la\ot\mu\ot\la_2)$ (with $\lambda_2$ being the $(0.5,0.5)$-measure) of this system:
$$
R(\omega, x, k):=(\omega, Tx, k+\phi_\omega(x)).
$$
We claim that $R$ is not URE.
Indeed, the space of ergodic components of $R$ is $(\Omega,\lambda)$.
For each $\omega\in\Omega$, the ergodic component of $R$
corresponding to $\omega$
 is the  $\phi_\omega$-skew product extension $T_{\phi_\omega}$ of $T$.
Take two points $\omega$ and $\omega'$ of $\Omega$.
If they are not tail equivalent then $T_{\phi_\omega}$ and $T_{\phi_{\omega'}}$ are not isomorphic.
It follows from Lemma~\ref{l:strongRD} below that $T_{\phi_\omega}$ and $T_{\phi_{\omega'}}$ are relatively disjoint
over the common factor $T$ in a strong sense.
Therefore,
if there exists an ergodic transformation $S$ on a standard probability space $(Y,\nu)$ containing
all systems $(X\times\Bbb Z/2\Bbb Z,T_{\phi_\omega},\mu\otimes \lambda_2)$, $\omega\in \Omega$, as factors then  $L^2(Y,\nu)$ contains $L^2(X\times\Bbb Z/2\Bbb Z, \mu\otimes \lambda_2)$ as closed subspaces for all $\omega\in \Omega$.
Since $T$ has pure point spectrum,  $(X,\mu, T)$ is a common factor for all dynamical systems
$(X\times\Bbb Z/2\Bbb Z,T_{\phi_\omega},\mu\otimes \lambda_2)$ inside $(Y,\nu)$.
Thus, we obtain that the Hilbert space $L^2(Y,\nu)$ contains uncountably many closed subspaces
$L^2(X\times\Bbb Z/2\Bbb Z, \mu\otimes \lambda_2)$ each of which includes a common
subspace $L^2(X,\mu)$.
Moreover,  the orthocomplements
$L^2(X\times\Bbb Z/2\Bbb Z, \mu\otimes \lambda_2)\ominus L^2(X,\mu)$
are pairwise orthogonal in $L^2(Y,\nu)$ by Lemma~\ref{l:strongRD} below.
Hence,  $L^2(Y,\nu)$
would not be separable, a contradiction.
\end{Example}

\begin{Lemma}\label{l:strongRD} Assume that $T\in{\rm Aut}\xbm$ is an ergodic automorphism with discrete spectrum and let $\phi,\psi \colon X\to\Z/2\Z$ be two ergodic cocycles. If $T_\phi$ and $T_\psi$ are not isomorphic then $T_\phi$ and $T_\psi$ are (strongly) relatively disjoint in the following sense: Let $\rho\in J(T_\phi,T_\psi)$. Then $\rho$ is the relatively independent extension of its restriction $\rho|_{X\times X}$:
$$
\rho=\widehat{\rho|_{X\times X}},
$$
that is, $\rho=\rho|_{X\times X}\otimes \la_2\otimes \la_2$.
Hence,
$\int F\ot G\,d\rho=0$ for all functions $F,G\in L^2(X\times \Z/2\Z,\mu\ot\la_2)\ominus L^2(X,\mu)$.\end{Lemma}

\begin{proof}
In view of \cite{Le-Me}, the fact that $T_\phi$ and $T_\psi$ are not isomorphic is equivalent to the ergodicity of
all cocycles
$$
X\ni x\mapsto (\phi(Wx),\psi(x))\in \Z/2\Z\times\Z/2\Z,\;W\in C(T),$$
that is, the measure $\mu\ot\la_{2}\ot\la_{2}$ is the only $T_{(\phi\circ W,\psi)}$-invariant measure projecting on $\mu$ (and it is ergodic).
The only ergodic self-joinings of $T$ are graph measures $\mu_W$ with $W\in C(T)$:
$$
\mu_W(A\times B):=\mu(A\cap WB) \quad \text{for all Borel subsets $A,B\subset X$.}
$$
It follows that if $\rho\in J^e(T_\phi,T_\psi)$ then there is $W\in C(T)$ such that
$$
\rho=\mu_W\otimes\la_{2}\otimes\la_{2}.
$$
Therefore, disintegrating an arbitrary joining $\rho\in J(T_\phi,T_\psi)$ into ergodic components, we obtain
that
\begin{align*}
\rho&=\int_{C(T)}\mu_W\otimes\la_{2}\otimes\la_{2}\,d\PP(W)\\
&=
\left(\int_{C(T)}\mu_W\,d\PP(W)\right)\otimes\la_{2}\otimes\la_{2}\\
&=\rho|_{X\times X}\otimes\la_{2}\otimes\la_{2}\\
&=\widehat{\rho|_{X\times X}},
\end{align*}
as desired.
\end{proof}

\begin{Remark} It $T$ is URE, then the Cartesian square of $T$ need not be URE.
Indeed, let $X=\Bbb T^2$, $\mu$ be the Haar measure on $X$ and $z$ be an aperiodic element of $\Bbb T$.
We define two  transformations $T$ and $S$ on $(X,\mu)$ by setting
$$
T(x,y):=(zx, yx),\quad S(x,y):=(x,yx)
$$
for all $(x,y)\in\Bbb T^2$.
It is straightforward to verify that  the mapping
$$
(x_1,y_1,x_2,y_2)\mapsto (x_2x_1^{-1},y_2y_1^{-1})
$$
establishes a factor map between $(X\times X,\mu\otimes\mu,T\times T)$ and
$(X,\mu,S)$.
According to Proposition~\ref{p:inac}(A),
 if $T\times T$ were URE, then so would be $S$.
 However the latter  is impossible in view of Example~\ref{p:inac}(ii) as the the set of ergodic components of
 $S$ includes  all irrational rotations (uncountable subfamily of which are mutually disjoint).
 \end{Remark}

\subsection{Around  piecewise Markov URE.  Proofs of Proposition~\ref{p:jjm} and Theorem~\ref{t:ajmt}}\label{s:tralala}
We first prove Proposition~\ref{p:jjm} and then extend it to the
piecewise Markov URE systems.
To this purpose,
we  observe that
thanks to  Corollary~\ref{c:jon17}, we can assume without loss of generality that
there exists a probability space $(\baX,\ov{\cb},\bamu)$ such that
$T=\text{Id}_{\baX}\times R$, where $R\in\,$Erg.
(We use here the notation from the statement of Proposition~\ref{p:jjm}.)
Then
Proposition~\ref{p:jjm} follows straightforwardly from the following result which has been already obtained in \cite[Proposition~4.4]{Go-Le-Ru}  but we give here a different proof.

\begin{Prop}\label{p:iderg}
Let $R$ be an ergodic automorphism on $\zdk$ and let $(\baX,\ov{\cb},\bamu)$ be another probability space. Then
$$
\ov{\rm span}\Bigg(\bigcup_{\rho\in J(R',{\rm Id}_{\baX}\times R), R'\in {\rm Erg}}{\rm Im}(\Phi_\rho)\Bigg)=L^2({\rm Inv}_{\bamu\otimes\kappa})^\perp.$$
\end{Prop}
\begin{proof}
 Given a measurable function $N \colon \baX\to\Z$, consider the operator
 $$
 \Phi_N \colon L^2(Z,\kappa)\to L^2(\baX\times Z,\bamu\ot\kappa),
 $$
$$
\Phi_N(f)(\bax,z):=f(R^{N(\bax)}z).$$
Then, clearly, $\Phi_N(1)=1$ and $\Phi_N(f)\geq0$ if $f\geq0$.
Moreover,
\begin{multline*}
\langle \Phi_N^\ast1,f\rangle=\langle 1,\Phi_N(f)\rangle=\int \Phi_N(f)\,d\bamu \,d\kappa\\=
\int\Big(\int f(R^{N(\bax)}z)\,d\kappa(z)\Big)\,d\bamu(\bax)=
\int\Big(\int f(z)\,d\kappa(z)\Big)\,d\bamu(\bax)=\int f\,d\kappa,
\end{multline*}
so $\Phi_N^\ast(1)=1$.
Furthermore,
\begin{multline*}
\Phi_N\circ U_R(f)=\Phi_N(U_Rf)=(U_Rf)(R^{N(\bax)}z)\\=
(f\circ R)(R^{N(\bax)}z)
=
f(R^{N(\bax)+1}z)
\end{multline*}
and
\begin{multline*}
U_{{\rm Id}\times R}\circ \Phi_N(f)=U_{{\rm Id}\times R}(\Phi_N f)=\Phi_Nf({\rm Id}\times R(\bax,z))
\\=
\Phi_Nf(\bax,Rz)=f(R^{N(\bax)}(Rz)),
\end{multline*}
so   $\Phi_N\circ U_R=U_{{\rm Id}\times R}\circ \Phi_N$.

Now, take any measurable $A\subset \baX$ and let $f\in L^2(Z,\kappa)$. Consider $N \colon \baX\to\Z$ being the zero function and $N' \colon \baX\to\Z$ being the characteristic function of $A$. Then
$$
\Phi_N(f)-\Phi_{N'}(f)=1_A\otimes (f-f\circ R).$$
Now, the coboundaries of the form $h=f-f\circ R$, $f\in L^2(Z,\kappa)$, are dense in the space $L^2_0(Z,\kappa)$ (by the von Neumann theorem). Since $A\subset \baX$ is arbitrary, also $g\otimes h$
are in the closed span of ergodic Markov quasi-images, and the result follows.
\end{proof}

As we have mentioned above, this result implies
Proposition~\ref{p:jjm}.

\begin{Cor}\label{c:urenew}
All piecewise Markov URE automorphisms satisfy~\eqref{Mproperty}.
\end{Cor}
\begin{proof}Let $T$ act on $\xbm$ and be  piecewise Markov URE.
By Proposition~\ref{p:iderg}, it follows that  each automorphism ${\rm Id}_Y\times E$ acting on $(Y\times Y_1,\mathcal{C}\otimes\mathcal{C}_1,\nu\otimes\nu_1)$ with $E$ ergodic satisfies~\eqref{Mproperty}.
To conclude, it suffices to observe that for each $\rho\in J({\rm Id}_Y\times E,T)$, we have
$$
\overline{\rm span}\Bigg(\bigcup_{\lambda\in J(R',{\rm Id}_Y\times E), R'\in{\rm Erg}}{\rm Im}(\Phi_\rho\circ\Phi_\lambda)\Bigg) \supset {\rm Im}(\Phi_\rho)\cap (L^2({\rm Inv}_\mu))^\perp.
$$
\end{proof}

We now use the technique related to the study of the Markov URE property to prove~Theorem~\ref{t:ajmt}.

 \begin{proof}[Proof of Theorem~\ref{t:ajmt}.]
We would like to prove that $\bfu\perp \mathcal{C}_{\cf_{\rm ec}}$.
By  \cite[Theorem~A]{Ka-Ku-Le-Ru}, this is equivalent to showing that
$$
\pi_0\perp L^2(\ca(\cf_{\rm ec}))$$
for each Furstenberg system $\kappa\in V(\bfu)$.
Let us fix $\kappa\in V(\bfu)$. Suppose that $\EE(\pi_0|\ca(\cf_{\rm ec}))\neq0$ and recall that $\pi_0\perp L^2({\rm Inv}_\kappa)$ since $\|\bfu\|_{u^1}=0$, cf.\ \cite{Go-Le-Ru}. In view of our assumption, it follows that we can find a topological system $(Z,R)$ together with an invariant measure $\mu\in M(Z,R)$ such that:
\begin{itemize}
\item $(Z,R)$ satisfies the strong $\bfu$-MOMO property,
\item for the measure-theoretic dynamical system $(Z,\mu,R)$, there is a joining $\rho\in J(R,S|_{\ca(\cf_{\rm ec})})$ with
$$
\pi_0\not\perp{\rm Im }(\Phi_\rho).$$
\end{itemize}
Then there exists $g\in C(Z)$ such that
$$\int\EE(\pi_0| \ca(\cf_{\rm ec}))\Phi_\rho(g)\,d\kappa|_{\ca(\cf_{\rm ec})}\neq 0.$$
Let $\widehat{\rho}$ denote the relatively independent extension of $\rho$ to an element of $J((R,\mu),(S,\kappa))$. Then
$$
\int g\ot \pi_0\,d\widehat{\rho}=\int g\ot
\EE(\pi_0|\ca(\cf_{\rm ec}))\,d\rho=$$$$
\int \EE(\pi_0|\ca(\cf_{\rm ec}))
\Phi_\rho(g)\,d\kappa|_{\ca(\cf_{\rm ec})}\neq0.$$
On the other hand, we can compute the integral $\int g\ot\pi_0\,d\widehat{\rho}$ using generic sequences. More precisely, in view of \cite{Ka-Ku-Le-Ru}, there exist a sequence $(N_r)\subset\N$ and a sequence $(z_n)\subset Z$ such that the increasing sequence $(b_k)$ coming from the set of $n$ for which $z_{n+1}\neq R(z_n)$ has the property that $b_{k+1}-b_k\to\infty$ and
$$\int g\ot \pi_0\,d\widehat{\rho}=\lim_{r\to\infty}\frac1{N_r}
\sum_{n\leq N_r}g(z_n)\bfu(n).$$
However, by slightly modifying $(b_k)$, we also have
$$
\lim_{r\to\infty}\frac1{N_r}
\sum_{n\leq N_r}g(z_n)\bfu(n)=\lim_{K\to\infty}\frac1{b_K}\sum_{k\leq K}
\sum_{b_k\leq n<b_{k+1}}g(z_n)\bfu(n)$$
and since the latter limit, by the strong $\bfu$-MOMO of $(Z,R)$, equals zero, we obtain a contradiction.
\end{proof}

 Our final remark in this subsection is related to
 Question~\ref{quest2} (see Section~\ref{s:open} below).

\begin{Remark}\label{r:RD} 
We would like to
point out that the non-URE automorphism  $R$ constructed in Example~\ref{e:22}
is not piecewise Markov URE either.
Indeed, let $S$ acting on $\ycn$ be ergodic and let $\rho\in J(R,S)$.
Then (see e.g.\ \cite{Go-Le-Ru})
$$
\rho=\int_\Omega\delta_\omega\otimes\rho_\omega\,d\lambda(\omega),
$$
where $\rho_\omega\in J(T_{\phi_\omega},S)$  for each $\omega$ ($\rho_\omega$ are not necessarily ergodic).
We say that $\rho_\omega$ is {\em non-trivial} if there exists $f_\omega\in L^2(X\times\Z/2\Z,\mu\otimes\la_{2})\ominus L^2(X,\mu)$ such that $\Phi_{\rho_\omega}(f_\omega)\neq0$.\\ \underline{Claim:}
If $\omega,\omega'$ are not tail equivalent then  $\Phi_{\rho_\omega}(f_\omega)\perp \Phi_{\rho_{\omega'}}(f_{\omega'})$.\\
Indeed, we have that
$$
\int_Y\Phi_{\rho_\omega}(f_\omega)\Phi_{\rho_{\omega'}}(f_{\omega'})\,d\nu=
\int_{X\times\Z/2\Z}\Phi^\ast_{\rho_{\omega'}}\Phi_{\rho_\omega}(f_\omega)f_{\omega'}\,d(\mu\otimes\lambda_2)=0$$
since $\Phi^\ast_{\rho_{\omega'}}\Phi_{\rho_\omega}(f_\omega)=0$ in view of Lemma~\ref{l:strongRD}
(as $\Phi^\ast_{\rho_{\omega'}}\Phi_{\rho_\omega}$ is the Markov operator corresponding to a joining of $T_{\phi_\omega}$ and $T_{\phi_\omega'}$ and $f_\omega$ is in the orthocomplement of $L^2$ of the first coordinate).

Take $F\in L^2(Y,\nu)$ and consider  decomposition
$$
\Phi^*_{\rho_\omega}(F)=f_\omega+f_\omega^\perp
$$
with $f_\omega\in L^2(X\times\Z/2\Z,\mu\otimes\lambda_2)\ominus L^2(X,\mu)$ and $f_{\omega}^\perp \in L^2(X,\mu)$.
Suppose that $f_\omega\neq0$.
Then
\begin{align*}
\int_Y\Phi_{\rho_\omega}(f_\omega)\cdot F\,d\nu&=
\int_{X\times\Bbb Z/2\Bbb Z} f_\omega \Phi^\ast_{\rho_\omega}( F)\,d(\mu\otimes\lambda_2)\\
&=\|f_\omega\|^2_{L^2(\mu\otimes\lambda_2)}>0.
\end{align*}
In particular,  $\Phi_{\rho_\omega}(f_\omega)\neq0$, so  $\rho_\omega$ is non-trivial.
However, it follows from the Claim that $\rho_\omega$ can be non-trivial only for a countable subset of pairwise non-equivalent $\omega\in\Omega$.
Thus, we obtain that $\Phi^*_{\rho_\omega}(F)\in L^2(X,\mu)$ at $\lambda$-a.e. $\omega\in\Omega$.
Therefore, $\Phi_\rho^*(F)\in L^2(\Omega\times X)$ and for example the function $(\omega,x,j)\mapsto (-1)^j$ is orthogonal to all ergodic Markov quasi-images.
Now, the assertion  of the remark follows from Corollary~\ref{c:urenew}.
\end{Remark}

\subsection{Some properties of the d-bar distance}

Let $(K,d)$ be a nonempty compact metric space, $\# K>1$, and let $S$ be the leftward shift on $K^\bbZ$.
The space $M(K^\bbZ,S)$ admits a generalization of Ornstein's ``d-bar'' metric:
$$
\ol{d}(\nu,\nu') := \inf\bigg\{\int d(y_0,y'_0)\ d\lambda(y,y')
\colon  \la\ \hbox{is a joining of}\ \nu\ \hbox{and}\ \nu'\bigg\},
$$
where $\nu,\nu' \in M(K^\bbZ,S)$ and $y=(y_n)_{n\in\Bbb Z}, \ y'=(y_n')_{n\in\Bbb Z}\in K^\Bbb Z$.
By the $*$-weak compactness of the set of joinings $J(\nu,\nu')=J((S,\nu), (S,\nu'))$, this infimum is always attained.
We note that $\ol{d}$ generates a topology on $M(K^\bbZ,S)$ which is  neither compact nor separable.
  Although $\ol{d}$ is strictly stronger than the weak topology, the following holds.

\begin{Lemma}\label{lem:lsc} The mapping $M(K^\bbZ,S)\times M(K^\bbZ,S)
\ni (\nu,\nu')\mapsto \ol{d}(\nu,\nu')$ is $*$-weak lower semicontinuous.
\end{Lemma}

Of course,  this claim is folklore  but were unable to find an exact reference to it.
Therefore, we provide a  proof.

\begin{proof}
We have to show that for each $a>0$, the set
$$
\{(\nu,\nu')\in M(K^\bbZ,S)^2\colon \ol{d}(\nu,\nu')\le a \}
$$
is $*$-weak closed.
Let a sequence $(\nu_n,\nu_n')$, $n\in\Bbb N$, $*$-weakly converges to $(\nu,\nu')$ and $\ol{d}(\nu_n,\nu_n')\le a$ for all $n$.
Peak a subsequence $(n_k)_{k=1}^\infty$
such that
$$
\ol{d}(\nu_{n_k},\nu_{n_k}')=\liminf_{n\in\Bbb N}\ol{d}(\nu_n,\nu_n').
$$
For each $k$, choose a joining $\rho_k\in J_S(\nu_{n_k},\nu'_{n_k})$ such that
$$
\ol{d}(\nu_{n_k},\nu_{n_k}')=\int\ol{d}(y_0,y_0')\,d\rho_k(y,y').
$$
Since $M(K^\bbZ,S)^2$ is $*$-weak compact, we can assume without loss of generality (passing to a further subsequence, if necessarily) that
the sequence $(\rho_k)_{k\in\Bbb N}$ converges to some  measure
$\rho\in M(K^\bbZ,S)^2$.
Let $\widetilde\nu$ and $\widetilde\nu'$ denote the left and the right marginal projections of $\rho$.
Since the two  projections are continuous as mappings from $M(K^\bbZ,S)^2$ to $M(K^\bbZ,S)$, it follows that
$\nu_{n_k}\to\widetilde\nu$ and $\nu_{n_k}'\to\widetilde\nu'$.
Hence, $\nu=\widetilde\nu$ and $\nu'=\widetilde\nu'$.
Thus, $\rho\in J_S(\nu,\nu')$.
We now have
$$
\ol{d}(\nu,\nu')\le
\int\ol{d}(y_0,y_0')\,d\rho(y,y')=\lim_{k\to\infty}\int\ol{d}(y_0,y_0')\,d\rho_k(y,y')
=\lim_{k\to\infty}\ol{d}(\nu_{n_k},\nu_{n_k}').
$$
Hence, $\ol{d}(\nu,\nu')\le a$, as desired.
\end{proof}

We  will also show that the d-bar topology is ``compatible'' with the inverse limit operation for the underlying  compact metric spaces.

Given $n\in\Bbb N$, let $(K_n,d_n)$ be a compact metric space.
Let $p_n \colon K_{n+1}\to K_n$ ne a continuous onto mapping.
Denote by $K$ the inverse limit of the sequence
$$
K_1 \overset{p_1}\longleftarrow K_2\overset{p_2}\longleftarrow\cdots
$$
Then $K$ is a compact metric space.
Denote by $q_n \colon K\to K_n$ the canonical projections.
The topology of $K$ is compatible with the metric
$$
d(k,k'):=\sum_{n=1}^\infty\frac1{2^n}d_n(q_n(k),q_n(k'))
$$
 for all  $k,
k'\in K$.
Denote by $S_n$ the leftward shift on $K_n^\Bbb Z$.
Consider an associated sequence of topological dynamical systems
$$
(K_1^\Bbb Z ,S_1)\overset{p_1^\Bbb Z}\longleftarrow (K_2^\Bbb Z,S_2)\overset{p_2^\Bbb Z}\longleftarrow\cdots,
$$
where the mappings $p_n^\Bbb Z$ are equivariant, i.e. $S_np_n^\Bbb Z=p_n^\Bbb ZS_{n+1}$
for all $n\in\Bbb N$.
Then the inverse limit of this sequence is the system $(K^\Bbb Z,S)$, where $S$ is the leftward shift.
Then $q_n^\Bbb Z$ is the canonical equivariant projection of $(K^\Bbb Z,S)$ onto $(K_n^\Bbb Z,S_n)$.

\begin{Lemma}\label{le:projdbar} The topology generated by the $\ol{d}$-distance on $M(K^\Bbb Z,S)$
is the inverse limit of the topologies generated by the $\ol{d_n}$-distance on
$M(K_n^\Bbb Z,S_n)$ as $n\to\infty$.
\end{Lemma}
\begin{proof}
First of all, for each $n$, we replace  $d_n$ with another metric $r_n$ compatible with the same topology on $K_n$:
$$
r_n(k,k'):=\sum_{j=1}^n\frac 1{2^j}d_j\big(p_j\circ\cdots \circ p_{n-1}(k),p_j\circ\cdots\circ p_{n-1}(k')\big ).
$$
If $j=n$ in this sum then by $p_{n}\circ p_{n-1}$ (which has no sense, of course) we mean the identity.
Then the topology generated by the $\ol{r_n}$-distance is the same as the topology generated by the $\ol{d_n}$-distance.
Now, take two measures
$\nu,\nu'\in M(K^\Bbb Z,S)$.
Then
$$
\ol{d}(\nu,\nu')=\inf_{\rho\in J(\nu,\nu')}\int_{K^\Bbb Z\times K^\Bbb Z}\sum_{n=1}^\infty\frac1{2^n}d_n(q_n(k_0),q_n(k_0'))\,d\rho(k,k').
$$
Since for every $N>0$,
$$
\sum_{n=1}^N\frac1{2^n}d_n(q_n(k_0),q_n(k_0'))=r_N(q_N(k_0), q_N(k_0')),
$$
it follows that
$$
\ol{d}(\nu,\nu')-\frac1{2^N}
\le \inf_{\rho\in J(\nu,\nu')}\int_{K^\Bbb Z\times K^\Bbb Z}r_N(q_N(k_0), q_N(k_0'))d\rho(k,k')\le \ol{d}(\nu,\nu').
$$
Let $\nu_N$ and $\nu_N'$ be the pushforwards of $\nu$ and $\nu'$ respectively onto $K_N^\Bbb Z$.
Then  $(q_N^{\Bbb Z}\times q_N^{\Bbb Z})_*(J(\nu,\nu'))=
J(\nu_N,\nu_N')$.
Therefore,
$$
\inf_{\rho\in J(\nu,\nu')}\int_{K^\Bbb Z\times K^\Bbb Z}r_N(q_N(k_0), q_N(k_0'))d\rho(k,k')=\ol{r_N}(\nu_N,\nu_N').
$$
Thus, we obtain that
$$
\ol{d}(\nu,\nu')-\frac1{2^N}
\le \ol{r_N}(\nu_N,\nu_N')\le \ol{d}(\nu,\nu').
$$
It follows that the topology generated by the $\ol{d}$-distance
coincides with the inverse limit of the $\ol{r_N}$-topologies as $N\to\infty$.
\end{proof}

\subsection{URE via the d-bar,  and  proof of Theorem~\ref{theorem:f}}\label{subsect}

In this subsection we will use the d-bar distance  to show the URE property for some classes of dynamical systems.
First, we  prove Theorem~\ref{theorem:f}.

\begin{proof}[Proof of Theorem F]
In view of Proposition~\ref{p:inac}(B),
it remains to prove only the hard part
(iii)$\Rightarrow$(ii).
By Lem\-ma~\ref{lem:sep}, the set of factors of $(Y,\kappa,S)$ is $\widebar d$-separable.
(It is not assumed in that lemma that $S$ is ergodic.)
Hence the set of ergodic components $(T,\nu_\gamma)$, $\gamma\in\Gamma$, is $\widebar d$-separable (upon deleting a subset of zero measure, if necessary).
On the other hand, each $\widebar d$-separable subset of ergodic dynamical systems has a common
ergodic extension by Corollary~\ref{cor:Tim}.
\end{proof}

Next, we will give a characterization of TURE via the d-bar.
For that we will need a topological lemma.

\begin{Lemma}\label{graphs}
Let $(K,r)$ be a compact metric space and let  $A$ be a symmetric $G_\delta$ subset in $K\times K$.
If there is an uncountable subset $R\subset K$ such that $(R\times R)\setminus\{(k,k')\in K\times K\colon k= k'\}\subset A$ then there is a perfect subset
$P\subset K$ such that
$(P\times P)\setminus\{(k,k')\in K\times K\colon k=k'\}\subset A$.
\end{Lemma}
\begin{proof}
Since $A$ is a symmetric $G_\delta$, there exists a sequence $(U_n)_{n=1}^\infty$
of symmetric open subsets of $K\times K$ such that $U_1\supset U_2\supset\cdots$
and $\bigcap_{n=1}^\infty U_n=A$.
Denote by $\mathcal L$ the set of all finite words over $\{0,1\}$, i.e.
$\mathcal L=\bigsqcup_{n=0}^\infty\{0,1\}^n$.
For  a word  $w\in\mathcal L$, let $l(w)$ denote the length of $w$.
We intend to construct for each $w\in\mathcal L$,
a nonempty open set $V_w\subset K$
 such that the following are satisfied:
 \begin{itemize}
 \item[(i)]
 $\widebar{V_{wi}}\subset V_w$ for all $w\in \mathcal L$ and  $i=0,1$,
 \item[(ii)]
  $\widebar{V_{w0}}\cap  \widebar{V_{w1}}=\emptyset$ for all $w\in \mathcal L$,
  \item[(iii)]
  diam\,$V_w<2^{-l(w)}$ for all $w\in \mathcal L$,
  \item[(iv)]
  $V_w\cap R$ is uncountable,
  \item[(v)]
  if $w,w'\in\{0,1\}^n$ and $w\ne w'$ then $\widebar{V_{w}}\times\widebar{V_{w'}}\subset U_n$ for all $n$.
 \end{itemize}
 The construction will be done recursively.
 On the first step, we let $V_{\emptyset}:=K$.
 Assume that we have defined $V_w$ for all $w\in\mathcal L$ with $l(w)\le n$.
 For each such $w$, since  the set $V_w\cap R$ is uncountable,
it contains two different condensation points.
Choose them and call them $x_{w0}$ and $x_{w1}$.
Now, take arbitrary $z,z'\in\{0,1\}^{n+1}$ with $z\ne z'$.
Then $x_z\ne x_{z'}$ and both belong to $R$.
Hence
$$
(x_z,x_{z'})\in (R\times R)\setminus\{(k,k)\colon k\in K\}\subset A\subset U_{n+1}.
$$
Since
$U_{n+1}$ is open,
there are neighbourhoods of $x_z$ and $x_{z'}$ whose Cartesian product lies in $U_{n+1}$.
Since there are only finitely many of such pairs $x_z\ne x_{z'}$, we can shrink these neighbourhoods so to obtain
open subsets $V_{w}\ni x_w$ satisfying (i)--(v) for all words $w$
with $l(w)=n+1$.
\newline\indent
Now, let
$$
P:=\bigcap_{n=1}^\infty\bigcup_{l(w)=n}\widebar{V_{w}}.
$$
Then $P$ is a closed subset of $K$.
Moreover, given $w=(i_n)_{n=1}^\infty\in\{0,1\}^\Bbb N$, the intersection
$\bigcap_{n=1}^\infty \widebar{V_{(i_1\dots i_n)}}$ is a singleton.
Denote it by $\vartheta(w)$.
It is a routine to verify that
the mapping
$$
\{0,1\}^\Bbb N\ni w \mapsto \vartheta(w)\in P
$$
is a homeomorphism.
Hence $P$ is  perfect.
If two sequences $w=(i_m)_{m=1}^\infty$ and $w'=(i_m')_{m=1}^\infty$ in $\{0,1\}^\Bbb N$ are different then there is $N>0$ such that $w_n:=(i_1,\dots,i_n)\ne
(i_1',\dots,i_n')=:w_n'$ for all $n>N$.
By (v),
$
\widebar{V_{w_n}}\times\widebar{V_{w_n'}}\subset U_n.
$
Hence,
$$
(\vartheta(w),\vartheta(w'))=\bigcap_{n=N+1}^\infty(\widebar{V_{w_n}}\times\widebar{V_{w_n'}})\subset \bigcap_{n=N+1}^\infty U_n=A.
$$
Therefore,
$(P\times P)\setminus\{(k,k')\in K\times K\colon k=k'\}\subset A$, as desired.
\end{proof}

Let $X$ be a compact metric space and let $T$ be a homeomorphism of $X$.
Take another compact metric space $K$ and a continuous one-to-one mapping $\pi \colon X\to K$.
Denote by $\pi^\Bbb Z$ the equivariant mapping
$$
X\ni x\mapsto \pi^\Bbb Z(x):=(\pi(T^nx))_{n\in\Bbb Z}\in K^\Bbb Z.
$$
Let $\pi^\Bbb Z_*$ stand for the associated mapping from $M(X,T)$ to
$M(K^\Bbb Z,S)$.
Of course,
\begin{enumerate}
\item[---]
$\pi^\Bbb Z_*(M(X,T))$ is a closed convex subset of $M(K^\Bbb Z,S)$ and
\item[---]
$\pi^\Bbb Z_*$ is an affine homeomorphism of $M(X,T)$ onto
$\pi^\Bbb Z_*(M(X,T))$.
\end{enumerate}

\begin{Prop}\label{pr:my} Assume that $\pi^\Bbb Z_*(M^e(X,T))$ is not  $\ol{d}$-separable.
Then there is a nonatomic probability Borel measure $\kappa$ on
$\pi^\Bbb Z_*(M^e(X,T))$ such that each $*$-weak Borel subset of
$\pi^\Bbb Z_*(M^e(X,T))$ of full measure $\kappa$
is  not  $\ol{d}$-separable.
\end{Prop}

\begin{proof} Let $\delta>0$.
It follows from  Zorn's lemma that there is a maximal subset $P_\delta\subset \pi^\Bbb Z_*(M^e(X,T))$
such that $\ol{d}(\nu,\nu')\ge\delta$ for all $\nu\ne\nu'\in P_\delta$.
We claim that there is $\delta_0>0$ such that
$P_{\delta_0}$ is uncountable.
Indeed, otherwise, for each $n>0$, the set
$P_{1/n}$ would be countable.
Hence, the union $\bigcup_nP_{1/n}$ is also countable.
However, it is straightforward to check that $\bigcup_nP_{1/n}$ is dense in
$\pi^\Bbb Z_*(M^e(X,T))$.
Hence, $\pi^\Bbb Z_*(M^e(X,T))$
is $\ol{d}$-separable, a contradiction.

It follows from Lemma~\ref{lem:lsc} that the set
$\{(\nu,\nu')\in M(K^\Bbb Z,S)^2\colon \ol{d}(\nu,\nu')> \delta\}$ is $*$-weakly open in $M(K^\Bbb Z,S)^{\times 2}$ for each $\delta>0$.
Hence, the set
$$
A(\delta_0):=\{(\nu,\nu')\in \pi^\Bbb Z_*(M^e(X,T))^{\times 2}\colon \ol{d}(\nu,\nu')\ge \delta_0\}
$$
 is a $G_\delta$ in  $\pi^\Bbb Z_*(M^e(X,T))^{\times 2}$.
 Of course, $A(\delta_0)\supset P_{\delta_0}^{\times 2}\setminus\{(\nu,\nu)\colon \nu\in P_{\delta_0}\}$.
 By Lemma~\ref{graphs}, there exists a $*$-weak perfect subset $J\subset \pi_*^\Bbb Z(M^e(X,T))$
 such that
 $$
 (J\times J)\setminus\{(\nu,\nu)\colon \nu\in J\}\subset A(\delta_0).
 $$
 Thus, for all points $\nu,\nu'\in J$, if $\nu\ne\nu'$ then $\ol{d}(\nu,\nu')\ge\delta_0$.
 Fix a Borel non-atomic probability $\kappa$ on $J$.
 If $B$ is a Borel subset of $\pi_*^\Bbb Z(M^e(X,T))$ and $\kappa(B)=1$
 then $\kappa(B\cap J)=1$.
Since $\kappa$ is nonatomic,  $B\cap J$ is uncountable.
Hence, $B$ is not $\ol{d}$-separable.
\end{proof}

From Appendix  and Proposition~\ref{pr:my} we deduce the following characterization of the TURE (topological URE)
property for  topological systems.

\begin{Cor}\label{co:d-bar}
{\em The following are equivalent:
\begin{enumerate}
\item[(i)]
$\pi^\Bbb Z_*(M^e(X,T))$ is $\ol{d}$-separable.
\item[(ii)]
For each measure $\mu\in M(X,T)$, the dynamical system $(X,\mu,T)$ is URE.
\item[(iii)]
There exists a dynamical system $(Y,\kappa, R)$ such that $(X,\mu,T)$ is a factor of $(Y,\kappa, R)$
for each $\mu\in M^e(X,T)$.
\item[(iv)]
There exists an ergodic dynamical system $(Y,\kappa, R)$ such that $(X,\mu,T)$ is a factor of $(Y,\kappa, R)$
for each $\mu\in M^e(X,T)$.
\end{enumerate}}
\end{Cor}

\begin{proof} Proposition~\ref{prop:sep} gives the implication (i)$\Rightarrow$(iv).
Proposition~\ref{pr:my} and Corollary~\ref{cor:Tim} yield the implication (ii)$\Rightarrow$(i).
The implication (iii)$\Rightarrow$(ii)  follows from the definition of URE (see Theorem~\ref{theorem:f}(iii)).
Of course, the implication (iv)$\Rightarrow$(iii) is trivial.
\end{proof}


Finally, in this section, we will show that the inverse limit of arbitrary URE systems is URE utilizing the d-bar metric.

Let
\begin{equation}\label{eq:inv}
(X_1,\mu_1,T_1)\longleftarrow (X_2,\mu_2,T_2)\longleftarrow\cdots
\end{equation}
 be an infinite inverse sequence of dynamical systems.
Denote by $(X,\mu,T)$ the inverse limit of this system.

\begin{Th}\label{th:invlimURE}
 If $(X_n,\mu_n,T_n)$ is URE for each $n$ then $(X,\mu,T)$ is URE.
\end{Th}

\begin{proof}
Without loss of generality we may assume that for each $n\in\Bbb N$, there exist
\begin{enumerate}
\item[---] a compact metric space
$(K_n,d_n)$,
\item[---]  a continuous onto mapping $p_n \colon K_{n+1}\to K_n$ and
\item[---] a probability measure $\nu_n$ on $K_n^\Bbb Z$
\end{enumerate}
such that  $\nu_n$ is invariant under the leftward shift $S_n$ on  $K_n^\Bbb Z$ and (\ref{eq:inv}) is isomorphic to the sequence
\begin{equation}\label{eq:new}
(K_1^\Bbb Z,S_1,\nu_1) \overset{p_1^\Bbb Z}\longleftarrow (K_2^\Bbb Z,S_2,\nu_2)\overset{p_2^\Bbb Z}\longleftarrow
\cdots,
\end{equation}
where $p_n^\Bbb Z(y):=(p_n(y_j))_{j\in\Bbb Z}$ for all $y=(y_j)_{j\in\Bbb Z}\in K_{n+1}^\Bbb Z$.
Let $K$ be the inverse limit of the sequence $K_1 \overset{p_1}\longleftarrow K_2\overset{p_2}\longleftarrow\cdots$.
Denote by $S$ the leftward shift  on  $K^\Bbb Z$.
Denote by $\nu$ the inverse limit of the sequence $(\nu_n)_{n=1}^\infty$.
Then $\nu$ is a Borel $S$-invariant probability on $K^\Bbb Z$ and
the dynamical system $(K^\Bbb Z,\nu, S)$ is the inverse limit of (\ref{eq:new}).
Then  $(p_n^\Bbb Z)_*$ is a continuous affine mapping from $M(K_{n+1}^\Bbb Z,S_{n+1})$ to
$M(K_{n}^\Bbb Z,S_{n})$ and  $(p_n^\Bbb Z)_*(\nu_{n+1})=\nu_n$ for each $n\in\Bbb N$.
Let $\kappa_n$ be the unique probability on $M^e(K_{n}^\Bbb Z,S_{n})$ such that $\nu_n=\int_{M^e(K_{n}^\Bbb Z,S_{n})}\tau\,d\kappa_n(\tau)$.
In other words, $\kappa_n$ can be interpreted as  {\it the ergodic decomposition}  of $\nu_n$.
Then $(p_n^\Bbb Z)_{**}(\kappa_{n+1})=\kappa_n$ for each $n\in\Bbb N$.
Thus, we obtain an inverse sequence of standard probability spaces
$$
(M^e(K_{1}^\Bbb Z,S_1),\kappa_1) \overset{(p_1^\Bbb Z)_*}\longleftarrow
(M^e(K_{n}^\Bbb Z,S_{2}) ,\kappa_2)\overset{(p_2^\Bbb Z)_*}\longleftarrow\cdots.
$$
The inverse limit of this sequence is exactly the space $(M^e(K^\Bbb Z,S),\kappa)$ such that
$\nu=\int_{M^e(K^\Bbb Z,S)}\tau\,d\kappa(\tau)$,
i.e. $\kappa$ is the ergodic decomposition of $\kappa$.
By the URE assumption and Corollary~\ref{cor:Tim},  for each $n\in\Bbb N$, there exists a Borel subset $E_n\subset M^e(K_{n}^\Bbb Z,S_{n})$
such that $\kappa_n(E_n)=1$ and $E_n$ is $\ol{d_n}$-separable.
Replacing, if necessarily, $E_n$ with ``smaller'' subsets of full measure we may assume without loss of generality that $(p_n^\Bbb Z)_*(E_{n+1})=E_n$ for all $n$.
Then the inverse limit $E$ of the sequence $(E_n)_{n=1}^\infty$ is well defined as a Borel subset of $M^e(K^\Bbb Z,S)$.
Of course, $\kappa(E)=1$.
It follows from  Lemma~\ref{le:projdbar}  that $E$ is $\ol{d}$-separable.
\end{proof}

\subsection{Remarks on the inverse limits of URE automorphisms }

\begin{Remark}
By Theorem~\ref{th:invlimURE}, the inverse limit of a sequence of URE-transfor\-mations is URE.
Hence, the ergodic components of the inverse limit have a common extension.
However, Theorem~\ref{th:invlimURE} does not
give any explicit description of the common extension.
We now provide such a description
in a particular case where each transformation in the sequence has countably many ergodic components.

Let $T_1\leftarrow T_2\leftarrow~\cdots$ be an
 inverse  sequence of probability preserving transformations
such  that each $T_n$ has at most countably many ergodic components.
Let a transformation $R$ be  the projective limit of this sequence.

Denote by $(W_1,\kappa_1)$  the space of ergodic components
of $T_1$.
Then $W_1$ is countable.
Hence, we
 can consider $T_1$ as a disjoint union of ergodic transformations
$T_{1,w_1}$, $w_1\in W_1$:
$$
T_1=\bigsqcup_{w_1\in W_1}T_{1,w_1}.
$$
Let $(X_{1,w_1},\mu_{1,w_1})$ denote the space of $T_{1,w_1}$.
Consider an ergodic joining of the family
$T_{1,w_1}$, $w_1\in W_1$.
Denote it by $S_1$.
Then for each $w_1\in W_1$, there is a factor mapping $q_{1,w_1}$
that intertwines $S_1$ with
$T_{1,w_1}$.

Next, let $(W_2,\kappa_2)$ be the space of ergodic components
of $T_2$.
We can consider $T_2$ as a disjoint union of ergodic transformations
$T_{2,w_2}$, $w_2\in W_2$:
$$
T_2=\bigsqcup_{w_2\in W_2}T_{2,w_2}.
$$
Let
$(X_{2,w_2}, \mu_{2,w_2})$ denote the space of
$T_{2,w_2}$, $w_2\in W_2$.
Then the factor mapping $T_2\to T_1$ determines  a canonical projection
$\pi_{1} \colon (W_2,\kappa_2)\to (W_1,\kappa_1)$
and factor mappings
$$
r_{1,w_2} \colon (X_{2,w_2}, \mu_{2,w_2},T_{2,w_2})\to
(X_{1,\pi_1(w_2)}, \mu_{1,\pi_1(w_2)},T_{1,\pi_1(w_2)})
$$
for each $w_2\in W_2$.
Hence, the product mapping
$$
r_1:=\prod_{w_2\in W_2} r_{1,w_2} \colon \prod_{w_2\in W_2}X_{2,w_2}\to \prod_{w_1\in W_1}X_{1,w_1}
$$
intertwines $\prod_{w_2\in W_2}T_{2,w_2}$ with
$\prod_{w_1\in W_1}T_{1,w_1}$.
We now select an ergodic joining of the family
$T_{2,w_2}$, $w_2\in W_2$, which projects onto $S_1$ under this mapping.
Denote it by $S_2$.
For each $w_2\in W_2$, the $w_2$-coordinate marginal
$$
q_{2,w_2} \colon  \prod_{w\in W_2}X_{2,w}\to X_{2,w_2}
$$
 intertwines
$S_2$ with $(X_{2,w_2},\mu_{2,w_2}, T_{2,w_2})$.

Continue this procedure inductively.
We obtain an inverse  sequence of probability countable spaces
$$
\begin{CD}
(W_1,\kappa_1)@<\pi_1<<(W_2,\kappa_2)@<\pi_2<<(W_3,\kappa_3)@<\pi_3<<\cdots,
\end{CD}
$$
a sequence of families $( T_{n,w_n})_{w_n\in W_n}$ of ergodic transformations,
a sequence of ergodic transformations
$$
\begin{CD}
S_1@<r_1<< S_2@<r_2<< S_3@<r_3<< \cdots,
\end{CD}
$$
and a sequence of  mappings $q_{n,w_n}$ from the space of $S_n$ onto the space of $T_{n,w_n}$,  $w_n\in W_n$,
such that

---   $\bigsqcup_{w_n\in W_n}T_{n,w_n}$ is the  ergodic decomposition of
of $T_n$,

--- $T_{n,w_n}$ is an extension of $T_{n-1, \pi_{n-1}(w_n)}$,

---  $q_{n,w_m}\circ S_n=T_{n,w_n}\circ q_{n,w_n}$  for each $w_n\in W_n$ and

--- the following diagram commutes for each $n$:
\begin{equation}\label{comD}
\begin{CD}
S_{n-1} @<r_{n-1}<< S_{n}\\
@V{q_{n-1,\pi_{n-1}(w_{n})}}VV @VVq_{n,w_{n}}V\\
T_{n-1,\pi_{n-1}(w_{n})} @<r_{n-1,w_n}<<  T_{n,w_{n}}
\end{CD}
\end{equation}
Let $S$ be the inverse limit of $(S_n)_{n=1}^\infty$.
Then $S$ is an ergodic transformation.
The space
$$
(W,\kappa):=\projlim_{n\to\infty}(W_n,\kappa_n)
$$
is the space of ergodic components of $R$.
Indeed, if  $W\ni w=(w_n)_{n=1}^\infty$ then $\pi_n(w_{n+1})=w_n$
and
the $w$-ergodic component of $R$ is  the projective limit of the sequence
$$
\begin{CD}
T_{1,w_1} @<r_{1,w_2}<<  T_{2,w_{2}}@<r_{2,w_3}<<\cdots.
\end{CD}
$$
Hence, the component is a factor of $S$ in view of the commutative diagram~(\ref{comD}).
Thus, $S$ is the common extension of all ergodic components of $R$.

\end{Remark}

In the following example we provide an inverse sequence of transformations whose ergodic components are all isomorphic to a fixed ergodic transformation $S$ but the inverse limit (which is URE) has uncountably many non-isomorphic ergodic components.

\begin{Example}
We construct here an  ergodic transformation $S$ and an
inverse sequence $I\times S\leftarrow I\times S\leftarrow\cdots$
such that

--- each arrow is  a finite-to-one extension,

--- the ergodic components of the inverse limit $R$ of this sequence
are pairwise non-isomorphic.\footnote{Thus,   all ergodic components of all  transformations in the sequence are isomorphic to each other. Of course, $R$ has uncountably many ergodic components because we assume that  $I$  is defined on a non-atomic standard measure space.}

Let $T$ be a 2-fold simple transformation on a standard probability space $(X,\mu)$.
Suppose that the centralizer $C(T)$ of $T$ is isomorphic to $\bigoplus_{n=1}^\infty\Bbb Z$.\footnote{For examples of simple $T$ with $C(T)$ isomorphic to $\Bbb Z^d$  we refer to \cite{Mad}. In the case $d=\infty$ the construction is similar.}
Let $\mathcal P$ denote the set of prime numbers.
Fix a mapping $\mathcal P\ni p\mapsto R_p\in C(T)$ such that
 $T, R_2,R_3,\dots$ is a family of independent generators of $C(T)$.
It follows from \cite{Dan} that there exist an ergodic  cocycle $\phi \colon X\to\Bbb T$ and a family of Borel  mappings $a_p \colon X\to\Bbb T$ such that
$$
\phi(R_px)=a_p(x)^{-1}\phi(x)^{p}a_p(Tx)
$$
at a.e. $x$ for each $p\in\mathcal P$.
Denote by $T_\phi$ the $\phi$-skew product extension of $T$.
Then $T_\phi$ is an ergodic transformation of $(X\times\Bbb T,\mu\otimes\text{Leb})$.\footnote{Note that since $T$ is 2-fold simple,  the centralizer $C(T_\phi)$ consists of the powers of $T_\phi$ and the rotations on the second coordinate. Indeed, the
elements $R_2^{a_2}\cdots R_q^{a_q}$  ($a_j\geq0$) lift to non-invertible elements that commute with $T_\phi$, while their inverses do not lift at all.}
Define mappings $\pi_p \colon X\times\Bbb T\to X\times\Bbb T$
by setting
$$
\pi_p(x,z):=(R_px,z^{p}a_p(x))
$$
Then $\pi_p$ is $p$-to-one, it preserves the product measure
$\widetilde\mu:=\mu\otimes\text{Leb}$
and $\pi_p T_\phi=T_\phi\pi_p$ for each $p\in\mathcal P$.
Fix a sequence $\boldsymbol p:=(p_n)_{n=1}^\infty$ with $p_n\in\mathcal P$ for every $n$.
Then an inverse sequence of dynamical systems
$$
\begin{CD}
(X\times\Bbb T,\widetilde\mu,T_\phi)@<\pi_{p_1}<<
(X\times\Bbb T,\widetilde\mu, T_\phi )@<\pi_{p_2}<<
\cdots.
\end{CD}
$$
is well defined.
Denote by $(X_{\boldsymbol p}, \mu_{\boldsymbol p},T_{\boldsymbol p})$ the corresponding inverse limit.
Then $(X_{\boldsymbol p}, \mu_{\boldsymbol p},T_{\boldsymbol p})$ is an ergodic dynamical system.
We now describe the structure of this system as a compact group extension of $T$.
Let $G_{\boldsymbol p}$ is the inverse limit of the inverse sequence of compact Abelian groups
$$
\begin{CD}
\Bbb T @<\widehat{p_1}<< \Bbb T @<\widehat{p_2}<< \Bbb T @<\widehat{p_3}<<\cdots,
\end{CD}
$$
where   $\widehat p_n$ denotes the homomorphism
$\Bbb T\ni z\mapsto z^{p_n}\in\Bbb T$ fore each $n$.
If  $\boldsymbol p$ and  $\boldsymbol q$ are two sequences
of primes without repetitions then
$G_{\boldsymbol p}$ and $G_{\boldsymbol q}$ are
isomorphic if and only if $\{p_1,p_2,\dots\}=\{q_1,q_2, \dots\}$.

Fix a Borel cross-section $s_n \colon \Bbb T\to\Bbb T$ of
 $\widehat p_n$.
Let
$$
b_n(x):=\prod_{k=1}^{n-1}s_{n-1}\circ s_{n-2}\circ
\cdots\circ  s_k(a_{p_k}(R^{-1}_{p_1}\cdots R^{-1}_{p_k}(x))).
$$
We define Borel mappings $\phi_n \colon X\to\Bbb T$ by setting $\phi_1:=\phi$
and
$$
\phi_n(x)=b_n(x)^{-1}\phi(R_{p_1}^{-1}\cdots R_{p_{n-1}}^{-1}x)   b_n(Tx)
$$
if $n>1$.
It is straightforward to verify that $\widehat p_{n-1}\circ\phi_{n}=\phi_{n-1}$ for each $n>1$.
Hence, a Borel mapping $\phi_{\boldsymbol p} \colon X\to G_{\boldsymbol p}$
is well defined by setting
$$
\phi_{\boldsymbol p}(x):= (\phi_n(x))_{n=1}^\infty\quad\text{for each $x\in X$.}
$$
It is routine to check that $(X_{\boldsymbol p}, \mu_{\boldsymbol p},T_{\boldsymbol p})$ is naturally isomorphic to
the skew product
$(X\times G_{\boldsymbol p}, \mu\otimes\lambda_{G_{\boldsymbol p}},
T_{\phi_{\boldsymbol p}})
$.
Again, since $T$ is 2-fold simple and the only non-trivial automorphism of the group $G_{\boldsymbol p}$ is the inversion,
$C(T_{\phi_{\boldsymbol p}})$ is
generated by $T_{\phi_{\boldsymbol p}}$ and
the compact group $G_{\boldsymbol p}$, acting on
$X\times G_{\boldsymbol p}$ by rotations along the second coordinate.
Thus, the topological group $C(T_{\phi_{\boldsymbol p}})$
is isomorphic to the direct product  $\Bbb Z\times G_{\boldsymbol p}$.
It follows that two transformations $T_{\phi_{\boldsymbol p}}$
and $T_{\phi_{\boldsymbol q}}$ are  not isomorphic if
$\{p_1,p_2,\dots\}\ne \{q_1,q_2, \dots\}$.

Let $Y:=([0,1], \text{Leb})$.
Consider a sequence of refining partitions $P_n$ of $Y$ into intervals of equal length $1/2^n$.
Then $(P_n)_{n=1}^\infty$ generates the entire Borel $\sigma$-algebra.
Fix a bijection $\iota$ between the set of atoms of all $P_n$, $n\in\Bbb N$, and  the prime numbers.
Consider  the product dynamical system
$$
(Y\times (X\times\Bbb T), \text{Leb}\otimes\widetilde\mu, I\times T_\phi).
$$
For $n\in\Bbb N$ and $y\in Y$, we find the unique atom of $P_n$
to whom $y$ belongs.
Let $\iota_n(y)$ be the corresponding prime number.
We now set
$$
\vartheta_n(y,x):=(y,\pi_{\iota_n(y)}(x)).
$$
Then $\vartheta_n\circ  (I\times T_\phi)= (I\times T_\phi)\circ\vartheta_n$ for each $n$.
Thus, we obtain an inverse sequence of dynamical systems
$$
\begin{CD}
(Y\times X\times\Bbb T,I\times T_\phi)@<\vartheta_1<<
(Y\times X\times \Bbb T, I\times T_\phi )@<\vartheta_2<<
\cdots.
\end{CD}
$$
Denote the inverse limit of this system as $(Y\times Z,\kappa, R)$.
Of course, the projection onto the first coordinate is the ergodic decomposition of $R$.
The corresponding marginal of $\kappa$ is Leb.
Disintegrate $\kappa$ with respect to Leb:
$$
\kappa=\int_{Y}\delta_y\otimes\kappa_y\, d\text{Leb}(y).
$$
Since for a.e. $y\in Y$, $y$ belongs to the $\iota_n(y)$-th atom of  $P_n$, it follows that the ergodic component
$(Z,\kappa_y, R)$ is the inverse limit of the inverse sequence
$$
\begin{CD}
(X\times\Bbb T,\widetilde\mu,T_\phi)@<\pi_{p_{\iota_1(y)}}<<
(X\times\Bbb T,\widetilde\mu, T_\phi )@<\pi_{p_{\iota_2(y)}}<<
\cdots.
\end{CD}
$$
It follows that  $(Z,\kappa_y, R)$ is isomorphic to
$(T_{\phi_{\boldsymbol\iota(y)}})$, where
$\boldsymbol\iota(y):=(\iota_n(y))_{n=1}^\infty$.
It follows from the construction that
$\{\iota_n(y)\colon n\in\Bbb N\}\ne \{\iota_n(y')\colon n\in\Bbb N\}$
for all $y,y'$ from a conull subset $B$ of $Y$.
Hence, the systems $(Z,\kappa_y, R)$, $y\in B$, are pairwise
non-isomorphic, as desired.

\end{Example}

\section{The strong $\bfu$-MOMO property}
Let us just recall, cf.~\eqref{jsjm2}, that a topological dynamical system $(Y,R)$ has the
strong $\bfu$-MOMO property if and only if for every $f\in C(Y)$, for every
increasing sequence $(b_k)_{k\ge 1}\subset\mathbb N$ such that
\[
b_{k+1}-b_k \xrightarrow[k\to\infty]{} \infty,
\]
and for every choice of points $y_k\in Y$, one has
\begin{equation}\label{eq:strongMOMO-bk}
\lim_{K\to\infty}\;
\frac{1}{b_K}
\sum_{k=1}^{K-1}
\left|
\sum_{b_k\leq n<b_{k+1}}\bfu(n)\, f\!\left(R^{n-b_k}y_k\right)
\right|
\;=\;0.
\end{equation}

\subsection{Proof of Proposition~\ref{p:rolemomo}
}
\begin{proof}[Proof of Proposition~\ref{p:rolemomo}]
We apply the Lifting Lemma from \cite{Ka-Ku-Le-Ru} to $T\times \text{Id}_{\mathbb{A}}$ and $S$. If $(N_k)$ is a sequence along which $\bfu$ is generic for $\kappa$, then there exists a sequence $(x_n,e_n)\subset X\times \mathbb{A}$ such that the density
$$
d(\{n\colon (x_{n+1},e_{n+1})\neq (Tx_n,e_n)\})
$$
 vanishes (in the logarithmic case, we assume that the logarithmic density $\delta(\{n\colon (x_{n+1},e_{n+1})\neq (Tx_n,e_n)\})$ is zero) and, by choosing a relevant $(b_k)$, we obtain $(y_k,e'_k)\subset X\times\mathbb{A}$ such that along a subsequence $(N_{k_\ell})$,
\begin{multline*}
\int_{X\times X_{\bfu}}(f\ot id_{\mathbb{A}})\ot\pi_0\,d\rho=
\lim_{\ell\to\infty}\frac1{N_{k_\ell}}\sum_{n<N_{k_\ell}}e_nf(x_n)
\pi_0(S^n\bfu)\\
=\lim_{K\to\infty}\frac1{b_K}\sum_{k<K}e_k'\sum_{b_k\leq n<b_{k+1}}f(T^ny_k)\bfu(n)=0,
\end{multline*}
the latter equality by the (strong) $\bfu$-MOMO property.

For the other direction, in order to show that
$$\lim\frac1{b_K}\sum_{k<K}\Big|\sum_{b_k\leq n<b_{k+1}}f(T^nx_k)\bfu(n)\Big|=0,$$
it is enough to show that for a relevant choice of $(e_k)\subset \mathbb{A}$, we have
$$
\lim\frac1{b_K}\sum_{k<K}\sum_{b_k\leq n<b_{k+1}}e_kf(T^nx_k)\bfu(n)=0.$$
To show this, we can assume that
$$
\frac1{b_K}\sum_{k<K}\sum_{b_k\leq n<b_{k+1}}\delta_{(T^nx_k,e_k,S^n\bfu)}\to \rho\in M(X\times\mathbb{A}\times X_{\bfu},T\times \text{Id}_{\mathbb{A}}\times S).$$
Moreover, the projection of $\rho$ on the $X_{\bfu}$-coordinate is $\kappa$ a measure which belongs to $V(\bfu)$.
Therefore, the limit we are interested in equals
$$
\int f\ot id_{\mathbb{A}}\ot\pi_0\,d\rho
$$
 which is  zero.
\end{proof}


\subsection{Some lemmas on the strong $\bfu$-MOMO property}
Throughout this section we assume that $\|\bfu\|_{u^1}=0$ which gives \eqref{eq:strongMOMO-bk} for $f=1$. Under this assumption, we recall that to show the property~\eqref{jsjm2} for all continuous functions, it is enough to consider linearly dense subset of continuous functions.

We collect below some facts concerning preservation of the $\bfu$-MOMO property by several simple extensions. The same results hold if we consider the logarithmic strong $\bfu$-MOMO property.

\paragraph{Product with an identity}

\begin{Lemma}\label{l:m19} If a homeomorphism $R'$ acting on a compact metric space $Y'$ satisfies the strong $\bfu$-MOMO property then $R'\times {\rm Id}$ (with the identity acting on $Y^{\prime\prime}$) also does.
\end{Lemma}
\begin{proof} It is enough to consider $F(y',y^{\prime\prime})=f(y')g(y^{\prime\prime})$ with $f\in C(Y')$ and $g\in C(Y^{\prime\prime})$. Given $(y'_k,y^{\prime\prime}_k)\in Y'\times Y^{\prime\prime}$ and $(b_k)\subset\N$ with $b_{k+1}-b_k\to\infty$, we have
\begin{align*}
\frac1{b_K}\sum_{k<K}&\Big|\sum_{b_k\leq n<b_{k+1}}F((R'\times \text{Id})^n(y'_k,y^{\prime\prime}_k))\bfu(n)\Big|\\
&=\frac1{b_K}\sum_{k<K}|g(y^{\prime\prime}_k)|\Big|\sum_{b_k\leq n<b_{k+1}}f({R'}^n(y'_k))\bfu(n)\Big|\\
&=O\Bigg(  \frac1{b_K}\sum_{k<K}\Big|\sum_{b_k\leq n<b_{k+1}}f({R'}^n(y'_k))\bfu(n)\Big|     \Bigg)\\
&=o(1)
\end{align*}
when $K\to\infty$.\end{proof}

\paragraph{Infinite product(s) of $A_2$}
\begin{Lemma}\label{p:momor}
If $A_2\colon(x,y)\mapsto (x,x+y)$ on $\T^2$ enjoys the strong $\bfu$-MOMO property, so does the infinite Cartesian product $R:=A_2^\Bbb N$ of $A_2$.
\end{Lemma}
\begin{proof}
To show the strong $\bfu$-MOMO condition, it suffices to study the characters.
Consider a character $\chi$ (we use additive notation) depending only on the $t$ first positions:
$$
\chi((x_1,y_1),\ldots,
(x_t,y_t),\ldots)=e^{2\pi i(\sum_{j}\alpha_jx_j+\sum_j\beta_jy_j)}
$$
with $\alpha_j,\beta_j\in\Z$ for $j=1,\ldots,t$.
So, we examine the sums of the form:
$$\frac1{b_K}\sum_{k\leq K}\Big|\sum_{b_k\leq n<b_{k+1}}\chi(R^n((x_1^{(k)},y_1^{(k)}),\ldots,
(x_t^{(k)},y_t^{(k)}),\ldots))\bfu(n)\Big|$$
and we want to show that they go to zero. We have
\begin{align*}
\Big|&\sum_{b_k\leq n<b_{k+1}}\chi(R^n((x_1^{(k)},y_1^{(k)}),\ldots,
(x_t^{(k)},y_t^{(k)}),\ldots))\bfu(n)\Big|\\
&=
\Big|\sum_{b_k\leq n<b_{k+1}}\chi((x_1^{(k)},nx_1^{(k)}+y_1^{(k)}),\ldots,
(x_t^{(k)},nx_t^{(k)}+y_t^{(k)}),\ldots))\bfu(n)\Big|\\
&=
\Big|\sum_{b_k\leq n<b_{k+1}}e^{2\pi i(\sum_j\alpha_jx_j^{(k)}+\sum_j\beta_j(nx_j^{(k)}
+y_j^{(k)}))}\bfu(n)\Big|\\
&=\Big|\sum_{b_k\leq n<b_{k+1}}e^{2\pi in(\sum_j\beta_jx_j^{(k)})}\bfu(n)\Big|,
\end{align*}
and the claim follows from the strong $\bfu$-MOMO of $A_2$  (for the character $(x,y)\mapsto e^{2\pi iy}$ and $(x_k,y_k)=(\sum_j\beta_jx^{(k)}_j,0)$ for $k\geq1$).
\end{proof}

\begin{Lemma}\label{l:momor1} Assume that $\Lambda\subset\T$ is a Borel subgroup of $\T$ with the property that for each closed $C\subset \Lambda$, the homeomorphism $A_2|_{C\times\T}\colon C\times\T\to C\times\T$ enjoys the strong $\bfu$-MOMO property. Let $\Lambda_i\subset\Lambda$ be closed for $i\geq1$. Then the homeomorphism $B$:
$$B(x_1,x_2,\ldots,y_1,y_2,\ldots):=(x_1,x_2,\ldots, y_1+x_1,y_2+x_2,\ldots)$$
acting on $\Lambda_1\times\Lambda_2\times\ldots\times\T^\infty$ enjoys the strong $\bfu$-MOMO property.\end{Lemma}
\begin{proof}
Again, notice that we only need to consider restrictions of characters of $\T^\infty\times\T^\infty$ to $(\prod_{i\geq1}\Lambda_i)\times\T^\infty$. If
$$
\chi(x_1,x_2,\ldots,y_1,y_2,\ldots)=e^{2\pi i(\sum_{j}\alpha_jx_j+\sum_j\beta_jy_j)}$$
with $\alpha_j,\beta_j\in\Z$ for $j=1,\ldots,t$ then we repeat the computation from the proof of the previous lemma and use the strong $\bfu$-MOMO property of
$A_2|_{\alpha_1\Lambda_1+\ldots+\alpha_t\Lambda_t}\times \T$
to obtain the assertion.\end{proof}

\paragraph{Shrinking wedge}
Assume now that $(X_n,d_n)$, $n\ge 1$, are compact metric spaces and let $T_n\colon X_n\to X_n$
be homeomorphisms. We recall first the {\em shrinking wedge} construction. Let
\[
X \;:=\; \{x_0\}\;\sqcup\;\bigsqcup_{n\ge 1} X_n
\]
be the disjoint union of the $X_n$ together with an additional point $x_0$.
Fix a sequence $(\delta_n)_{n\ge 1}$ with $\delta_n>0$ and $\delta_n\to 0$.
Assume moreover that each $X_n$ has been rescaled (if necessary) so that
\[
{\rm diam}(X_n)\le \delta_n.
\]

Define a metric $d$ on $X$ in the following way:
\begin{itemize}
\item if $x,y\in X_n$, set $d(x,y):=d_n(x,y)\, (\leq \delta_n)$;
\item if $x\in X_n$ and $y\in X_m$ with $n\neq m$, set
\[
d(x,y):=\delta_n+\delta_m;
\]
\item if $x\in X_n$, set $d(x,x_0):=\delta_n$, and $d(x_0,x_0):=0$.
\end{itemize}
One checks that $d$ is a metric and that $(X,d)$ is compact. Moreover, we have
${\rm diam}_d(X_n)\le \delta_n\to 0$ and every sequence $x_n\in X_n$ satisfies
$d(x_n,x_0)=\delta_n\to 0$, hence $x_n\to x_0$.

Consider the set of functions
\[
\mathcal D
:=
\Big\{f\in C(X): (\exists N\geq1)\; (\forall n\ge N)\; f|_{X_n}\equiv f(x_0)\Big\},
\]
that is, the set of functions eventually constant on the tail components and which agree there with their value at the wedge point $x_0$. It is not hard to see that it is dense in $C(X)$ with respect to $\|\cdot\|_\infty$.

Assume that $T_n$ is a homeomorphism of $X_n$ for all $n\geq1$ and define $T\colon X\to X$ by
\[
T(x)=
\begin{cases}
T_n(x), & x\in X_n,\\
x_0, & x=x_0.
\end{cases}
\]
Since each $X_n$ is $T$-invariant and $T_n$ is a homeomorphism, and since
$x_n\to x_0$ whenever $x_n\in X_n$, it follows that $T$ is a homeomorphism
of $X$.

Let $(b_k)_{k\ge 1}$ be any increasing sequence with $b_{k+1}-b_k\to\infty$.
Since $\|\bfu\|_{u^1}=0$,
\begin{equation}\label{eq:MR-needed}
\lim_{K\to\infty}\;
\frac{1}{b_K}
\sum_{k=1}^{K-1}
\left|
\sum_{b_k\leq n<b_{k+1}}\bfu(n)
\right|
\;=\;0.
\end{equation}

\begin{Prop}\label{thm:wedge-strongMOMO}
Assume that each $(X_n,T_n)$ satisfies strong $\bfu$-MOMO. Then the shrinking wedge system $(X,T)$ also satisfies
strong $\bfu$-MOMO.
\end{Prop}

\begin{proof}
Fix $f\in \mathcal{D}$, an increasing sequence $(b_k)$ with $b_{k+1}-b_k\to\infty$,
and points $x_k\in X$. We must show \eqref{eq:strongMOMO-bk} for $(X,T)$.

 Choose $N$ such that
\begin{equation}\label{eq:tail-const}
f|_{X_n}\equiv f(x_0)\qquad \text{for all }n\ge N.
\end{equation}
Split indices $k$ into:
\[
I_{\rm core}:=\{k:\ x_k\in X_i \text{ for some }1\le i\le N-1\} \text{ and }
I_{\rm tail}:=\mathbb N\setminus I_{\rm core}.
\]
For each $k$ define the block sum
\[
S_k(f):=\sum_{b_k\leq n<b_{k+1}}\bfu(n)\, f\!\left(T^{\,n-b_k}x_k\right).
\]

\smallskip
\noindent\emph{Tail indices.}
If $k\in I_{\rm tail}$, then either $x_k=x_0$ or $x_k\in X_n$ for some $n\ge N$.
In both cases, the orbit segment $\{T^{j}x_k:0\le j<b_{k+1}-b_k\}$ stays in
$\{x_0\}\cup\bigcup_{n\ge N}X_n$, hence by \eqref{eq:tail-const}, it is constant under $f$:
\[
f\!\left(T^{\,n-b_k}x_k\right)=f(x_0) \quad \text{for } b_k\le n<b_{k+1}.
\]
Therefore
\[
S_k(f)= f(x_0)\sum_{b_k\leq n<b_{k+1}}\bfu(n),
\]
and
\begin{equation}\label{eq:tail-MR}
\lim_{K\to\infty}\frac{1}{b_K}\sum_{\substack{1\le k\le K-1\\k\in I_{\rm tail}}}|S_k(f)|
\le |f(x_0)|\cdot
\lim_{K\to\infty}\frac{1}{b_K}\sum_{k=1}^{K-1}\left|\sum_{b_k\leq n<b_{k+1}}\bfu(n)\right|
=0.
\end{equation}

\smallskip
\noindent\emph{Core indices.}
If $k\in I_{\rm core}$, then $x_k\in X_{i(k)}$ for some $i(k)\in\{1,\dots,N-1\}$ and
$T^{j}x_k=T_{i(k)}^{j}x_k$. Hence
\[
S_k(f)=\sum_{b_k\leq n<b_{k+1}}\bfu(n)\, f\!\left(T_{i(k)}^{\,n-b_k}x_k\right),
\]
which is a block sum for $(X_{i(k)},T_{i(k)})$ with test function $f|_{X_{i(k)}}$.
Since each $(X_i,T_i)$, $1\le i\le N-1$, satisfies strong $\bfu$-MOMO, we have for each fixed $i$:
\[
\lim_{K\to\infty}
\frac{1}{b_K}\sum_{\substack{1\le k\le K-1\\ i(k)=i}}
|S_k(f)| = 0.
\]
Summing over the finitely many indices $i\in\{1,\dots,N-1\}$ yields
\begin{equation}\label{eq:core}
\lim_{K\to\infty}
\frac{1}{b_K}\sum_{\substack{1\le k\le K-1\\ k\in I_{\rm core}}}
|S_k(f)| = 0.
\end{equation}

\smallskip
\noindent
Adding \eqref{eq:tail-MR} and \eqref{eq:core} gives
\[
\lim_{K\to\infty}
\frac{1}{b_K}\sum_{k=1}^{K-1}|S_k(f)|=0,
\]
which concludes.
\end{proof}

\begin{Remark} The shrinking wedge model is ubiquitous when we want to extend the $\bfu$-orthogonality for uniquely ergodic systems to $\bfu$-orthogonality of topological systems having countably many ergodic measures (up to measure-theoretic isomorphism). For example, let us prove that under the assumption of Corollary~\ref{corollary:g}, we get the $\bfu$-orthogonality for systems belonging to $\cc_{\cf_{\rm ec}}$ having countably many ergodic measures (up to measure-theoretic isomorphism). Indeed, first of all, the orthogonality of uniquely ergodic systems extends to the strong $\bfu$-MOMO property of such systems, cf.\ Footnote~\ref{f:domomo}. Take now $(X,T)\in\cc_{\cf_{\rm ec}}$ with $M^e(X,T)$, up to measure-theoretic isomorphism, equal to $\{\nu_j\colon j\geq1\}$. Then find uniquely ergodic models $(Z_j,\kappa_j, R_j)$ for $(T,\nu_j)$, $j\geq1$. Next, use the shrinking wedge construction to obtain a dynamical system $(Z,R)$ satisfying the strong $\bfu$-MOMO property and having as ergodic invariant measures $\kappa'_j$ so that $(R,\kappa'_j)$  is isomorphic to  $(Z_j,R_j,\kappa_j)$, $j\geq1$. Finally, use Theorem~\ref{theorem:a}.
\end{Remark}

\subsection{A lemma on the graph measures in the product spaces}
Let us recall that if $R_i$ is an automorphism of $(Z_i,\cd_i,\kappa_i)$, $i=1,2$, and if $J \colon Z_1\to Z_2$ establishes an isomorphism of $R_1$ and $R_2$, then the formula
$$
\Delta_J(D_1\times D_2):=\kappa_1(D_1\cap J^{-1}(D_2)),$$
$D_i\in\cd_i$,
extends to a joining $\Delta_J\in J(R_1,R_2)$ called a {\em graph joining}. According to \cite{Le-Me} a joining $\rho\in J(R_1,R_2)$ is a graph joining if and only if it identifies the marginal sigma-algebras or, equivalently,
$$
L^2(Z_1,\kappa_1)\ot \raz_{Z_2}=\raz_{Z_1}\ot L^2(Z_2,\kappa_2)
$$
as   closed subspaces of $L^2(Z_1\times Z_2,\rho)$.

\begin{Lemma}\label{l:sm1} Let $(X,T)$, $(Y,S)$ be two topological systems. Then the set
\begin{multline*}
M^{\rm graph}(X\times Y,T\times S):=\big\{\rho\in M(X\times Y,T\times S)\colon \rho\text{ is a graph-joining}\\
\text{ of } (\pi_X)_\ast(\rho)\text{ and }(\pi_Y)_\ast(\rho)\big\}
\end{multline*}
is a Borel subset of $M(X\times Y,T\times S)$.
\end{Lemma}
\begin{proof}
Let $\{f_i\}_{i\geq1}$, $\{g_j\}_{j\geq1}$ be dense subsets in $C(X)$ and $C(Y)$, respectively. Given $\vep>0$, $i,j\geq1$, set
$$A_{i,j,\vep}:=\{\rho\in M(X\times Y,T\times S)\colon \|f_i\ot\raz_Y-\raz_X\ot g_j\|_{L^2(\rho)}<\vep\}.$$
Note that $A_{i,j,\vep}$ is an open set in the $*$-weak topology on $M(X\times Y,T\times S)$.\footnote{If the integral $\int_{X\times Y}|f_i\ot\raz_Y-\raz_X\ot g_j|^2\,d\rho$ is small then $\int_{X\times Y}|f_i\ot\raz_Y-\raz_X\ot g_j|^2\,d\rho'$ remains small if $\rho'$ is $*$-weakly close to $\rho$.}
Consider the set $$
D:=\bigcap_{\Q\ni \vep>0}\bigcap_{i\geq 1}\bigcup_{j\geq1}A_{i,j,\vep}
\cap\bigcap_{\Q\ni \vep>0}\bigcap_{j\geq 1}\bigcup_{i\geq1}A_{i,j,\vep}.$$
If we fix $i\geq1$ and $\rho\in D$ then for each $\vep>0$ the function $f_i\ot\raz_Y$ is $\vep$-close (in $L^2(\rho)$) to a certain function $\raz_X\ot g_j$, so letting $\vep\to0$, the function $f_i\ot\raz_Y$ is $\rho$-a.e.\ equal to a function of the form $\raz_X\ot g$, with $g\in L^2((\pi_Y)_\ast(\rho))$. The result now follows
since
$$
M^{\rm graph}(X\times Y,T\times S)=D.
$$
\end{proof}

\subsection{Strong $\bfu$-MOMO and Sarnak's conjecture in topological models: proofs of Theorem~\ref{theorem:a}, Theorem~\ref{theorem:e} and Corollary~\ref{corollary:g}}
Our first aim is to prove Theorem~\ref{theorem:a} which in a sense reduces the problem of orthogonality to ergodic measures in a topological system while their ``multiplicity'' turns out to be irrelevant.

The proof is based on the following.

\begin{Lemma}\label{l:irrmult1} Under the assumptions of Theorem~\ref{theorem:a}
there exists an (analytically) measurable map
$$
\Phi \colon M^e(Z,R)\to M^e(M(Z,R)\times X, \text{{\rm Id}}\times T)$$
which is 1-1 and $(R,\nu)$ is measure-theoretically isomorphic to $(\text{{\rm Id}}\times T,\Phi(\nu))$ for each $\nu\in M^e(Z,R)$.
\end{Lemma} \begin{proof}
Let us recall that the set
$$
\mathcal{R}:=M^{\rm graph}\big((M(Z,R)\times X)\times Z,(\text{Id}\times T)\times R\big)
$$
(see Lemma~\ref{l:sm1}) is a Borel subset
of  $M\big((M(Z,R)\times X)\times Z, (\text{Id}\times T)\times R\big)$.
Since the mapping $M(Z,R)\ni\kappa\mapsto\delta_\kappa\in M(M(Z,R))$ is continuous,
and the two marginal projections  $\mathcal R\ni \lambda\mapsto (\pi_{M(Z,R)})_*\lambda\in M(M(Z,R))$
and  $\mathcal R\ni \lambda\mapsto (\pi_Z)_*\lambda\in M(Z,R)$ are also continuous,
 it follows that
 the subset
$$
\mathcal{R}_1:=\{\lambda\in\mathcal{R}\colon \delta_{(\pi_Z)_\ast\lambda}=
(\pi_{M(Z,R)})_\ast\lambda\}
$$
is closed in $\mathcal R$.
Moreover (by our assumption),
 $$
 (\pi_Z)(\mathcal{R}_1)\supset M^e(Z,R).
 $$
By the Jankov-von Neumann theorem, there exists  an analytic selector $\phi \colon M^e(Z,R)\to \mathcal{R}_1$ of the projection $(\pi_Z)\restriction\mathcal R_1$.
It is straightforward to check that   the map $\Phi:=(\pi_{M(Z,R)\times X})\circ \phi$ satisfies  our claim.
\end{proof}

\begin{Remark}\label{r:allmeas} Under the assumptions of Lemma~\ref{l:irrmult1}, there is a measurable map $\Psi \colon M(Z,R)\to M(M(Z,R)\times X,{\rm Id}\times T)$ such that $(R,\nu)$ is measure-theoretically isomorphic to $({\rm Id}\times T,\Psi(\nu)).$
Indeed, if $\nu=\int_{M^e(Z,R)}\kappa\,d\PP(\kappa)$ then we set $\Psi(\nu):=\int_{M^e(M(Z,R)\times X, {\rm Id}\times T)} \xi\, d(\Phi_\ast(\PP))(\xi)$. Note that an isomorphism of the systems given by $\nu$ and $\Psi(\nu)$, respectively, follows from Proposition~\ref{fields} because $\Phi$ establishes an isomorphism of the corresponding spaces of ergodic components, while the ergodic components corresponding to $\kappa$ and $\xi=\Phi(\kappa)$ are isomorphic.
\end{Remark}

\begin{proof}[Proof of Theorem~\ref{theorem:a}]
Since $(X,T)$ satisfies the strong $\bfu$-MOMO property, by Lemma~\ref{l:m19}, so does $(M(Z,R)\times X,\text{Id}\times T)$. So with no loss of generality\footnote{Note that if the assumption of Theorem~\ref{theorem:a} is satisfied, the same assumption is satisfied for the extension as the only ergodic measures for ${\rm Id}\times T$ are of the form $\delta_\kappa\ot \nu'$.} we can assume that $\Phi \colon M^e(Z,R)\to M^e(X,T)$ is 1-1 and measurable (and sends an ergodic measure to an ergodic measure yielding isomorphic automorphisms).

Suppose that
$\bfu\not\perp (Z,R)$, then by a standard argument, there exist  $\xi\in M(Z,R)$, $\kappa\in V(\bfu)$ and $\rho\in J((Z,\xi),(S,\kappa))$ such that
$$
\int_{Z\times X_{\bfu}}f\ot\pi_0\,d\rho\neq0
$$
 for some $f\in C(Z)$.
Let $P$ be a probability measure on $M^e(Z,R)$ such that $\xi=\int_{M^e(Z,R)}\mu\, dP(\mu)$. Set
$$
\nu:=\int_{M^e(X,T)}\theta\, d(\Phi_\ast(P))(\theta).$$
Since $\Phi$ sends an ergodic measure to an ergodic measure yielding isomorphic automorphisms, and $\Phi$ is 1-1, in view of Proposition~\ref{fields}, we obtain that the automorphisms $(T,\nu)$ is isomorphic to $(R,\xi)$, say,  via a  map $J$.
Let $\rho':=(J\times \text{Id}_{X_{\bfu}})_\ast\rho$.
Then $\rho'$ is a joining  of $(T,\nu)$ with $(S,\kappa)$, and
$$
\int J(f)\ot\pi_0\,d\rho'\neq0.
$$
 Approximating $J(f)$ by a continuous function $G\in C(X)$, we obtain that
$\int G\ot\pi_0\,d\rho'\neq0$.
Since $(X,T)$ has the strong $\bfu$-MOMO property, we obtain a contradiction with Proposition~\ref{p:rolemomo}.
\end{proof}

\begin{proof}[Proof of Theorem~\ref{theorem:e}]
We repeat a simplified version of the proof of Theorem~\ref{theorem:a}.  Indeed, if there is no orthogonality, then we have a joining $\rho$ of $(R,\nu)$ (for some $\nu\in V(z_0)$) and a Furstenberg system $(X_{\bfu},\theta,S)$ with $\int f\ot\pi_0\,d\rho\neq 0$ for some continuous function $f\in C(Z)$.
Then, by our assumption we see that $(R,\nu)$ is a measure-theoretic factor of $\text{Id}_{[0,1]}\times T_\nu$ (with product measure,\footnote{We consider $P'\ot \kappa_\nu$ for a relevant $P'$ projecting onto $P$ on $\Gamma$.} see also Theorem~\ref{theorem:f}), moreover, the homeomorphism $\text{Id}\times T_\nu$ satisfies the strong $\bfu$-MOMO property by Lemma~\ref{l:m19}.
Use the relatively independent extension construction to obtain a joining $\rho'$ between $\text{Id}_{[0,1]}\times (T_\nu, \kappa_\nu)$ and $(X_{\bfu},\theta,S)$ with $\int f'\ot\pi_0\,d\rho'\neq 0$ for some $0\neq f'\in L^2$.
Since $f'$ can be approximated by a continuous function, we obtain a contradiction with the strong $\bfu$-MOMO of
$\text{Id}_{[0,1]}\times T_\nu$.
\end{proof}

\begin{proof}[Proof of Corollary~\ref{corollary:g}]
Take $(X,T)$ satisfying the assumptions of Corollary~\ref{corollary:g}.
\begin{itemize}
\item By assumption, for each visible $\nu\in M(X,T)$, the automorphism $(T,\nu)$ is a factor of  $\text{Id}\times E$, with $E$ ergodic.
\item $(X,T)\in \mathcal{C}_{\cf_{\rm ec}}$ yields additionally that $(T,\nu)$ is a factor of $\text{Id}\times E'$, where $E'$ is ergodic and belongs to $\cf$.
\item We can also assume that $E'$ is uniquely ergodic.
\item Hence  (by assumption), $E'\perp \bfu$, but in fact, due to \cite{Ka-Ku-Le-Ru}, we can assume that $E'$  satisfies the strong $\bfu$-MOMO property.
\item Hence, $(T,\nu)$ is a measure-theoretic factor of $\text{Id}\times E'$, with this homeomorphism satisfying strong $\bfu$-MOMO property.
\item This implies  $(X,T)\perp\bfu$ as in the proof of Theorem~\ref{theorem:e}.
\end{itemize}

As for the second part of Corollary~\ref{corollary:g}, we make a modification in the proof of Theorem~\ref{theorem:e}, by noticing that\footnote{We use consecutively: $f$ is in an $\cf$-system, $\ca_{\mathcal{F}}(S,\theta)=\ca_{\mathcal{G}}(S,\theta)$ by assumption, $ \EE(\pi_0|\ca_{\mathcal{G}}(S,\theta))$ is in a $\mathcal{G}$-system and $\EE(f|\ca_{\mathcal{G}}(T,\nu))$ is in a $\mathcal{G}$-system.}
\begin{align*}
\int f\ot\pi_0\, d\rho&=\int f\otimes \EE(\pi_0|\ca_{\mathcal{F}}(S,\theta))\,d\rho\\
&=
\int f\otimes \EE(\pi_0|\ca_{\mathcal{G}}(S,\theta))\,d\rho\\
&=\int \EE(f|\ca_{\mathcal{G}}(T,\nu))\otimes \EE(\pi_0|\ca_{\mathcal{G}}(S,\theta))\,d\rho\\
&=
\int \EE(f|\ca_{\mathcal{G}}(T,\nu))\otimes \pi_0\,d\rho
\end{align*}
since $(X_{\bfu},\theta,S)\in\mathcal{G}$, and then, continuing the proof with $\EE(f|\ca_{\mathcal{G}}(T,\nu))$ and $T|_{\ca_{\mathcal{G}}(T,\nu)}$ instead of $f$ and $(T,\nu)$.
\end{proof}

\begin{Remark} The proof of the second part requires only that $({\rm UE}\cap \mathcal{C}_{\mathcal{G}})\perp\bfu$. Note however that once this is satisfied and all $\ca_{\cf}$-factors of Furstenberg systems of $\bfu$ are in $\mathcal{G}$ then $({\rm UE}\cap \mathcal{C}_{\mathcal{F}})\perp\bfu$ whenever $\cf\supset\mathcal{G}$. Indeed, starting with a uniquely ergodic $(X,\nu,T)\in\cf$  if $\int f\ot\pi_0\,d\rho\neq 0$, then also
$\int \EE(f|\ca_{\mathcal{G}}(T,\nu))\ot\pi_0\,d\rho\neq0$ and ``rewriting'' this in the form $\int f'\ot\pi_0 \,d\rho'$ using a uniquely ergodic model of the factor $\ca_{\mathcal{G}}(T,\nu)$, we obtain a contradiction with the strong $\bfu$-MOMO property of this model. \end{Remark}

\subsection{Proof of Corollary~\ref{p:m17}}
We will show that the assumption of Theorem~\ref{t:ajmt} is satisfied. Indeed, fix any Furstenberg system $\kappa\in V(\bfu)$.
Using definition, there are ergodic automorphisms $R_i$ (and we can assume that $R_i\in\cf$) and joinings $\rho_i\in J(R_i\times \text{Id},T)$, $i\in I$, ``covering'' the space $L^2(X_{\bfu},\kappa)\ominus L^2({\rm Inv}_\kappa)$.
By the Jewett-Krieger theorem, no harm to assume that the $R_i$'s are uniquely ergodic.
Because $\bfu\perp {\rm UE}\cap\cf_{\rm ec}$, in fact, all $R_i$'s satisfy the strong $\bfu$-MOMO property.\footnote{This follows from \cite{Ab-Le-Ru} from the observation that the orbital systems coming from a uniquely ergodic system remain uniquely ergodic models of the same system.}

It follows from Lemma~\ref{l:m19} that $R_i\times \text{Id}\in\widehat{\mathcal{L}}$ and therefore,
the assumption in Theorem~\ref{t:ajmt} is satisfied. \bez


\section{Universal models for the characteristic class of automorphisms with relatively discrete spectrum over  identity}\label{s:relK}

\subsection{A necessary condition for existence of  universal models}

Assume that $\cf$ is a characteristic class and assume that a topological system $(X,T)$ is a universal model of $\cf$, that is:
\beq\label{model1}
(X,\mu,T)\in\cf \text{ for each }\mu\in M(X,T)\eeq
and for each $R\in{\rm Aut}\zdk$, $R\in\cf$,  there is $\mu\in M(X,T)$ such that
\beq\label{model2}
(X,\mu,T)\text{ and }(Z,\kappa,R)\text{ are isomorphic}.\eeq

\begin{Prop}\label{p:noum} If  $\cf$ has a universal model then
$\cf=\cf_{\text{{\rm ec}}}$.
\end{Prop}
\begin{proof}
Let $R\in {\rm Aut}\zdk$ belong to $\cf_{\rm ec}$.
We denote by $(\widebar Z,\widebar \kappa)$ the space of ergodic components of $R$ and by $\kappa=\int_{\widebar Z}\kappa_{\widebar z}d\widebar\kappa(\widebar z)$
the ergodic decomposition of $\kappa$.
Then there exists  a Borel subset $\widebar Z_0\subset\widebar Z$ with
$\widebar \kappa(\widebar Z_0)=1$ such that
$(Z,\kappa_{\widebar z}, R)\in\cf$ for all $\widebar z\in\widebar Z_0$.
Hence, thanks to (\ref{model2}),
there is a measure $\mu_{\widebar z}\in M^e(X,T)$ such that
$(X,\mu_{\widebar z},T)$ is isomorphic to $(Z,\kappa_{\widebar z},R)$
for each $\widebar z\in\widebar Z_0$.
It follows that
$$
(\pi_Z)_*(M^{\rm graph}(X\times Z,T\times R))\supset
\{\kappa_{\widebar z}\colon \widebar z\in\widebar Z_0\}.
$$
Since the mapping $(\pi_Z)_*\colon M^{\rm graph}(X\times Z,T\times R)\to M(Z,R)$ is continuous,
it follows from the Jankov-von Neumann theorem (see Theorem~\ref{th:JvN})
that there is an analytic mapping
$$
\tau \colon \{\kappa_{\widebar z}\colon \widebar z\in\widebar Z_0\}\to M^{\rm graph}(X\times Z,T\times R)
$$
 such that $(\pi_Z)_*\circ\tau=\text{Id}$.
Of course, the three dynamical systems:
$(Z,\kappa_{\widebar z},R)$, $(X\times Z,\tau(\kappa_{\widebar z}), T\times R)$ and $(X,(\pi_X)_*(\tau(\kappa_{\widebar z})), T)$ are pairwise isomorphic.
The three of them are ergodic.
Let $\mu:=\int_{\widebar Z}(\pi_X)_*(\tau(\kappa_{\widebar z}))\,d\widebar\kappa(\widebar z)$.
Then $\mu\in M(X,T)$.
Consider two dynamical systems: $(X,\mu, T)$ and $(Z,\kappa, R)$ and apply Proposition~\ref{p:Claim} to them.
By that proposition, $ (Z,\kappa, R)$ is a factor of the dynamical system $(X\times[0,1],\mu\otimes\text{{\rm Leb}}, T\times\rm{Id})$.
It follows that $R\in \mathcal F$.
Thus, $\cf_{\rm ec}\subset\cf$.

To prove the converse inclusion, take $R\in \cf$.
Then there exists $\mu\in M(X,T)$ such that $(X,\mu,T)$ is isomorphic to $R$.
Consider the ergodic decomposition
 $\mu=\int_{M^e(X,T)}\xi\,d\PP(\xi)$.
 By (\ref{model1}),  $(X,\xi,T)\in\cf$ for each $\xi\in M^e(X,T)$.
 Thus,   all the ergodic components of $R$ are in $\cf$.
 Hence,  $R\in\cf_{\rm ec}$.
\end{proof}

\begin{Cor} 
There is no universal model for the class ${\rm{\bf{DISP}}}$.
\end{Cor}

\begin{Remark} Although  we will show below that some ec-characteristic classes have universal models, the problem
of existence of universal models
 is open for many classical classes like the class DIST of distal automorphisms or NIL$_{s}$ for $s\geq2$, see \cite{Ka-Ku-Le-Ru}. See also Question~\ref{quest6}.\end{Remark}

\subsection{Algebraic preparations} \label{s:model}
Assume that $G$ is a compact abelian Polish group with Haar measure $\la_G$. We consider the space $M(G)$ of all Borel probability measures on $G$ endowed with the weak$^\ast$-topology. Given $g\in G$, set $K(g):=\ov{\{kg\colon k\in\Z\}}$. Then the map
\begin{equation}\label{eq:canonical}
G\ni g\mapsto  \la_{K(g)}\in M(G)
\end{equation}
 is measurable.
Indeed, for each character $\chi\in \widehat{G}$, the map $J_\chi\colon g\mapsto \int \chi\,d\la_{K(g)}$ is measurable as it takes only two values $0,1$ and
$$
\{g\in G\colon J_\chi(g)=1\}=\{g\in G\colon \chi(g)=1\}.
$$
 We consider (\ref{eq:canonical}) as a {\em canonical system of ergodic conditional measures on the fibers}.
We let $\ov{G}:=G\times G\times G$ and define $A \colon \ov{G}\to \ov{G}$
by setting
$$
A(g_1,g_2,g_3):=(g_1,g_2,g_3+g_1)$$
for all $(g_1,g_2,g_3)\in\ov{G}$. Then $A$ is a continuous group automorphism of $\ov{G}$. Now, given $\eta\in M(G\times G)$, we can define a probability $\widetilde{\eta}\in M(\ov{G})$ by setting
\beq\label{wzorm}\widetilde{\eta}=\int_{G\times G}\delta_{(g_1,g_2)}\otimes \la_{K(g_1)}\,d\eta(g_1,g_2).\eeq
Note that  $\widetilde{\eta}$ is $A$-invariant.\footnote{Indeed, for all $\chi_i\in\widehat{G}$ ($i=1,2,3$), we have
$$\int_{\ov{G}}\chi_1(g_1)\chi_2(g_2)\chi_3(g_3+g_1)\,
d\widetilde\eta(g_1,g_2,g_3)=$$$$
\int_{G\times G}\chi_1(g_1)\chi_2(g_2)\chi_3(g_1)
\Big(\int_{K_{g_1}}\chi_3(g_3)\,d\la_{K_{g_1}}(g_3)\Big)
\,d\eta(g_1,g_2)=$$$$
\int_{\ov{G}}\chi_1(g_1)\chi_2(g_2)\chi_3(g_3)\,
d\widetilde\eta(g_1,g_2,g_3)$$ since $\chi_3(g_1)=1$ iff $\int \chi_3\,d\la_{K_{g_1}}=1$. Furthermore, \eqref{wzorm} is the ergodic decomposition of~$\widetilde{\nu}$.}
Moreover, $\mathcal{I}(A)=\mathcal{B}_{G\times G}\ot\mathcal{N}_{G}$ since the fiber over $(g_1,g_2)$ is $K({g_1})$ with the corresponding Haar measure which in turn is ergodic for the rotation by $g_1$.
As usual, $\mathcal I(A)$ stands for the sigma-algebra of $A$-invariant subsets,
$\mathcal B_{G}$ for the entire Borel sigma-algebra on $G$, and $\mathcal N_G$ for the trivial sigma-algebra on $G$.
Note that if we set $F:=\chi_1\ot\chi_2\ot\chi_3$ then
\beq\label{ee1}
F\circ A (g_1,g_2,g_3)=\chi_3(g_1)F(g_1,g_2,g_3).\eeq
It follows that for each $n\in\Z$,
\beq\label{wzors}
\widehat{\sigma}_{F,A,\widetilde{\eta}}(n)=
\Big((\chi_3)_\ast(\eta|_G)\Big)\widehat{}(n),
\eeq
that is, the spectral measure of $F$ (for $(A,\widetilde{\eta})$) is the $\chi_3$ image of the projection of $\eta$ to the first coordinate.

A particular case arises, when we consider $G=\T^{\N}$.
If we write $y\in G$ as a sequence $(y_n)_{n\in\Bbb N}$ with $y_n\in\Bbb T$
for each $n$, then
$K_y$ is the dual group to the countable discrete Abelian group generated by set $\{y_n\colon n\in\Bbb N\}\subset\Bbb T$.

 Given $n\in\Bbb N$ and $m\in\Bbb Z$,
 we define a function $f_{n,m} \colon (\Bbb T^\Bbb N)^3\to\Bbb T$ by setting
 $f_{n,m}(y,v,z):=z_n^m$ for all $y=(y_n)_{n=1}^\infty$, $v=(v_n)_{n=1}^\infty$, $z=(z_n)_{n=1}^\infty\in \Bbb T^\Bbb N$.
 Then $f_{n,m}\in\widehat{(\T^{\N})^3}$, and \eqref{ee1} reads as
 $$
 f_{n,m}(A(y,v,z))=y_n^mf_{n,m}(y,v,z).
 $$

 We will also need the following.

 \begin{Lemma}\label{l:sama1}
 Assume that $G$ is a compact abelian group  and $H\subset G$ is its closed subgroup. Let $\kappa\in M(G/H)$. Then there exists a unique $\widetilde{\kappa}\in M(G)$ satisfying:
 \begin{enumerate}
\item[{\rm (i)}] $\widetilde{\kappa}$ is invariant under  rotations $R$ by the elements from $H$: $\widetilde{\kappa}\circ R_h^{-1}=\widetilde{\kappa}$ for each $h\in H$ and
\item[{\rm (ii)}]
  $\pi_\ast(\widetilde{\kappa})=\kappa$, where $\pi \colon G\to G/H$ is the canonical quotient map.
 \end{enumerate}
 \end{Lemma}
 \begin{proof} Let $s \colon G/H\to G$ be a Borel selector of $\pi$ and set
 \begin{equation}\label{eq:*}
 \widetilde{\kappa}:=\int_{G/H}\la_H\circ R_{s(gH)}^{-1}\,d\kappa(gH).
 \end{equation}
 Note that since $gH=s(gH)H$,
 the measure $\la_H\circ R_{s(gH)}^{-1}$ is supported on the coset $gH$ which
 is $\pi^{-1}(gH)$.
 Moreover,
 $$
 \pi_*(\widetilde\kappa)=\int_{G/H}\pi_*(\la_H\circ R_{s(gH)}^{-1})\,d\kappa(gH)
= \int_{G/H}\delta_{gH}\,d\kappa(gH)=\kappa.
 $$
 This means that (\ref{eq:*}) is the disintegration of $ \widetilde{\kappa}$ over $\kappa$
 along $\pi$.
 In particular,~(ii) holds.
 We also observe that
 $\la_H\circ R_{s(gH)}^{-1}\circ R_h^{-1}=\la_H\circ R_{s(gH)}^{-1}$
 for each $h\in H$ which reads as  (i).

On the other hand, let a measure $\xi\in M(G)$ satisfy (i) and (ii), i.e.
$\xi\circ R_h^{-1}=\xi$ for each $h\in H$ and
$\pi_*(\xi)=\kappa$.
Using (b), we disintegrate
 $\xi$ over $\kappa$ along $\pi$:
$$
\xi=\int_{G/H}{\xi}_{gH}\,d\kappa(gH).
$$
Then ${\xi}_{gH}$ is a probability measure supported on  $\pi^{-1}(gH)=gH\subset G$ for every point $gH\in G/H$.
Using (i),
we obtain that
\begin{equation}\label{eq:**}
 \xi=\xi\circ R_h^{-1}=\int_{G/H}\xi_{gH}\circ R_h^{-1}\,d\kappa(gH)\end{equation}
 for each $h\in H$.
 However, since $\xi_{gH}$ is supported on $gH$, it follows that
 the measure $\xi_{gH}\circ R_h^{-1}$ is also
 supported on  $gH$.
 This means that (\ref{eq:**}) is also a disintegration of $\xi$ over $\kappa$ along $\pi$.
 By the uniqueness of the disintegrations, we deduce from (\ref{eq:*}) and $(\ref{eq:**})$ that
 $\xi_{gH}\circ R_h^{-1}=\xi_{gH}$  at $\kappa$-a.e. $gH$ for every $h\in H$.
 Hence, $\xi_{gH}=\lambda_H\circ R_g^{-1}$ for $\kappa$-a.e. $gH$
 by the uniqueness of Haar measure on $H$.
 Thus, we obtain that   $\xi=\widetilde\kappa$.
 \end{proof}

 \begin{Lemma}\label{l:sama2} Let $\xi\in M(G, (R_h)_{h\in H})$. 
 Then
 $$
 \xi=\int_{G}\la_H\circ R_{g}^{-1}\,d\xi (g).
 $$
\end{Lemma}
 \begin{proof}
 We will utilize Lemma~\ref{l:sama1}.
 First, we note that the two measures $ \xi$ and  $\int_{G}\la_H\circ R_{g}^{-1}\,d\xi (g)$
 are invariant under $R_h$ for each $h\in H$.
 It remains to verify that both of them have the same pushforward onto $G/H$. Indeed, we have
 $$
 \begin{aligned}
 \pi_\ast\left(\int_{G}\la_H\circ R_{g}^{-1}\,d\xi (g)\right)&=
 \int_{G}\pi_*(\la_H\circ R_{g}^{-1})\,d\xi(g)\\
  &=\int_{G/H}\pi_*(\la_H\circ R_{s(gH)}^{-1})\,d\pi_*(\xi)(gH)\\
 &=
 \int_G\delta_{gH}\,d\pi_*(\xi)(gH)\\
 &=\pi_*(\xi),
 \end{aligned}
 $$
where the second equality follows from the fact that given $f$ on $G$ which is well defined on $G/H$ (i.e. $f(gh)=f(g)$ for each $h\in H$), we have that $\int_Gf(g)\,d\xi(g)=\int_{G/H}\widetilde{f}(gH)\,d\pi_\ast(\xi)(gH)$.
 \end{proof}

Our next purpose in this section is to prove the following theorem.

\begin{Th}\label{c:sasza17} 
There is an analytic subset $L\subset G\times G$ such that
for each $g\in G$, the set
$$
L_g:=\{s\in G\colon (g,s)\in L\}
$$
 is a cross-section
of the quotient $G\to G/K(g)$.
\end{Th}

To this end,
we first remind the concept of Borel field of Polish groups from \cite{Su}.
Let $\cal P$ stand for the set of all Polish groups.

\begin{Def}[\cite{Su}, Theorem 2.3] \label{d:borel} Let $X$ be a standard Borel space.
A mapping $X\ni x\mapsto G_x\in\cal P$ is called Borel if there is a standard Borel structure on the set
$Y:=\bigsqcup_{x\in X}G_x$ (the RHS set can be understood as $\bigcup_{x\in X}\{x\}\times G_x$) such that:
\begin{itemize}
\item
the projection $\pi \colon Y\to X$ is Borel,\footnote{This imply, in particular, that the set $Y\ast Y$ below is a Borel subset of $Y\times Y$.}
\item
the relative Borel structure on $\pi^{-1}(x)$ coincides with the Borel structure generated by the topology on $G_x$,
\item
the mappings
$$
Y*Y:=\{(y,y')\in Y\times Y\colon \pi(y)=\pi(y')\}\ni (y,y')\mapsto yy'\in Y
$$
 and
$Y\ni y\mapsto y^{-1}\in Y$ are Borel,
\item
there are countably many Borel maps $g_k \colon X\to Y$ with $g_k(x)\in G_x$ for all $x\in X$
and metrics $\delta_x$ on $G_x$ compatible with the topology such that $\{g_k(x)\colon k\in\Bbb N\}$ is dense in $G_x$ and the mapping $Y\ni y\mapsto\delta_{\pi(y)}(y, g_k\circ\pi(y))$ is Borel for each $k$.
\end{itemize}
\end{Def}

\begin{Prop}\label{Bstructure}
\begin{enumerate}
\item[\rm (i)]
The standard Borel $\sigma$-algebra on $Y$ is the smallest $\sigma$-algebra with respect to which the mappings $\pi$ and
$$
\theta_k \colon Y\ni y\mapsto \theta_k(y):=\delta_{\pi(y)}(y,g_k(\pi(y)))\in\Bbb R_+,\quad k\in\Bbb N,
$$
 are measurable.
 \item[\rm (ii)]
 The mapping $Y*Y\ni (y,y')\mapsto\delta_{\pi(y)}(y,y')$ is Borel.
 \end{enumerate}
 \end{Prop}

 \begin{proof}  (i) We first note that $\pi$ and $\theta_k$, $k\in\Bbb N$, are Borel.
 Hence, it suffices to show that these mappings separate points in $Y$.
 Take $y_1,y_2\in Y$, $y_1\ne y_2$.
 Assume that $\pi(y_1)=\pi(y_2)$ and $\theta_k(y_1)=\theta_k(y_2)$ for all $k\in\Bbb N$.
 It follows that $y_1,y_2\in G_{\pi(y_1)}$.
 Since the subset $\{g_k(\pi(y_1))\colon k\in\Bbb N\}$ is dense in $G_{\pi(y_1)}$, we obtain that
 $\delta_{\pi(y_1)}(y_1,z)=\delta_{\pi(y_1)}(y_2,z)$ for each $z\in G_{\pi(y_1)}$.
 Hence, $y_1=y_2$, a contradiction.

 (ii) Given $y\in Y$,  we can find an infinite sequence $k_1(y)<k_2(y)<\cdots$ such that
 $y=\lim_{n\to\infty}g_{k_n(y)}(\pi(y))$.
 Of course, we can do this choice in such a way that the mapping $Y\ni y\mapsto k_n(y)\in\Bbb N$
 is measurable for each $n$.
 For instance, $k_n(y)$ is the least number $m$ such that   $\delta_{\pi(y)}(y, g_{m}(\pi(y))< \frac1n$.
 Then, for each $(y,y')\in Y* Y$, we have:
 $$
 \delta_{\pi(y)}(y,y')=\lim_{n\to\infty}\delta_{\pi(y)}(y, g_{k_n(y')}\circ\pi(y)).
 $$
\end{proof}
 The following statement was implicitly  proved in \cite[Theorem 3.1(ii)]{Su}.
For the completeness of our argument, we present here an explicit proof.\\
\begin{Prop} \label{sasza9}
Let  $X\ni x\mapsto G_x$ and $X\ni x\mapsto H_x$
be  Borel mappings into Polish groups and $H_x$ is a closed normal subgroup of $G_x$.
If $\bigsqcup_{x\in X}H_x$ is a Borel subset of
$\bigsqcup_{x\in X}G_x$ then $X\ni x\mapsto G_x/H_x$ is a Borel mapping.
It follows also that the quotient mapping
\beq\label{sasza10}
q \colon Y\ni y\mapsto yH_{\pi(y)}\in \bigsqcup_{x\in X}G_x/H_x
\text{ is Borel.}\eeq
\end{Prop}

\begin{proof}  Let $\widebar Y:=\bigsqcup_{x\in X}G_x/H_x$ and let $\widebar\pi \colon \widebar Y\to X$
stand for the natural quotient mapping.
Given $y_1,y_2\in \bigsqcup_{x\in X}G_x/H_x$ with $\widebar \pi(y_1)=\widebar\pi(y_2)$, we let
\begin{align*}
\widebar\delta_x(z_1,z_2)&:=\inf_{y_1\in z_1,y_2\in z_2}\delta_x(y_1,y_2),\\
\widebar g_k(x)&:=g_k(x)H_x\in G_x/H_x
\end{align*}
for all $k\in\Bbb N$ and  $x\in X$.
According to the proof of \cite[Theorem 3.1(ii)]{Su}, $X\ni x\mapsto G_x/H_x$ is a Borel mapping to Polish groups with $\widebar\delta_x$ and $\widebar g_k$ playing the same role for this mapping  as
 $\delta_x$ and $ g_k$  for $X\ni x\mapsto G_x$ in  Definition~\ref{d:borel}.
It follows from  Proposition~\ref{Bstructure}(i) that  $q$ is Borel if and only if  the mappings
$\widebar\pi\circ q$ and
$$
Y\ni \widebar\delta_{ \widebar\pi\circ q(y)}(q(y),\widebar g_k( \widebar\pi\circ q(y)))\in\Bbb R_+,
$$
 $k\in\Bbb N$, are all Borel.
Let $h_k \colon X\to \bigsqcup_{x\in X}H_x$, $k\in\Bbb N$,
denote the Borel functions  from the definition that the mapping $x\mapsto H_x$ is Borel.
By assumption, the Borel $\sigma$-algebra on $\bigsqcup_{x\in X}H_x$ is the restriction
of the Borel $\sigma$-algebra on $Y$.
Hence, we can consider $h_k$ as a Borel function to $Y$ for each $k$.
Of course, $\widebar\pi\circ q=\pi$ and
$$
\widebar\delta_{ \pi(y)}(q(y),\widebar g_k( \pi(y)))=
\inf_{l,m\in\Bbb N}\delta_{ \pi(y)}\big(yh_l(\pi(y)), g_k( \pi(y))h_m(\pi(y))\big)
$$
for each $y\in Y$.
Given $l,k,m\in\Bbb N$, the mappings
$$
Y\ni y\mapsto (y, h_l(\pi(y)))\in Y*Y\quad\text{and}\quad Y\ni y\mapsto\big(g_k(\pi(y)),h_m(\pi(y))\big)\in Y*Y
$$
 are  Borel.
 Then it follows from Definition~\ref{d:borel} that the mappings
 $$
Y\ni y\mapsto yh_l(\pi(y)))\in Y\quad\text{and}\quad Y\ni y\mapsto g_k(\pi(y))h_m(\pi(y))\in Y
$$
 are also Borel.
 Applying Proposition~\ref{Bstructure}(ii), we obtain that
 $$
 Y\ni y\mapsto \delta_{ \pi(y)}(yh_l(\pi(y)), g_k( \pi(y))h_m(\pi(y)))\in\Bbb R_+
 $$
 is Borel  too.
Hence, the mapping $Y\ni y\mapsto \widebar\delta_{ \pi(y)}(q(y),\widebar g_k( \pi(y)))$ is Borel,
as desired.
\end{proof}

Let $X$ be a Polish space.
Let $\frak Z(X)$ stand for the set of all (nonempty) closed subsets of $X$.
Given an open subset $O\subset X$, let
$$
[O]:=\{Z\in\frak Z(X)\colon Z\cap O\ne \emptyset \}.
$$
Let $\mathcal B$  be the smallest $\sigma$-algebra on $\frak Z(X)$ which contains
$[O]$ for all open subsets in $X$.
It is called {\it Effros Borel}.
The following two facts are well known.
\begin{Th}[\cite{Au-Mo}]
The Effros Borel $\sigma$-algebra  is standard.
\end{Th}
\begin{Th}[Theorem 2.5 in~\cite{Su}]
Let $G$ be a Polish group and $X$ a standard Borel space. Let a mapping $X\ni x\mapsto G_x\in\frak Z(G)$ be Effros Borel.
\begin{enumerate}
\item[(i)] If $G_x$ is a subgroup of $G$ then $X\ni x\mapsto G_x$ is Borel (in the sense of Definition~\ref{d:borel}).
\item[
(ii)] the subset $\bigcup_{x\in X}\{x\}\times G_x$ is a Borel subset of $X\times G$.
\end{enumerate}
\end{Th}

\begin{Lemma}\label{Claim 4} Let $G$ be a compact Polish Abelian group.
Given $g\in G$, denote by $K(g)$ the closure of the cyclic group generated by
$g$.
Then, the mapping $G\ni g\mapsto K(g)$ is Effros Borel.\end{Lemma}
\begin{proof}  Let $\delta$ stand for a metric on $G$ compatible with the topology and let
$O$ be a nonempty open subset of $G$.
Then the following set
$$
\{g\in G\colon K(g)\in [O]\}=\{g\in G\colon K(g)\cap O\ne \emptyset\}=\bigcup_{n\in\Bbb Z}\{g\in G\colon g^n\in O\}
$$
is open in $G$.
\end{proof}

\begin{Prop}\label{pr:ququ} Let $G$ be a compact Polish Abelian group.
Then there is a standard Borel structure on the disjoint union $Y:=\bigsqcup_{g\in G}G/K(g)$ such that the quotient mapping
$
\tau \colon G\times G\ni (g,b){\mapsto} (g,bK(g))\in Y
$
is Borel.
\end{Prop}

\begin{proof}
It follows from
Lemma~\ref{Claim 4} and Fact 2(i) that $G\ni g\mapsto K(g)$ is  Borel (in the sense of Definition~\ref{d:borel}).
Of course, the constant mapping $G\ni g\mapsto G\in \mathcal P$ is Borel.
Indeed, we identify naturally
 the disjoint union $\bigsqcup_{g\in G}G$ with $G\times G$ and endow it with the product Borel structure.
 Then all the conditions from Definition~\ref{d:borel} are satisfied.
Hence, by  Lemma~\ref{Claim 4} and Fact 2(ii), the mapping $G\ni g\mapsto G/K(g)$ is Borel.
\end{proof}

\begin{proof}[Proof of Theorem~{\rm \ref{c:sasza17}}]
By Proposition~\ref{pr:ququ}, $\tau$ is a Borel onto mapping from $G\times G$ to $Y$.
Hence, the Jankov-von Neumann uniformization theorem (Theorem~\ref{th:JvN}) yields that there is an analytically measurable function $s \colon Y\to G\times G$ which is  a cross-section of $\tau$.
It remains to set $L:=s(Y)$.
\end{proof}

\subsection{Relatively discrete spectrum over  identity}
The concept of relatively discrete spectrum with respect to a factor is fundamental in ergodic theory \cite{Fu}, \cite{Zi}. However, in almost all results on this subject it is assumed that the factor relative to which we examine a given automorphism is ergodic. In particular, the Halmos-von Neumann theorem has been proved  \cite{Zi}, \cite{Le-Th-We} under such assumption.  In this section, we will consider exactly the opposite case, namely we study the relatively discrete spectrum over the sigma-algebra of invariant sets.\footnote{In this situation the whole automorphism is relatively ergodic over the factor.}

We start with recalling necessary definitions. Let $T\in{\rm Aut}\xbm$ and let $\ca\subset\cb$ be a factor. We say that a measurable map $f \colon X/\ca\to U(n)$  (where $U(n)$ stands for the group of unitary matrices of degree $n$) is a {\em relative eigenvalue} if there exists a measurable $F \colon X\to \C^n$ such that:
\beq\label{ltw1}f(\ov{x})F(x)=F(Tx)
\eeq
for $\mu$-a.e.\ $x\in X$, where $x\mapsto \ov{x}$
is the factor map
$$
(X,\cb,\mu,T)\to (X/\ca,\ca,\ov{\mu},\ov{T}).
$$
Then $F=(F_1,\ldots,F_n)$ is a {\em relative
eigenvector}.
Note that by \eqref{ltw1},
$$
\bigg(\sum_{j}|F_j(x)|^2\bigg)^{1/2}=\|F(x)\|=\|F(Tx)\|,
$$
 so we can assume without loss of generality that $F$ is bounded by truncating $F$ (say, $\tilde{F}(x)=F(x)$ when $\|F(x\|\leq M$ and $\tilde{F}(x)=0$ otherwise).
Note that the scalar functions $F_j$ generate a finite dimensional  $L^\infty(\ca)$-module which is $T$-invariant.

\begin{Def}
 $T$ is said {\em to have relative discrete spectrum over} $T|_{\ca}$ if $L^2(X,\mu)$ is generated by the coordinate (scalar) functions of relative eigenfunctions.
\end{Def}
This  is equivalent, see \cite{Fu0}, \cite{Le-Th-We}, to the fact that  $L^2(X,\mu)$ is generated by the $T$-invariant finite dimensional $L^\infty(\ca)$-modules.
Note also that if $f$ is a relative eigenvalue and $\eta \colon X/\ca\to U(n)$ is measurable then
\beq\label{coh1}f'(\ov{x}):=\eta(\ov{T}\ov{x})f(\ov{x})\eta(\ov{x})^{-1}
\eeq
is also a relative eigenvalue with the corresponding relative eigenvector $\eta\cdot F$. Furthermore, the
$L^\infty(\ca)$-modules generated by $F$ and $\eta \cdot F$ are the same. It follows that if we want to check whether $T$ has relatively discrete spectrum over $\ca$, it is enough to choose relative eigenvalues up to cohomology~\eqref{coh1}.

We are now interested in $\ca=\mathcal I(T)$, i.e. $\ca$
is the the sigma-algebra of $T$-invariant subsets.
In other words, $\mathcal  I(T)$ is the maximal factor of $T$ on which $T$ acts as the identity.
 In this case, we claim, we can restrict ourselves to $n=1$. Indeed, given a relative eigenvalue $\la$, we can find   measurable\footnote{This requires a proof: For $D(\ov{x})$ this looks obvious since the map which to $U\in U(n)$ associates the roots of the characteristic polynomial of $U$ is continuous. Now, $\eta(\ov{x})$ is of course not unique, but the ``fiber'' over $\ov{x}$ are compact, as it must be a coset of the centralizer of $D(\ov{x})$ in $U(n)$, so the Kallman-Mauldin theorem \cite{Kal-Ma} will give us a required measurable selector.} maps $D \colon X/{\mathcal I(T)}\to {\rm Diag}(n)\cap U(n)$ and $\eta \colon X/\mathcal I(T)\to U(n)$ such that
$$
\la(\ov{x})=\eta(\ov{x})D(\ov{x}) \eta(\ov{x})^{-1}\qquad\text{at a.e. }\ov{x}.
$$
Hence, $D$ is a relative eigenvalue with a relative eigenvector $G(x)=\eta(\ov{x})F(x)$. This means however that
$$  d_{jj}(\ov{x})G_j(x)=G_j(Tx),$$
where $D$ is the diagonal matrix with $d_{jj}$'s on the diagonal ($j=1,\ldots,n$). Since, as we have already noticed, the modules generated by $G$ and $F$ are the same, we have proved the following:\footnote{\label{f:chards} In \cite{Ka-Ku-Le-Ru} it was proved that $T\to \text{Id}$ has relatively discrete spectrum if and only if  (a.a.) fiber automorphisms (from~\eqref{edecomp}) $T_{\bax}$ have discrete spectrum.}

\begin{Prop}\label{p:m1} The factor mapping $T\to T|_{\mathcal I(T)}=\text{{\rm Id}}_{\mathcal I(T)}$ has relatively discrete spectrum if and only if the set of bounded measurable functions $F \colon X\to \C$ such that $F(Tx)=f(\ov{x})F(x)$ for some measurable $f \colon X/{\mathcal I(T)}\to\Bbb T$ is linearly dense (total) in $L^2(X,\mu)$.
\end{Prop}

\begin{Remark} \label{r:alge} Note that the linear space generated by the functions $F \colon X\to \C$ such that $F(Tx)=f(\ov{x})F(x)$, is an algebra of functions.\end{Remark}

\begin{Remark}\label{r:kiedy} If $F\in L^\infty(X,\mu)$ satisfies the equation $F(Tx)=f(\ov{x})F(x)$ for some measurable $f \colon X/\mathcal I(T)\to\Bbb T$ then the set $[F=0]$ is $T$-invariant.
Hence, we can assume  that $|F|=1$ on the remaining part (otherwise replace $F$ by a new function $\tilde{F}(x):=F(x)/|F(x)|$, which  is also a $\mathcal I(T)$-relative eigenfunction corresponding to $f$).
Then, for each $n\in\Z$,
$$
\int_XF(T^nx)\ov{F(x)}\,d\mu(x)=\int_{[F\neq0]}
f(\ov{x})^n\,d\ov\mu(\ov{x}).$$
It follows that the spectral measure of $F$ equals to
 the image of the measure $\ov\mu|_{[F\neq 0]}$ under the map $f$.
 \end{Remark}

For example, take a probability measure $
\sigma$ on $\Bbb T$ and consider the transformataion $(x,y)\mapsto (x,xy)$ of the torus $\Bbb T^2$ furnished with the product measure $\sigma\ot {\rm Leb}_{\T}$.
This transformation has relatively discrete spectrum (over the sigma-algebra corresponding to the first coordinate).
If $F(x,y)=y$ then $F$ is a relative eigenfunction corresponding to $f(x)=x$, and the spectral measure of $F$ equals $\sigma$.


\subsection{Relatively Kronecker factors over identity and {\bf DISP}$_{\text{ec}}$}
Let $(X,\mu,T)$ be  a dynamical system.
Let $\mu=\int_{\widebar X}\mu_{\widebar x}\,d\widebar \mu(\widebar x)$ stand for the ergodic decomposition of $\mu$ and
let $\pi \colon X\to\widebar X$ denote the corresponding factor mapping.
Then we obtain a corresponding decomposition of $L^2(X,\mu)$:
$$
L^2(X,\mu)=\int^{\oplus}_{\widebar X}L^2(X,\mu_{\widebar x})\,d\widebar\mu(\widebar x).
$$
Denote by $\mathcal K_{\rm rel}=\mathcal{K}_{\rm rel}(T)$ the $\mathcal I(T)$-relative Kronecker factor of $T$, i.e.
 the factor generated by the $\mathcal I(T)$-relative eigenfunctions of $T$.
Since  $\mathcal K_{\rm rel}\supset\mathcal I(T)$, the $\mathcal I(T)$-relative Kronecker factor is determined uniquely by the corresponding measurable
field $\widebar X\ni\widebar x\mapsto \mathcal K_{\widebar x}$ of factors $K_{\widebar x}$ of the fiber systems $(X,\mu_{\widebar x}, T)$, where
$\mathcal K_{\widebar x}=\mathcal K_{\rm rel}\cap\pi^{-1}(\widebar x)$.
Then $L^2(\mathcal K_{\rm rel},\mu)$ is a closed subspace of $L^2(X,\mu)$ and $L^2(\mathcal K_{\widebar x}, \mu_{\widebar x})$ is a closed subspace of $L^2(X,\mu_{\widebar x})$ generated by some genuine eigenfunctions (restrictions of relative eigenfunction to the fibers), and
\beq\label{rugu}
L^2(\mathcal K_{\rm rel},\mu)=\int^{\oplus}_{\widebar X}L^2(\mathcal K_{\widebar x},\mu_{\widebar x})\,d\widebar\mu(\widebar x).
\eeq

\begin{Lemma}\label{pr:giraf} $\mathcal K_{\widebar x}$ is the Kronecker factor of the (ergodic) system
$(X,\mu_{\widebar x}, T)$ for a.e. $\widebar x\in\widebar X$.
\end{Lemma}

\begin{proof}
For each $\widebar x\in\widebar X$, denote by
$\mathcal L_{\widebar x}$  the Kronecker factor of the (ergodic) system
$(X,\mu_{\widebar x}, T)$.
Then the system $(X,\mathcal L_{\widebar x},\mu_{\widebar x}|\mathcal L_{\widebar x},T)$ is ergodic with pure point spectrum.
Denote this spectrum by $\Lambda_{\widebar x}$.
Then the subset
$$
\Lambda:=\bigsqcup_{{\widebar x}\in{\widebar X}}\{\widebar x\}\times\Lambda_{\widebar x}
$$
is  a measurable subset of ${\widebar X}\times\Bbb T$.
Since $\Lambda_{\widebar x}$ is countable for each $\widebar x\in\widebar X$, there exist a countable family of functions
$f_n \colon \widebar X\to\Bbb T$ such that $\Lambda$ is exhausted by the graphs of $f_n$, $n\in\Bbb N$.
It follows that
$\Lambda_{\widebar x}=\bigcup_{n\in\Bbb N}f_n(\widebar x)$
at $\widebar\mu$-a.e. $\widebar x\in \widebar X$.
For each $n$, we can find a measurable
field $\widebar X\ni\widebar x\mapsto \widehat F_{n,\widebar x}$
of functions $\widehat F_ {n,\widebar x}\in L^2(\mathcal L_{\widebar x},\mu_{\widebar x}|\mathcal L_{\widebar x})$ such that
$$
\widehat F_{n,\widebar x}\circ T=f_n(\widebar x) \widehat F_n\quad\text{ and }\quad|\widehat F_{n,\widebar x}|=1.
$$
Of course, the linear span of $\{\widehat F_{n,\widebar x}\colon n\in\Bbb N\}$ is dense in
$L^2(\mathcal L_{\widebar x},\mu_{\widebar x}|\mathcal L_{\widebar x})$ for a.e. $\widebar x$ because it contains all the eigenfunctions of $(X,\mu_{\widebar x}, T)$.
We now let $F_n(x):=\widehat F_{n,\pi(x)}(x)$ for all $x\in X$.
Then $F_n$ is a measurable map from $X$ to $\Bbb T$
and $F_n\circ T=f_n\circ \pi \cdot F_n$.
 Thus, $F_n$ is  a relative eigenfunction  of $T$ with respect to $\mathcal I(T)$.
The family $\{F_{n}\colon n\in\Bbb N\}$ jointly with the family $\{D\circ \pi\colon D\in L^2(\widebar X,\widebar\mu)\}$
 generate a factor $\mathcal L$ of $(X,\mu, T)$.
 Of course, $\mathcal K_{\rm rel}\supset\mathcal L\supset\mathcal I(T)$ and hence
 \beq\label{guru}
L^2(\mathcal K_{\rm rel},\mu)\supset L^2(\mathcal L,\mu)=\int^{\oplus}_{\widebar X}L^2(\mathcal L_{\widebar x},\mu_{\widebar x})\,d\widebar\mu(\widebar x).
\eeq
 Since $\mathcal K_{\widebar x}\subset \mathcal L_{\widebar x}$ and, hence,
 $L^2(\mathcal K_{\widebar x},\mu_{\widebar x})\subset L^2(\mathcal L_{\widebar x},\mu_{\widebar x})$, we deduce from~(\ref{rugu}) and (\ref{guru}) that
 $L^2(\mathcal K_{\widebar x},\mu_{\widebar x})= L^2(\mathcal L_{\widebar x},\mu_{\widebar x})$, i.e.\
 $\mathcal K_{\widebar x}= \mathcal L_{\widebar x}$
 at $\widebar\mu$-a.e. $\widebar x$, as desired.
 \end{proof}

Denote by $\bf R$ the family of automorphisms which have  relatively discrete spectrum over the sigma-algebra of invariant subsets.
\begin{Prop}\label{re:mrak}  We have
${\bf R}={\bf DISP}_\text{{\rm ec}}$.
Hence, for each automorphism $T$, we have $\mathcal{K}_{\rm rel}(T)=
\ca_{{\rm {\bf DISP}}_{\rm ec}}(T).$
\end{Prop}
\begin{proof}
Take a system
 $(X, \mu,T)\in {\bf R}$.
 Let $\mu=\int_{\widebar X}\mu_{\widebar x}\,d\widebar\mu(\widebar x)$ stand for the ergodic disintegration of $\mu$.
 It follows from Lemma~\ref{pr:giraf}
 that  $(X,\mu_{\widebar x}, T)$ has pure point spectrum.
 Thus, $(X,\mu,T)\in {\bf DISP}_\text{ec}$.

Conversely, take a dynamical system
 $(X,\mathcal B,\mu,T)\in {\bf DISP}_\text{ec}$.
 Let $\mathcal K$ be the $\mathcal I(T)$-relative Kronecker factor of $(X,\mu,T)$.
As above,  $\mu=\int_{\widebar X}\mu_{\widebar x}d\widebar \mu(\widebar x)$ stands for the ergodic disintegration of $\mu$.
Then by Lemma~\ref{pr:giraf}, $\mathcal K_{\widebar x}$ is the Kronecker factor
of the fiber system $(X,\mu_{\widebar x}, T)$.
As   $(X,\mu,T)\in {\bf DISP}_\text{ec}$, this Kronecker factor is
 the entire
system $(X,\mu_{\widebar x}, T)$.
Thus, $\mathcal K_{\widebar x}=\mathcal B_{\widebar x}$ at a.e. $\widebar x\in\widebar X$.
It follows from Proposition~\ref{2factors} that $\mathcal K=\mathcal B$,
 i.e.
$(X,\mathcal B,\mu, T)\in\bf R$.
\end{proof}

\subsection{Universal model for  {\bf DISP}$_{\text{ec}}$ }
Let $(X,\mathcal  B,\mu)$ be a standard probability space and let $T\in {\rm Aut}(X,\mu)$.
 In Proposition~\ref{p:m1}, we showed that $T$ has relative discrete spectrum over $\mathcal{I}(T)$ if and only if  the liner span of all $\mathcal I(T)$-relative eigenfunctions of $T$ is dense in $L^2(X,\mu)$.
Note that if $f$ is an $\mathcal I(T)$-relative eigenfunction of $T$ then $|f|\circ T=|f|$.
Hence, the mapping $X\ni x\mapsto |f(x)|\in \Bbb R$ is $\mathcal  I(T)$-measurable. Therefore, if $\tau \colon X\to X/{\mathcal{I}(T)}$ is the (natural) factor map, $\tau(x)=\bax$, then $|f(x)|=|f(y)|$ whenever $\bax=\ov{y}$.

\begin{Th}\label{discr} Let $T\in {\rm Aut}(X,\mathcal B,\mu)$ have an $\mathcal I(T)$-relatively discrete spectrum.
Then there is  a probability $\eta$ on $\Bbb T^\Bbb N\times \Bbb T^\Bbb N$
such that  $(X,\mu,T)$ is isomorphic to  $(\Bbb T^\Bbb N\times \Bbb T^\Bbb N\times \Bbb T^\Bbb N,\widetilde\eta, A)$, where
$$
A(y,v,z)=(y,v,yz)\qquad\text{and}\qquad\widetilde\eta=\int_{\Bbb T^\Bbb N\times \Bbb T^\Bbb N}\delta_{y,v}\otimes\lambda_{K(y)}\,d\eta(y,v)
$$
 and $\mathcal I(A)=\mathcal B_{\Bbb T^\Bbb N\times \Bbb T^\Bbb N}\otimes \mathcal N_{\Bbb T^{\Bbb N}}$.
\end{Th}
\begin{proof}
We  note that the linear space $\mathcal S$ generated  by all bounded $\mathcal I(T)$-relative eigenfunctions of $T$ is
a $T$-invariant $*$-subalgebra in $L^\infty(X,\mu)$.
It contains $L^\infty\big(X,\mathcal I(T),\mu|_{\mathcal{I}(T)}\big)$ and therefore, there is a $T$-factor $\mathcal F\subset\mathcal B$ such that $\mathcal F
\supset\mathcal I(T)$ and
$\mathcal S=L^\infty(X,\mathcal F,\mu)$.
But  each $\mathcal  I(T)$-relative eigenfunction of $T$ is a limit  (in $L^2(X,\mu)$) of a sequence of bounded $\mathcal  I(T)$-relative eigenfunctions of $T$, so
since $T$ has an
$\mathcal  I(T)$-relatively discrete spectrum, it follows that $L^\infty(X, \mathcal F,\mu)$ is dense in $L^2(X,\mu)$.
This yields that $\mathcal F=\mathcal B$.
Hence, there exists a countable family
$\{f_n\}_{n=2}^\infty$ of $\mathcal I(T)$-relative eigenfunctions of $T$ such that the smallest sigma-algebra making all of $f_n$ measurable is equal to $\mathcal B$. It follows that
$\{f_n\}_{n=2}^\infty$ separate points of $X$ \cite[Section 12.A]{Ke}. 
We have  that $f_n(Tx)=a_n(\bax)f_n(x)$ $\mu$-a.e.\ (with $|a_n|=1$).
For each $n>1$ and $x\in X$, we now set
$$
\widetilde f_n(x)=
\begin{cases}
f_n(x)/|f_n(x)|, &\text{if $f_n(x)\ne 0$ and}\\
1, &\text{if $f_n(x)= 0$,}
\end{cases}
$$
and
$$
\widetilde a_n(x)=
\begin{cases}
a_n(\bax), &\text{if $f_n(x)\ne 0$ and}\\
1, &\text{if $f_n(x)= 0$.}
\end{cases}
$$
Let ${f}_1 \colon X/\mathcal{I}(T)\to\Bbb T$ be a one-to-one Borel  mapping and let $\widetilde f_1:=f\circ\tau$
and
$\widetilde a_1:=1$.
Then,  $\widetilde f_n\circ T=\widetilde a_n\widetilde f_n$ for each
$n\ge 1$.

For $x\in X$, let
$a(\bax):=(\widetilde a_n(\bax))_{n=1}^\infty\in  \Bbb T^\Bbb N$ and
$f(x):=(\widetilde f_n(x))_{n\in\Bbb N})\in \Bbb T^\Bbb N$.
We now consider the mapping
$$
\Phi \colon X\ni x\mapsto\Phi(x):=(a(\bax), f(x))\in \Bbb T^\Bbb N \times \Bbb T^\Bbb N.
$$
We claim that  $\Phi$ is one-to-one.
Indeed, let $x,y\in X$ and $x\ne y$.
If there is a subset $B\in\mathcal I(T)$ such that $x\in B$ but $y\not\in B$
then $\bax\neq\ov{y}$, whence $\widetilde f_1(x)={f}_1(\bax)\ne
{f}_1(\ov{y})=\widetilde f_1(y)$.
Otherwise, 
$\bax=\ov{y}$ and there is $n>1$ such that $ f_n(x)\ne  f_n(y)$.
Since  $|f_n|$ is $\mathcal I(T)$-measurable, it follows that $| f_n(x)|=|  f_n(y)|$.
Of course, $| f_n(x)|\ne 0$ (otherwise, $0=f_n(x)=f_n(y)$).
Hence,
$\widetilde f_n(x)\ne \widetilde f_n(y)$.
Thus, if $x\ne y$ then $f(x)\ne f(y)$.
This, in turn, implies that $\Phi(x)\ne \Phi(y)$.
Hence, $\Phi$ is one-to-one, as claimed.

Set $B \colon \T^{\N}\times \T^{\N}\to\T^{\N}\times\T^{\N}$, $B(y,z):=(y,yz)$. It is straightforward to verify that $\Phi \circ T=B\circ \Phi$.
Then $\Phi$ is an isomorphism of $(X,\mu,T)$ onto $(\Bbb T^\Bbb N\times \Bbb T^\Bbb N,\mu\circ \Phi^{-1},B)$.
We denote the push-forward of $\mu\circ \Phi^{-1}$ onto the first coordinate $Y=\Bbb T^\Bbb N$ by $\eta^*$.
Consider the disintegration
$$
\mu\circ \Phi^{-1}=\int_{\Bbb T^\Bbb N}\delta_y\otimes\eta_y\,d\eta^*(y)
$$
of $\mu\circ \Phi^{-1}$ with respect to $\eta^*$.
Then
 the mapping $Y\ni y\mapsto\eta_y\in {M}( \Bbb T^\Bbb N)$ is measurable.
 For each $y
 \in  \Bbb T^\Bbb N$, denote by $R_y$ the $y$-rotation on $ \Bbb T^\Bbb N$.
Of course,  $\eta_y$  is invariant  under $R_y$.
It follows from Lemma~\ref{l:sama2} that
\beq\label{sasza18}
\eta_y=\int_{\Bbb T^{\Bbb N}}\lambda_{K(y)}\circ R_z^{-1}\,d\eta_y(z).\eeq
We now apply Corollary~\ref{c:sasza17}.
According to this corollary, there exists an analytic subset $L\subset\Bbb T^\Bbb N\times
\Bbb T^\Bbb N$ such that for each $y\in\Bbb T^\Bbb N$, the set
$$
L_y:=\{s\in\Bbb T^\Bbb N\colon (y,s)\in L\}
$$
 is an (analytic) cross section
of the quotient mapping $\Bbb T^\Bbb N\to \Bbb T^\Bbb N/K(y)$.
Hence, for each $y\in Y$, there is  a bijection $\Psi_y \colon \Bbb T^\Bbb N\to L_y\times K(y)$ such that
$$
\Psi_y^{-1}(a,k):=ak \qquad\text{for all $a\in L_y$ and $k\in K(y)$.}
$$
It follows from \eqref{sasza18} that $\eta_y\circ \Psi_y^{-1}=\overline\eta_y\otimes\lambda_{K(y)}$ for some probability $\overline\eta_y$ on $L_y$.
We now define a mapping
$\Psi \colon \Bbb T^\Bbb N\times \Bbb T^\Bbb N\to \Bbb T^\Bbb N\times \Bbb T^\Bbb N\times \Bbb T^\Bbb N$
by setting
$$
\Psi(y,z)=
(y, \Psi_y(z)).
$$
Since $A$ is analytic, $\Psi$ is analytic.
Therefore, the mapping $\Bbb T^\Bbb N\ni y\mapsto\overline\eta_y$
 is analytic and, hence,
measurable  mod $\eta^*$.
Thus,
 $\Psi$ is a measure theoretical isomorphism
of $(\Bbb T^\Bbb N\times \Bbb T^\Bbb N, \mu\circ \Phi^{-1})$ onto $( \Bbb T^\Bbb N\times \Bbb T^\Bbb N\times \Bbb T^\Bbb N, \mu\circ \Phi^{-1}\circ \Psi^{-1})$ and
$$
\mu\circ (\Psi\Phi)^{-1}=
\int_{\Bbb T^\Bbb N}\delta_y\otimes\overline\eta_y\otimes\lambda_{K(y)}\,d\eta^*(y)=
\int_{\Bbb T^\Bbb N\times \Bbb T^\Bbb N}\delta_{(y,v)}\otimes\lambda_{K(y)}\,d\eta(y,v),
$$
where $\eta:=\int_{\Bbb T^\Bbb N}\delta_y\otimes\overline \eta_yd\eta^*(y)$.
Thus, $\mu\circ (\Psi\Phi)^{-1}=\widetilde\eta$.
Of course, $\Psi T=A\Psi$.
\end{proof}

We call the dynamical system $( \Bbb T^\Bbb N\times \Bbb T^\Bbb N\times \Bbb T^\Bbb N,\widetilde \eta,A)$ {\it a continuous  algebraic model of $(X,\mu,T)$}.
Of course, $( \Bbb T^\Bbb N\times \Bbb T^\Bbb N\times \Bbb T^\Bbb N,\widetilde \eta,A)$ has
a $(\mathcal B_{\Bbb T^\Bbb N\times \Bbb T^\Bbb N}\otimes \mathcal N_{\Bbb T^{\Bbb N}})$-relatively discrete spectrum.

\begin{Remark}\label{Respect} It is easy to see that given an $A$-invariant probability measure $\nu$ on
$\Bbb T^\Bbb N\times \Bbb T^\Bbb N\times \Bbb T^\Bbb N$, the dynamical system
$(\Bbb T^\Bbb N\times \Bbb T^\Bbb N\times \Bbb T^\Bbb N,\nu, A)$ has a $\mathcal I(A)$-relatively discrete spectrum.
Indeed, each character of $\Bbb T^\Bbb N\times \Bbb T^\Bbb N\times \Bbb T^\Bbb N$ is an $\mathcal I(A)$-relative eigenfunction of $A$.
It remains to note that the set of all characters is total in $L^2(\Bbb T^\Bbb N\times \Bbb T^\Bbb N\times \Bbb T^\Bbb N,\nu)$.
\end{Remark}
From Theorem~\ref{discr}, Proposition~\ref{re:mrak}  and  {Remark}~\ref{Respect} we deduce the main result of this subsection.

\begin{Cor}\label{cor:DISP-mod}
The topological  system
$( \Bbb T^\Bbb N\times \Bbb T^\Bbb N\times \Bbb T^\Bbb N,A)$
is a universal model for the
 characteristic class {\bf DISP}$_\text{ec}$.
\end{Cor}

\subsection{More corollaries from Theorem \ref{discr}: Theorem~\ref{t:m1}, relative Halmos-von Neumann theorem, etc.}

Lemmas~\ref{p:momor}, \ref{l:m19} and Corollary~\ref{cor:DISP-mod}  yield the following.

\begin{Cor}\label{c:strongmomo}
Let $\bfu \colon \N\to\mathbb{U}$. Assume that $A_2(x,y):=(x,y+x)$ (on $\T^2$) satisfies the strong $\bfu$-MOMO property.
Then each automorphism $T\in{\bf DISP}_{{\rm ec}}
$ has a topological model satisfying the strong $\bfu$-MOMO property.
\end{Cor}

Utilizing Corollary~\ref{c:strongmomo} we can now prove
Theorem~\ref{t:m1}.

\begin{proof}[Proof of Theorem~\ref{t:m1}]
{\bf Necessity.}
By   \cite[Corollary 1.10]{Ka-Ku-Le-Ru},
the following conditions are equivalent:
\begin{itemize}
\item[---]
$
\|\bfu\|_{u^2}=0,
$
\item[---]
 $\bfu\perp \mathcal{C}_{{\rm {\bf{DISP}}}_{\rm ec}}$ and
\item[---]
each topological system in which all ergodic measures yield
discrete spectrum enjoys the strong $\bfu$-MOMO property.
\end{itemize}
In particular, $A_2\colon (x,y)\mapsto (x,y+x)$ (on $\T^2$) satisfies the strong $\bfu$-MOMO property. The latter is known to be equivalent to the local 1-Fourier uniformity, see e.g.\ \cite{Ka-Le-Ri-Te}.

{\bf Sufficiency.} So, we assume the local 1-Fourier uniformity of $\bfu$, that is, the strong $\bfu$-MOMO property of $A_2$. We would like to prove that $\bfu\perp \mathcal{C}_{{\rm {\bf{DISP}}}_{\rm ec}}$ (see the proof of necessity). All we need to show is that the assumption of Theorem~\ref{t:ajmt} is satisfied. However, this last claim holds thanks to Corollary~\ref{c:strongmomo}.
\end{proof}

The following statement can be considered as a non-ergodic version of  the Halmos-von Neumann isomorphism theorem.
It follows from
Theorem~\ref{discr}, Proposition~\ref{fields} and the classical Halmos-von Neumann theorem on isomorphism of ergodic transformations with pure discrete spectrum.

\begin{Th}\label{HvN} Let $T\in {\rm Aut}(X,\mu)$ be an automorphism with $\mathcal I(T)$-relatively discrete spectrum
and let $T'\in {\rm Aut}(X',\mu')$ be another automorphism with $\mathcal I(T')$-relatively discrete spectrum.
Denote by $\Lambda_{\bax}$ the group  of eigenvalues of $T_{\bax}$ (and use similar notation for $\Lambda_{\bax'}$).
 Then, $T$ and $T'$ are isomorphic if and only if there is a measure theoretical isomorphism
$\phi \colon (\baX,\bamu)\to(\baX',\bamu')$
such that $\Lambda_{\phi(\bax)}=\Lambda_{\bax}$ for $\bamu$-a.e. $\bax\in\baX$.
\end{Th}

The following result explains when an automorphism with a relative discrete spectrum (over the identity) has the smallest possible Kronecker factor.

\begin{Prop} \label{p:RelKroMax}
Assume that  $T\in {\rm {\bf{DISP}}}_{\rm ec}$.
Then  the  Kronecker factor of $T$ (which always contains $\mathcal I(T)$) is equal to $\mathcal I(T)$
if and only if $T_{\bax}\perp T_{\bax'}$ for $\bamu\ot\bamu$-a.a.\ $(\bax,\bax')\in \baX\times \baX$.
\end{Prop}
\begin{proof}
Suppose first that the   Kronecker factor of $T$ is $\mathcal I(T)$.
All we need to show is that $T\perp {\rm Erg}$ \cite{Go-Le-Ru}.
So, suppose  that $T$ is not disjoint with an ergodic automorphism $R$.
 As $T\in {\rm {\bf{DISP}}}_{\rm ec}$,
there is a non-trivial joining of $T$ with
  $\mathcal A_{{\rm {\bf{DISP}}}_{\rm ec}}(R)$.
Since $R$ is ergodic, $\mathcal A_{{\rm {\bf{DISP}}}_{\rm ec}}(R)=\mathcal A_{{\rm {\bf{DISP}}}}(R)$, i.e. $\mathcal A_{{\rm {\bf{DISP}}}_{\rm ec}}(R)$
 is the Kronecker factor of $R$.
 Then $T$ has non-trivial Kronecker factor, a contradiction.

 Conversely,  if the Kronecker factor of $T$ contains non-trivial invariant subsets, then there exist  $\lambda\in\Bbb T\setminus\{1\}$ and  a nontrivial $T$-eigenfunction $F$ such that $F\circ T=\lambda F$.
 Denote by $R_\lambda$ the rotation by $\lambda$ on $\Bbb T$.
 Then, the dynamical system $(\Bbb T, F_*(\mu), R_\lambda)$
 is a factor of $T$ (via the mapping $F$).
 If $\lambda$ is irrational then
$(\Bbb T, F_*(\mu), R_\lambda)$ is ergodic.
If $\lambda$ is rational and $n$ is the least period of $\lambda$,
then it is always possible to choose  $F$ in such a way that $F$ takes values in the subgroup
$\{z\in\Bbb T\colon z^n=1\}$.
Then  $(\Bbb T, F_*(\mu), R_\lambda)$ is ergodic.
Thus, in every case, $T$ has a nontrivial ergodic factor and, hence,
$T\not\perp {\rm Erg}$.
This implies that $T_{\bax}\not\perp T_{\bax'}$ for $\bamu\ot\bamu$-a.a.\ $(\bax,\bax')\in \baX\times \baX$ \cite{Go-Le-Ru}.
 \end{proof}



\subsection{Generalization to the relatively $\Lambda$-discrete spectrum over identity. Universal models for {\bf DISP}$(\Lambda)_{{\rm ec}}$}

Let $T$ be an automorphism of $\xbm$.
As above, let  $(\widebar X,\widebar\mu)$, $\pi \colon X\to\widebar X$ and $\mu=\int_{\widebar X}\mu_{\widebar x}\,d\widebar \mu(\widebar x)$ denote the space of ergodic components of $T$, the corresponding factor mapping and the corresponding disintegration of $\mu$ respectively.
Note that if
 $\{F_n\}_{n=1}^\infty$ is a total family of $\mathcal I(T)$-relative eigenfunctions in $L^2(\mathcal K_{\rm rel}(T),\mu)$ and $F_{n,\widebar x}:=F_n\restriction \pi^{-1}(\widebar x)$
 then
 $(F_{n,\widebar x})_{n=1}^\infty$
is a total orthogonal family in $L^2(\mathcal K_{\widebar x},\mu_{\widebar x})$ (some elements of this family can be 0).
Hence,
\beq\label{slon}
L^2(\mathcal K_{\rm rel},\mu)=\int_{\widebar X}^{\oplus}L^2(\mathcal K_{\widebar x},\mu_{\widebar x})\,d\widebar \mu(\widebar x).
\eeq
Let $\Lambda$ be a Borel subgroup  in $\Bbb T$.
Denote by $\mathcal K_{\rm rel}^\Lambda=\mathcal K_{\rm rel}^\Lambda(T)$ the $T$-factor generated by
the algebra of $\mathcal I(T)$-relative eigenfunctions $F$ of $T$ such that
$F\circ T=(a\circ\pi)\cdot F$  $\mu$-almost everywhere with $a$ taking values in $\Lambda$. We call it {\em relative $\Lambda$-Kronecker factor} of $T$.\footnote{Note that $\mathcal{K}_{\rm rel}^\Lambda(T)\supset\mathcal{I}(T)$ and the relative $\Lambda$-Kronecker factor has relatively discrete spectrum over $\mathcal{I}(T)$.}
In a similar way, let
$\mathcal L^\Lambda_{\widebar x}$ be the factor of
$(X,\mu_{\widebar x},T)$ generated by
the   eigenfunctions whose eigenvalues
belong to  $\Lambda$.

The following statement is a $\Lambda$-analogue of Lemma~\ref{pr:giraf}.

\begin{Lemma}\label{kapitan}
$(\mathcal K^\Lambda)_{\widebar x}=\mathcal L^\Lambda_{\widebar x}$
for a.a. $\widebar x\in\widebar X$.
\end{Lemma}

\begin{proof}
We define a linear operator $P^\Lambda_{\widebar x}$ in $L^2(\mathcal K_{\widebar x},\mu_{\widebar x})$ by setting for each
$n$:
\beq\label{eq:total}
P_{\widebar x}^\Lambda F_{n,\widebar x}:=
\begin{cases}
F_{n,\widebar x} &\text{if $F_{n,\widebar x}\circ T=\lambda F_{n,\widebar x}$ with $\lambda\in\Lambda$ and }\\
0 &\text{otherwise}.
\end{cases}
\eeq
Since $(F_{n,\widebar x})_{n=1}^\infty$
is a total orthogonal family in $L^2(\mathcal K_{\widebar x},\mu_{\widebar x})$,
the equation (\ref{eq:total}) determines  $P^\Lambda_{\widebar x}$ correctly and uniquely.
Moreover, $P^\Lambda_{\widebar x}$ is an orthogonal projection
with
$$
P_{\widebar x}^\Lambda L^2(\mathcal K_{\widebar x},\mu_{\widebar x})=
L^2(\mathcal L_{\widebar x}^\Lambda,\mu_{\widebar x})
\qquad
\text{at a.e. $\widebar x$.}
$$
We now define a linear operator $P^\Lambda$ in $L^2(\mathcal K_{\rm rel},\mu)$
by setting for each $n\in\Bbb N$,
$$
P^\Lambda F_n:=1_{\{\widebar x\in\widebar X\, \colon \, a_n(\widebar x)\in\Lambda\}}\cdot F_n
$$
It is a routine to verify that $P^\Lambda$ is a well defined orthogonal projection
in $L^2(X,\mu)$
and $P^\Lambda L^2(\mathcal K_{\rm rel},\mu)=L^2(\mathcal K_{\rm rel}^\Lambda,\mu)$.
Moreover,
 $P^\Lambda|L^2(\mathcal K_{\widebar x},\mu_{\widebar x})=P_{\widebar x}^\Lambda$  at a.e. $\widebar x$.
Thus,
$$
P^\Lambda =\int_{\widebar X}^\oplus P^\Lambda_{\widebar x}\,d\widebar \mu(\widebar x).
$$
Applying this to (\ref{slon}), we obtain
$$
L^2(\mathcal K_{\rm rel}^\Lambda,\mu)=\int_{\widebar X}^{\oplus}P^\Lambda_{\widebar x}L^2(\mathcal K_{\widebar x},\mu_{\widebar x})\,d\widebar \mu(\widebar x)
=\int_{\widebar X}^{\oplus}L^2(\mathcal L_{\widebar x}^\Lambda,\mu_{\widebar x})\,d\widebar \mu(\widebar x).
$$
Since
$L^2(\mathcal K_{\rm rel}^\Lambda,\mu)=\int_{\widebar X}^{\oplus}
L^2((\mathcal K^\Lambda)_{\widebar x},\mu_{\widebar x})\,d\widebar \mu(\widebar x),
$
it follows that the closed subspaces $L^2((\mathcal K^\Lambda)_{\widebar x},\mu_{\widebar x})$ and $L^2(\mathcal L^\Lambda_{\widebar x},\mu_{\widebar x})$ of  $L^2(\mu_{\widebar x})$ are equal.
Hence, $\mathcal K^\Lambda_{\widebar x}=\mathcal L^\Lambda_{\widebar x}$,
as desired.
\end{proof}

Denote by ${\bf R}^\Lambda$ the subclass of those automorphisms with
relatively discrete spectrum (over the invariant sigma-algebra) whose eigenvectors have eigenvalue functions with values in $\Lambda$.
By {\bf DISP}$(\Lambda)_{\text{ec}}$ we mean those automorphisms whose ergodic components have pure point spectrum and this spectrum is contained in $\Lambda$.

In order to see concrete automorphisms belonging to ${\rm {\bf{DISP}}}(\Lambda)_{\rm ec}$ when $\Lambda$ is uncountable, first notice that such $\Lambda$ is Borel isomorphic to the circle. Hence, there are continuous measures supported on $\Lambda$, say $\sigma$ is such a measure. Consider then the automorphism
$A_2$ acting on $(\T\times\T,\sigma\ot{\rm Leb}_{\T})$ as $A_2(x,y):=(x,xy)$.
Then $(A_2,\sigma\ot{\rm Leb}_{\T})\in {\rm {\bf{DISP}}}(\Lambda)_{\rm ec}$.\footnote{We also note  that $A_2$ has continuous spectrum in the orthocomplement  to space of  functions depending on the first coordinate.
The    maximal spectral type in the orthocomplement  equals to $\sum_{m\in\Z}\frac1{2^{|m|}}\sigma^{(m)}$, where $\sigma^{(m)}$ is the image of $\sigma$ via the map $z\mapsto z^m$.
This measure is also supported on $\Lambda$.}

The following corollary is proved in the same way as Proposition~\ref{re:mrak}
but with reference to Lemma~\ref{kapitan} instead of Lemma~\ref{pr:giraf}.
We leave details to the reader.

\begin{Cor}\label{rabit}
${\bf R}^\Lambda={\text{{\bf{DISP}}}(\Lambda)}_{\text{{\rm ec}}}$.
\end{Cor}

From now on we will assume that
 $\Lambda$ is a {\bf $\sigma$-compact} subgroup of $\Bbb T$.
The class of such groups includes all
\begin{enumerate}
\item[---]  countable subgroups of $\Bbb T$,
\item[---]  $L^\infty$-spectra of  ergodic conservative nonsingular automorphisms,
\item[---] compactly generated subgroups of $\Bbb T$ (in particular, the group generated by $\{z_n\colon n\geq1\}$ for any $z_n\to 1$),
\item[---] $H_\alpha$-subgroups, i.e.\ subgroups of the form $\{z\in\Bbb T\colon\sum_{j=1}^{+\infty}a_j|z^{n_j}|^\alpha<\infty\}$ for any sequences $(n_j)_{n=1}^\infty$ of natural numbers and   $(a_j)_{j=1}^\infty$ of non-negative  reals and $\alpha>0$.
\end{enumerate}
In the latter three cases $\Lambda$ can be a proper uncountable subgroup of $\Bbb T$.
For instance, this happens if $\Lambda$ is generated by an infinite Kronecker subset of $\Bbb T$~\footnote{To obtain $\Lambda$ as in Remark~\ref{r:LambdaS}, it is enough to consider $p$-Kronecker sets (for a fixed prime $p$), these are (perfect) closed subsets $K\subset\T$ such that each modulus~1 function $f\in C(K)$ can be uniformly approximated by characters of the form $z\mapsto z^{p^k}$ (it follows that $p^{k_n}x\to 0$ for $x\in K$ uniformly by approximating 1 by characters given by powers of $p$). Constructions of perfect $p$-Kronecker sets are provided in Complements. },
      or if $\Lambda$ is an $H_2$-group with $\sum_{j=1}^\infty a_j=\infty$
and $\sum_{j=1}^\infty a_j\big(\frac{n_j}{n_{j+1}}\big)^2<\infty$
(see \cite{Ho-Me-Pa} and references therein).


Our purpose  is  construct a continuous universal model for ${\rm{\bf{DISP}}}(\Lambda)_{\rm ec}$.
To this end, fix  a sequence of compact subsets $\Lambda_n$ in $\Lambda$ such that
\beq\label{incr}
\{1\}\subset\Lambda_1\subset \Lambda_2\subset
\cdots\qquad\text{and}\qquad
 \bigcup_n\Lambda_n=\Lambda.
\eeq
Let $\Lambda_\infty:=\Lambda_1\times\Lambda_2\times\cdots$.
Endow this set with the infinite product topology.
Define a continuous transformation $A_\Lambda$ of the compact space $\Lambda_\infty\times\Bbb T^\Bbb N\times\Bbb T^\Bbb N$ by setting
$$
A_\Lambda(y,v,z):=(y,v, yz).
$$
If $\nu$ be an $A_\Lambda$-invariant measure, we denote by $\nu^*$ the projection of $\nu$ onto the first two coordinates.
Then we disintegrate $\nu$ with respect to $\nu^*$:
$$
\nu:=\int_{\Lambda_\infty\times\Bbb T^\Bbb N}\delta_{y,v}\otimes\nu_{y,v}\,d\nu^*(y,v).
$$
Then the probability measure $\nu_{y,v}$ on $\Bbb T^\Bbb N$ is invariant under $R_y$ for $\nu^*$-a.e. $(y,v)$.
Take $y=(y_1,y_2,\dots)\in\Lambda_\infty$.
Denote by $\Lambda(y)$ the group generated by $y_1,y_2,\dots$.
Of course, $\Lambda(y)\subset \Lambda$ and each ergodic component of the rotation $(\Bbb T^\Bbb N,\nu_{y,v},R_y)$ is isomorphic to the transformation with pure discrete spectrum
$\Lambda(y)$.
It follows that $(\Lambda_\infty\times\Bbb T^\Bbb N\times\Bbb T^\Bbb N,\nu, A_\Lambda)\in {\rm{\bf{DISP}}}(\Lambda)_{\rm ec}$.
The converse is also true.

\begin{Th}\label{displambda} Let $(X,\mathcal B,\mu,T)\in {\rm{\bf{DISP}}}(\Lambda)_{\rm ec}$.
Then there is an $A_\Lambda$-invariant probability measure $\eta$ on $\Lambda_\infty\times\Bbb T^\Bbb N\times\Bbb T^\Bbb N$
such that $(X,\mathcal B,\mu)$ is measure theoretically isomorphic to
$(\Lambda_\infty\times\Bbb T^\Bbb N\times\Bbb T^\Bbb N,\eta, A_\Lambda)$
and $\mathcal I(A_\Lambda)=\mathcal B_{\Lambda_\infty\times\Bbb T^\Bbb N}\otimes\mathcal N_{\Bbb T^\Bbb N}.$
\end{Th}
\begin{proof}
By Corollary~\ref{rabit},
 $(X,\mathcal B,\mu,T)\in {\rm{\bf{DISP}}}(\Lambda)_{\rm ec}$ if and only if there exist  functions
$f_n \colon X\to\Bbb T$ and $a_n \colon X\to\Lambda$ such that
$a_n$ is $\mathcal I(T)$-measurable, $f_n\circ T=a_nf_n$ for each $n\in\Bbb N$
and $(f_n)_{n=1}^\infty$ generate the entire Borel $\sigma$-algebra on $X$.
Without loss of generality (modifying, if necessary, $a_n$ and $f_n$ on subsets of ``small'' measure) we may assume that
for each $n$, there is $m_n\in\Bbb N$ such that
$a_n$ takes values in $\Lambda_{m_n}$.
We now define $\Phi \colon X\to \Lambda_\infty\times\Bbb T^\Bbb N$
by setting
$$
\Phi(x)_k:=
\begin{cases}
(a_n(x), f_n(x))\in \Lambda_{m_n}\times\Bbb T &\text{if $k=m_n$ for some $n$},\\
(1,1)\in\Lambda_k\times\Bbb T &\text{otherwise}
\end{cases}
$$
and $\Phi(x):=(\Phi(x)_k)_{k=1}^\infty$.
Then $\Phi$ is an isomorphism of $(X,\mathcal B,T)$ onto
$(\Lambda_\infty\times\Bbb T^\Bbb N, \mu\circ\Phi^{-1}, A'_\Lambda)$,
where the transformation $A_\Lambda'$ is given by
$A_\Lambda'(y,z):=(y,yz)$.
Then, as in the proof of Theorem~\ref{discr}, we  construct
a mapping
$$
\Psi \colon  \Lambda_\infty\times\Bbb T^\Bbb N\to\Lambda_\infty\times\Bbb T^\Bbb N\times\Bbb T^\Bbb N
$$
 that intertwines $A'_\Lambda$ with $A_\Lambda$.
\end{proof}

\begin{Cor}\label{c:Fsigma}
If $\Lambda$ is  $\sigma$-finite subgroup of $\Bbb T$ then
$(\Lambda_\infty\times\Bbb T^\Bbb N\times\Bbb T^\Bbb N, A_\Lambda)$ is a universal model for ${\rm {\bf{DISP}}}(\Lambda)_{\rm ec}$.
\end{Cor}

\begin{Remark} By changing (closed, increasing) $\Lambda_i$ satisfying~\eqref{incr}, we obtain different universal models of $A_{\Lambda}$.
In particular, this is  applied to $\Lambda=\T$  and  $A$.
\end{Remark}

\subsection{Automorphisms with pure discrete  spectrum}

Let $T$ be
an automorphism of a standard probability space $(X,\mathcal B,\mu)$.
As above, we denote by $\Lambda(T)\subset\Bbb T$ the set of eigenvalues of $T$.
Then $\Lambda(T)$ is a countable subset of $\Bbb T$.
Let $(\widebar X,\widebar \mu)$ stand for the space of ergodic components of $T$.
Denote by $\Lambda_{\widebar x}(T)$ the set of eigenvalues of the $\widebar x$-ergodic component of $T$.

\begin{Remark}
We note that $\Lambda(T)$ is not necessarily a group.
 However,   if $\lambda\in\Lambda(T)$ then $\lambda^n\in\Lambda(T)$ for each $n\in\Bbb Z$.
Thus, $\Lambda(T)$ is a union of countably many cyclic groups.
\end{Remark}

\begin{Prop}\label{eigen}
Let
$T\in {\bf DISP}$.
Let  $f\in L^2(X,\mu)$  be an  $\mathcal I(T)$-relative eigenfunction, $|f(x)|\in\{0,1\}$ at a.e. $x\in X$,
and $f\circ T=af$
for a
$\mathcal I(T)$-measurable function $a \colon X\to \Bbb T$.
Then $a(x)\in\Lambda(T)$ at a.e. $x\in \text{\rm{supp\,}} f$.
Hence, $\Lambda_{\widebar x}(T)\subset\Lambda(T)$
for a.e. $\widebar x$.
\end{Prop}
\begin{proof}
Denote by $U_T$ the Koopman operator generated by $T$.
Denote by $\sigma_f$ the $U_T$-spectral measure of the vector $f$.
Let $B$ stand for the support of $f$.
For each $n\in\Bbb Z$,
$$
\int_\Bbb Tz^nd\sigma_f(z)=\langle U_T^nf,f\rangle=\langle a^nf,f\rangle=\int_{B}a(x)^nd\mu(x)=\int_\Bbb Tz^nd((1_{B}\mu)\circ a^{-1})(z).
$$
Hence, $\sigma_f=(1_{B}\mu)\circ a^{-1}$.
Since $T$ has pure discrete spectrum, $\sigma_f$ is supported on $\Lambda(T)$.
 The assertion of the proposition follows.
\end{proof}

This implies the following corollary.

\begin{Prop}\label{c:m2} Assume that $T\in
{\bf DISP}_{{\rm ec}}$.
Then $T\in
{\bf DISP}$ if and only if the $\mathcal I(T)$-relative eigenfunctions have discrete spectral measures.
 \end{Prop}
\begin{proof} The {\it only if\,} part of the assertion  follows from Proposition~\ref{eigen}.

To show the {\it  if\,} part,  assume now that a relative eigenfunction $f$ has a discrete spectral measure supported on $\Lambda(T)$.
This means that this spectral measure consists of countably many atoms.
But then it means the relative eigenvalue $a$ of $f$ takes countably many values.
Let $B_{\la}\subset \ov{X}$ be the set of $\ov{x}$ for which $a(\ov{x})=\la$.
Set $f_{\la}(x):=f(x)$ if  $\ov{x}\in B_{\la}$ and $f_{\la}(x):=0$ otherwise.
Then $F_\la$ is an eigenfunction of $T$.
Furthermore, $F=\sum_{\la}F_{\la}$.
\end{proof}

Of course,  ${\bf DISP}\subset {\bf DISP}_{\text{ec}}$.
Therefore, if $T$ has pure point spectrum,
we can apply Theorem~\ref{discr} about
continuous  algebraic models for transformations from
${\bf DISP}_{\text{ec}}$.
According to this theorem,
there is  a probability $\eta$ on $\Bbb T^\Bbb N\times \Bbb T^\Bbb N$
such that
 $(X,\mu,T)$ is isomorphic to
$( \Bbb T^\Bbb N\times \Bbb T^\Bbb N\times \Bbb T^\Bbb N,\widetilde \eta,A)$.
It follows from the construction of the algebraic model (see the proof of Theorem~\ref{discr}) and Proposition~\ref{eigen} that
$\eta(\Lambda(T)^\Bbb N\times  \Bbb T^\Bbb N)=1$.
Thus, we obtain the following  refinement of Theorem~\ref{discr} for the automorphisms with pure discrete spectrum.

\begin{Prop}\label{katrin}
Let $(X,\mu,T)\in {\bf DISP}$.
Then there is
 a probability  $\eta$ on
$\Lambda(T)^\Bbb N\times \Bbb T^\Bbb N$ such that
 $(X,\mu,T)$ is isomorphic to
$( \Lambda(T)^\Bbb N\times \Bbb T^\Bbb N\times \Bbb T^\Bbb N,\widetilde \eta,A)$,
where
$$
\widetilde\eta=\int_{\Lambda(T)^\Bbb N\times \Bbb T^\Bbb N}\delta_{y,v}\otimes\lambda_{K(y)}\,d\eta(y,v)
$$
\end{Prop}

\begin{Cor}\label{cor:smile} Let $T\in {\bf DISP}_{\text{ec}}$.
Then  $T\in {\bf DISP}$ if and only if $T$ is URE.
\end{Cor}
\begin{proof} Let  $T\in {\bf DISP}$.
 By Proposition~\ref{katrin}, a.e. ergodic component of $T$ has discrete spectrum and $\Lambda_{\widebar x}(T)\subset\Lambda(T)$.
 Let $R$ be an ergodic transformation with discrete spectrum such that  $\Lambda(R)\supset \Lambda(T)$.
 Then a.e. ergodic component of $T$ is a factor of $R$.
 Hence, $T$ is URE by Proposition~\ref{p:inac}.

 Conversely, if $T\in {\bf DISP}_{\text{ec}}$ and $T$ is URE then $T\in {\bf DISP}$ by
 Proposition~\ref{p:inac}(C).
\end{proof}

The Corollary~\ref{cor:smile} implies, in turn, the following result (that has been proved in \cite{Ed} using different methods).

\begin{Cor}\label{pds}
A  transformation $T$ has pure discrete spectrum if and only if there is a probability measure $\nu$ on $\Bbb T$ such that $T$ is isomorphic to a factor of the direct product $I\times R$, where $I$ is the identity on $(\Bbb T,\nu)$ and $R$ is an ergodic transformation with a pure discrete spectrum and $\Lambda(R)$ is the group \rm{gr}($\Lambda(T))$ generated by $\Lambda(T)$.
\end{Cor}

 Let $T_n$ be an ergodic transformation with pure discrete spectrum.
 Let $I_n$ be an identity transformation of a Lebesgue space.
 Then $T:=\bigsqcup_n (I_n\times T_n)$ is  a transformation with pure discrete spectrum and
 $\Lambda(T)=\bigcup_{n}\Lambda(T_n)$.
 It may seem that each transformation with pure discrete spectrum has such a structure.
 That is not true.
 See, e.g, Example~\ref{e:e2}.

\begin{Remark}\label{re:boost}
Note also that for an automorphism $S$, the factor $\ca_{{\rm {\bf{DISP}}}_{\rm ec}}(S)$ is exactly the
$\mathcal I(S)$-relative Kronecker factor $\ck_{\rm rel}(S)$ of $S$.
Hence, $\ca_{{\rm {\bf{DISP}}}_{\rm ec}}(S)$ is
URE if and only if  $\ck_{\rm rel}(S)$ is equal to the Kronecker factor $\ck(S)$.
\end{Remark}

\section{On the characteristic classes generated by $A_{2,\alpha}$
}
Given an $\alpha\in[0,1)$ {\bf irrational}, consider the homeomorphism $A_{2,\alpha}$ of $\T^2$ given by
$$
A_{2,\alpha}(x,y)=(x+\alpha,x+y).$$
Then $(\T^2,A_{2,\alpha})$ is a uniquely ergodic topological system (with Leb$_{\T^2}$ being the unique invariant measure), which not only is orthogonal to $\mob$ but it also satisfies the strong $\mob$-MOMO property (see \cite{Li-Sa}, \cite{Ku-Le}, \cite{Ab-Ku-Le-Ru}). However,
\beq\label{topfacalfa}
\mbox{$A_2$ is a {\bf topological} factor of $A_{2,\alpha}\times A_{2,\alpha}$}\eeq
(indeed,  the map $\Psi \colon \T^2\times\T^2\to\T^2$,
$((x_1,y_1),(x_2,y_2))\mapsto (x_2-x_1,y_2-y_1)$ settles a relevant factor map), so the strong $\mob$-MOMO  for the Cartesian square remains an open problem.
Furthermore, though the set of invariant measures for the Cartesian square seems to be restricted as all such measures are self-joinings of $A_{2,\alpha}$, among them, we have measures
$$
\rho_\beta:=\int_\Bbb T\int_{\Bbb T}\delta_x\otimes\delta_y\otimes\delta_x\otimes\delta_{y+\beta}\,dxdy.
$$
Then, the dynamical system $(\Bbb T^4,\rho_\beta, A_{2,\alpha}\times A_{2,\alpha})$
has $\beta$ as an eigenvalue.
It follows that all rotations are in $\cf(A_{2,\alpha})$.
Hence (by taking products with all identities,  using the Halmos-von Neumann theorem and applying Corollary~\ref{pds}), we see that
$\text{{\bf DISP}}\subset \cf(A_{2,\alpha})$.
This reasoning about necessary elements in a characteristic class has the following extension.

\begin{Lemma} \label{l:gaj1} Let $(X,T)$ be uniquely ergodic: $M(X,T)=\{\nu\}$. Assume that $(Y,S)$ is another topological system for which $\Psi\colon X^\infty\to Y$ establishes a topological factor map between $(X^\infty,T^{\times\infty})$ and $(Y,S)$ (in particular, $\Psi$ is onto). Then, for each $\kappa\in M(Y,S)$ the measure-theoretic system $(Y,\kappa,S)\in \cf(T)$.\end{Lemma}
\begin{proof} Let $\psi\colon Y\to X^\infty$ be a Borel cross-section of
$\Psi$.
Then $\gamma:=\int_Y\delta_{\psi(y)}d\kappa(y)$
is a probability measure on $X^\infty$ such that $\Psi_\ast(\gamma)=\kappa$.
Let $\rho$ be a limit point of $(\frac1N\sum_{n=0}^{N-1}(T^{\times\infty})^n_\ast(\gamma
)$.
Now, clearly $\rho$ is an infinite self-joining of $(X,\nu,T)$ and $(S,\kappa)$ is a factor of $(T^{\times\infty},\rho)$, so our claim follows.
\end{proof}

Note that we obtain the same result if we consider $\Psi\times {\rm Id}_W$ establishing a factor map between $T^{\times\infty}\times {\rm Id}_W$ and $S\times {\rm Id}_W$ (a pull-back via $\Psi\times {\rm Id}_W$ of 
a measure $\kappa\in M(Y\times W, S\times {\rm Id}_W)$ will be an infinite joining of $T$ joined with ${\rm Id}_W$, and ${\rm Id}_W$ belongs to any non-trivial characteristic class).

\begin{Prop} \label{p:gaj1} For each irrational $\alpha\in\T$, we have
\beq\label{aff12} \cf(A_{2,\alpha})\supset {\rm {\bf{DISP}}}_{\rm ec}.\eeq
\end{Prop}
\begin{proof} In view of~\eqref{topfacalfa}, we have (as topological factors)  $A_{2,\alpha}^\infty\times {\rm Id}\to A_2^{\infty}\times {\rm Id}$.
By Theorem~\ref{theorem:b},  $A_2^\infty\times {\rm Id}$ is a universal model for the class ${\rm{\bf{DISP}}}_{\rm ec}$. It is now enough to apply Lemma~\ref{l:gaj1}.\end{proof}

\begin{Lemma}\label{l:gaj2} If\, $1$, $\beta$ and $\alpha$ are rationally independent  then $A_{2,\beta}\notin\cf(A_{2,\alpha})$.
\end{Lemma}
\begin{proof}
Suppose that $A_{2,\beta}\in \cf(A_{2,\alpha})$. Then, $A_{2,\beta}$ is a factor of a certain self-joining of $A_{2,\alpha}$ and since $A_{2,\beta}$ is ergodic, it must be a factor of an ergodic self-joining of $A_{2,\alpha}$. By a general theory of cocycles over ergodic rotations, we obtain that there are $\beta_i\in\T$ ($i\geq1$), a closed subgroup $G\subset\T^\infty$ and a Borel mapping $\psi\colon\T\to G$ such that:
\begin{enumerate}
\item[(i)] $T_\psi(x,g):=(Tx,\psi(x)+g)=(x+\alpha,\psi(x)+g)$ (with $g=(y_1,y_2,\ldots)\in G$) is ergodic;
\item[(ii)] $A_{2,\beta}$ is a factor of $T_\psi$;
\item[(iii)] $\psi$ is cohomologous (as a cocycle of $T$) with
$$
\theta\colon \Bbb T\ni x\mapsto (x+\beta_1,x+\beta_2,\ldots)\in\T^\infty,$$
that is, $\psi-\theta=h\circ T-h$ for a measurable $h \colon \T\to\T^\infty$.
\end{enumerate}
Since $\beta$ is an eigenvalue of $A_{2,\beta}$, it is also an eigenvalue of $T_\psi$. Hence, there exist $\eta\in\widehat{G}$ and a measurable $\xi \colon \T\to\mathbb{S}^1$ such that
\beq\label{gaj7}
\eta(\psi(x))=e^{-2\pi i\beta}\xi(Tx)/\xi(x).\eeq
The corresponding eigenfunction (which is unique up to a multiplicative constant) $f=f(x,g)$ is hence of the form
$$
f(x,g)=\xi(x)\eta(g).$$
Moreover, $A_{2,\beta}$ as a factor of $T_\psi$ is represented by a $T_\psi$-invariant sigma-algebra $\ca\subset\cb(X\times G)$. Hence, there exists an $L^2$-function $F=F(x,g)$ (even of constant modulus~1 and  $\ca$-measurable) such that
$$F\circ T_\psi=f\cdot F$$
(this is so, because this holds for $A_{2,\beta}$: consider $e^{2\pi ix}$ which is an eigenfunction corresponding to $\beta$ and then $e^{2\pi i y}$ which satisfies
$$e^{2\pi i y}\circ A_{2,\beta}=e^{2\pi i (x+y)}=e^{2\pi ix}e^{2\pi iy}.)$$
Now, $F(x,g)=\sum_{\chi\in\widehat{G}}a_\chi(x)\chi(g)$, so
$$
F\circ T_\psi(x,g)=\sum_{\chi\in \widehat{G}}a_\chi(Tx)\chi(\psi(x))\chi(g),$$
and
$$
f(x,g)\sum_{\chi\in \widehat{G}}a_\chi(x)\chi(g)=\sum_{\chi\in \widehat{G}}a_\chi(x)\xi(x)\eta(g)\chi(g).$$
Therefore, for each $\chi\in\widehat{G}$:
$$
a_{\eta\cdot\chi}(Tx)(\eta\cdot\chi)(\psi(x))=
a_\chi(x)\xi(x).$$
Thus $|a_{\eta\cdot\chi}\circ T|=|a_\chi|$, so
$\|a_{\eta\cdot\chi}\|_{L^2}=\|a_\chi\|_{L^2}$, and we obtain that for each $k\geq1$, we have $\|a_{\eta^k\cdot\chi}\|_{L^2}=\|a_\chi\|_{L^2}$ and since $\eta^k\neq1$ for each $k\geq1$ (when, assuming $\eta^k=1$, we take the $k$-th power in \eqref{gaj7}, we easily deduce that $k\beta$ is an eigenvalue of $T$, contradicting the independence of $\alpha$ and $\beta$), we obtain that $a_\chi=0$, which is a contradiction as $F\neq0$.
\end{proof}

\begin{Remark} \label{r:gaj3} Note that in the above lemma we proved 
the following, more general, fact: if for an ergodic automorphism $R$,  there are functions $f,F$ (of modulus~1) and $\beta$ irrational such that $f\circ R=e^{2\pi i\beta}f$ and  $F\circ R=f\cdot F$ then $R$ does not belong to $\cf(A_{2,\alpha})$ whenever $\alpha$, $\beta$  and $1$ are rationally independent.
\end{Remark}

We now discuss some facts around Conjecture~\ref{quest4}
from Section~\ref{s:open}.
We say that an automorphism $R$ acting on a space $\zdk$ has {\em quasi-discrete spectrum of order~2} if there are non-zero functions $F,f$ and $\beta\in\R$ such that $f\circ R=e^{2\pi i \beta}f$, $F\circ R
=F\cdot f$, and the set of all such $F$ and $f$ is total in $L^2\zdk$.
Note that in this definition we can additionally assume that $F$ and $f$ are bounded.
We will call $F$ a {\it quasi-eigenfunction of $T$}.

\begin{Lemma}\label{l:gaj20}
If $R$ has quasi-discrete spectrum of order 2, then for any $\lambda\in J_\infty(R)$, we have $(R^{\times\infty},\la)$ has also quasi-discrete spectrum of order~2.\end{Lemma}
\begin{proof} Consider bounded $f_i$ and $F_i$ and tensors of the form  $\bigotimes_{i\in I} F_i\bigotimes_{j\in J}  f_j$, for all finite $I$ and $J$.
\end{proof}

If we additionally assume that $R$ is ergodic, then we can additionally assume that $F$ and $f$ have constant modulus~1.

\begin{Lemma}\label{l:gaj21} Assume that  $R$ is ergodic and has quasi-discrete spectrum of order~2. If the spectrum of $R$ is rational, then $R$ has purely discrete spectrum.
\end{Lemma} \begin{proof} We have $f\circ R=e^{2\pi i/m}f$ and then $F\circ R=f\cdot F$. It follows that
$$
F\circ R^m=e^{2\pi i (m(m-1)/2)/m}f^m \cdot F=e^{2\pi i (m-1)/2}\cdot F$$
since we can assume that $f$ takes values in the group of roots of degree~$m$ (in fact, by ergodicity,  there is a unique Rokhlin tower $A,RA,\ldots, R^{m-1}A$ fulfilling the whole space, and we can assume that $f=1$ on $A$, and then $f=e^{2\pi ik/m}$ on $R^kA$;  in general we use the fact that the $m$-th power of $f$ is a constant). It follows that $F$ is an eigenfunction of $R^m$, so its spectral measure (for $R^m$) is Dirac, concentrated at $e^{2\pi i(m-1)/m}$. It follows that the spectral measure of $F$ for $R$ must be discrete (as its image via the map $z\mapsto z^m$ is Dirac). Hence, all functions $f,F$ have discrete spectral measures and therefore the maximal spectral type of $R$ is also discrete.
\end{proof}

\begin{Lemma}\label{l:gaj22} Assume that $R$ is ergodic and has quasi-discrete spectrum of order~2. Let $\ca\subset\cd$ be a factor of $R$ with partly continuous spectrum.
Then there exist $G,g\in L^2(\ca)$ (of constant modulus~1) such that $G\circ R=g\cdot G$,
$g\circ R=e^{2\pi i \beta}g$ and $\beta$ is irrational.
\end{Lemma}
\begin{proof}

Since $R$ has quasi-discrete spectrum of order 2, $\mathcal A$ also
has quasi-discrete spectrum of order 2 (by ergodicity, when $f$ is an eigenfunction and $\EE(f|\ca)\neq0$ then $f$ must be equal to $\E(f|\ca)$ up to a constant, hence, $f$ is $\ca$-measurable).
 Consider a quasi-eigenfunction $F$ of $R\restriction\mathcal A$.
 Then $F\circ R=f\cdot F$ and  $f\circ R=e^{2\pi i \alpha} f$ for a function $f\in L^2(\mathcal A)$ and $\alpha\in [0,1)$.
 If $\alpha$ is rational for all quasi-eigenfunctions of $R\restriction\mathcal A$ then
 $R\restriction\mathcal A$ has pure point spectrum by Lemma~\ref{l:gaj21}.
 That contradicts to the condition of the lemma.
 Hence, there are $F$, $f$ and $\beta$ as above with $\beta\not\in\Bbb Q$.
\end{proof}

\begin{Prop}\label{p:gaj20} Assume that
 $\alpha$, $\beta$ and $1$ are rationally independent.
   Then
$\cf(A_{2,\alpha})\cap \cf(A_{2,\beta})\cap {\rm Erg}={\rm {\bf{DISP}}}_{\rm ec}\cap{\rm Erg}=
{\rm {\bf DISP}}\cap {\rm Erg}$.
\end{Prop}
\begin{proof}
Let $T$ be an ergodic automorphism with partly continuous spectrum and belonging to $\cf(A_{2,\alpha})$.
 Since $T$ is ergodic, $T$ is a factor of an infinite, ergodic self-joining of $A_{2,\alpha}$:
$$
A_{2,\alpha}\vee A_{2,\alpha}\vee\ldots\to T.$$
This (ergodic) infinite self-joining has quasi-discrete spectrum by Lemma~\ref{l:gaj20}.
Thanks to  Lemma~\ref{l:gaj22}, there exist
 non-trivial $f$, $F$ such that
$$
f\circ {T}=e^{2\pi i\gamma}f\text{ (with $\gamma$ irrational) and }\;F\circ {T}=f\cdot F.$$
Then, by Remark~\ref{r:gaj3},
$\gamma\in\Bbb Q\alpha+\Bbb Q$.
In a similar way, if $T\in \cf(A_{2,\beta})$ then
$\gamma\in\Bbb Q\beta+\Bbb Q$.
We get a contradiction because 1, $\alpha$ and $\beta$ are rationally independent.
Hence, $T\in{\rm {\bf DISP}}$.
\end{proof}

\section{Universal model for the characteristic class generated by an MSJ automorphism}

Assume that $T\in {\rm Aut}\xbm$ is an MSJ automorphism.
Our standing assumption is  that $\mu$ is nonatomic.
This implies that $T$ is weakly mixing.
We are interested in properties of the  infinite Cartesian product $T^{\times\infty}$ of $T$.

\subsection{Characteristic class for simple automorphisms}
Recall that, following \cite{Ka-Ku-Le-Ru}, given an automorphism $T$, the smallest characteristic class $\cf(T)$ containing $T$ consists of all factors of all infinite self-joinings of $T$.

Assume that $T\in{\rm Aut}\xbm$ is a simple automorphism\footnote{The definition of simplicity is similar to the MSJ property with the difference that instead of off-diagonal measures, we consider images of $\mu$ under the maps
$x\mapsto (x,S_1x,S_2x,\ldots)\in X^\infty$, where all $S_j$ are in the centralizer $C(T)$ of $T$.}(\cite{Ju-Ru}, \cite{Ve}). We recall that all ergodic discrete spectrum are simple. However, if a simple automorphism has no purely discrete spectrum, it has to be weakly mixing.

\begin{Prop}\label{p:simfac}Assume that $T$ is simple and weakly mixing\footnote{If $T$ is ergodic and has discrete spectrum then, by Theorem~\ref{pds}, $\cf(T)$ consists of all factors of $({\rm Id}_{[0,1]}\times T, {\rm Leb}_{[0,1]}\times \mu)$.}. Then $\cf(T)$ is equal to the family of all factors of the automorphism ${\rm Id}_{[0,1]}\times T^{\times\infty}$ considered on
 $([0,1]\times X^\infty,{\rm Leb}_{[0,1]}\otimes\mu^{\otimes \infty})$.
 Moreover, $\cf(T)=\cf(T)_{\rm ec}$.
\end{Prop}
\begin{proof} Note that ${\rm Id}_{[0,1]}\in\cf(T)$ (as each non-trivial characteristic class contains all identities \cite{Ka-Ku-Le-Ru}), and $T^{\times \infty}$ is also in $\cf(T)$. Hence, ${\rm Id}_{[0,1]}\times T^{\times \infty}\in \cf(T)$ and therefore all factors of ${\rm Id}_{[0,1]}\times T^{\times \infty}$ belong to $\cf(T)$.

Let now $R\in\cf(T)$, where $R\in{\rm Aut}\zdk$. Since $\cf(T)$ consists of all factors of infinite self-joinings of $T$, up to isomorphism, we have
$$
\cd\subset (\cb^{\otimes\infty},\la), \text{ where }\la\in J_\infty(T).$$
Let $\la=\int_{J^e_\infty(T)}\rho\,dQ(\rho)$ be the ergodic decomposition of $\la$. Then
\beq\label{jm18}
\la|_{\cd}=\int_{J^e_\infty(T)}\rho|_{\cd}\,dQ(\rho).\eeq
Here, $\rho|_{\cd}$ represents a factor of the ergodic automorphism $(T^{\times \infty},\rho)$. By the simplicity of $T$, it follows that $(T^{\times \infty},\rho)$ is a factor of $(T^{\times \infty},\mu^{\ot\infty})$. By
Corollary~\ref{p:random-factor?} and \eqref{jm18}, it follows that $R$ (identified with $\la|_{\cd}$) is a factor of $Id_{[0,1]}\times T^{\times \infty},{\rm Leb}_{[0,1]}\otimes \mu^{\otimes\infty})$.

Finally, note that
$$
R\in \cf(T)_{\rm ec} \Leftrightarrow \kappa=\int_{\Gamma}\rho_\gamma\,dQ(\gamma),\; \text{ with }(R,\rho_\gamma)\in {\rm Erg}\cap \cf(T).$$
So $(R,\rho_\gamma)$ is an ergodic member of $\cf(T)$, i.e. it is an ergodic factor of $(T^{\times \infty},\mu^{\otimes\infty})$ by simplicity. Therefore, the same proof shows that $R$ is a factor of ${\rm Id}_{[0,1]}\times T^{\times\infty}\in\cf(T)$.\end{proof}

\subsection{Compact subgroups of the centralizer $C(T^{\times\infty})$ of $T^{\times\infty}$}

Let $D$
be a countable (or finite) set.
Denote by $\Sigma(D)$ the group of all permutations of $D$.
Endow $\Sigma(D)$ with the natural Polish topology of pointwise convergence.
Then $\Sigma(D)$ is a Polish 0-dimensional group acting continuously on $D$.
Fix a discrete countable torsion free group $G$.
Then $\Sigma(D)$ acts naturally on $G^D$ by permutation of coordinates.
Denote by $G^D\rtimes\Sigma(D)$ the corresponding semidirect product.
We remind that the multiplication law in $G^D\rtimes\Sigma(D)$
is given by the formula:
$$
(h,\alpha)(l,\beta):=(h\cdot (l\circ\alpha^{-1}), \alpha\beta),\quad\text{for all $h,l
\in G^D$ and  $\alpha,\beta\in\Sigma(D)$.}
$$
Let $p\colon G^D\rtimes\Sigma(D)\to\Sigma(D)$ stand for the canonical projection.

Fix a compact subgroup $K$ of $G^D\rtimes\Sigma(D)$.
Then $p(K)$ is a compact subgroup of $\Sigma(D)$.
Consider the finest   partition  $D=\bigsqcup_iD_i$ of $D$ into  $p(K)$-invariant subsets $D_i$.
In other words, each $D_i$ is a $p(K)$-orbit.
Since $p(K)$ is compact, $D_i$ is finite for each $i$.
Then
$p(K)\subset\bigotimes_i\Sigma(D_i)$.
Denote by $p_i$
the projection of $p(K)$ to $\Sigma(D_i)$.
Let $K_i:=p_i(K)$.
Then
\begin{itemize}
\item
$K_i$ is a transitive group of permutations on $D_i$ for each $i$ and
\item
$p(K)\subset\bigotimes_iK_i$.
\end{itemize}
Since $G^D\rtimes\big(\bigotimes_iK_i\big)=\bigotimes_i(G^{D_i}\rtimes K_i)$,
we obtain that
$$
K\subset\bigotimes_i \big(G^{D_i}\rtimes K_i\big).
$$
Suppose that $K$ contains  an element $k=(g_i)_i$ such that for some $i_0$,
there exists $h\in G^{D_i}\setminus \{1\}$ with  $g_{i_0}=(h,1)\in
G^{D_{i_0}}\rtimes K_{i_0}$.\footnote{We use a general notation that $1$ stands for the unit in a relevant group.}
But $G$ is torsion free, so the subgroup generated by $h$ is discrete, infinite and cyclic.
It is contained in $K$, which is compact.
This contradiction implies that $h=1$.
Thus, we obtain that
for each $i$, there exists a function $l_i\colon K_i\to G^{D_i}$
such that
$$
K\subset \bigotimes_i\{(l_i(k),k)\colon k\in K_i\}.
$$
The righthand side of this formula is a group if and only if $l_i$ is a {\it skew homomorphism}, i.e.
$$
l_i(kh)=l_i(k)\, (l_i(h))\circ k^{-1}\quad\text{for all }k,h\in K_i.
$$

\begin{Example}
Given $d\in\Bbb N$, let
 $D:=\{1,\dots,d\}$ and let $\sigma$ stand for the cycle $(1,\dots, d)$.
Denote by $\Sigma$ the (finite, cyclic) group generated by $\sigma$.
It is clear that a skew homomorphism  $l\colon \Sigma\to G^D$ is completely determined by
a single value $l(\sigma)\in G^D$.
This value is a sequence $l(\sigma)[i]\in G$, $i\in D$.
Since
$$
1=l(\text{id})=l(\sigma^d)=l(\sigma)\cdot (l(\sigma)\circ\sigma^{-1})\cdot\ldots\cdot (l(\sigma)\circ\sigma^{-d+1}),
$$
 we obtain that   $l(\sigma)[d]\cdots l(\sigma)[1]=1$.
Conversely,\footnote{Use an obvious observation that $s_1\cdots s_d=1$ in $G$ implies $s_ds_1\cdots s_{d-1}=1$.} each sequence $a:=(a[i])_{i\in D}$ of $G$-elements such that $a[d]\cdots a[1]=1$ determines a unique skew homomorphism  $l\colon \Sigma\to G^D$ such that
$l(\sigma)=a$.

\end{Example}

Thus, we have proved the following proposition.

\begin{Prop}\label{structure of subgroups} For each compact subgroup $K$ of $\Sigma(D)\rtimes G^D$, there exist
a partition $D=\bigsqcup_{i}D_i$ of $D$ into finite subsets $D_i$, transitive subgroups $K_i\subset\Sigma(D_i)$ and skew homomorphisms $l_i\colon K_i\to G^{D_i}$
such that $$
K\subset \bigotimes_{i} \{(l_i(k),k)\colon k\in K_i\}.
$$
If $p$ is the natural projection $K\to\Sigma(G)$ and $p_i$ is the natural projection $\bigotimes_i\Sigma(D_i)\to \Sigma(D_i)$ then
 $K_i=p_i\circ p(K)$ for each $i$.

\end{Prop}

Suppose that $(X,\mathcal B,\mu, T)$ is a measure preserving system.
Consider the infinite product $(X^\Bbb N,\mu^{\otimes \Bbb N}, T^{\times\infty})$.
Then the Polish group $\Bbb Z^\Bbb N\rtimes \Sigma(\Bbb N)$ embeds naturally into the centralizer $C(T^{\times\infty})$ of $T^{\times\infty}$: for each $x=(x_n)_{n\in N}\in X^\Bbb N$, $k=(k_n)_{n\in \Bbb N}\in\Bbb Z^\Bbb N$, $\sigma\in \Sigma(\Bbb N)$ and $m\in\Bbb N$,
\begin{equation}\label{action}
\big((k,\sigma)x\big)_m:= T^{k_m}x_{\sigma(m)}.
\end{equation}
Moreover, this embedding is a topological isomorphism of  $\Bbb Z^\Bbb N\rtimes \Sigma(\Bbb N)$ onto its image furnished with the weak topology.
It was shown in \cite{Rud} that if $T$ has MSJ then this image equals the entire $C(T^{\times\infty})$.

\subsection{Semisimplicity and factors of the infinite power of MSJ automorphisms}
The notion of semisimplicity has been introduced in \cite{Ju-Le-Me}.
An ergodic automorphism $S$ acting on $\ycn$ is called {\it semisimple} if for each self-joining $\rho\in J_2^e(S)$, the two factor mappings $(Y\times Y,\rho,S\times S)\to (Y,\nu,S)$ are relatively weakly mixing.\footnote{\label{f:relp}Recall that given  an ergodic automorphism $R$ of $\zdk$ with a factor $\ce\subset\cd$ and the (corresponding to $\ce$) disintegration $\kappa=\int_{\widebar Z}\kappa_{\widebar{z}}\,d\widebar\kappa(\widebar z)$ of $\kappa$,
we say that $R$ is {\em relatively weakly mixing over} $\ce$   if the measure
$$
\kappa\otimes_{\ce}\kappa:=\int_{\widebar Z} \kappa_{\widebar z}\otimes\kappa_{\widebar z}\,d\widebar\kappa(\widebar z)
$$
is $(R\times R)$-ergodic.
}

\begin{Lemma}\label{l:msj1} If $T$ is an MSJ automorphism then the automorphism $(T^{\times\infty},\mu^{\otimes\Bbb N})$ is semisimple.
 \end{Lemma}
\begin{proof} Let $\la\in J^e_2(T^{\times\infty})$. It follows from  the definition of MSJ that
 there exist a subset $P\subset\N$, a one-to-one mapping $\sigma\colon P\to \Bbb N$ and
 a mapping $l\colon P\to\Bbb Z$ such that for all subsets $A=\prod_{i\in\Bbb N} A_i\subset X^{\Bbb N}$ and
 $B=\prod_{i\in\Bbb N} B_i\subset X^\Bbb N$ such that $A_i=B_i=X$ eventually in $i$,
 \beq\label{2-joi}
 \la(A\times B)=\Pi_{i\not\in P}\mu(A_i)\cdot \Pi_{i\in P}\mu(A_i\cap T^{l(i)}B_{\sigma(i)})\cdot \Pi_{i\not\in\sigma(P)}\mu(B_i).
 \eeq
 Disintegrating $\lambda$ with respect to the projection  $\pi_1\colon X^{\Bbb N}\times X^{\Bbb N}\to X^{\Bbb N}$ to the first coordinate,
 we obtain that
 $$
 \begin{aligned}
 \lambda&=\int_{X^{\Bbb N}}\bigotimes_{i\not\in P}\delta_{x_i}\otimes\bigotimes_{i\in P}(\delta_{x_i}\otimes\delta_{T^{l(i)}x_{\sigma(i)}})\otimes\mu^{\otimes (\Bbb N\setminus\sigma(P))}\,d\mu^{\otimes\N}(x).\\
 \end{aligned}
 $$
 Of course, we have  assumed  here that $x=(x_i)_{i\in \Bbb N}\in X^{\Bbb N}$.
Consider the mapping
$$
 \tau\colon X^{\Bbb N}\in X^{\Bbb N}\ni(x,y)\mapsto (x, (y_i)_{i\in\Bbb N\setminus\sigma(P)})\in X^{\Bbb N}\times X^{\Bbb N\setminus\sigma(P)}.
 $$
Then $\tau$ is an isomorphism of the dynamical system
$(X^{\Bbb N}\times X^{\Bbb N},\lambda, T^{\times\infty}\times T^{\times\infty})$ with the direct product
$(X^{\Bbb N},\mu^{\otimes\Bbb N}, T^{\times\infty}) \otimes (X^{\Bbb N\setminus\sigma(P)},\mu^{\otimes( \Bbb N\setminus\sigma(P))}, T^{\times(\Bbb N\setminus\sigma(P))})$
such that $\pi_1\circ\tau^{-1}$ is the projection of this direct product onto the first
coordinate.
 Since $T$ is weakly mixing, it follows that $\pi_1$ is relatively weakly mixing.
 In a similar way one can show that $\pi_2$ (i.e.\ the projection to the second coordinate in $X^\Bbb N\times X^{\Bbb N}$) is relatively weakly mixing.
\end{proof}

\begin{Lemma}\label{l:msj2}Let $T$ be an MSJ automorphism and let $\ca\subset\cb(X^{\N})$ be a factor such that $(T^{\times\infty},\mu^{\otimes\N})$ is relatively weakly mixing over $\ca$. Then, there exists $P\subset \N$ such that (up to a natural identification) $\ca=\cb(X^P)$.
\end{Lemma}
\begin{proof} Consider the $\mathcal A$-relatively independent product
$\la= \mu^{\otimes\N}\ot_{\ca}\mu^{\otimes\N}\in J^e_2(T^{\times\infty})$, cf.\ Footnote~\ref{f:relp}.
Let $\cb_1(\la)$ denote the largest factor of $\cb(X^{\N})\ot\{\emptyset,X^\infty\}$ identified with a factor $\cb_2(\la)$ ``sitting'' on the second coordinate
$\{\emptyset,X^\infty\}\ot \cb(X^{\N})$.
In view of a general property of the relative product (see, e.g.\ \cite{Le-Pa}):
$$
\cb_1(\mu^{\ot\Bbb N}\ot_{\ca}\mu^{\ot\Bbb N})=\ca=
\cb_2(\mu^{\ot\Bbb N}\ot_{\ca}\mu^{\ot\Bbb N})  \mod\la,
$$
that is, the  Markov operator $\Phi_\lambda\colon L^2(X^\Bbb N,\mu^{\otimes\Bbb N})\to
L^2(X^\Bbb N,\mu^{\otimes\Bbb N})$ corresponding to $\lambda$ is isometric only on the elements of $L^2(\ca)$.
Moreover, $\Phi_\lambda f=f$ for each $f\in L^2(\ca)$.

On the other hand, as $T$ has MSJ, $\lambda$ is of the form (\ref{2-joi}).
It follows from~(\ref{2-joi}) that $\cb_1(\la)=\mathcal B^{\otimes P}$,
$\cb_2(\la)=\mathcal B^{\otimes \sigma(P)}$
and
$$
\Phi_\lambda\Bigg(\bigotimes_{i\in P}f_i\Bigg)=\bigotimes_{i\in \sigma(P)}f_{\sigma^{-1}(i)}\circ T^{-l(\sigma^{-1}(i))}
$$
 for all $f_i\in L^2(X,\mu)$.
Therefore, $\sigma(P)=P$, $\sigma(i)=i$  and $l(i)=0$ for each $i\in P$.
The result follows.
\end{proof}

Given a  factor $\ca\subset \cb(X^\Bbb N)$ of $T^{\times\infty}$, there exists the smallest factor $\widehat{\ca}$ of $T^{\times\infty}$ such that $\widehat{\ca}\supset\ca$ and
$T^\infty$ is $\widehat{\ca}$-relatively weakly mixing \cite{Ju-Le-Me}.
By Lemma~\ref{l:msj2},
$\widehat{\ca}=\cb(X^{P})$ for some subset $P=P(\ca)\subset \N$.
 We note that $P$ is the smallest subset of $\Bbb N$ such  that
 ${\ca}\subset (\bigotimes_{i\in P}{\mathcal B})\otimes\{\emptyset, X^{\Bbb N\setminus P}\}$.
 We can now apply the main result of \cite{Ju-Le-Me} about  the  factors of semisimple automorphisms
 (a generalization of  the Veech theorem \cite{Ve} and \cite{Ju-Ru}) to obtain a full description for the factors
 of $(T^{\times\infty},\mu^{\ot\N})$.

\begin{Prop}\label{p:msj1} Let $T$ be an MSJ automorphism. Let $\ca\subset \cb(X^\infty)$ be a factor of $T^{\times\infty}$.
Then there exists a subset $F=F(\ca)\subset\N$ and a compact subgroup $K={K}(\ca)\subset C(T^{\times F},\mu^{\otimes F})$ such that $\ca$ is the $\sigma$-algebra $\mathcal I( K)$ of subsets fixed by all elements of ${K}$.
\end{Prop}


\subsection{The characteristic class generated by an MSJ automorphism}


As above,  the dynamical system  $(X,\mu,T)$ has  the MSJ property and $\mu$ is continuous (hence, $T$ is weakly mixing).
By the Jewett-Krieger theorem, we can assume that $T$ is a uniquely ergodic homeomorphism acting on a compact metric space $X$ and $\mu$ is the unique $T$-invariant probability
measure on $X$.
Then $T$ has no periodic orbits.

Denote by $\mathcal R$ the set of all triplets $(D,K,l)$, where $D$ is a finite subset of $\Bbb N$, $K$ is a transitive subgroup of $\Sigma(D)$ and $l\colon K\to \Bbb Z^D$ is a skew homomorphism.
Of course, $\mathcal R$ is countable.
Given $(D,K,l)\in \mathcal R$, we let
$$
C_{D,K,l}:=\{(l(\sigma), \sigma)\colon \sigma\in K\}.
$$
Then $C_{D,K,l}$ is a subgroup of $\Bbb Z^D\rtimes K$.
This subgroup acts continuously on $X^D$ by the formula
\beq\label{mavpa}
(l(\sigma), \sigma)(x_d)_{d\in D}:=(T^{l(\sigma)_d}x_{\sigma(d)})_{d\in D}\quad\text{for each $(x_d)_{d\in D}\in X^D$}.
\eeq

Denote by $\mathcal R\mathcal R$ the set of all pairs $P=((D_n,K_n,l_n)_{n=1}^N, H)$ consisting of a finite sequence of triplets $(D_n,K_n,l_n)\in\mathcal R$ and a
subgroup $H\subset \bigotimes_{n=1}^NC_{D_n,K_n,l_n}$
such that the projection of $H$ to the second coordinate of the $n$-th marginal
is $K_n$ for each $n$.
Of course, $H$ acts continuously on $\bigotimes_{n=1}^N X^{D_n}$ (see (\ref{mavpa})).
This action commutes with the homeomorphism $\bigotimes_{n=1}^N T^{\times D_n}$.
Let
$$
Y_P:=\bigg(\bigotimes_{n=1}^N X^{D_n}\bigg)/H
$$
stand for  the corresponding quotient space, i.e. the space of $H$-orbits furnished with the quotient topology.
Let $\pi_{P}\colon \bigotimes_{n=1}^N X^{D_n}\to Y_{P}$ denote the corresponding projection.
Then $Y_{P}$ is a compact metric space, $\pi_{P}$ is continuous and there exists a homeomorphism $S_{P}$ of
$Y_{P}$ such that
 $$
 \pi_P\circ \bigotimes_{n=1}^NT^{\times D_n}= S_{P}\circ \pi_P.
 $$
Let
$$
Y:=\bigotimes_{P\in\mathcal R\mathcal R}(Y_P\sqcup\{*\}).
$$
Endow $Y$  with the product topology.
 Then $Y$ is a compact metric space.
 Given $P\in\mathcal R\mathcal R
 $, we extend $S_{P}$ to a homeomorphism of $Y_{P}\sqcup \{*\}$ by setting
 $S_{P}*:=*$.
 Let
 $$
 S:=\bigotimes_{P\in \mathcal R\mathcal R}S_{P}.
 $$
 Then $S$ is a homeomorphism of $Y$.
 It has a unique fixed point.
 All other points have infinite orbits.
 Hence, $S$ has a unique one-atomic invariant measure.
 All the other invariant measures are continuous.

We first show that there is a one-to-one correspondence between the ergodic $S$-invariant measures
on $Y$ and the factors of  $(X^\Bbb N,\mu^{\otimes \Bbb N},T^{\times\infty})$ considered up to isomorphism.
More precisely,  we will prove the following proposition.

\begin{Prop}\label{ergodic case}
\begin{itemize}
\item[{\rm (i)}]
For each factor $\mathcal A$  of  $(X^\Bbb N,\mu^{\otimes \Bbb N},T^{\times\infty})$
there is an $S$-invariant measure $\alpha$ on $Y$ such that the dynamical system
$(Y,\alpha, S)$ is isomorphic to $T^{\times\infty}|_{\mathcal A}$.
\item[{\rm (ii)}]
For each ergodic $S$-invariant measure $\beta$ on $Y$, the system
$(Y,\beta, S)$ is isomorphic to a factor of $(X^\Bbb N,\mu^{\otimes \Bbb N},T^{\times\infty})$.

\end{itemize}
\end{Prop}
\begin{proof}
(i)
The trivial factor corresponds to the unique one-atomic $S$-invariant measure
on $Y$.
Suppose now that $\mathcal A$ is nontrivial.
 It follows from Propositions~\ref{p:msj1}, \ref{structure of subgroups}  and \cite{Rud} that there exists
 a countable (or finite) sequence $(D_n,K_n,l_n)\in \mathcal R$
 and a compact subgroup $H\subset \bigotimes_n C_{D_n,K_n,l_n}$
 such that $\mathcal A$ is isomorphic to   the factor $\mathcal I(H)$
 of the product system $\bigotimes_n(X^{D_n}, \mu^{\otimes D_n}, T^{\times D_n})$.
 We remind that  $H$ (and the entire $\bigotimes_n C_{D_n,K_n,l_n}$) is contained in the centralizer of  $\bigotimes_n T^{\times D_n}$.
 Given $N$, we denote by $H_N$ the image of $H$ in $\bigotimes_{n=1}^N C_{D_n,K_n,l_n}$
 along the natural projection $\bigotimes_{n} C_{D_n,K_n,l_n}\to \bigotimes_{n=1}^N C_{D_n,K_n,l_n}$.
 Then  the pair
 $$
 P_N:=((D_n,K_n,l_n)_{n=1}^N, H_N)
 $$
  belongs to $\mathcal R\mathcal R$.
 We let $\xi_N:=\bigotimes_{n=1}^N\big(\mu^{\otimes D_n}\big)\circ\pi_{P_N}^{-1}$.
Then we obtain an ``inverse'' sequence
$$
(Y_{P_1},\xi_{P_1}, S_{P_1})\leftarrow (Y_{P_2},\xi_{P_2}, S_{P_2})\leftarrow\cdots
$$
of measure preserving dynamical systems.
All the arrows are equivariant mappings.
Hence, the inverse  limit $(\widetilde Y,\widetilde\xi, \widetilde S)$  of this sequence is well defined.
 Of course, $(Y_{P_N},\xi_{P_N}, S_{P_N})$
 is isomorphic to $\mathcal I(H_N)$ of the product system $\bigotimes_{n=1}^N(X^{D_n}, \mu^{\otimes D_n}, T^{\times D_n})$ for each $N$.
Hence, $(\widetilde Y,\widetilde\xi, \widetilde S)$ is isomorphic to $\mathcal I(H)$ and, hence, to $\mathcal A$.

To define a measure on $Y$, we first note that
the inverse  limit $\widetilde Y$ of compact spaces $Y_{P_N}$ embeds canonically into the Cartesian product
$\bigotimes_N Y_{P_N}$.
We use this embedding to transfer the measure $\widetilde\xi$ to $\bigotimes_N Y_{P_N}$.
Thus, we obtain an $\bigotimes_N S_{P_N}$-invariant probability measure on $\bigotimes_N Y_{P_N}$.
In turn, $\bigotimes_N Y_{P_N}$ embeds  naturally  into  $\bigotimes_N (Y_{P_N}\sqcup\{*\})$
as an $\bigotimes_N S_{P_N}$-invariant compact subset.
Hence,
 we obtain an $\bigotimes_N S_{P_N}$-invariant measure on $\bigotimes_N (Y_{P_N}\sqcup\{*\})$.
Denote it by $\xi$.
Let $\mathcal R\mathcal R':=\mathcal R\mathcal R\setminus\{P_1,P_2,\dots\}$.
 We now define a  measure $\alpha$ on $Y=\bigotimes_N(Y_{P_N}\sqcup\{*\})\otimes \bigotimes_{Q\in \mathcal R\mathcal R'}(Y_{Q}\sqcup\{*\})$ by setting
 $$
 \alpha:=\xi\otimes\bigotimes_{Q\in \mathcal R\mathcal R'}\delta_{*}.
  $$
Then $\alpha$ is invariant under $S$ and $(Y,\alpha, S)$ is isomorphic to $\mathcal A$, as desired.

(ii)
 Conversely, let $\beta$ be an ergodic $S$-invariant measure on
 $Y$.
 For each  $P\in\mathcal R\mathcal R$, let $\beta_{ P}$
denote the projection of $\beta$  onto the $P$-coordinate of $Y$, i.e. to the compact set $Y_{P}\sqcup\{*\}$.
Then $\beta_{ P}$ is an ergodic $S_{P}$-invariant probability on this set.
Hence, either $\beta_{ P}(\{*\})=0$ or $\beta_{P}(\{*\})=1$.
Let
$$
\mathcal R\mathcal R_\beta:=\{P\in\mathcal R\mathcal R\colon \beta_{ P}(\{*\})=0\}.
$$
Then $\beta$ is supported on the closed $S$-invariant subset
$$
Y_\beta:=\left(\bigotimes_{P\in\mathcal R\mathcal R_\beta}Y_{P}\right)\times
\bigotimes_{P\not\in\mathcal R\mathcal R_\beta}\{*\}\subset Y.
$$
We now set
 $$
 X_\beta:=\bigotimes_{ \mathcal R\mathcal R_\beta\,\ni \,P=\big((D_n,\dots)_{n=1}^N,H\big)}\bigotimes_{n=1}^N X^{D_n}\quad\text{and}\quad \pi^\beta:=\bigotimes_{P\in \mathcal R_\beta}\pi_{P}.
 $$
 Then $\pi^\beta$ is an equivariant continuous mapping from the topological system
 $\Big(X_\beta, \bigotimes_{(D_n,\dots)_{n=1}^N\in \mathcal R\mathcal R_\beta}
 \bigotimes_{n=1}^NT^{\times D_n}\Big)$ onto $(Y_\beta, S)$.
 Since the former system is a compact extension of the latter one, we can
 choose an ergodic  $\bigotimes_{(D_n,\dots)_{n=1}^N\in \mathcal R\mathcal R_\beta}
 \bigotimes_{n=1}^NT^{\times D_n}$-invariant measure $\alpha$
 on $X_\beta$ that projects onto $\beta$ under $\pi^\beta$.
 As $T$ is uniquely ergodic, $\alpha$ is an ergodic ``multiple'' self-joining of $T$.
 Without loss of generality, we may assume that $\alpha$ is an infinite self-joining of $T$.
 Thus, we showed that $(Y,\beta, S)$ is a factor of  $(X^{\times\Bbb N}, \mu^{\otimes\Bbb N}, T^{\times\infty})$,
as desired.
\end{proof}

Prior to state the main result of this section, we prove two more auxiliary facts.

 Let $Z$ be  an uncountable compact Polish space and let
 $Q\colon Z\to Z$  be a homeomorphism.
Let $M(Z,Q)_*$ be the subset of $M(Z,Q)$ consisting of all
nonatomic measures.
Then  $M(Z,Q)_*$ is a $G_\delta$-subset of $M(Z,Q)$.
Hence, it is Polish in the induced topology.

\begin{Lemma}\label{Bor} For each $\lambda\in M(Z,Q)_*$, there is Borel bijection $\varphi_\lambda\colon Z\to [0,1]$ (mod $\lambda$) such that $\lambda\circ\varphi_\lambda^{-1}=\text{\rm Leb}$
and the mapping
$$
\boldsymbol a_Z\colon M(Z,Q)_*\ni\lambda\mapsto \boldsymbol a_ Z(\lambda):=\varphi_\lambda Q\varphi_\lambda^{-1}\in \text{\rm Aut}([0,1],\text{\rm Leb})
$$
is Borel.
\end{Lemma}

\begin{proof} Without loss of generality we may assume that
$Z=[0,1]$.\footnote{This means that there is  a Borel isomorphism of $Z$ onto $[0,1]$.
Then
 $Q$ (or rather the corresponding copy of $Q$) is no longer a homeomorphism of $[0,1]$.
 However, this copy is a Borel isomorphism
of $[0,1]$.
Moreover, passing from $Z$ to $[0,1]$ transfers  $M_*(Z,Q)$ bijectively
onto the set of Borel non-atomic probability measures on $[0,1]$.
This mapping is  Borel if we endow the latter set with the usual Borel structure which does not depend on the choice of compact topology generating it.}
Take a probability $\lambda\in M(Z,Q)_*$.
Since $\lambda$ is nonatomic,
the $\lambda$-distribution mapping
$$
\psi_\lambda\colon [0,1]\ni u\mapsto\psi_\lambda(u):=\lambda([0,u])\in [0,1]
$$
is continuous, non-decreasing, $\psi_\lambda(0)=0$
and $\psi_\lambda(1)=1$.
Hence, the pseudo-inverse  to $\psi_\lambda$ mapping
$$
[0,1]\ni t\mapsto\varphi_\lambda(t):=\min\{u\colon \psi_\lambda(u)\ge t\}
$$
is well defined and continuous.
It is well known that
$\psi_\lambda$ is an isomorphism (mod 0) of $([0,1],\lambda)$ onto $([0,1], \text{Leb})$.
Moreover, $\varphi_\lambda=\psi_\lambda^{-1}$ mod 0.
Then, of course, $\varphi_\lambda Q\varphi_\lambda^{-1}\in \text{\rm Aut}([0,1],\text{\rm Leb})$ for each $\lambda\in M(Z,Q)_*$.
It remains to show that for all Borel subsets $A,B\subset[0,1]$, the mapping
\beq\label{rural}
M(Z,Q)_*\ni\lambda\mapsto\text{Leb}( \varphi_\lambda^{-1} Q\varphi_\lambda A\cap B)
\eeq
is Borel.
We note that
 the mappings
\begin{align*}
M(Z,Q)_*\times [0,1]\ni(\lambda, t)\mapsto \psi_\lambda(t)\in[0,1]\\
M(Z,Q)_*\times [0,1]\ni(\lambda, t)\mapsto \varphi_\lambda(t)\in[0,1]\\
\end{align*}
are Borel.
Hence,
the mapping
$$
M(Z,Q)_*\times [0,1]\ni(\lambda,t)\mapsto \psi_\lambda Q\varphi_\lambda( t)
$$
is also Borel.
As
\begin{align*}
\text{Leb}( \varphi_\lambda Q^{-1}\varphi_\lambda^{-1}A\cap B)
&=
\int_0^1 1_A(\varphi_\lambda^{-1} Q\varphi_\lambda( t))1_B(t)\,dt\\
&=
\int_0^1 1_A(\psi_\lambda Q\varphi_\lambda( t))1_B(t)\,dt,
\end{align*}
it follows that the mapping (\ref{rural}) is Borel, as desired.
\end{proof}


Suppose that a $Q$-invariant measure $\kappa$ is fixed on $Z$.
Denote by $\frak F$ the set of all $Q$-invariant sub-sigma-algebras (mod $\kappa$) on $Z$.
To topologize $\frak F$, consider for each $\mathcal E\in\frak F$, the orthogonal projection
$P_\mathcal E$ in $L^2(Z,Q)$ onto the closed subspace
$L^2(Z,\mathcal E)$.
Then the mapping $\frak F\ni\mathcal E\mapsto P_\mathcal E$ is a one-to-one embedding
of $\frak F$ into the set of orthogonal projections in $L^2(Z,Q)$.
It is straightforward to verify that the image of this embedding is $G_\delta$. Hence, it is Polish in the induced topology.
We now furnish $\frak F$ with the weak operator topology using this embedding.
Then $\frak F$ is a  Polish space.
Select a sequence $(f_n)_{n=1}^\infty$ of continuous functions
from $Z$ to $[0,1]$ such that $\{f_n\colon n\in\Bbb N\}$ is dense in the unit ball of the  Banach space $(C([0,1]), \|.\|_\infty)$.
Given $\mathcal E\in\frak F$, we let $\EE_\mathcal E=\EE(\cdot|\mathcal{E}):L^1(Z,\kappa)\to L^1(Z,\mathcal E,\kappa\restriction\mathcal E)$.
Denote by $W$ the infinite Cartesian product  $[0,1]^{\Bbb N}$.
The functions $\EE_\mathcal E(f_n)\colon Z\to[0,1]$, $n\in\Bbb N$, are $\mathcal E$-measurable
and generate (together) the entire $\mathcal E$.
Hence, the mapping
$$
\boldsymbol b_\mathcal E\colon Z\ni x\mapsto \boldsymbol b_\mathcal E(x):=\Big(\big(
\EE_\mathcal E(f_n)(Q^mx)\big)_{n=1}^\infty\Big)_{m\in\Bbb Z}\in W^\Bbb Z
$$
is a measure theoretical isomorphism of  $(Z,\mathcal E,\kappa|_{\mathcal E})$
onto $(W^\Bbb Z, (\kappa|_{\mathcal E})\circ\boldsymbol b_\mathcal E^{-1})$ that intertwines
$Q$ with the two-sided shift $D$ on $W^{\Bbb Z}$.
Moreover, the mapping
$$
\boldsymbol b\colon \frak F\ni\mathcal E\mapsto\boldsymbol b(\mathcal E):=(\kappa|_{\mathcal E})\circ\boldsymbol b_\mathcal E^{-1}\in M(W^{\Bbb Z},D)
$$
is Borel.
Thus, we have shown the following lemma.

\begin{Lemma}\label{factors-measures} There is a Borel
mapping $\boldsymbol b\colon \frak F\to M(W^{\Bbb Z},D)$ such that the dynamical systems
$(Z,\mathcal E,\kappa|_{\mathcal E}, Q)$ and $(W^\Bbb Z, \boldsymbol b(\mathcal E), D)$
are isomorphic for each $\mathcal{E}\in\frak F$.
\end{Lemma}
We now prove the main result of this section.


\begin{Th}\label{t:universalMSJ} Let $(X,\mu,T)$ have the MSJ property. Then the characteristic class $\cf(T)$ has a universal model.\end{Th}
\begin{proof}
We claim that $([0,1]\times Y,\text{Id}\times S)$ is a universal model for $\mathcal F(T).$

 First,  it follows from Proposition~\ref{ergodic case}(ii) and Corollary~\ref{p:random-factor?} that for each $(\text{Id}\times S)$-invariant
probability measure $\omega$ on the space $[0,1]\times Y$, the corresponding dynamical system
$([0,1]\times Y, \omega, \text{Id}\times S)$ is a factor of the product  $([0,1]\times X^\Bbb N,\text{Leb}\otimes\mu^{\ot\infty}, \text{Id}\times T^{\times \Bbb N})$.
Thus, $([0,1]\times Y, \omega, \text{Id}\times S)\in \mathcal F(T)$.\footnote{
Indeed, the ergodic decomposition of $({\rm Id}\times S,\omega)$ is given by a choice of measure on the set of ${\rm Id}\times S$-ergodic measures on $[0,1]\times Y$. All such ergodic measures are of the form $\delta_t\otimes\omega'$, where $\omega'$ is an $S$-ergodic measure on $Y$. Hence, each such ergodic measure determines an ergodic system $(Y,\omega',S)$ which is isomorphic to a factor of $(T^{\times\infty},\mu^{\ot\N})$ and therefore Corollary~\ref{p:random-factor?} applies.
}


Conversely, take an arbitrary  dynamical system  belonging to $\mathcal F(T)$.
By  Proposition~\ref{p:simfac}, this system is isomorphic to a factor
 of the product
 $$
 ([0,1]\times X^\Bbb N,\text{Leb}\otimes\mu^{\ot\infty}, \text{Id}\times T^{\times \Bbb N}).$$
 Denote this factor by $\mathcal C$.
We remind that for an arbitrary measure preserving transformation $Q$, the sigma-algebra of
$Q$-invariant subsets is denoted by $\mathcal I(Q)$.
We now have that
$$
\mathcal I((\text{Id}\times T^{\times \Bbb N})\restriction \mathcal C)=\mathcal I( \text{Id}\times T^{\times \Bbb N})\cap\mathcal C=\Big(\mathcal B([0,1])\otimes\{\emptyset, X^{\times \Bbb N}\}\Big)\cap\mathcal C.
$$
Hence, there exists a sigma-algebra $\mathcal G\subset\mathcal B([0,1])$ such that
$$
\mathcal I((\text{Id}\times T^{\times \Bbb N})\restriction \mathcal C)=\mathcal G\otimes\{\emptyset, X^{\times \Bbb N}\}.
$$
Then $\mathcal G$ determines a quotient mapping
$$
([0,1],\mathcal G,\text{Leb}\restriction\mathcal G)\to([0,1], \mathcal B([0,1]),\kappa)
$$
 for some Borel measure $\kappa$ on $[0,1]$.
 The space $([0,1], \mathcal B([0,1]),\kappa)$ is the space of ergodic components
 of $(\text{Id}\times T^{\times \Bbb N})\restriction \mathcal C$.
 Restricting $\mathcal C$ to the fibers over points of this space we obtain
a measurable field $[0,1]\ni t\mapsto \mathcal C_t$ of factors of
  $(X^\Bbb N,\mu^{\otimes N},T^{\times \Bbb N})$.\footnote{We remind that a.e. ergodic component of a factor of a dynamical system is a factor of an ergodic component of the system.}
We also note that in view of Proposition~\ref{2factors}, $\mathcal C$
  is completely determined by this field.
 Denote by $\frak F$ the set of all factors of
 $(X^\Bbb N,\mu^{\otimes N},T^{\times \Bbb N})$
 and endow $\frak F$ with the Polish topology as in Lemma~\ref{factors-measures}.
Then we  have shown  that a measurable field
  $$
  ([0,1], \mathcal B([0,1]),\kappa)\ni t\mapsto \mathcal C_t\in\frak F
  $$
  is well defined.
  We now can write $\kappa$ as a convex combination $\kappa=q\kappa_d+(1-q)\kappa_c$ of two mutually disjoint probability measures $\kappa_d$ and $\kappa_c$ such that
  $\mathcal C_t$ is trivial for $\kappa_d$-a.e. $t\in[0,1]$ and  $\mathcal C_t$ is non-trivial
  for $\kappa_c$-a.e. $t\in[0,1]$.
  We have to show that the two systems
  $([0,1],\kappa_d, (\mathcal C_t)_{t\in[0,1]})$
  and $([0,1],\kappa_c, (\mathcal C_t)_{t\in[0,1]})$ can be represented via some
  invariant measures for $([0,1]\times Y, \text{Id}\times S)$.
  The first one corresponds obviously to the measure $\kappa_d\otimes\delta_{*}$.
  It remains to find a corresponding measure for $([0,1],\kappa_c, (\mathcal C_t)_{t\in[0,1]})$.
Since $\mathcal C_t$ is non-trivial and ergodic, it follows from Proposition~\ref{ergodic case}(i) that there exists a continuous measure $\alpha_t$ on $Y$ such that $(Y,\alpha_t,S)$
is isomorphic to $\mathcal C_t$ for $\kappa_c$-a.e. $t\in[0,1]$.
Consider a subset
$$
V:=\{(t,\lambda)\in [0,1]\times M_*(Y,S)\colon (Y,\lambda,S)\text{ is isomorphic to $\mathcal C_t$}\}.
$$
Then the projection of $V$ to the first coordinate
is the entire segment $[0,1]$ mod $\kappa$.

We claim that  $V$ is an analytic subset of $[0,1]\times M_*(Y,S)$.
Indeed, given $\lambda\in M_*(Y,S)$, the system
$(Y,\lambda,S)$ is isomorphic to $\mathcal C_t$ if and only if
$\boldsymbol a_Y(\lambda)$ is isomorphic to $\boldsymbol b(\mathcal C_t)$.
Let $\boldsymbol A:=\text{Aut}([0,1],\text{Leb})$.
We note that  the subset
$$
\mathcal P:=\{ (V,V')\in\boldsymbol A\times\boldsymbol  A\colon
\text{$V$ and $V'$ are conjugate in $\boldsymbol A$}\}
$$
is analytic  \cite{Fo-Ru-We}.
It follows from Lemmas~\ref{Bor} and \ref{factors-measures} that
the mapping
$$
[0,1]\times M_*(Y,S)\ni(t,\lambda)\mapsto \big(\boldsymbol a_Y(\lambda),
\boldsymbol a_{W^\Bbb Z}( \boldsymbol b(\mathcal C_t))\big)\in \boldsymbol A\times\boldsymbol  A
$$
is Borel.
Hence,
$$
V=\{(t,\lambda)\in [0,1]\times M_*(Y,S)\colon (\boldsymbol a_Y(\lambda),
\boldsymbol a_{W^\Bbb Z}( \boldsymbol b(\mathcal C_t))\in \mathcal P\}
$$
is  analytic because it is the inverse image of $\mathcal P$ with respect to a Borel mapping.

Hence, we are in a position to apply  the Jankov-von Neumann theorem.
It follows from this theorem that there is a measurable (analytic)
mapping $[0,1]\ni t\mapsto\lambda_t$ such that $(t,\lambda_t)\in V$ for $\kappa_c$-a.e. $t$.
Therefore,  the dynamical system
$\Big([0,1]\times Y, \int_{[0,1]}\delta_t\otimes\lambda_t\,d\kappa_c(t), \text{Id}\times S\Big)$ is isomorphic to $([0,1],\kappa_c, (\mathcal C_t)_{t\in[0,1]})$.
Finally, we obtain that
$\mathcal C$ is isomorphic to
$\Big([0,1]\times Y, q\kappa_d\otimes\delta_\ast+(1-q)\int_{[0,1]}\delta_t\otimes\lambda_t\,d\kappa_c(t), \text{Id}\times S\Big)$, as desired.
\end{proof}


\subsection{Proofs of Theorem~\ref{t:misejo} and Corollary~\ref{c:ortogtomsj}}
\begin{proof}[Proof of Theorem~\ref{t:misejo}.]
It follows from Theorem~\ref{t:universalMSJ} that $\cf(T)$ has a universal model. This model (see the proof of Theorem~\ref{t:universalMSJ}) has the form $([0,1]\times Y,\text{Id}\times S)$. However, $Y$ is of the form:
$$
Y=\bigotimes_{P\in\mathcal{R}\mathcal{R}}\Big(\Big(
\bigotimes_{n=1}^NX^{D_n}\Big)/H\cup\{\ast\}\Big)
$$
with the corresponding product action. Notice that such a homeomorphism is a {\bf topological} factor of
$$
\bigotimes_{P\in\mathcal{R}\mathcal{R}}\Big(\Big(
\bigotimes_{n=1}^NX^{D_n}\Big)\cup\{\ast\}\Big)$$
and it is not hard to see that the strong $\bfu$-MOMO of the corresponding homeomorphism is
equivalent to the strong $\bfu$-MOMO of
$$
\bigotimes_{P\in\mathcal{R}\mathcal{R}}\Big(
\bigotimes_{n=1}^NX^{D_n}\Big).$$
But this homeomorphism is topologically isomorphic to the infinite Cartesian product of $T$. To complete the proof use Lemma~\ref{l:m19}.
\end{proof}

\begin{proof}[Proof of Corollary~\ref{c:ortogtomsj}.]
Let us first notice that for each $s\geq1$, the prime powers of the automorphism $T^{\times s}$ acting on $(X^{s},\cb^{\ot s},\mu^{\ot s})$ are (Furstenberg) disjoint: for $p\neq q$, we have
\beq\label{abc1}
((T^{\times s})^p,\mu^{\ot s})\perp
((T^{\times s})^q,\mu^{\ot s}).\footnote{We have $T^p\times\ldots \times T^p\perp T^q\times\ldots\times T^q$ since $T^p\perp T^q$ and MSJ implies the PID property.}\eeq
It follows from \cite{Ab-Ku-Le-Ru}, \cite{Bo-Sa-Zi} that
\beq\label{abc2}\begin{array}{l}
\mbox{each uniquely ergodic model
of $(T^{\times s},\mu^{\ot s})$}\\
\mbox{satisfies the strong $\bfu$-MOMO property}\end{array}\eeq
for each bounded $\bfu$ which is multiplicative (recall that we always assume that $\|\bfu\|_{u^1}=0$).

We only need to prove that, given $(X_0,T_0)$ which is a uniquely ergodic model of $T$, the topological system $(X_0^\infty,T_0^{\times\infty})$
enjoys the strong $\bfu$-MOMO property (see (iii) in Theorem~\ref{t:misejo} and Lemma~\ref{l:m19}).
To obtain this we only need to prove that, for all $k\geq1$, $(X_0^k,T_0^{\times k})$ has the strong $\bfu$-MOMO property.
For that note that $M^e(X_0^k,T_0^{\times k})$ consists of countably many members. Moreover, each such member determines a measure-theoretic dynamical system which is isomorphic to a Cartesian power of $T$. Therefore, in view of \eqref{abc2}, each ergodic invariant measure for $(X_0^k,T_0^{\times k})$ has a uniquely ergodic model which satisfies the strong $\bfu$-MOMO property. It follows that we are in the framework of condition PF1  in~\cite[Theorem~4.1]{Ka-Le-Ra}\footnote{Theorem 4.1 in \cite{Ka-Le-Ra} is an infinite extension of the main result in \cite{Ab-Ku-Le-Ru}.}. Now, the implication PF1 $\Rightarrow$ PF3 in  \cite[Theorem~4.1]{Ka-Le-Ra} tells us that
$(X_0^k,T_0^{\times k})$ must also enjoy the strong $\bfu$-MOMO property.
\end{proof}



\section{On the main result of \cite{Hu-Ta-Xu}}\label{Hu-Xu}

\textcolor{red}{
\subsection{
Countable measure-theoretic spectrum = almost countability in the sense of Huang-Tan-Xu}\label{s:uncountable}}

Let now $(X,T)$ be a topological dynamical system.
For $\nu\in M^e(X,T)$, let $\Lambda(\nu)\subset \Bbb T$ denote the set of eigenvalues of the automorphism $(T,\nu)$.
It is well known that $\Lambda(\nu)$ is a countable subgroup of $\Bbb T$.
Let
$$
\Lambda(X,T):=\bigcup_{\nu\in M^e(X,T)}\Lambda(\nu).
$$

\begin{Th}\label{thm:mainMS}
Assume that $\Lambda(X,T)$
is uncountable.
Then there exists a non-atomic Borel probability measure $P$ on $M^e(X,T)$ such that
for every measurable set $E\subset M^e(X,T)$ with $P(E)=1$ one has
\[
\bigcup_{\nu\in E} \Lambda(\nu)
\quad\text{is uncountable.}
\]
\end{Th}

Before we prove Theorem~\ref{thm:mainMS} we need a lemma.
Let
$$
R:=\{(\nu,\lambda)\in M^e(X,T)\times \Bbb T\colon
 \lambda\in \Lambda(\nu)\}.
$$

\begin{Lemma}\label{lem:analytic}
The set $R$ is  Borel in the product space $M^e(X,T)\times \Bbb T$.
\end{Lemma}
\begin{proof}
Let $\{f_n\}_{n=1}^\infty$ be  a total family of continuous functions on $\Bbb T$ of norm $1$.
This means that the liner space spanned by all $f_n$, $n\in\Bbb N$, is dense in $C(T)$ in the norm topology.
Given $n\in\Bbb N$ and $\nu\in M^e(X,T)$, let $\sigma_{\nu,n}$ denote the  spectral measure of $f_n$ for the automorphism $(T,\nu)$, i.e.\
$\sigma_{\nu,n}\in M(\Bbb T)$ and  (see Section~\ref{f:spectral})
$$
\int_Xf_n\circ T^m\,\cdot\overline{ f_n}\,d\nu=\int_{\Bbb T}z^md\sigma_{\nu,n}
$$
for each $m\in\Bbb Z$.
Then, the mapping $M^e(X,T)\ni \nu\mapsto\sigma_{\nu,n}\in M(\Bbb T)$
is continuous (see \cite{Be-Go-Ru} for details).
Since the measure $\sigma_\nu:=\sum_{n=1}^\infty 2^{-n}\sigma_{\nu,n}$ is a measure of maximal spectral type for $(X,\nu,T)$, it follows that
the mapping
$$
M^e(X,T)\ni \nu\mapsto\sigma_\nu\in M(\Bbb T)
$$
is Borel. Moreover, the mapping $M (\T) \times \T \ni (\rho, z) \to \rho(\{z\})$ is also Borel.\footnote{The sets  $C_n:=\{(x,\la)\in\T\times\T\colon \|x-\la\|<1/n\}$ are open, whence $\raz_{C_n}$ is Borel and therefore $g_n(\rho,\la):=\int_{\T}\raz_{C_n}(x,\la)\,d\rho(x)$ are Borel, $n\geq1$. Finally, $\rho(\{\la\})=\inf_{n\geq1}g_n(\rho,\la)$. }
Since
$$
R=\{(\nu,\lambda)\in M^e(X,T)\times\Bbb T\colon \sigma_\nu(\{\lambda\})>0\},
$$
we deduce that $R$ is Borel.
\end{proof}

\begin{proof}[Proof of Theorem~\ref{thm:mainMS}]
Since $\Lambda(X,T)$ is the projection of $R$ onto the second coordinate, it follows that
$\Lambda(X,T)$ is analytic.
Let $\frak A$ denote the smallest $\sigma$-algebra (of  subsets of  $\Lambda(X,T)$) that contains all analytic (in $\Bbb T$) subsets.
It follows from the Jankov--von Neumann that there exists a  $\frak A$-measurable mapping\footnote{On $M^e(X,T)$ we consider the sigma-algebra of Borel sets.}
$s\colon \Lambda(X,T)\to M^e(X,T)$ such that $(s(\lambda),\lambda)\in R$ for each $\lambda\in \Lambda(X,T)$.
Take a nonatomic probability measure $\mu$ on $\frak A$\footnote{It exists. Indeed, as $\Lambda(X,T)$ is uncountable and analytic,
$\Lambda(X,T)$ contains a Cantor subset $C\subset\T$, see \cite{Ke}, Thm.\ 14.5. Take an arbitrary nonatomic Borel measure $\mu$ on $C$. Then this measure extends on $\frak A$ because each analytic set is universally measurable. Hence, we can consider $\mu$ as a measure on $\frak A$.}.
Then $P:=\mu\circ s^{-1}$ is a well defined Borel measure on $M^e(X,T)$.
If $E$ is a Borel subset of  $M^e(X,T)$ and $P(E)=1$ then $\mu(s^{-1}(E))=1$.
 Since $\mu$ is nonatomic, $s^{-1}(E)$ is uncountable.
 On the other hand,
 $s^{-1}(E)\subset \bigcup_{\nu\in E}\Lambda(\nu)$.
\end{proof}

We will say that a topological dynamical system $(X,T)$ has countable {\em measure-theoretic spectrum} if $\Lambda(X,T)=\bigcup_{\nu\in M^e(X,T)}\Lambda(\nu)$ is countable.
\begin{Remark}\label{r:ACequiv}
It follows from Theorem~\ref{thm:mainMS} that $(X,T)$ is almost countable in the sense of Huang, Tan and Xu \cite{Hu-Ta-Xu} if and only if it has countable measure-theoretic spectrum.
\end{Remark}

\subsection{$\mathcal{K}_{\rm
rel}(T)=\mathcal{K}(T)$ for all invariant measures implies logarithmic strong $\mob$-MOMO}\label{s:hutaxu}

Let us recall that, for an automorphism $T$ acting on $\xbm$, we denote by $\ck(T)$ its Kronecker factor and by $\ck_{\rm rat}(T)$,  its rational Kronecker factor, that is, the largest factor with rational discrete spectrum.  We write $\Lambda(T)$ for the point spectrum of $T$ and $\Lambda_{\rm irr}(T)$ for the irrational part of $\Lambda(T)$. Moreover, we denote by $\mathcal{K}_{\rm rel}(T)$ the relative (with respect to the sigma-algebra of invariants sets) Kronecker factor of $T$.
By Proposition~\ref{re:mrak}, $\ck_{\rm rel}(T)=\ca_{{\rm {\bf DISP}}_{\rm ec}}(T)$.

\subsubsection{Lemmas}\label{s:lemmas}
In the  lemmas below we will assume that $T$ is an automorphism of $\xbm$ with $\mu=\int\mu_\gamma\,d\PP(\gamma)$ being its ergodic decomposition.
The following is an immediate  corollary from Remark~\ref{re:boost}.

\begin{Lemma}\label{l:hutaxu1}
We have $\ck_{\rm rel}(T)=\ck(T)$ if and only if there exists  $\Gamma_0\subset \Gamma$, $\PP(\Gamma_0)=1$, such that $\bigcup_{\gamma\in\Gamma_0}\Lambda(T,\mu_\gamma)$ is countable (if and only if the set $\bigcup_{\gamma\in\Gamma_0}\Lambda_{irr}(T,\mu_\gamma)$ is countable).
\end{Lemma}

The five lemmas below (except for Lemma~\ref{l:hutaxu3}) come from \cite{Hu-Ta-Xu}, we give short proofs.

\begin{Lemma}\label{l:hutaxu2}
Let $g\in L^2\xbm$. Then
\[
g\perp L^2(\ck_{\rm rat}(T)) \iff g\perp L^2(\ck_{\rm rat}(T,\mu_\gamma)) \text{ for }\PP\text{-a.a. }\gamma\in\Gamma.
\]
\end{Lemma}
\begin{proof}
Recall that $g\perp L^2(\ck_{\rm rat}(T))$ iff $\sigma_g(\{\lambda\})=0$ for each $\lambda\in\T$ rational (cf.\ \eqref{mutualsing}). The result now follows immediately from Lemma~\ref{l:mutualsing} and the equality  $\int j(z)\,d\sigma_g(z)=
\int\Big(\int j(z)\,d\sigma_{g,\gamma}(z)\Big)\,d\PP(\gamma)$
for $j=\raz_{\{\lambda\}}$.
\end{proof}

For the definition of $\infty$-step pro-nil factor (of a given automorphism), we refer the reader to \cite{Fr-Ho}.
We use only the property of such a factor given below and proved in \cite[Lemma 6.1(i)]{Fr-Ho}:
\begin{Lemma} \label{l:hutaxu3}
Suppose that $T$ is an ergodic $\infty$-step pro-nil system, that $S$ acting on $\ycn$ is ergodic with
$\Lambda_{\rm irr}(T)\cap\Lambda_{\rm irr}(S)=\emptyset.$
Then, for all $\rho\in J(T,S)$, we have
$$ \int f\ot g\,d\rho=0$$
for all  $f\in L^2\xbm$ satisfying $f\perp L^2(\ck_{\rm rat}(T))$ and all $g\in L^2\ycn$.
\end{Lemma}

From this lemma we deduce the following.

\begin{Lemma}\label{l:hutaxu4}
Suppose that $T$ is an ergodic $\infty$-step pro-nil system, that $S$ acting on $\ycn$ is ergodic with
$\Lambda_{\rm irr}(T)\cap\Lambda_{\rm irr}(S)=\emptyset.$
 Then
for all $\rho\in J(T,S)$, we have
$$ \int f\ot g\,d\rho=0$$
for all $f\in L^2\xbm$ and all $g\in L^2\ycn$ satisfying $g\perp L^2(\ck_{\rm rat}(S))$.
\end{Lemma}
\begin{proof} We have that
$$
\int f\ot g\,d\rho=\int f_1\ot g\,d\rho+\int f_2\ot g\,d\rho,
$$
where $f_1\in L^2(\ck_{\rm rat}(T))$ and $f_2\perp L^2(\ck_{\rm rat}(T))$.
Now, the first integral is zero  in view of
\eqref{mutualsing}, since the spectral measure of $f_1$ is concentrated on the set of roots of unity, while the spectral measure of $g$ has no atom there.
 The second integral vanishes thanks to  Lemma~\ref{l:hutaxu3}.
\end{proof}

This lemma, in turn, yields  immediately the following.

\begin{Lemma}\label{l:hutaxu5}
Suppose that $T$ is an ergodic Cartesian product of an $\infty$-step pro-nil system and a Bernoulli system, and that $S$ is ergodic of zero entropy with
$\Lambda_{\rm irr}(T)\cap\Lambda_{\rm irr}(S)=\emptyset.$
Then,
for all $\rho\in J(T,S)$, we have
$$ \int f\ot g\,d\rho=0$$
for all $f\in L^2\xbm$ and all $g\in L^2\ycn$ satisfying $g\perp L^2(\ck_{\rm rat}(S))$.
\end{Lemma}
\begin{proof}
It suffices to note that
$$
\int f\ot g\,d\rho=\int \EE(f|\infty-\text{step pro-nil factor})\ot g\,d\rho
$$
as this pro-nil factor is the Pinsker factor of the Cartesian product,
and apply
 Lemma~\ref{l:hutaxu4}.
\end{proof}

In the following lemma we do not assume that the automorphisms $T$ and $S$ on $(X,\mu)$ and $(Y,\nu)$ respectively, are ergodic.

\begin{Lemma}\label{l:hutaxu6}
Suppose that $\Lambda_{\rm irr}(T)=\emptyset$
 and  a.e.\ ergodic component of $T$ is the Cartesian product of some $\infty$-step pro-nil system and a Bernoulli automorphism.
  Suppose that $S$ is of zero entropy and $\mathcal{K}_{\rm rel}(S)=\mathcal{K}(S)$. Then,
for each $\rho\in J(T,S)$, we have
$$\int f\ot g\,d\rho=0$$
for all $f\in L^2\xbm$ with continuous spectral measure and for all $g\in L^2\ycn$.
\end{Lemma}
\begin{proof}
Without loss  of generality, we can assume that $X$ and $Y$ are compact metric spaces
and
 $T$ and $S$ are homeomorphisms of $X$ and $Y$  respectively.
Then $\rho\in M(X\times Y,T\times S)$.
Let $\rho=\int_{M^e(X\times Y,T\times S)}\eta\,d\Bbb L(\eta)$
stand for the ergodic decomposition of $\rho$.
Since $\mu=(\pi_X)_\ast(\rho)$ and $\nu=(\pi_Y)_\ast(\rho)$, we obtain that
\begin{align*}
\mu&=\int_{M^e(X\times Y,T\times S)}(\pi_X)_\ast\eta\,d\Bbb L(\eta)
=\int_{M^e(X,T)}\kappa\,d(\pi_Y)_{\ast\ast}(\Bbb L)(\kappa)\quad\text{ and}\\
\nu&=\int_{M^e(X\times Y,T\times S)}(\pi_Y)_\ast\eta\,d\Bbb L(\eta)
=\int_{M^e(Y,S)}\kappa\,d(\pi_Y)_{\ast\ast}(\Bbb L)(\kappa).
\end{align*}
Hence,
 the probabilities  $(\pi_X)_{\ast\ast}(\Bbb L)$ and $(\pi_Y)_{\ast\ast}(\Bbb L)$ represent
  the ergodic decomposition of $\mu$ and $\nu$ respectively.
Since $\mathcal{K}_{\rm rel}(S)=\mathcal{K}(S)$, it follows from Lemma~\ref{l:hutaxu1}
that there is a subset $\mathcal L_Y$
of full $(\pi_Y)_{\ast\ast}(\Bbb L)$-measure in $M^e(Y,S)$ such that the union
$\Lambda_0:=\bigcup_{\kappa\in \mathcal L_Y}\Lambda(S,\kappa)$ is countable.
Since $h(S)=0$, we can (and will) assume without loss of generality that
$h(S,\kappa)=0$ for  each $\kappa\in \mathcal L_Y$.
As $\Lambda_0$ is countable, the set $\Lambda_0\setminus\mathbb{Q}$ is also countable.
  Since  $\Lambda_{\rm irr}(T)=\emptyset$, there exists a subset $\mathcal L_X$
  of full $(\pi_X)_{\ast\ast}(\Bbb L)$-measure in $M^e(X,T)$ such that
  $\bigcap_{\kappa\in\mathcal L_X}\Lambda(T,\kappa)\cap (\Lambda_0\setminus\mathbb{Q})=\emptyset$.
Thus, there is a subset $\mathcal L$ of full $\Bbb L$-measure in $M^e(X\times Y,T\times S)$
  such that for every $\eta\in \mathcal L$ we have that $(\pi_X)_*(\eta)\in\mathcal L_X$ and
  $(\pi_Y)_*(\eta)\in\mathcal L_Y$.
  Hence,  $h(S,(\pi_Y)_*(\eta))=0$ and
  $$
  \Lambda_{\rm irr}(T,(\pi_X)_*(\eta))\cap\Lambda_{\rm irr}(S,(\pi_Y)_*(\eta))=\emptyset.
  $$
  Therefore, by Lemma~\ref{l:hutaxu5},    for each $\eta\in \mathcal L$,
  $$
  \int_{X\times Y} f\otimes( g-E(g|\mathcal K_{\text{rat}}(S,(\pi_Y)_*(\eta))\,d\eta=0.
  $$
  Applying Lemma~\ref{l:hutaxu2}, we obtain that
   $$
  \int_{X\times Y} f\otimes( g-E(g|\mathcal K_{\text{rat}}(S))\,d\rho=0.
  $$
It remains to note that
\begin{align*}
\int f\ot g\,d\rho&=\int f\ot \EE(g|\ck_{\rm rat}(S))\,d\rho+
\int f\ot(g- \EE(g|\ck_{\rm rat}(S)))\,d\rho\\
&=\int f\ot(g- \EE(g|\ck_{\rm rat}(S)))\,d\rho
\end{align*}
since the first integral vanishes (as the spectral measure of $f$ is continuous, while the spectral measure of $\EE(g|\ck_{\rm rat})$ is discrete).
 \end{proof}

\subsubsection{Proof of Theorem~\ref{t:hutaxu1}}
Recall now that Frantzikinakis and Host proved in \cite{Fr-Ho} that each logarithmic Furstenberg system $(X_{\mob},\kappa,T)$ ($T$ stands for the shift), $\kappa\in V^{\rm log}(\mob)$ is a factor of a system which has no irrational spectrum and a.e. ergodic component is of the form an
$\infty$-step pro-nil system times a Bernoulli automorphism. Moreover, Tao in \cite{Ta00} proved that the spectral measure of $\pi_0$ is Lebesgue for each logarithmic Furstenberg system of the M\"obius function. We now apply Lemma~\ref{l:hutaxu6} (with $f=\pi_0$) to obtain the short interval extension of the main result in \cite{Hu-Ta-Xu}:

\begin{proof}[Proof of Theorem~\ref{t:hutaxu1}]
(a) This argument is standard: if the sequence
$$
\Big(\frac1{\log N}\sum_{n<N}\frac1ng(S^ny)\mob(n)\Big)
$$
 does not go to zero, then for some subsequence $(N_k)$, we have $$\frac1{\log N_k}\sum_{n<N_k}\frac1n\delta_{(S^ny,T^n\mob)}\to\rho\in J((S,\nu),(T,\kappa)),$$ where $\nu\in V^{\rm log}(y)$,
$\kappa\in V^{\rm log}(\mob)$ and $\int g\ot\pi_0\,d\rho\neq0$ which is in contradiction with Lemma~\ref{l:hutaxu6}.\\
(b) According to \cite{Ab-Le-Ru}, \cite{Ab-Ku-Le-Ru} and  \cite[Lemma~1.8]{Ka-Le-Ri-Te}, we need to prove the logarithmic strong $\mob$-MOMO property of $(Y,S)$, that is, we need to prove that if $(b_k)$ is an increasing sequence of natural numbers satisfying
\beq\label{momo67}
\delta(\{b_k\colon k\geq1\})=0\eeq
and $(y_k)\subset Y$, then
\beq\label{momo68}
\lim_{K\to\infty} \frac1{\log b_K}\sum_{k<K}\Big|\sum_{b_k\leq n<b_{k+1}}\frac{\mob(n)}nf(S^ny_k)\Big|=0.\eeq
We now repeat the main arguments from \cite{Ka-Ku-Le-Ru}. Let $\mathbb{A}_3:=\{e^{2\pi ij/3}\colon j=0,1,2\}$, and set
$$
\underline{x}=(x_n),\;x_n=(S^{n-b_k}y_k,e_k)\text{ for }b\leq n<b_{k+1}\text{ and }x_n=x_0\text{ for }n<0,$$
where $x_0=(y_0,e_0)\in Y\times\mathbb{A}_3$ is arbitrary. Let $T$ stand for the shift on $X:=(Y\times\mathbb{A}_3)^{\Z}$. The $e_k$s are selected so that
$$
e_k\sum_{b_k\leq n<b_{k+1}}\frac{\mob(n)}nf(S^{n-b_k}y_k)\in\{0\}\cup\{z\in\C^\ast\colon {\rm arg}(z)\in [-\pi/3,\pi/3]\}.$$
So, we need to show that
$$
\frac1{\log b_K}\sum_{k<K}\sum_{b_k\leq n<b_{k+1}}e_kf(S^{n-b_k}y_k)\frac{\mob(n)}n\to 0\text{ when }K\to\infty.$$
By setting $F=f\ot id$ (i.e. $F(v,a)=a_0f(v_0)$), we need to prove
\beq\label{momo69}\frac1{\log b_K}\sum_{n<b_K} F(T^n\underline{x})\frac{\mob(n)}n=
\frac1{\log b_K}\sum_{k<K} \sum_{b_k\leq n<b_{k+1}}F(T^n\underline{x})\frac{\mob(n)}n\to 0.\eeq
Set $X_{\underline x}:=\ov{\{T^n\underline x\colon n\in\Z\}}$. We intend to show that (i) can be applied to $(X_{\underline x},T)$, so take $\nu$ a logarithmically visible measure in $(X_{\underline x},T)$. Assume first that $\underline x$ is logarithmically generic for $\nu$, that is, for some $(N_k)$, we have
\beq\label{momo70}
\frac1{\log N_s}\sum_{n<N_s}\frac1n\delta_{T^n\underline x}\to \nu.\eeq
Set $C:=\{v\in X_{\underline x}\colon (v_1,a_1)=(Sv_0,a_0)\}$. In view of \eqref{momo67},
$$
\Big(\frac1{\log N_s}\sum_{n<N_s}\frac1n\delta_{T^n\underline x}\Big)(C)=\frac1{\log N_s}\sum_{n<N_s; x_{n+1}=(S\times id)x_n}\frac1n \to 1$$
when $s\to\infty$. Since $C$ is closed, $\nu(C)=1$ and also $\nu(\bigcap_{n\in\Z}T^{-n}C)=1$, that is,
\beq\label{momo71}
\nu(\{(v_j,a_j))_{j\in\Z}\in X\colon (v_j,a_j)(S^jv_0,a_0)\text{ for all }j\in\Z\})=1.\eeq
Let $\nu^{(0)}$ be the projection of $\nu$ on the $0$-coordinate. Then, using~\eqref{momo71}, for $A\subset Y$, we have
$$
\nu^{(0)}(S^{-1}A)=\nu(Y^{\N}\times S^{-1}A\times Y^{\N})=\nu(Y^{\N}\times S^{-1}A\times A\times Y^{\N})=
\nu(Y^{\N}\times A\times Y^{\N}),$$
so $\nu^{(0)}$ is $S\times Id$-invariant. Since $Y\times\mathbb{A}_3$ consists of three $S\times id$-invariant sets $Y\times\{e^{2\pi ij/3}\}$, the measure $\nu^{(0)}$ is a (up to some obvious identification) convex combination of three $S$-invariant measures. Moreover, via \eqref{momo71}, the map
$$
(v_0,a_0)\mapsto ((S\times id)^j(v_0,a_0))_{j\in\Z}$$
settles a measure-theoretic isomorphism $(S\times Id,\nu^{(0)})$ and $(T,\nu)$, so the relative Kronecker factor of $(T,\nu)$ is the same as its Kronecker factor, and we can apply~(i).
To complete the proof we should consider an arbitrary point $\underline{x}'$ being logarithmically generic for $\nu$. In this situation the topological structure of $X_{\underline{x}}$ is decisive. For example, if $\underline{x}'=\lim_{r\to \infty} T^{n_r}\underline x$ with $n_r\to\infty$  then by writing $\underline{x}'=(v_j,a_j)_{j\in\Z}$ there can be at most one $j$ such that $T(v_j,a_j)\neq (v_{j+1},a_{j+1})$, see for example \cite{Ab-Ku-Le-Ru}.
\end{proof}

\begin{Cor}\label{c:stmomo}
Let  $(Y,S)$ be a zero entropy topological system whose all ergodic measures yield weakly mixing automorphisms.
Then $(Y,S)$ satisfies the logarithmic strong $\mob$-MOMO property (in particular, $(Y,S)\perp_{\rm log} \mob$).
\end{Cor}

\begin{Remark}\label{r:apropos} Note that if $(Y,S)$ does not satisfy the assumptions in Theorem~\ref{t:hutaxu1}, then for some $\nu\in M(X,T)$, we have $\ck(S,\nu)\subsetneq \ca_{{\rm {\bf DISP}}
_{\rm ec}}(S,\nu)$. By taking the latter factor, we obtain a system which is represented by an invariant measure in a universal model for the class DISP$_{\rm ec}$ and for which the logarithmic strong $\mob$-MOMO property remains unproven (and which is the main obstacle to prove the logarithmic Sarnak's conjecture for the class DISP$_{\rm ec}$).\end{Remark}

\subsubsection{Proof of Corollary~\ref{corollary:h}}
We now provide a direct proof of Corollary~\ref{corollary:h} using Section~\ref{s:lemmas}. Indeed, we come back to
Lemma~\ref{l:hutaxu6}, by considering as $S\in {\rm Aut}(Y,\cc,\nu)$ the Pinsker factor $(X_{\mob}/\Pi(\kappa),\kappa|_{\Pi(\kappa)}, S|_{\Pi(\kappa)})$ (we can apply the lemma since $\Pi(\kappa)$ being URE implies that the largest DISP$_{\rm ec}$-factor is also URE).
On the other hand, the automorphism $S|_{\Pi(\kappa)}$ is a factor of a certain $T$ as in that lemma. We can now consider the diagonal joining $\Delta$ of the automorphism $S|_{\Pi(\kappa)}$ with itself to obtain that
$$\int f\ot \ov{f}\,d\Delta=0$$
whenever the spectral measure $\sigma_{f}$ of $f$ has no atoms. But this means that $f=0$, whence $\Pi(\kappa)$ is precisely the Kronecker factor of $S|_{\Pi(\kappa)}$. As $\sigma_{\pi_0}$ has indeed no atoms \cite{Ma-Ra-Ta}, we obtained that $\pi_0\perp L^2(\Pi(\kappa))$, that is, so called Veech's condition is satisfied and the logarithmic Sarnak's conjecture holds \cite{Ka-Ku-Le-Ru}.

\begin{proof}[Proof of Remark~\ref{r:logSchar}]
We repeat word for word the proof of Corollary~\ref{corollary:h} just having the largest DISP$_{\rm ec}$-factor of $(S,\kappa)$ URE by assumption.
\end{proof}

\section{Concluding remarks and open problems}\label{s:open}

\begin{Question}\label{quest7}
For which characteristic classes $\mathcal F$, the product formula
$$\ca(\cf)(T\times S)=\ca(\cf)(T)\ot\ca(\cf)(S)$$
holds true?
\end{Question}

It works for ${\rm{\bf{ZE}}}$,  ${\rm{\bf{DISP}}}$ and ${\rm{\bf{DIST}}}$ for example
but it breaks down
already at the smallest (see \cite{Ka-Ku-Le-Ru}) nontrivial characteristic class ${\rm{\bf{ID}}}$ (note that ${\rm{\bf{ID}}}={\rm{\bf{ID}}}_{\rm ec}$. Indeed, if $T$ is ergodic then $\ca({\rm ID})(T)=\{\ast\}$, so $\ca({\rm ID})(T)\ot\ca({\rm ID})(T)$ is also trivial, while $\ca({\rm ID})(T\times T)$ is trivial if and only if  $T$ is weakly mixing.
We can also prove the following particular case.

\begin{Prop}\label{t:ms}Let $\cf_{\rm ec}$ be an  ec-characteristic class. Assume that $T\in{\rm Aut}\xbm$ is ergodic and consider the identity {\rm Id}$_Y$ in ${\rm Aut}\ycn$. Then, we have
$$
\ca(\cf_{\rm ec})(T\times {\rm Id}_Y)=\ca(\cf_{\rm ec})(T)\ot\cc.$$
\end{Prop}
\begin{proof} Let $\cd:=\ca(\cf_{\rm ec})(T\times {\rm Id}_Y)$. Note that $\cd\supset \ca(\cf_{\rm ec})(T)\ot\cc$, in particular, it contains $\{\emptyset, X\}\ot\cc$. It follows that we can apply Proposition~\ref{p:faktory}, and the factor $\cd$ is given (by \eqref{can3}) by a measurable choice
\beq\label{can3a}y\mapsto\cd_y\subset\cb\eeq
of factors of $T$ (i.e. of $\cb$). But \eqref{can3a} yields also the ergodic decomposition of $\cd$. Since $(T\times {\rm Id}_Y)|_{\cd}\in\cf_{\rm ec}$ and we consider the $\cf_{\rm ec}$-class, the factors $\cd_y$ must be $\cf$-factors of $T$. Hence, all of them are contained in $\ca(\cf_{\rm ec})(T)$, and therefore, $\cd\subset\ca(\cf_{\rm ec})(T)\ot \cc$.
\end{proof}

\begin{Remark} We do not know if Proposition~\ref{t:ms} works for any characteristic class. It certainly goes beyond ec-classes, as it works for ${\rm{\bf{DISP}}}$, but also it works for $(q_n)$-rigidity classes $\cf={\rm {\bf{RIG}}}((q_n))$~\footnote{This class consists of all automorphisms $R$ on $(Z,\cd,\kappa)$ such that $h\circ R^{q_n}\to h$ in $L^2(Z,\kappa)$.} as the following reasoning shows.
Let $F\in L^2(X\times Y,\mu\ot\nu)$ satisfy $F\circ(T\times Id)^{q_n}\to F$ (in $L^2$).  Then, for $F_y(x):=F(x,y)$ we have $F_y\circ T^{q_n}\to F_y$. It follows that when $F\in L^2(\ca(\cf)(T\times Id))$ then $F_y\in L^2(\ca(\cf)(T))$. The claim follows by the following classical theorem:
if $\ca\subset\cb\ot\cc$ is a sigma-field such that $\ca_y=\cb$ for a.e. $y\in Y$ and  if $\ca\supset \{\emptyset,X\}\ot\cc$ then $\ca=\cb\ot\cc$, see Proposition~\ref{2factors}.  \end{Remark}


\begin{Question}\label{quest0}
Has the spectral condition formulated in Proposition~\ref{theorem:c} a purely combinatorial reformulation for $\bfu$?\end{Question}

\begin{Question}\label{quest1}
Is the property~\eqref{Mproperty} a characterization of the piecewise Markov URE property?\end{Question}

\begin{Question}\label{quest2}To which extent the three notions: URE, Markov URE and piecewise Markov URE  are equivalent?\end{Question}

\begin{Question}\label{quest3}
It is not hard to see that the Bernoulli measures on the space $(K^\Bbb Z,S)$  is   $d$-bar separable.
 Indeed, they yield automorphisms that have a common extension, namely, Bernoulli automorphism with infinite entropy.
  In \cite{Ba-Ca-Kw-Op}, it was proved that  the Kolmogorov measures form a closed set.
Is the set of Kolmogorov measures separable in $\overline{d}$?
\end{Question}

\begin{Question}\label{quest3.5}
Has the class $\cf(A_{2,\alpha})$ a universal model for all (or some) irrational $\alpha$?
\end{Question}

\begin{Conj}\label{quest4} We conjecture that
$$\bigcap_{\alpha\notin\Q}\cf(A_{2,\alpha})={\rm {\bf{DISP}}}_{\rm ec}.
$$
\end{Conj}

\begin{Question}\label{quest6}
Are there universal model for the following classical characteristic classes: ${\rm {\bf{DIST}}}$, ${\rm {\bf{RIG}}}{((q_n))}$, $\mathcal{Z}_s$ for $s\geq2$ \cite{Ho-Kr}, $\cm({\rm WM}^\perp)$ \cite{Le-Pa}? \end{Question}

\begin{Question}\label{quest7}
Does the assertion of Corollary~\ref{c:Fsigma} hold for non-$\sigma$-finite  Borel subgroups of $\T$?
The problem seems to be open even in the important (in harmonic analysis) case of saturated subgroups.
We refer to \cite{Ho-Me-Pa} for definition, properties and examples of saturated subgroups.
For instance, given an arbitrary  singular measure $\kappa$ on $\Bbb T$,
the group of $\kappa$-quasi-invariance $H(\kappa):=\{z\in\Bbb T\colon \kappa*\delta_z\sim\kappa\}$  is  proper and saturated  \cite{Ho-Me-Pa}.
Does Corollary~\ref{c:Fsigma} hold for $H(\kappa)$?
\end{Question}

\begin{Question}\label{quest5} For  an arbitrary simple automorphism  $T$ \cite{Ju-Ru}, is there a universal model for the class $\cf(T)$?
What about an arbitrary ergodic  $T$ with zero entropy?
\end{Question}

\begin{Question}\label{q:ec} For which ergodic automorphisms $T$
 with zero entropy, we have the equality $\cf(T)=\cf(T)_{\rm ec}$?\end{Question}

\section{Appendix ``Separable subsets for $d$-bar metric''}
\begin{center} by T. Austin
\end{center}

In this note, the term `automorphism' refers to an invertible measure-preserving transformation on a standard Borel probability space, or to the resulting measure-preserving system as a whole.

Let $(K,d)$ be a nonempty compact metric space and let $S$ be the leftward shift on $K^\bbZ$.  Let $P$ be the space of $S$-invariant Borel probability measures on $K^\bbZ$, and by default give it the vague topology, which is compact and metrizable.  The subset $P_{\rm{e}}$ of ergodic elements of $P$ is G$_\delta$ in this topology, and hence Borel.

The space $P$ admits this generalization of Ornstein's ``d-bar'' metric:
\[\ol{d}(\nu,\nu') := \inf\Big\{\int d(y_0,y'_0)\ \la(\d y,\d y'):\ \la\ \hbox{a joining of}\ \nu\ \hbox{and}\ \nu'\Big\} \qquad (\nu,\nu' \in P).\]
By vague compactness, this infimum is always attained.  Furthermore, by taking suitable ergodic components, it is always attained by an ergodic joining if $\mu$ and $\nu$ are themselves ergodic.  If $K$ has at least two points, then $\ol{d}$ generates a topology on $P$ which is strictly stronger than the vague topology, and which is neither compact nor separable.  Other recent works using this metric include \cite{Ba-Ca-Kw-Op,Ba-La}.

Intuitively, the relationship between $\ol{d}$ and the vague topology on $P$ is similar to the relationship between the $\ell^\infty$-metric and the product topology on $[0,1]^\bbN$.

Although $\ol{d}$ is strictly stronger than the weak topology, we do have the following comparison.

\begin{Lemma}\label{lem:Borel}
If $E$ is a $\ol{d}$-closed and $\ol{d}$-separable subset of $P$, then it is K$_{\sigma\delta}$ and hence Borel for the vague topology.
\end{Lemma}

\begin{proof}
For any $\nu \in P$ and $\delta > 0$, let $C(\nu,\delta)$ be the set of all shift-invariant probability measures $\la$ on $A^\bbZ \times A^\bbZ$ which satisfy
\[\int d(y_0,y'_0)\ \la(\d y,\d y') \le \delta \qquad \hbox{and} \qquad (\hbox{first coord. proj. of $\la$}) = \nu.\]
This is compact for the vague topology.  If $\psi\colon K^\bbZ \times K^\bbZ \to K^\bbZ$ is the second coordinate projection, then pushing forward by $\psi$ acts continuously on probability measures for the vague topology, and therefore the set
\[\{\psi_\ast\la:\ \la \in C(\nu,\delta)\}\]
is also compact in the vague topology.  This set is precisely the $\ol{d}$-closed ball $\ol{B}_{\ol{d}}(\nu,\delta)$.

Since $E$ is $\ol{d}$-separable, it has a countable $\ol{d}$-dense subset $J$.  Then, since $E$ is also $\ol{d}$-closed, we have
\[E = \bigcap_{k=1}^\infty \bigcup_{\nu \in J}\ol{B}_{\ol{d}}(\nu,1/k).\]
So this is K$_{\sigma\delta}$ for the vague topology.
\end{proof}

Separable subsets of $(P,\ol{d})$ arise naturally in the following way.

\begin{Lemma}\label{lem:sep}
Fix an automorphism $\boldsymbol{X} = (X,\mu,T)$, and let
\[E := \{\nu \in P:\ (K^\bbZ,\nu,S)\ \hbox{is a factor of}\ \boldsymbol{X}\}.\]
Then $E$ is $\ol{d}$-separable.
\end{Lemma}

\begin{proof}
Let $F$ be the set of all measurable maps from $X$ to $K$ modulo agreement $\mu$-a.e., and give $F$ the topology of convergence in probability.  This topology is generated by the metric
\[d_F(\pi,\pi') := \int d(\pi(x),\pi'(x))\ \mu(\d x)  \qquad (\pi,\pi' \in F).\]
Since $X$ is standard and $K$ is compact, the metric $d_F$ is complete and separable.

For each $\pi \in F$, let $\pi^\bbZ:= (\dots,\pi\circ T^{-1},\pi,\pi\circ T,\dots)$, regarded as an equivariant Borel map from $(X,T)$  to $(K^\bbZ,S)$.  Then $E$ is the image of $F$ under the mapping $\pi \mapsto \pi^\bbZ_\ast \mu$.  This mapping is a contraction from $d_F$ to $\ol{d}$.  Indeed, if $\pi,\pi' \in F$, then the joint distribution of $\pi^\bbZ$ and $(\pi')^\bbZ$ is a joining $\la$ of $\pi_\ast^\bbZ \mu$ and $(\pi')^\bbZ_\ast \mu$ which satisfies
\[\int d(y_0,y'_0)\ \la(\d y,\d y') = \int d(\pi(x),\pi'(x))\ \mu(\d x) .\]
So $E$ is a continuous image of a separable space, hence also separable.
\end{proof}

The implication from the previous lemma can be reversed in a way that respects ergodicity.

\begin{Prop}\label{prop:sep}
If $E\subset P_{\rm{e}}$ and $E$ is $\ol{d}$-separable, then there is an ergodic automorphism $\boldsymbol{X} = (X,\mu,T)$ such that $(K^\bbZ,\nu,S)$ is a factor of $\boldsymbol{X}$ for every $\nu \in E$.
\end{Prop}

\begin{proof}
\emph{Step 1.}\quad Let $\bbN^\ast = \{\emptyset\}\cup \bbN \cup \bbN^2\cup \cdots$ be the countable set of finite-length strings of positive integers.  We regard it as the vertex set of a tree in which the children of $\a = (\a_1,\dots,\a_n)$ are all possible extensions $(\a_1,\dots,\a_n,\b)$ for $\b \in \bbN$.  We sometimes abbreviate this child to $\a\b$. We include the empty string $\emptyset$ as an element of $\bbN^\ast$, so its children are the elements of $\bbN$.

In addition, pick an enumeration $\bbN^\ast = \{\a^{(1)},\a^{(2)},\dots\}$ in which parents always appear before their children.  In particular, it must be the case that $\a^{(1)} = \emptyset$.  For example, such an enumeration can be obtained by defining the `weight' of a string in $\bbN^\ast$ to be the sum of its entries; then writing $\emptyset$, followed by a list of all strings of weight $1$, followed by all strings of weight $2$, and so on.

We may clearly assume that $E$ is nonempty.  Then, since it is $\ol{d}$-separable, we can choose a mapping
\[\bbN^\ast \to E:\a \mapsto \nu_\a\]
with the following two properties:
\begin{enumerate}
\item[i)] (Cauchy estimate) If $\a  \in \bbN^n$ with $n\ge 1$ and $\b \in \bbN$, then
\[\ol{d}(\nu_\a,\nu_{\a\b}) < 2^{-n}.\]
Let $\la_{\a\b}$ be an ergodic joining of $\nu_\a$ and $\nu_{\a\b}$ which witnesses this inequality.
\item[ii)] (Density) For any $\nu \in E$, there exists $(\a_1,\a_2,\dots) \in \bbN^\bbN$ such that
\[\ol{d}(\nu_{(\a_1,\dots,\a_n)},\nu) \to 0 \qquad \hbox{as}\ n\to\infty.\]
\end{enumerate}
Here is one way to construct such a mapping.  Let $J$ be a countable $\ol{d}$-dense subset of $E$, and now choose the measures $\nu_\a$ recursively as follows.  First pick $\nu_\emptyset$ arbitrarily.  Next, let $\nu_1$, $\nu_2$, \dots be a sequence that includes every member of $J$, with repetitions if necessary.  Beyond that, once we have already chosen $\nu_\a$ for some $\a \in \bbN^n$ with $n\ge 1$, let $\nu_{\a 1}$, $\nu_{\a 2}$, \dots be a sequence that includes every member of $J\cap B_{\ol{d}}(\nu_\a,2^{-n})$, with repetitions if necessary.  Once this recursion is complete, the mapping obtained has property (i) by construction, and it also has property (ii) because every $\nu \in E$ is the limit of a sequence in $J$ that converges  as fast as we wish according to $\ol{d}$.

\vspace{7pt}

\emph{Step 2.}\quad Now let $(X_\a,\mu_\a,T_\a)$ be a copy of $(K^\bbZ,\nu_\a,S)$ for every $\a \in \bbN^\ast$.  Let
\[X := \prod_{\a \in \bbN^\ast} X_\a \qquad \hbox{and} \qquad  T := \bigtimes_{\a \in \bbN^\ast} T_\a.\]
Since $\bbN^\ast$ is countable, $X$ is still standard.

We construct $\mu$ on $X$ as a certain joining of all the measures $\mu_\a$.  We first construct recursively a joining $\mu^{(n)}$ of the finite collection $\mu_{\a^{(1)}}$, \dots, $\mu_{\a^{(n)}}$ for each $n$:
\begin{itemize}
\item To begin, let $\mu^{(1)} := \mu_\emptyset$.
\item To continue, suppose that $\mu^{(n)}$ has already been constructed, and consider $\a^{(n+1)}$.  Its parent in $\bbN^\ast$ equals $\a^{(i)}$ for some $i \in \{1,2,\dots,n\}$, and we can write $\a^{(n+1)} = \a^{(i)}\b$ for some $\b \in \bbN$.

We already have a joining $\mu^{(n)}$ of $X_{\a^{(1)}}$, \dots, $X_{\a^{(n)}}$, and property (i) above gives an ergodic joining $\la_{\a^{(i)}\b}$ of $X_{\a^{(i)}}$ and $X_{\a^{(n+1)}}$.   Form the relative product of these two joinings over the common coordinate factor $X_{\a^{(i)}}$, and let $\mu^{(n+1)}$ be a typical ergodic component of it.  This continues the recursion.
\end{itemize}

By construction, this sequence of measures on the spaces $X_{\a^{(1)}} \times \cdots \times X_{\a^{(n)}}$ is consistent under coordinate projection.  Therefore these measures are all the images of a unique measure $\mu$ on the infinite product space $X$ by the Kolmogorov extension theorem. Since $\mu$ is an inverse limit of ergodic invariant measures, it is still invariant and ergodic.

\vspace{7pt}

\emph{Step 3.}\quad To finish, consider any $\nu \in E$, and pick an infinite sequence $(\a_1,\a_2,\dots) \in \bbN^\bbN$ as in property (ii) from Step 1.

For each $i$, there is a coordinate projection
\[\pi^{(i)}\colon X\to X_{(\a_1,\dots,\a_i)}.\]
By our construction in Step 2, the joint distribution of $\pi^{(i)}$ and $\pi^{(i+1)}$ under $\mu$ is equal to the ergodic joining $\la_{(\a_1,\dots,\a_{i+1})}$ given by property (i) from Step 1.  Therefore
\[\int d(\pi^{(i)}_0(x),\pi^{(i+1)}_0(x))\ \mu(\d x) = \int d(y_0,y_0')\ \la_{(\a_1,\dots,\a_{i+1})}(\d y,\d y') < 2^{-i}.\]
By Markov's inequality, it follows that the sequence $(\pi_0^{(i)}(x))_{i\ge 1}$ is Cauchy for $\mu$-a.e. $x$, and so the maps $\pi^{(i)}$ converge $\mu$-a.e. to a limit map $\pi\colon X \to K^\bbZ$.  The distribution of this limit map is
\[\lim_{i\to\infty} \pi^{(i)}_\ast \mu = \lim_{i\to\infty} \nu_{(\a_1,\dots,\a_i)} = \nu,\]
using again property (ii) from Step 1.  So $\pi$ is a factor map from $(X,\mu,T)$ to $(K^\bbZ,\nu,S)$.
\end{proof}

We now use Proposition~\ref{prop:sep} to characterize the following property.

\begin{Def}
Let $\boldsymbol{X}$ and $\boldsymbol{Y}$ be automorphisms and assume that
$\boldsymbol{X}$ is ergodic.  We say $\boldsymbol Y$ is \textbf{enclosed} by $\boldsymbol X$ if $\boldsymbol Y$ is a factor of ${\boldsymbol W}\times \boldsymbol X$ for some trivial automorphism $\boldsymbol W$.  An automorphism is \textbf{enclosed} if it is enclosed by some ergodic automorphism.
\end{Def}

In this definition, we can always take the space of $\boldsymbol W$ to be $[0,1]$ with Lebesgue measure, but we do not need this restriction.

Let $(K^\bbZ,\nu,S)$ be a shift system for a compact metric space $K$ as before.  Up to isomorphism, any automorphism can be put into this form.  One version of the ergodic decomposition theorem asserts that there is a unique probability measure $\theta_\nu$ on $P_{\rm{e}}$ such that
\[\nu = \int_{P_{\rm{e}}}\omega \ \theta_\nu(\d \omega).\]

\begin{Cor}\label{cor:Tim}
The shift system $(K^\bbZ,\nu,S)$ is enclosed if and only if there is a $\ol{d}$-separable and $\ol{d}$-closed subset $E$ of $P_{\rm{e}}$ such that $\theta_\nu E = 1$.
\end{Cor}

\begin{proof}
\emph{($\Rightarrow$).}\quad Let $(X,\mu,T)$ be an ergodic automorphism and $\pi$ be a factor map from $\boldsymbol{W}\times \boldsymbol{X}$ to $(K^\bbZ,\nu,S)$.  Let $\boldsymbol{W} := (W,\kappa,\mathrm{id})$ After modifying $\pi$ on a negligible set if necessary, we may assume that
\[\pi(w,Tx) = S\pi(w,x) \qquad \hbox{for strictly all}\ (w,x) \in W\times X.\]
It follows that
\[\nu = \pi_\ast (\kappa\times \mu) = \int_W \pi(w,\cdot)_\ast \mu\ \kappa(\d w).\]
Also, $\pi(w,\cdot)_\ast \mu$ is an element of $P_{\rm{e}}$ for each $w$, because it is a factor of the ergodic system $\boldsymbol{X}$.  In the formula above, the right-hand side is an integral of a measurable family of ergodic invariant measures, so by uniqueness it must agree with the ergodic decomposition of $\nu$.  Therefore the measure $\theta_\nu$ is supported on the closure of the separable set $E$ from Lemma~\ref{lem:sep}, which is Borel for the vague topology by Lemma~\ref{lem:Borel}.

\vspace{7pt}

\emph{($\Leftarrow$).}\quad Let $E \subset P_{\rm{e}}$ be $\ol{d}$-separable, $\ol{d}$-closed, and satisfy $\theta_\nu E = 1$.  By Proposition~\ref{prop:sep}, there is an ergodic automorphism $(X,\mu,T)$ such that $(K^\bbZ,\omega,S)$ is a factor of $(X,\mu,T)$ for every $\omega \in E$.  As in the proof of Lemma~\ref{lem:sep}, let $F$ be the space of all measurable maps from $X$ to $K$ modulo agreement $\mu$-a.e., and give $F$ the topology of convergence in probability.  This topology is complete and separable.  Now consider the set
\[\t{E} := \{(\omega,\pi) \in E\times F:\ \pi^\bbZ_\ast \mu = \omega\}.\]
This is a Borel subset of the standard measurable space $E\times F$, and by our choice of $(X,\mu,T)$ the projection map from $\t{E}$ to $E$ is surjective.  Therefore the measurable selector theorem provides a Borel subset $W$ of $E$ with $\theta_\nu W = 1$ and a Borel map $W \to F\colon \omega \mapsto \pi_\omega$ such that $(\pi_\omega)^\bbZ_\ast\mu = \omega$ for every $\omega \in W$.  Now standard constructions give a single Borel map $\pi\colon W\times X\to K$ such that $\pi_\omega$ agrees a.e. with $\pi(\omega,\cdot)$ for every $\omega$ (see, for instance, the proof of~\cite[Theorem 1]{Mac}).  

Let $\kappa$ be given by restricting $\theta_\nu$ to $W$.  To finish, observe that $\pi^\bbZ\colon W\times X\to K^\bbZ$ is Borel, intertwines $\mathrm{id}\times T$ with $S$, and satisfies
\[\pi^\bbZ_\ast (\kappa \times \mu) = \int_W \pi^\bbZ(\omega,\cdot)_\ast\mu\ \kappa(\d \omega) = \int_W (\pi_\omega)^\bbZ_\ast \mu\ \theta_\nu(\d \omega) = \int_W \omega\ \theta_\nu(\d\omega) = \nu.\]
So $(K^\bbZ,\nu,S)$ is a factor of $(W,\kappa,\mathrm{id})\times \boldsymbol{X}$.
\end{proof}


%
%
%
%
%

\section{Complements}

\subsection{Construction of $p$--Kronecker sets}

Let $p\ge2$ be a fixed prime and let $\mathbb T=\mathbb R/\mathbb Z$.
We consider the characters
\[
\chi_k(x)=e^{2\pi i p^k x},\qquad k\ge0.
\]

A compact set $K\subset\mathbb T$ is called \emph{$p$--Kronecker} if for every
$f\in C(K,\mathbb S^1)$ and every $\varepsilon>0$ there exists $k\ge0$ such that
\[
\sup_{x\in K} |\chi_k(x)-f(x)|<\varepsilon .
\]

We need some preparatory facts.

\begin{Lemma}[Arc lifting for $x\mapsto p^k x$]\label{lem:arcs}
Let $U_1,\dots,U_m\subset\mathbb T$ be nonempty open arcs and
$V_1,\dots,V_m\subset\mathbb T$ be nonempty open arcs.
Then there exists $k\ge1$ and, for each $j$, two disjoint nonempty open sub-arcs
\[
U_{j}^{(0)},U_{j}^{(1)}\subset U_j
\]
such that
\[
p^k(U_{j}^{(i)})\subset V_j \qquad (i=0,1,\; j=1,\dots,m).
\]
\end{Lemma}
\begin{proof}
For fixed $k$, the map $T_k(x)=p^k x\pmod1$ is a $p^k$--fold covering of $\mathbb T$.
Hence $T_k^{-1}(V_j)$ is the disjoint union of $p^k$ open arcs, each of length
$|V_j|/p^k$, distributed uniformly around the circle.

Choose $k$ so large that
\[
p^k |U_j| > 2 \quad\text{for all } j .
\]
Then $U_j$ intersects at least two distinct components of $T_k^{-1}(V_j)$.
Choosing these components as $U_{j}^{(0)}$ and $U_{j}^{(1)}$ gives the claim.
\end{proof}

Fix once and for all a sequence $(f_n)_{n\ge1}\subset C(\mathbb T,\mathbb S^1)$
which is dense in $C(\mathbb T,\mathbb S^1)$ with respect to the supremum norm  (this is our choice of test functions).
Fix also a sequence $\varepsilon_n\downarrow0$.

We construct inductively compact sets
\[
K_0\supset K_1\supset K_2\supset\cdots
\]
such that:
\begin{itemize}
\item each $K_n$ is a finite union of pairwise disjoint closed arcs,
\item each arc in $K_n$ has exactly two descendants in $K_{n+1}$,
\item for each $n$ there exists $k_n$ with
\[
\sup_{x\in K_n}|\chi_{k_n}(x)-f_n(x)|<\varepsilon_n .
\]
\end{itemize}

\medskip
\noindent\textbf{Step $0$.}
Let $K_0=\mathbb T$, written as a single closed arc.

\medskip
\noindent\textbf{Induction step.}
Assume that
\[
K_{n-1}=\bigcup_{\omega\in\{0,1\}^{n-1}} I_\omega,
\]
where the $I_\omega$ are pairwise disjoint closed arcs.

For each $\omega$, choose a point $x_\omega\in\mathrm{int}(I_\omega)$ and define
the target arc
\[
V_\omega:=\{z\in\mathbb T:\ |z-f_n(x_\omega)|<\varepsilon_n\}.
\]
Each $V_\omega$ is a nonempty open arc.

Apply Lemma~\ref{lem:arcs} to the family
\[
U_\omega:=\mathrm{int}(I_\omega),\qquad V_\omega,
\]
to obtain an integer $k_n$ and, for each $\omega$, two disjoint open sub-arcs
\[
U_{\omega}^{(0)},U_{\omega}^{(1)}\subset U_\omega
\]
such that
\[
p^{k_n}(U_{\omega}^{(i)})\subset V_\omega \qquad (i=0,1).
\]

Define
\[
I_{\omega i}:=\overline{U_{\omega}^{(i)}} \qquad (i=0,1),
\]
and set
\[
K_n:=\bigcup_{\omega\in\{0,1\}^{n-1}} (I_{\omega0}\cup I_{\omega1}).
\]
By construction, the arcs in $K_n$ are pairwise disjoint and each has diameter
tending to zero as $n\to\infty$.

Define
\[
K:=\bigcap_{n\ge0} K_n .
\]
Since each arc has two descendants and diameters  tend to zero, $K$ is a perfect,
totally disconnected, uncountable compact set (a Cantor set).

\subsubsection{Verification of the $p$--Kronecker property}

Fix $n$ and $x\in K$. Then $x\in I_\omega$ for a unique
$\omega\in\{0,1\}^{n-1}$, hence
\[
p^{k_n}x\in V_\omega,
\]
which implies
\[
|\chi_{k_n}(x)-f_n(x_\omega)|<\varepsilon_n .
\]
By continuity of $f_n$ and smallness of $I_\omega$,
\[
|f_n(x)-f_n(x_\omega)|<\varepsilon_n
\]
for all $x\in I_\omega$, hence
\[
\sup_{x\in K}|\chi_{k_n}(x)-f_n(x)|<2\varepsilon_n .
\tag{$\ast$}
\]

Now let $f\in C(K,\mathbb S^1)$ and $\varepsilon>0$.
Extend $f$ to $\tilde f\in C(\mathbb T,\mathbb S^1)$ (Tietze extension theorem).
Choose $n$ such that
\[
\|\tilde f-f_n\|_\infty<\varepsilon/3
\quad\text{and}\quad
2\varepsilon_n<\varepsilon/3 .
\]
Then by $(\ast)$,
\[
\sup_{x\in K}|\chi_{k_n}(x)-f(x)|
\le
\sup_{x\in K}|\chi_{k_n}(x)-f_n(x)|
+\|f_n-f\|_\infty
<\varepsilon .
\]
Thus $K$ is $p$--Kronecker.
\qed

\subsection{GHK seminorms and atoms of spectral measures}\label{s:ghk1}
The material below is well-known (we enclose it for sake of completeness).

For $\lambda\in\T$ define
\[
A_H(\lambda,z):=\frac1H\sum_{h=0}^{H-1}(\lambda\overline{z})^h.
\]
Then, by the boundedness of these averages and the pointwise convergence  $A_H(\lambda,z)\to \raz_{\{z=\lambda\}}$ as $H\to\infty$, we immediately obtain:
For any finite measure $\sigma$ on $\T$ and any $\lambda\in\T$,
\beq
\label{lem:cesaro-atom}
\lim_{H\to\infty}\int_{\T} A_H(\lambda,z)\, d\sigma(z)=\sigma(\{\lambda\}).
\eeq


Recall that we are given $\bfu\colon \N\to\mathbb{D}$ and $\kappa\in V(\bfu)$ and we consider the subshift $(X_{\bfu},\kappa, S)$. Let $U$ denote the corresponding Koopman operator: $f\mapsto f\circ S$. Recall that for $f\in L^2(X_{\bfu},\kappa)$ one defines
\begin{Def}[$u^1$ seminorm]
\begin{equation}\label{eq:u1-f}
\|f\|_{u^1}^2:=\lim_{H\to\infty}\frac1H\sum_{h=0}^{H-1}\int_X (U^h f)\cdot \overline{f}\,d\kappa
=\lim_{H\to\infty}\frac1H\sum_{h=0}^{H-1}\langle U^h f,f\rangle_{L^2(\kappa)}.
\end{equation}
\end{Def}


\begin{Prop}\label{prop:a}
In the system $(X_{\bfu},\kappa,S)$, one has
\[
\|\pi_0\|_{u^1}^2=\sigma_{\pi_0}(\{1\}).
\]
\end{Prop}

\begin{proof}
Let $f:=\pi_0$. Using \eqref{eq:u1-f}, we obtain
\[
\|f\|_{u^1}^2
=\lim_{H\to\infty}\frac1H\sum_{h=0}^{H-1}\int_{\T} z^h\,d\sigma_f(z)
=\lim_{H\to\infty}\int_{\T}\left(\frac1H\sum_{h=0}^{H-1} z^h\right)d\sigma_f(z).
\]
By \eqref{lem:cesaro-atom} with $\lambda=1$, this equals $\sigma_f(\{1\})$, hence $\|\pi_0\|_{u^1}^2=\sigma_{\pi_0}(\{1\})$.
\end{proof}


\begin{Def}[Modulated $u^1$-seminorm]
For $f\in L^2(X,\kappa)$ and $\lambda\in\T$, define
\begin{equation}\label{eq:u1-lambda}
\|f\|_{u^1(\lambda)}^2:=\lim_{H\to\infty}\frac1H\sum_{h=0}^{H-1}\int_X \lambda^h (U^h f)\cdot \overline{f}\,d\kappa,
\end{equation}
whenever the limit exists.
\end{Def}

\begin{Prop}\label{prop:b}
For every $\lambda\in\T$,
\[
\|\pi_0\|_{u^1(\lambda)}^2=\sigma_{\pi_0}(\{\lambda\}).
\]
In particular, in the symbolic model this quantity corresponds to the ``$u^1$-size'' of the modulated sequence $n\mapsto u(n)\lambda^n$.
\end{Prop}

\begin{proof}
Let $f:=\pi_0$. Using \eqref{eq:u1-lambda},
\[
\|f\|_{u^1(\lambda)}^2
=\lim_{H\to\infty}\frac1H\sum_{h=0}^{H-1}\int_{\T}\lambda^h z^h\,d\sigma_f(z)
=\lim_{H\to\infty}\int_{\T}\left(\frac1H\sum_{h=0}^{H-1}(\lambda\overline{z})^h\right)\,d\sigma_f(z).
\]
Now,~\eqref{lem:cesaro-atom} gives $\|f\|_{u^1(\lambda)}^2=\sigma_f(\{\lambda\})$.
\end{proof}

\subsubsection{The $u^2$-seminorm and ergodic decomposition}

Let $\kappa=\int_\Gamma \kappa_\gamma\, d\PP(\gamma)$ be the ergodic decomposition of $\kappa$.
For $\PP$-a.e.\ $\gamma$ the system $(X_{\bfu},\kappa_\gamma,S)$ is ergodic.

\begin{Def}[$u^2$-seminorm]
For $f\in L^\infty(X_{\bfu},\kappa)$ define
\begin{equation}\label{eq:u2-def}
\|f\|_{u^2}^{4}
:=
\lim_{H\to\infty}\frac1H\sum_{h=0}^{H-1}
\bigl\| (U^h f)\cdot \overline{f}\bigr\|_{u^1}^{2},
\end{equation}
whenever the limit exists.
\end{Def}

For each $\gamma$, let $\sigma_{f,\gamma}$ be the spectral measure of $f$ in $L^2(X_{\bfu},\kappa_\gamma)$.  The following is classical.

\begin{Lemma}[Wiener's lemma on an ergodic component]\label{lem:wiener-fiber}
If $\kappa_\gamma$ is ergodic and $f\in L^2(\kappa_\gamma)$, then
\[
\lim_{H\to\infty}\frac1H\sum_{h=0}^{H-1}\left|\int f\cdot \overline{f\circ S^h}\,d\kappa_\gamma\right|^2
=\sum_{\lambda\in\T}\sigma_{f,\gamma}(\{\lambda\})^2,
\]
where the sum runs over all atoms of $\sigma_{f,\gamma}$.
\end{Lemma}


\begin{Prop}\label{thm:c}
For $f\in L^\infty(X_{\bfu},\kappa)$ one has
\[
\|f\|_{u^2}^{4}
=
\int_\Gamma \left(\sum_{\lambda\in\T}\sigma_{f,\gamma}(\{\lambda\})^2\right)\, d\PP(\gamma).
\]
In particular, for $f=\pi_0$,
\[
\|\pi_0\|_{u^2}^{4}
=
\int_\Gamma \left(\sum_{\lambda\in\T}\sigma_{\pi_0,\gamma}(\{\lambda\})^2\right)\, d\PP(\gamma).
\]
\end{Prop}

\begin{proof}
Fix $h\ge 0$ and set $g_h:=(U^h f)\cdot \overline{f}$.
By \eqref{eq:u2-def},
\[
\|f\|_{u^2}^4=\lim_{H\to\infty}\frac1H\sum_{h=0}^{H-1}\|g_h\|_{u^1}^2.
\]
On each ergodic component $(X_{\bfu},\kappa_\gamma,S)$, applying the definition \eqref{eq:u1-f} to $g_h$, gives
\[
\|g_h\|_{u^1(\kappa_\gamma)}^2=\left|\int g_h\,d\kappa_\gamma\right|^2
=\left|\int f\cdot \overline{f\circ S^h}\,d\kappa_\gamma\right|^2.
\]
Integrating over $\gamma$ and exchanging Ces\`aro averages with the integral, yields
\[
\|f\|_{u^2}^4
=\int_\Gamma \left(\lim_{H\to\infty}\frac1H\sum_{h=0}^{H-1}\left|\int f\cdot \overline{f\circ S^h}\,d\kappa_\gamma\right|^2\right)\, d\PP(\gamma).
\]
Now, apply Lemma~\ref{lem:wiener-fiber} for $\PP
$-a.e.\ $\gamma$ to obtain the claim.
\end{proof}

\subsubsection{A Borel subgroup $\Lambda\subset\T$ and the $\Lambda$--Kronecker factor}

Let $\Lambda\subset\T$ be a Borel subgroup. On each ergodic component $(X_{\bfu},\kappa_\gamma,S)$, let $\mathcal{K}^\Lambda_\gamma$
denote the maximal factor with discrete spectrum whose eigenvalue group is contained in $\Lambda$
(equivalently: its $L^2$-space is the closed linear span of eigenfunctions with eigenvalues in $\Lambda$).
Let $\mathcal{K}^\Lambda_{\rm rel}$ denote the corresponding relative $\Lambda$-Kronecker  factor,  see Lemma~\ref{kapitan}.

\begin{Prop}\label{prop:lambda-orth}
For $f=\pi_0\in L^2(X_{\bfu},\kappa)$ the following are equivalent:
\begin{enumerate}
\item[\textup{(i)}] $\pi_0\perp L^2\bigl(\mathcal{K}^\Lambda_{\rm rel}\bigr)$.
\item[\textup{(ii)}] $\displaystyle \int_\Gamma\left(\sum_{\lambda\in\Lambda}
    \sigma_{\pi_0,\gamma}(\{\lambda\})^2\right)\,d\PP(\gamma)=0$.
\end{enumerate}
\end{Prop}

\begin{proof}
Fix an ergodic component $\gamma$. In the ergodic system $(X_{\bfu},\kappa_\gamma,S)$, the factor $\mathcal{K}^\Lambda_\gamma$
is the closed span of eigenfunctions whose eigenvalues lie in $\Lambda$.
If $\sigma_{\pi_0,\gamma}(\{\lambda\})>0$ for some $\lambda\in\Lambda$, then $\pi_0$ has a nontrivial spectral atom at $\lambda$,
hence a nonzero projection onto the corresponding eigenspace, and therefore $\pi_0\not\perp L^2(\mathcal{K}^\Lambda_\gamma)$.
Conversely, if $\sigma_{\pi_0,\gamma}(\{\lambda\})=0$ for all $\lambda\in\Lambda$, then the spectral measure of $\pi_0$ has no point mass on $\Lambda$,
so $\pi_0$ is orthogonal to the span of all eigenfunctions with eigenvalues in $\Lambda$, i.e.\ $\pi_0\perp L^2(\mathcal{K}^\Lambda_\gamma)$.

Since the terms are nonnegative,
\[
\sum_{\lambda\in\Lambda}\sigma_{\pi_0,\gamma}(\{\lambda\})^2=0
\quad\Longleftrightarrow\quad
\sigma_{\pi_0,\gamma}(\{\lambda\})=0\ \text{ for all }\lambda\in\Lambda.
\]
Thus, condition (ii) is equivalent to $\pi_0\perp L^2(\mathcal{K}^\Lambda_\gamma)$ for $\PP$-a.e.\ $\gamma$.

Finally, by construction, $\mathcal{K}^\Lambda_{\rm rel}$ is the direct integral of the fibers $\mathcal{K}^\Lambda_\gamma$,
so $\pi_0\perp L^2(\mathcal{K}_{\rm rel}^\Lambda)$ holds iff $\pi_0\perp L^2(\mathcal{K}^\Lambda_\gamma)$ for $\PP$-a.e.\ $\gamma$.
\end{proof}
We recall that by Corollary~\ref{rabit} the factor $\mathcal{K}^\Lambda_{\rm rel}$ is equal to
$\mathcal{A}_{{\rm {\bf{DISP}}(\Lambda)_{ec}}}(S,\kappa)$. Therefore, condition (ii) (applied to all Furstenberg systems $\kappa\in V(\bfu)$) in the above proposition is equivalent to the Veech condition (for $\bfu$).

\vspace{5ex}

\noindent
School of Math. Sciences,\\
 Tel Aviv University,\\
  69978 Tel Aviv, Israel

\vspace{5ex}

\noindent
Faculty of Mathematics and Computer Science\\
Nicolaus Copernicus University, \\
Chopin street 12/18, 87-100 Toru\'n, Poland

\vspace{1ex}
and

\vspace{1ex}\noindent
B. Verkin Institute for Low Temperature Physics and Engineering\\
of the National Academy of Sciences of Ukraine\\
47 Nauky Ave., Kharkiv, 61103, Ukraine \\

\noindent
Faculty of Mathematics and Computer Science\\
Nicolaus Copernicus University, \\
Chopin street 12/18, 87-100 Toru\'n, Poland\\

\noindent
aaro@tau.ac.il,\\alexandre.danilenko@gmail.com,\\joasiak@mat.umk.pl, \\mlem@mat.umk.pl,


\addcontentsline{toc}{section}{References}
\begin{thebibliography}{99}







\bibitem{Aa-Na}J.\ Aaronson, M.\ Nadkarni, {\em $L^\infty$-Eigenvalues and $L^2$-Spectra of Non-Singular Transformations}, Proc. London Math. Soc.\ s3-55 (1987), 538-570.
\bibitem{Ab-Le-Ru}H.\ El Abdalaoui,  M.\ Lema\'nczyk, T.\ de la Rue,  {\em Automorphisms with quasi-discrete spectrum, multiplicative functions and average orthogonality along short intervals}, International Math.\ Res.\ Notices {\bf 14} (2017), 4350-4368.

\bibitem{Ab-Ku-Le-Ru}H. El Abdalaoui, J.\ Ku\l aga-Przymus, M. Lema\'nczyk, T. de la Rue,  {\em M\"obius disjointness for models of an ergodic system and beyond},  Israel J.\ Math.\ {\bf 228}  (2018), 707-751.

\bibitem{Au-Mo}L.~Auslander, C.C.\ Moore,  {\em Unitary representations of solvable Lie groups},
Memoirs of Amer.\ Math.\ Soc.\ 62 (1966).

\bibitem{Au-Le} T.\ Austin, M.\ Lema\'nczyk, {\em Relatively finite measure-preserving extensions and lifting multipliers by Rokhlin cocycles},
J. Fixed Point Theory Appl. 6 (2009),  115–131.


\bibitem{Ba-Ca-Kw-Op} S.\ Babel, M.E.\ Can, D.\ Kwietniak, P.\ Oprocha, {\em  Spectrum of invariant
measures via generic points},  arXiv:2510.19711

\bibitem{Ba-La} S.\ Babel, M.\ L\c{a}cka, {\em On the closedness of ergodic measures in a characteristic class}, arxiv:2510.26564

\bibitem{Be-Go-Ru} P.\ Berk, M. G\'orska, T.\ de la Rue, {\em Joining properties of automorphisms disjoint with all ergodic systems}, Ergodic Theory Dynam.\ Systems 45 (2025), 1988--2002.

\bibitem{Bo-Sa-Zi}J.\ Bourgain, P.\ Sarnak, T.\ Ziegler, {\em Disjointness of Moebius from horocycle flows},
in: From Fourier Analysis and Number Theory to Radon Transforms and Geometry,
Dev. Math. 28, Springer, New York, 2013, 67–83.





\bibitem{Co-Do-Se} J.-P.\ Conze, T.\ Downarowicz, J.\ Serafin, {\em Correlation of sequences and of
measures, generic points for joinings and ergodicity of certain cocycles},
Trans. Amer. Math. Soc. 369 (2017), 3421--3441.



\bibitem{Dan} A. I. Danilenko, {\em Endomorphisms of measured equivalence relations, cocycles with values in non-locally compact groups and applications},
Ergodic Theory \& Dynam. Systems, {\bf 19} (1999), 571--590.

\bibitem{Da-Le}A. I.\ Danilenko, M.\ Lema\'nczyk, {\em A class of multipliers for $W^\perp$}, Probability in mathematics, Israel J. Math. 148 (2005), 137–168.


\bibitem{Do-Se} T. Downarowicz, J.\ Serafin, {\em Almost full entropy subshifts uncorrelated to the M\"obius function}, Int. Math. Res. Not. IMRN 2019, no. 11, 3459–3472.


\bibitem{Do-We}T.\ Downarowicz, B.\ Weiss, {\em Pure strictly uniform models of non-ergodic measure automorphisms}, Discrete Contin. Dyn. Syst. 42 (2022), 863–884.


\bibitem{Gr-Ta-Zi}  B. Green, T. Tao, T. Ziegler, An inverse theorem for the Gowers U s`1rN s-
norm, Ann. of Math. (2) 176 (2012), 1231–1372.

\bibitem{Ed}N. Edeko, {\em On the isomorphism problem for non-minimal transformations with discrete spectrum}, Discrete Continuous Dynam. Systems 39 (2019), 6001–6021.

\bibitem{EiWa}
M. Einsiedler, T. Ward,
{\em
Ergodic theory with a view towards number theory},
Grad. Texts in Math., 259
Springer-Verlag London, Ltd., London, 2011.



\bibitem{deFa} A.\ de Faveri, {\em M\"obius disjointness for $C^{1+\epsilon}$ skew products}, Int. Math. Res. Not. IMRN 2022, no. 4, 2513–2531.



\bibitem{Fe-Ku-Le}S.\ Ferenczi, J.\ Ku\l aga-Przymus, M.\ Lema\'nczyk, {\em Sarnak's Conjecture -- what's new}, in:
Ergodic Theory and Dynamical Systems in their Interactions with Arithmetics and Combinatorics,  CIRM Jean-Morlet Chair, Fall 2016,
Editors: S. Ferenczi, J. Ku\l aga-Przymus, M. Lema\'nczyk,
    Lecture Notes in Mathematics {\bf 2213},  Springer International Publishing, pp. 418.

\bibitem{Fo-Ru-We}M. Foreman, D. J. Rudolph, B. Weiss,
{\em
 The conjugacy problem in ergodic theory}, Ann. Math., 173 (2011)
 1529--1586.


\bibitem{Fr}N.\ Frantzikinakis, {\em Ergodicity of the Liouville system implies the Chowla conjecture},
Discrete Analysis 2017, paper no 8, pp. 23.

\bibitem{Fr-Ho}N.\ Frantzikinakis, B.\ Host, {\em The logarithmic Sarnak conjecture for ergodic weights},
   Annals Math.\ (2) {\bf 187} (2018), 869--931.

\bibitem{Fr-Le-Ru}  N.\ Frantzikinakis, M.\ Lema\'nczyk, T. de la Rue,  {\em Furstenberg systems of pretentious and MRT multiplicative functions},  Ergodic Theory Dynam. Systems  45 (2025),  2765--2844.

\bibitem{Fr-Ku-Le} K. Fr\c{a}czek, J.\ Ku\l aga, M. Lema\'nczyk, {\em On the self-similarity problem for Gaussian-Kronecker flows},  Proc. Amer.\ Math.\ Soc. {\bf 141} (2013), 4275-4291.
\bibitem{Fr-Le}K.\ Fr\c{a}czek, M.\ Lema\'nczyk, {\em A note on quasi-similarity of Koopman operators}, Journal
London Math.\ Soc. (2) {\bf 82} (2010), 361-375.


\bibitem{Fu0}H. Furstenberg, {\em Ergodic behaviour of diagonal measures and a theorem of
Szemeredy on arithmetic progressions}, J.\ Anal.\ Math.\ 31 (1977), 204–256.
\bibitem{Fu} H. Furstenberg,  {\em Recurrence in  Ergodic Theory and Combinatorial Number
Theory}, Princeton: Univ Press, 1981.
\bibitem{Gl} E. Glasner, {\em Ergodic Theory via Joinings},  American Mathematical Society, Providence, 2003.

\bibitem{Go-Le-Ru} A. Gomilko, M. Lema\'nczyk, T. de la Rue, {\em
On Furstenberg systems of aperiodic multiplicative functions of Matom\"aki, Radziwi\l\l \ and Tao}, J.\ Modern Dynamics {\bf 17} (2021), 529--555. doi: 10.3934/jmd.2021018

\bibitem{Go-Le-Ru1} A. Gomilko, M. Lema\'nczyk and T. de~la~Rue, {\em M\"obius orthogonality in density for zero entropy dynamical systems}, Pure Appl. Funct. Anal. {\bf 5} (2020), no.~6, 1357--1376.

\bibitem{Go-Le-Ru}M.\ G\'orska, M.\ Lema\'nczyk, T.\ de la Rue, {\em On
    orthogonality to uniquely ergodic systems}, to appear in J. d'Analyse Math., arXiv:2404.07907




\bibitem{Ho-Kr} B. Host, B. Kra, {\em Nilpotent structures in ergodic theory}, Mathematical Surveys and
Monographs, 236. American Mathematical Society, Providence, RI, 2018.


\bibitem{Ho-Me-Pa}B. Host, F.\ M\'ela, F.\ Parreau, {\em  Non singular transformations and spectral analysis of measures},
Bulletin de la Soci\'et\'e Math\'ematique de France 119 (1991), 33-90.

\bibitem{Hu-Ta-Xu}W.\ Huang, M.\ Tan, L.\ Xu, {\em Almost countable spectrum and logarithmic Sarnak conjecture}, arXiv 2511.04419


\bibitem{Hu-Wa-Ye}W.\ Huang, Z.\ Wang, X.\ Ye, {\em Measure complexity and Möbius disjointness}, Advances
Math. 347 (2019), 827–858.

\bibitem{Hu-Wa-Zh} W.\ Huang, Z.\ Wang, G.\ Zhang, {\em M\"obius disjointness for topological models of
ergodic systems with discrete spectrum}, J.\ Modern Dyn.\ 14 (2019), 277–290.

\bibitem{Ju-Le-Me}A.\ del Junco, M.\ Lema\'nczyk, M. K. Mentzen, {\em  Semisimplicity, joinings and group extensions}, Studia Math. 112 (1995), 141--164.

\bibitem{Ju-Ru1}
A.\ del Junco, D.J.\ Rudolph, {\em A rank-one, rigid, simple, prime map}, Ergodic Theory Dynam. Systems 7 (1987),  229–247.


\bibitem{Ju-Ru}A.\
del Junco, D.J.\ Rudolph, {\em On ergodic actions whose self-joinings are graphs}, Ergodic Theory Dynam. Systems
7 (1987), 537-551.


\bibitem{Kal-Ma} R.R.\ Kallman, R.D.\ Mauldin,
{\em A cross section theorem and an application to $C^*$-algebras},
Proc. Amer. Math. Soc. {\bf 69} (1978),  57--61.


\bibitem{Ka-Ku-Le-Ru}  A.\ Kanigowski, J.\ Ku\l aga-Przymus, M.\ Lema\'nczyk, T.  de la Rue, {\em On arithmetic functions orthogonal to deterministic sequences}, Advances Math. {\bf 428} (2023), https://doi.org/10.1016/j.aim.2023.109138

\bibitem{Ka-Le-Ra}   A.\ Kanigowski, M.\ Lema\'nczyk, M. Radziwi\l\l, {\em Rigidity in dynamics and M\"obius disjointness}, Fundamenta Math.\ {\bf 255} (2021), 309-336.
\bibitem{Ka-Le-Ri-Te} A.\ Kanigowski, M.\ Lema\'nczyk, F.\ Richter, J. Ter\"av\"ainen,
{\em  On the local Fourier uniformity problem for small sets}, International Math. Research Notices  IMRN 2024, no 15, 11488-11512.

\bibitem{Ke} A.S.\ Kechris, {\em Classical Descriptive Set Theory}, Springer, 1995.

\bibitem{Ku-Le}J.\ Ku\l aga-Przymus, M.\ Lema\'nczyk, {\em The M\"obius function and continuous extensions of rotations}, Monatsh. Math. 178 (2015),  553–582.

\bibitem{Kw}J.\ Kwiatkowski, {\em Isomorphism of regular Morse dynamical systems}, Studia Math. {\bf 77} (1982), 59-89.

\bibitem{Le-Me} M.\ Lema\'nczyk, M.K.\ Mentzen,
{\em Compact
subgroups in the centralizers of natural factors of an ergodic
group extension of a rotation determine all factors}, Ergodic
Theory Dynam. Systems {\bf 10} (1990), 763-776.
\bibitem{Le-Pa} M.\ Lema\'nczyk, F.\ Parreau,
{\em Rokhlin extensions and lifting disjointness}, Ergodic Theory Dynam. Systems 23 (2003), 1525-1550.
\bibitem{Le-Pa1}  M.\ Lema\'nczyk, F.\ Parreau, {\em Lifting mixing properties by Rokhlin cocycles} Ergodic Theory Dynam. Systems 32 (2012), 763–784.
\bibitem{Le-Pa-Th} M. Lema\'nczyk, F. Parreau, J.–P. Thouvenot, {\em Gaussian automorphisms whose ergodic
self–joinings are Gaussian}, Fundamenta Math.\ 164 (2000), 253-293.

\bibitem{Le-Th-We}  M. Lema\'nczyk, J.–P. Thouvenot, B. Weiss, {\em
Relative discrete spectrum and joinings}, Monatsh.\ Math.\ {\bf
137} (2002), 57-75.
\bibitem{Li-Sa} J.\ Liu, P.\ Sarnak, {\em The M\"obius function and distal flows}, Duke Math. J.\ 164 (2015),  1353–1399.


\bibitem{Mac} G.\ Mackey, {\em Point realizations of transformation groups}, Illinois J. Math.,\ 6 (1962), 327-335.


\bibitem{Mad} B.\ Madore, {\em Rank-one group actions with simple mixing $\Bbb Z$-subactions},
New York J. Math. {\bf 10} (2004) 175–194.

\bibitem{Ma-Ra} K.\ Matom\"aki, M.\ Radziwi\l\l, {\em Multiplicative functions in short intervals}, Annals of Mathematics 183 (2016), 1015–1056.

\bibitem{Ma-Ra-Ta} K.\ Matom\"aki, M.\ Radziwi\l\l, T.\ Tao,
{\em  An averaged form of Chowla’s conjecture},  Algebra Number Theory {\bf 9} (2015),  2167–2196.

\bibitem{Na} M.G.\ Nadkarni, {\em Spectral Theory of Dynamical Systems}, Birhauser Advanced Texts 1998.

  \bibitem{Rud}
 D.J.\ Rudolph, {\em An example of a measure preserving map with minimal self-joinings, and applications}, J. Analyse Math. 35 (1979), 97–122.


\bibitem{Ru} T.\ de la Rue, {\em Notes on Austin’s multiple ergodic theorem}, \\
\noindent \url{http://hal.archives-ouvertes.fr/hal-00400975}.



\bibitem{Sa} P.\ Sarnak,  {\em Three lectures on the M\"obius function, randomness and
dynamics},
https://publications.ias.edu/sites/default/files/
MobiusFunctionsLectures(2).pdf.

\bibitem{Su} C.\ Sutherland, {\em A Borel parametrization of Polish groups}, Publ. RIMS., 21 (1985) 1067–1086.

\bibitem{Sch}K.\ Schmidt, {\em Cocycles of Ergodic Transformation Groups}, Mac Millian, Delhi, 1977.



\bibitem{Se} J.\ Serafin, {\em Non-existence of a universal zero-entropy system}, Israel J. Math. 194 (2013),  349–358.

\bibitem{Ta00} T.\ Tao, {\em The logarithmically averaged Chowla and Elliott conjectures for two-point correlations}, Forum Math.\ Pi 4 (2016), e6, 36 pp.
\bibitem{Ta1} T.\ Tao, {\em Equivalence of the logarithmically averaged Chowla and Sarnak
conjectures}, Number Theory – Diophantine Problems, Uniform Distribution and Applications: Festschrift in Honour of Robert F. Tichy’s 60th
Birthday (C.\ Elsholtz and P.\ Grabner, eds.), Springer International Publishing, Cham, 2017, pp.\ 391–421.

\bibitem{Ve} W. A.\ Veech, {\em A criterion for a process to be prime},
Monatsh. Math. 94 (1982), no. 4, 335--341.

\bibitem{Ver} A. M.\ Vershik, {\em Polymorphisms, Markov processes, and quasi-similarity}, Discrete and Continuous Dynam. Systems 13 (2005), 1305-1324.

\bibitem{Wa} Z.\ Wang, {\em M\"obius disjointness for analytic skew products}, Inventiones Mathematicae
209 (2017), 175–196.

\bibitem{Zi} R. Zimmer, {\em Extensions of ergodic group actions}, Illinois J.\ Math.\ 20 (1976), 373–409.
\end{thebibliography}
\end{document}